\documentclass[12pt]{amsart} 
\usepackage{amsmath,amsthm,amssymb,amsfonts,amscd}

\theoremstyle{plain}
\newtheorem{thm}{Theorem}[section]
\newtheorem{lemma}[thm]{Lemma}

\newtheorem{pro}[thm]{Proposition}

\newtheorem{definition}[thm]{Definition}

\theoremstyle{remark}

\newcommand{\sz}{\scriptsize}
\newcommand{\fz}{\footnotesize}
\newcommand{\nz}{\normalsize}
\newcommand{\ti}{\tilde}
\newcommand{\al}{\alpha}
\newcommand{\be}{\beta}
\newcommand{\ga}{\gamma}
\newcommand{\GA}{\Gamma}
\newcommand{\dd}{\delta}
\newcommand{\DD}{\Delta}

\newcommand{\f}{\varphi}
\newcommand{\h}{\phi}

\newcommand{\lam}{\lambda}
\newcommand{\ra}{{\rightarrow}}
\newcommand{\lra}{{\longrightarrow}}

\newcommand{\Nn}{{\mathcal N}}

\newcommand{\EE}{{\mathbb E}}

\newcommand{\RR}{{\mathbb R}}

\newcommand{\ZZ}{{\mathbb Z}}

\numberwithin{equation}{section}

\begin{document}
\title[\ Real Bott Tower]
{Real Bott Tower }
\author{ADMI NAZRA}
\address{Department of Mathematics, Tokyo Metropolitan University.}
\email{nazra-admi@ed.tmu.ac.jp}
%\subject{5}
\keywords{Crystallographic group, Bieberbach group, Flat Riemannian Manifold}
\subjclass{}
%\date{\today}
\begin{abstract}
In this paper we proved that there exists four distinct diffeomorphism classes of three dimensional real Bott tower 
$M(A)=(S^1)^3/(\mathbb{Z}_2)^3$,
and 12 distinct diffeomorphism classes of four dimensional real Bott tower 
$M(A)=(S^1)^4/(\mathbb{Z}_2)^4$, where matrix $A$
corresponds to the action of $(\mathbb{Z}_2)^n$ on $(S^1)^n$ for n=3,4.
\end{abstract}
\maketitle

\tableofcontents
\section{Introduction}

In this paper we study free actions of the product of cyclic groups 
of order 2, $(\mathbb{Z}_2)^n$ on the n-dimensional Torus $T^n=(S^1)^n$ for n=3,4.
The action of $(\mathbb{Z}_2)^n$ can be expressed by an upper triangular matrix $A$ where 
the diagonal entries are one and the other entries are either one or zero. 
We call such a matrix $n$-th {\em Bott matrix} and 
the $n$-dimensional orbit space $M(A)=T^n/(\mathbb{Z}_2)^n$
is said to be a {\bf real Bott tower}.
By the definition, Bott matrices
consist of $8$ for $n=3$ and $64$ for $n=4$.

The main purpose of this paper is to determine the diffeomorphism classes of 
real Bott towers for dimension $n=3$ and $4$.
The original conjecture is whether {\em $A$ is conjugate to $A'$  if and only if$M(A)$ and $M(A')$ 
are diffeomorphic.}
%\par\noindent
However we found that 
%\begin{enumerate}
%\item {\em
there are non-diffeomorphic Bott towers $M(A), M(A')$ with
representative matrices $A,A'$ are conjugate.
%}
%\item {\em
%There exist representative matrices $A$ and $A'$ which are conjugate but
% $M(A)$ and $M(A')$ are diffeomorphic.}
%\end{enumerate}
So the necessary condition of the conjecture is not true. But   
the sufficient condition of the conjecture is not known so far.
On the other hand the real Bott tower $M(A)$ turns out to be an
$n$-dimensional compact euclidean space form (Riemannian flat manifold).
Then we can apply the Bieberbach theorems to classify Bott towers.

If $M$ is a compact flat Riemannian manifold
of dimension n, then $M$ is affinely diffeomorphic to $\mathbb{R}^n/\Gamma $, where $\Gamma $, 
the fundamental group of $M$, is a torsion-free discrete
uniform subgroup of the Euclidean group $\mathbb {E}(n)=\mathbb {R}^n\rtimes O(n)$. 
Such groups $\Gamma$ are called Bieberbach groups.

\section{Preliminaries}

Let us start with some definitions and theorems used throughout this paper.

\subsection{Groups Acting on Spaces}

\begin{definition}
Suppose that $X$ is a topological space and $G$ is a group. Then we say that X is a $G\hspace{-0.15cm}-\hspace{-0.15cm}space$ if $G$
 acts on $X$ such that the correspondence $G \times X\longrightarrow X$, $(g,x)\longrightarrow gx$, is continuous.
\end{definition}
%Certain $G\hspace{-0.15cm}-\hspace{-0.15cm}space$ lead to covering spaces. 
%Suppose that $X$ is a $G\hspace{-0.15cm}-\hspace{-0.15cm}space$.We say that the action of $G$ on $X$ is \textit {even} if for each $x\in X$
%there is an open neighbourhood $V$ of $x$ such that $gV\cap g'V=\emptyset$ for all $g,g'\in G$
%with $g\ne g'$. Notice that if the action is \textit {even} then $gx\ne x$ for all
%$g\in G$, $g\ne 1$ and all $x\in X$, because if $x\in V$ then $gx\in gV$.

\begin{thm}\label{T:1} %\cite{K80}
If $G$ is a finite group acting freely on a Hausdorff space $X$ then $p: X \rightarrow X/G$ is a covering.
\end{thm}

\begin{proof}
Let $G=\{ 1=g_0,g_1,g_2,...,g_n\}$. Since $X$ is Hausdorff there are open neighborhoods $U_0,U_1,...,U_n$
of $g_0x,g_1x,g_2x,...,g_nx$ respectively with $U_0\cap U_j=\emptyset $ for $j=1,2,...,n$.
Let $U$ be the intersection $\cap_{j=0}^{n}\,g^{-1}_jU_j$,
 which is clearly an open neighborhood of $x$
$(X$ is topological space). 

Now $g_iU=\cap_{j=0}^{n}\, g_i(g^{-1}_jU_j)\subseteq  U_i$ and
\begin{align*}
g_iU\cap g_jU&=g_j((g^{-1}_jg_iU)\cap  U)\\
               &=g_j(g_kU\cap  U)  \,\,\,\,\,\,\ \text {(for some k)}\\
		   &=\emptyset
\end{align*}
since $g_kU\subseteq U_k$ and $U\subseteq U_0$. 
%Hence the action of $G$ on $X$ is \textit {even}.
%\end{proof}

%\begin{thm}\label{T:2}
%\cite{K80} 
%Let $X$ be a $G$-space. If the action of $G$ on $X$ is even
%then $p: X \rightarrow X/G$ is a covering.
%\end{thm}
%\begin{proof}
Next note that $p:X\longrightarrow X/G$ is a continuous surjective open map. 
%Let $U$ be an open neighborhood of $x\in X$ satisfying the condition 
%of \textit {evenity}.
Since $p$ is an open map, $p(U)$ is an open neighborhood
of $Gx=p(x)$ and $p^{-1}(p(U))={\cup}_{g\in G} \,gU$ with $\{gU;g\in G\}$ being
disjoint open subsets of $X$. Furthermore $p|gU:gU\rightarrow p(U)$
is a continuous open bijective mapping and hence a homeomorphism.
\end{proof}
\newpage

We say $G$ is \textit {discontinuous} at $x\in X$ if, given sequence $\{\gamma_i\}\subset G $
of distinct elements, the sequence $\{\gamma_i x\}\subset X$ has no accumulation point.
$G$ is \textit {discontinuous} on $X$ if it is discontinuous at every point of $X$.
The action of $G$ is \textit {properly discontinuous} on $X$ if for every compact 
subset $K$ of $X$ the set of $\gamma\in G$ such that $\gamma K\cap K\ne \emptyset $ is finite.
If $\{\ga \in G|\gamma K\cap K\ne \emptyset \}$ is compact, then the action of $G$
is called \textit {proper} action.

\begin{lemma}\label{L1}
If $G$ acts properly discontinuously on $X$ then $G$ is discontinuous on $X$.
\end{lemma}
\begin{proof}
Suppose $G$ is not discontinuous on $X$. That is, given a sequence $\{\gamma_i\}$
of distinct elements of $G$ such that  $\{\gamma_i x: x\in X\}$ convergence to $y\in X$.
Then $K=\{x,y,\gamma_1x,\gamma_2x,....\}$ is compact subset of $X$. As $\gamma_ix\in K\cap \gamma_iK$
for each $i$, we have a contradiction.
\end{proof}

Let $G$ be a Hausdorff topological group. A \textit {discrete subgroup} is
a group with a countable number of elements and the discrete topology
(every point is an open set).
%If $H$ is a closed subgroup of $X$,
%then the coset space $X/H=\{xH:x\in X\}$ has quotient topology for
%the projection $p:X\rightarrow X/H$. This means that $p(x)=xH$ and open 
%subsets of $X/H$ are the $p(U)=\{xH:x\in U\}$ for all open sets $U\subset X$;
%$p$ is continuous and open. A closed subgroup $H\subset X$ is called \textit {uniform}
%if $X/H$ is compact. Let $H$ and $\Gamma$ be subgroups of $X$ with $H$ closed.
%Then $\Gamma$  acts on $X/H$ by $\gamma:xH\rightarrow (\gamma x)H$.

\begin{lemma}\label{L2}
Let $\Gamma$ and $K$ be subgroups of $G$ with $K$ compact and $G$ locally
compact.
%\begin{enumerate}
%\item 
These conditions are equivalent:
      \begin{itemize}
      \item[(i)]  $\Gamma$ is discontinuous at some point of $G/K$,
      \item[(ii)] $\Gamma$ is discontinuous on $G/K$,
      \item[(iii)] $\Gamma$ is properly discontinuous on $G/K$,
      \item[(iv)] $\Gamma$ is discrete in $G$.
      \end{itemize}
%\item 
%\end{enumerate}
\end{lemma}
\begin{proof}$((i)\Rightarrow (iv))$.
Let $\Gamma$ be discontinuous at $y=gK\in G/K$, for some $g \in G$.
Then $y$ has a neighborhood $UK$, $U$ open in $G$ such that
$\{\gamma\in \Gamma|\gamma(y)\in UK \}$ is finite.
Now $Ug^{-1}$ is a neighborhood of $1$ in $G$ which
has finite intersection with $\Gamma$.
For this, let $Ug^{-1}\cap \GA=\{\ga_i=u_i g^{-1}|u_i\in U, \ga_i\in \GA \}$. 
Then $\ga_i g =u_i g^{-1}g=u_i$, $\ga_i gK=u_iK\in UK$, so $\ga_i y=u_iK\in UK$.
Thus  $Ug^{-1}\cap \GA$ is finite.
%, because $1=xx^{-1},$ $U$ is a neighborhood of $x$,
%so $1\in Ux^{-1}$.  
% and $\Gamma$ is discontinuous at $\bar 1\in X/K$.
As $G$ is Hausdorff, a smaller neighborhood of $1$ meets $\Gamma$ only in $1$.
Let $Ug^{-1}=U'$.
Then for any $\gamma\in \Gamma$, $V=\gamma U'$ contains only the element $\gamma$
of $\Gamma$.
%It is even possible to choose then a neighborhood $W$ of $1$ such that
%$WW^{-1}\subset U'$, in which case $\gamma_1W\cap \gamma_2W=\emptyset $
%unless $\gamma_1=\gamma_2$. For if $\gamma_1w_1=\gamma_2w_2$ is the intersection, then
%$\gamma_2^{-1}\gamma_1=w_2w_1^{-1}\in W\subset U'$ which implies
%$\gamma_2^{-1}\gamma_1=1$ or $\gamma_2=\gamma_1$. Then $\Gamma$ must be countable since 
%$\{\gamma W|\gamma \in \Gamma\}$ is a nonintersecting family
%of open sets in one-one correspondence with $\Gamma$.
Thus $\Gamma$ is discrete in $G$.\\
$((iv)\Rightarrow (iii))$.
Let $\Gamma$ be discrete. Since $G$ locally compact, 
we choose a neighborhood $U$ of $1$ in $G$
with compact closure ($\bar U$ is compact). If $g\in G$, put $V_g=gUKU^{-1}g^{-1}$
has compact closure (i.e. $\bar V_g$ is compact). For if $\gamma\in V_g$, $\gamma g=g \Rightarrow \gamma gU=gUK $,
so $\forall \gamma\in V_g$, $\gamma\in gUKU^{-1}g^{-1}$.
Now $\gamma(p(gU))\cap p(gU)\ne \emptyset \Leftrightarrow \gamma \in V_g$. 
Note that $\GA \cap \bar V_g \subset \bar V_g.$ Since $\Gamma$ is discrete, then $\GA \cap \bar V_g$
should be closed. We know that a closed subset of compact set is compact. Therefore
$\GA \cap \bar V_g$ is compact. Hence $\GA \cap \bar V_g$ is finite.
So, for any compact subset $C=p(g \bar U)\subset G/K$, $\{\ga \in\GA| \ga C\cap C\not=\emptyset \}$
is finite. Thus $\Gamma$ is properly discontinuous.\\
$((iii)\Rightarrow (ii))$ by Lemma \ref {L1}, and $((ii)\Rightarrow (i))$ trivially.
\end{proof}

\subsection {Group Extension and 2-cocycle}

A group extension is a short exact sequence 
\[
\begin{CD}
1 @>>> \DD @>>> \pi @>p>> Q @>>> 1
\end{CD}
\]
where $\DD$ is a normal subgroup of $\pi$. There is a homomorphism 
\[
\mu:\pi \lra Aut(\DD)\,\,\,\,\, \text {by} \,\,g\mapsto \mu_g;\,\,\mu_g(n)=gng^{-1}\,\,(g\in \pi,\, n\in \DD).
\]
It is easily checked that 
\begin{equation}\label{Ps0}
\mu_{gh}(n)=\mu_g\mu_h(n). 
\end{equation}
Choose a section 
%\begin{equation}\label{Ps1}
$s:Q\lra \pi$ such that $p\circ s=id$ and $s(1)=1$.
%\end{equation}
Consider the map $\h:Q\lra Aut(\DD)$ given by $\h(\al)=\mu_{s(\al)}$ $(\al \in Q).$
Then suppose
\[
\begin{CD}
1 @>>> Inn(\DD) @>>> Aut(\DD) @>\nu >> Out(\DD) @>>> 1
\end{CD}
\]
is the natural exact sequence.
Then $\h$ induces a homomorphism
\begin{equation}\label{Ps2}
\f:Q\lra Out(\DD)\,\,\ \al \mapsto \nu\circ \mu_{s(\al)}.
\end{equation}
Since $s(\al \be)$, $s(\al)s(\be)$ are mapped to $\al \be \in Q$, i.e 
\begin{align*}
ps(\al \be)&=\al \be=ps(\al)ps(\be)\\
       &=p(s(\al)s(\be)),
\end{align*}
so there exist an element $f:Q\times Q \ra \DD$,  $f(\al,\be)\in \DD$ such that
\begin{align*}
f(\al,\be)s(\al\be)=s(\al)s(\be).
\end{align*}
The function $f$ 
satisfies the following conditions.
\begin{itemize}
%\begin{equation}
\item[(i)] $\phi(\al)(\phi(\be)(n))=f(\al,\be)\phi(\al\be)(n)f(\al,\be)^{-1}$
\item[(ii)] $f(\al,1)=f(1,\al)=1$
\item[(iii)] $\phi(\al)(f(\be,\ga))f(\al,\be\ga)=f(\al,\be)f(\al\be,\ga)$.
%\end{equation}
\end{itemize}

Conversely, given a function $\h:Q\ra Aut(\DD)$ and a function $f:Q\times Q\ra \DD$
satisfying (i)-(iii), we define a group law in the product $\DD \times Q:$
\begin{equation}\label{cycle1}
(n,\al)(m,\be)=(n\phi(\al)(m)f(\al,\be), \al\be).
\end{equation} 
By this product \eqref{cycle1},
we have a group $\pi=\Delta\times Q$.
Setting $p((n,\al))=\al$,
we obtain a group extension
$1\ra \Delta \ra \pi \stackrel {p}\ra Q\ra 1$.

Next let us consider
\begin{align*}
f(\al,\be)s(\al \be) n s(\al \be)^{-1}f(\al,\be)^{-1}&=s(\al)s(\be)n s(\al \be)^{-1}f(\al,\be)^{-1}\\
\mu_{f(\al,\be)s(\al \be)}(n)                  &=s(\al)s(\be)n s(\be)^{-1}s(\al)^{-1}\\
\mu_{f(\al,\be)} \mu_{s(\al \be)}(n)           &=\mu_{s(\al)s(\be)}(n)\,\,\,\,\,\,\,\,\text {by}\,\, \eqref{Ps0}\\
					           &=\mu_{s(\al)}\mu_{s(\be)}(n).
\end{align*}
Since $\nu(Inn(\DD))=1$, and let $a'=f(\al,\be)\in \DD$, $\mu_{f(\al,\be)}(n)=\mu_{a'}(n)=a'na'^{-1}\in Inn(\DD),$
then for any $n\in \DD$,
$\nu(\mu_{f(\al,\be)})=1$. So
\begin{align*}
\nu(\mu_{f(\al,\be)} \mu_{s(\al \be)})          &=\nu(\mu_{s(\al)}\mu_{s(\be)})\\
\nu(\mu_{f(\al,\be)}) \nu( \mu_{s(\al \be)})    &=\nu(\mu_{s(\al)}) \nu(\mu_{s(\be)})\\
                   \nu( \mu_{s(\al \be)})    &=\nu(\mu_{s(\al)}) \nu(\mu_{s(\be)})\\
                   \f(\al \be)               &=\f(\al) \f(\be).
\end{align*}
Hence $\f$ is homomorphism. We call $\f$ an \textit {abstract kernel}.

Next, let us consider the case where $\Delta=A$ is abelian. In other words, we try to classify
all group extensions of an abelian group $A$ by $Q$ with abstract kernel $\f:Q\lra Out(A)$.
Since $A$ is abelian, $Aut(A)=Out(A)$, so $\f=\h:Q\lra Aut(A)$. Moreover the condition function $f$ above, (i) holds automatically. 
A map $f:Q\times Q \lra A$ satisfying the two equalities (ii) and (iii) above is called a \textit {2-cocycle}
for the abstract kernel $\h$. The set of all 2-cocycles is denoted by $Z_{\h}^2(Q;A)$ or simply, by $Z^2(Q;A)$.

Let $\lam:Q\ra A$ be any map such that $\lam (1)=1$. Define $\dd' \lam:Q\times Q \ra A$ by
\begin{equation}
\dd' \lam (\al, \be)=\h(\al)(\lam (\be))\lam (\al) \lam (\al \be)^{-1}.
\end{equation}
It turns out that such a $\dd' \lam$ is a 2-cocycle and is called a \textit{2-coboundary}.
The set of all  2-coboundaries is denoted by $B_{\h}^2(Q;A)$. Clearly, $B_{\h}^2(Q;A)$ is 
a subgroup of $Z_{\h}^2(Q;A)$. Let $f_1, f_2: Q\times Q \ra A$ be 2-cocyles. We say 
$f_1$ is \textit{cohomologous} to $f_2$ if there is a map $\lam$ such that
$f_1(\al, \be)=\dd'\lam (\al, \be) f_2(\al, \be)$ $(\forall \al, \be \in Q)$.

We define the second cohomology group as the quotient group
\[
H_{\h}^2(Q;A)=Z_{\h}^2(Q;A)/B_{\h}^2(Q;A).
\]

From hypothesis 1 below, when $\Nn=\RR^n$, $H^2_{\bar\phi}(Q;\mathcal N)=0$
for any finite group $Q$.
Then there is a function
$\chi:Q\ra \mathcal N$ ($1$-chain) such that
$f=\delta'\chi$;
\begin{equation}\label{coboun}
f(\al,\be)=\bar\phi(\al)(\chi(\be))\chi(\al)\chi(\al\be)^{-1}
\  (\al,\be\in Q).
\end{equation}
Now let's check the property of $\bar\phi$.
As $\phi(\al)(n)=s(\al)ns(\al)^{-1}$, we see that
\begin{equation*}\begin{split}
\phi(\al)(\phi(\be)(n))&=s(\al)(s(\be)ns(\be)^{-1})s(\al)^{-1}\\
&=f(\al,\be)s(\al\be)ns(\al\be)^{-1}f(\al,\be)^{-1}\\
&=f(\al,\be)\phi(\al\be)(n)f(\al, \be)^{-1} 
\end{split}\end{equation*}
By the uniqueness of $\bar\phi$,
the following holds.

\begin{equation}\label{rel1}
\bar\phi(\al)(\bar\phi(\be)(x))=
f(\al,\be)\bar\phi(\al\be)(x)f(\al\be)^{-1} \ (x\in \mathcal N).
\end{equation}
\vspace{.2cm}

Now, we give three {\bf hypotheses} as follows:
%\noindent
\begin{itemize}
\item[{\bf 1.}]
Let $\Nn$ be connected Lie group contains $\Delta$ as a discrete uniform subgroup. 
$(\Delta, \Nn)$ has the unique extension property; If $\h(\al):\Delta\lra \Delta$ is 
an automorphism, then $\h$ extends to a unique automorphism $\bar \h(\al):\Nn \lra \Nn$.
\item[{\bf 2.}]
$\Nn$ acts properly and freely on $X$ such that there is a principal bundle
\[
\Nn \lra X\lra W \,\,\, \text{with}\,\, W=X/\Nn \,\, \text {is manifold}.
\]
%\noindent
\item[{\bf 3.}]
Let $(\pi,X)$ be a properly discontinuous action. 
\end{itemize}
\vspace{.3cm}

By hypothesis 1, we can consider the \textit {pushout} $\pi\Nn=\{(n,\ga)|n\in \Nn, \ga \in Q\}$;
\begin{equation}\label{it1}
\begin{CD}
1 @>>> \Delta @>>> \pi @>>> Q @>>> 1\\
@.     @VVV        @VVV      {||} @.\\
1@>>> \Nn @>>>\pi\Nn @>>> Q @>>> 1,
\end{CD}
\end{equation}
particularly, note that
\begin{equation}\label{it2}
 \pi \cap \Nn =\Delta,
\end{equation}
and
\begin{equation}\label{it3}
 \pi \,\,\text{normalizes}\,\, \Nn, 
\end{equation}
since for any $(x,\ga)\in \pi$ and $(n,1)\in \Nn$, $(x,\ga)(n,1)(x,\ga)^{-1}=$\\
$(x \bar \h (\ga)(n)x^{-1},1)\in \Nn.$

\begin{thm}\label{fibration}
There is a proper action $(\pi\Nn, X)$ such that the following is equivariant fibration;
\[
(\Nn,\Nn)\lra (\pi\Nn,X) \stackrel {p} \lra (Q,W).
\]
\end{thm}
\begin{proof}
We note that $\Nn$ acts properly and $\pi$ acts properly discontinuous on $X$ (hypothesis 2,3). 
Consider the natural projection $\Nn \lra \pi\Nn/\pi$ which induces a quotient diffeomorphism
$\Nn/\Nn\cap \pi=\pi\Nn/\pi$.
Since $\Nn\cap \pi=\Delta$ by \eqref{it2} and $\Nn/\Delta$ is compact ($\Delta$ is an uniform subgroup of $\Nn$ ),
we can choose a compact subset $K\subset \Nn$ such that $\pi\Nn=\pi K (\subset \pi\Nn)$.

Now we prove $\pi\Nn$ acts properly. 
Let $(n_i,\ga_i)\in \pi\Nn$ acts on $x_i\in X$ where lim $x_i=x\in X\,\,(i\ra \infty )$.
Suppose  $(n_i,\ga_i)x_i\ra y\in X (i\ra \infty )$. As above, $(n_i,\ga_i)=s_ik_i$ $(s_i\in \pi, \,\,k_i\in K)$.
Then $(n_i,\ga_i)x_i=s_i k_i x_i=s_i(k_i x_i)$. As $K$ is compact, we may assume $k_i \ra k\in K\,\,(i\ra \infty )$
so that $k_i x_i \ra kx\,\,(i\ra \infty )$ and as $s_i(k_i x_i)\ra y\,\,(i\ra \infty )$, by properness of $\pi$, $\{s_i\}$ consist of 
finitely many elements, in fact $s_i\ra s\in \pi$ for sufficiently larger $i$.
Hence $(n_i,\ga_i)=s_ik_i \ra sk\in \pi\Nn$.
\end{proof}

\subsection {Rigid Motions}

\begin{definition}
A rigid motion is an ordered pair $(s,M)$ with $M\in O(n)$ and $s\in\mathbb{R}^n$.
A rigid motion acts on $\mathbb{R}^n$ by
\begin{equation}\label{Ri1}
(s,M)x=Mx+s \,\,\,\, for \,\, x\in\mathbb{R}^n.
\end{equation}
\end{definition}

\noindent We multiply rigid motions by composing them, so
\begin{equation}\label{Ri2}
(s,M)(t,N)=(Mt+s,MN).
\end{equation}
We denote the group of rigid motions (of $\mathbb{R}^n$) by $\mathbb{E}(n)$.

\begin{definition}
An affine motion is an ordered pair $(s,M)$ with $M\in GL(n,\mathbb{R})$ and $s\in\mathbb{R}^n$.
\end{definition}
Affine motions act on $\mathbb{R}^n$ via \eqref{Ri1} and form a group, $\mathbb{A}(n)$, under
\eqref{Ri2}.

\begin{definition}
Define a homomorphism $L:\mathbb{E}(n) \rightarrow O(n)$ by
\[
L(s,M)=M.
\] 
If $\alpha\in\mathbb{E}(n)$, $L(\alpha)$ is called the rotational part of $\alpha$.
We can also define the translational part of $\alpha$ by 
\[
T(s,M)=s,
\] 
but the map $T:\mathbb{E}(n)\rightarrow \mathbb{R}^n$ is not a homomorphism.
\end{definition}

We say $(s,M)$ is a {\itshape pure translation} if $M=I,$ the identity matrix.
If $\Gamma $ is a subgroup of $\mathbb{E}(n)$, $\Gamma\cap \mathbb{R}^n$ will denote
the subgroup of $\Gamma$ of pure translations.

Recall that such a group $\Gamma$ is {\itshape torsionfree} if $\Gamma$
has no elements of finite order.

%\begin{definition}
%We say a subgroup of $\mathbb{E}(n)$ (or $\mathbb{A}(n)$) is discontinuous if
%all the orbits of $\Gamma$ are discrete.
%\end{definition}

\begin{definition}
We say that a subgroup $\Gamma$ of $\mathbb{E}(n)$ is \textit {uniform} if $\mathbb{R}^n/\Gamma$ is compact
where the orbit space, $\mathbb{R}^n/\Gamma$, is the set of orbits with the quotient
(or identification) topology. 
We say $\GA$ is \textit {reducible} if we can find $\al \in \mathbb {A(}n)$, so that
$T(\al \GA \al^{-1})$ does not span $\RR^n$. In other words, $\GA$ is reducible if after
an affine change of coordinates, all the elements of $\GA$ have translational parts
in a proper subspace of $\RR^n$. If $\GA$ is not reducible, we say $\GA$ is \textit {irreducible}.
\end{definition}

\begin{pro}\label{P:1}
Let $\Gamma$ be a subgroup of $\mathbb{E}(n)$. Then if $\Gamma$ acts
freely, $\Gamma$ is torsionfree. Furthermore, if $\Gamma$ is discontinuous
and torsionfree, $\Gamma$ acts freely.
\end{pro}
\begin{proof}
Suppose $\alpha\in\Gamma$, $\alpha\ne(0,I)$, and $\alpha^k=(0,I)$.
Let $x$ be arbitrary in $\mathbb{R}^n$, and set
\[
y=\sum_{i=0}^{k-1} \alpha^ix.
\]
Then $\alpha y=y$, so $\Gamma$ does not act freely.

%Now suppose $\exists x\in \mathbb{R}^n$ and $\alpha\in \Gamma,$ $\alpha\ne (0,I)$, s.t
%$\alpha x=x$. Move the origin to $-x$ . More precisely, conjugate $\Gamma$ by $(-x,I)$. Let
%$\Gamma'=(-x,I)\Gamma(x,I)$. Then if $\alpha=(s,M)$, we see that $(-x,I)\alpha(x,I)=(s+(M-I)x,M)\in \Gamma'$.
%Then
%\begin{align*}
%\alpha x &= x\\
%\alpha(x,I) &= (x,I)\\
%\alpha(x,I)0 &= (x,I)0   \,\,\,\,\, \text{(evaluted at $0\in \mathbb{R}^n$)} \\
%(x,I)^{-1}\alpha(x,I)0 &=0\\
%(-x+s+Mx,M)0&=0.
%\end{align*}
%We get $s+(M-I)x=0$, so $(0,M)\in\Gamma'$. Clearly if
%$\Gamma$ is discontinuous, $\Gamma'$ is also. Hence $(0,M)$ is a discrete element. Therefore $M$ cannot be a rotation through
%an irrational angle (i.e. the rotation would be an element of finite order),
% so $\exists k$ s.t $(0,M)^k=(0,M^k)=(0,I)$. 
%Then we have
%\begin{align*}
%(x,I)^{-1}\alpha(x,I) &= (0,M)\\
%((x,I)^{-1}\alpha(x,I))^k &= (0,M)^k=(0,I)\\
%(x,I)^{-1}\alpha^k(x,I) &= (0,I)\\
%\alpha^k &= (0,I),
%\end{align*}
%so $\alpha$ is a torsion element. This yields a contradiction.

If $\GA$ is discontinuous on $\RR^n$, then $\GA$ acts properly discontinuously on $\RR^n$. Then stabilizer
$\GA_x$ is finite $\forall x\in \RR^n$. Hence if $\GA$ is torsionfree, then $\GA_x=\{1\}$ $\forall x\in \RR^n$.
So $\GA$ acts freely on $\RR^n$.
\end{proof}

In this paper we will be concerned with discrete subgroups $\Gamma$ of $\mathbb{E}(n)$
which act freely on $\mathbb{R}^n$.

\begin{definition}
We say that $\Gamma$ is crystallographic if $\Gamma$ is uniform and discrete.
If $\Gamma$ is crystallographic and torsionfree in $\mathbb{E}(n)$, we say
$\Gamma$ is a Bieberbach subgroup of $\mathbb{E}(n)$.
\end{definition}

\begin{thm}\label{T4}
Let $\Gamma$ be a subgroup of the Euclidean group
 $\mathbb {E}(n)=\mathbb {R}^n\rtimes O(n)$.
Then $\Gamma$ is a torsionfree discrete uniform subgroup of $\mathbb {E}(n)$
 if and only if $\Gamma$ acts properly discontinuously and freely
on $\mathbb {R}^n$ with compact quotient.
\end{thm}
\begin{proof}
Let $X=\mathbb {E}(n)=\mathbb {R}^n\rtimes O(n)$. Let $K=O(n)$, compact subgroup.
Then $p:X\longrightarrow X/K=\mathbb {R}^n$  by $p(a, A)=a$.\\
$(\Rightarrow )$. Let $\Gamma$ is a torsionfree discrete uniform subgroup of $\mathbb {E}(n)$.
Then $\Gamma$ is discontinuous by Lemma \ref{L2}. Therefore $\Gamma$ acts freely by Proposition \ref{P:1}.
Also $\Gamma$ acts properly discontinuously by Lemma \ref{L2}. Finally, since $\Gamma$ is uniform
then $\mathbb {R}^n/\Gamma$ is compact.\\
$(\Leftarrow )$. Let $\Gamma$ acts properly discontinuously and freely
on $\mathbb {R}^n$ with compact quotient. By Proposition \ref{P:1}, $\Gamma$ is torsionfree.
Since $\Gamma$ acts properly discontinuously then $\Gamma$ is discrete in $\mathbb {E}(n)$ by Lemma \ref{L2}. 
$\Gamma$ is uniform, because $\mathbb {R}^n/\Gamma$ is compact.
\end{proof}

We have defined the group of rigid motion $\EE(n)$ which  acts on $\RR^n$.
Suppose torsionfree crystallographic subgroup $\GA\subset \EE(n) \subset \mathbb{A}(n)$
acts on $\RR^n$, and by first Bieberbach theorem there is an exact sequence
\[
1\lra \ZZ^n \lra \GA \lra F \lra 1, 
\]
where $F$ is a finite group.

Suppose $\mathcal {C}(\GA)$ the center of $\GA$ is non trivial. If $\al=(s,M)\in \mathcal {C}(\GA)$,
$\al n \al^{-1}=n$ \, $\forall n\in \ZZ^n (\ZZ^n \triangleleft \GA)$. Then 
$(n,I)=(s,M)(n,I)(s,M)^{-1}=(Mn,I)$  $\Rightarrow M=I$. 
Therefore $L(\mathcal {C}(\GA))=I$,  $\mathcal {C}(\GA)\subsetneqq \ZZ^n$. 
Since $\GA$ is torsionfree, then $\mathcal {C}(\GA)\cong \ZZ^k \subset \ZZ^n$, $1\leq k< n$.  
Hence we get
\[
1 \lra \ZZ^k \lra \GA \lra \GA/\ZZ^k =Q \lra 1.
\]
%with  $\ZZ^k=\GA\cap \RR^n$. 

What is the form of $Q$ which acts on $\RR^{(n-k)}\,\,(k<n)?$.
Let 
\begin{align*}
\ga = \fz
\Bigl ( 
\left(\begin{array}{c}
a\\
b
\end{array}\right)
\left(\begin{array}{cc}
A_{k\times k}& B_{k\times (n-k)}\\
C_{(n-k)\times k}& D_{(n-k)\times (n-k)}
\end{array}\right)
\Bigr )\in \GA \,\, \nz \text {and } \,\, 
\fz \\
\left(\begin{array}{c}
x\\
y
\end{array}\right) \fz \in \RR^n=
\left(\begin{array}{c}
\RR^k\\
\RR^{n-k}
\end{array}\right) \text {where $x\in\RR^k, \, y\in \RR^{n-k}$}.
\end{align*}
\nz

Now consider $\h:Q \lra Aut(\ZZ^k)$ by $\h(\al)(x)=\ga x \ga^{-1}$ which is represented by matrix $A\in GL(k,\ZZ)$
such that  $\h(\al)(x)=Ax$ for $\al \in Q$.
Therefore $\h(\al)$ naturally extends to $\bar \h(\al):\RR^k \lra \RR^k$ ($\GA$ normalizes $\RR^k$ also, $\ga \RR^k \ga^{-1}=A\RR^k$).

Then we have for any $x\in \RR^k$, 
\begin{align*}
\Bigl ( \left(\begin{array}{c}
x\\
0
\end{array}\right), I \Bigr )=&
\ga
\Bigl ( \left(\begin{array}{c}
x\\
0
\end{array}\right), I \Bigr )\ga^{-1}= 
\Bigl (
\left(\begin{array}{cc}
A& B\\
C& D
\end{array}\right) \left(\begin{array}{c}
x\\
0
\end{array}\right), I \Bigr )\\
=&
\Bigl (\left(\begin{array}{c}
Ax\\
Cx
\end{array}\right),I \Bigr )\\
\Rightarrow  A=I, \,\, C=0.
\end{align*}

By Theorem \ref{fibration} there is an equivariant fibration
\[
(\ZZ^k,\RR^k)\lra(\GA,\RR^n)\stackrel {(\nu, p)}\lra(Q,W).
\]
For $\al\in Q$, $y\in W=\RR^{n-k}$
\begin{align*}
\al y &=
\al p
\left(\begin{array}{c}
x\\
y
\end{array}\right)
= \nu(\ga)p 
\left(\begin{array}{c}
x\\
y
\end{array}\right)
=
p \Bigl (\ga 
\left(\begin{array}{c}
x\\
y
\end{array}\right) \Bigr )\\
&=p \Bigl (\left(\begin{array}{c}
a+x+By\\
b+Dy
\end{array}\right)\Bigr )=b+Dy\\
 &=(b,D)y
\end{align*}
we get $\al=(b,D)$. 
If 
$\left(\begin{array}{cc}
I& B\\
0& D
\end{array}\right)\in O(n)$ then $D\in O(n-k)$. Thus $Q=\{(b,D)\subset \EE(n-k)\}$.

\subsection{Bieberbach Theorems}

Bieberbach theorems are about discrete, uniform subgroups of $\mathbb{E}(n)$, but for geometric
applications of his theorems (to flat manifolds), we require the subgroups to have an additional
property.

\begin{thm}\label{T:4} %\cite{C85}
(First Bieberbach)
Let $\Gamma $ be a crystallographic  subgroup of $\mathbb {E}(n)$. Then
\begin{itemize}
\item[(i)]  $L(\Gamma )$ is finite, $L:\mathbb {E}(n)\rightarrow O(n).$
\item[(ii)] $\Gamma \cap \mathbb {R}^n$ is a lattice (finitely generated free abelian group) which spans $\mathbb {R}^n$.
(cf. \cite{C85} for proof)
\end{itemize}
\end{thm}

This means that for an n-dimensional crystallographic group $\Gamma$, the translation part
$\Gamma^*=\Gamma\cap \mathbb{R}^n$ of $\Gamma$ is a free abelian group isomorphic with $\mathbb{Z}^n$,
such that the vector space spanned by $\Gamma^*$ is the whole space $\mathbb{R}^n$. If $\Gamma^*=\Gamma$
then $\Gamma$ is torsionfree and the manifold $\mathbb{R}^n/\Gamma$ is an n-dimensional torus.

Recall that a sequence of groups and homomorphisms is \textit {exact} if the image of any of the
homomorphisms is equal to the kernel of the next one. By the first Bieberbach  theorem, if
$\Gamma$ is any crystallographic group, then $\Gamma$ satisfies an exact sequence
\begin{equation}\label{ESq}
0\longrightarrow \Gamma^*\longrightarrow \Gamma\longrightarrow \Phi \longrightarrow 1
\end{equation}
where $\Gamma^*=\Gamma\cap \mathbb{R}^n$ is a lattice of rank $n$, and $\Phi =L(\Gamma)$
is a finite group. We call $\Phi $, the \textit{holonomy group} of $\Gamma$.
We use the "0" at the left of \eqref{ESq} to indicate that we usually write $\Gamma^*$
additively, while the "1" at the right indicates that we usually write $\Phi $ multiplicatively.

\begin{pro} \label{T:5} %\cite{C85}
Let $\Gamma $ be a crystallographic  subgroup of $\mathbb {E}(n)$. Then
$\Gamma \cap \mathbb {R}^n$ is the unique normal, maximal abelian subgroup of $\Gamma $.
\end{pro}
\begin{proof}
%Let $(v,I)$ be any element of $\Gamma\cap \mathbb{R}^n=\Gamma^*$. Since $\alpha(v,I)\alpha^{-1}\in \Gamma\cap \mathbb{R}^n$ for 
%$\alpha\in\Gamma$, $\Gamma^*$ is normal, and it is abelian, because $\Gamma^*\subset \mathbb{R}^n$.
Let $\rho \subset \Gamma$ be a normal abelian subgroup. It suffices to show
$\rho\subset \mathbb{R}^n$, i.e. $L(\rho)=I$. Let $(s,M)\in\rho$. We can assume
Ms=s, because $(s,I)(s,M)(-s,I)=(s+s-Ms,M)=(s,M)$ $\Leftrightarrow Ms=s$.
Let $(v,I)$ be any element of $\Gamma^*$. Then $\rho$ normal implies
\[
(v,I)(s,M)(-v,I)=(v-Mv+s,M)\in\rho.
\]
But $\rho$ abelian $\Rightarrow $ $[(s,M),(v-Mv+s,M)]=(0,I)$ which yields 
\[
Mv-M^2v-v+Mv=0, 
\]
or $(M-I)^2v=0$.\\
Let $X=\{x\in\mathbb{R}^n: Mx=x\}$ and write $\mathbb{R}^n=X\oplus Y$ orthogonally.
Write $v=v_x+v_y$ with $v_x\in X$ and $v_y\in Y$, then
\[
0=(M-I)^2v=(M-I)^2(v_x+v_y)=(M-I)^2v_y.
\]
Now on $Y$, $(M-I)$ is non-singular, so $(M-I)^2v_y=0\Rightarrow v_y=0$. 
$\Rightarrow v=v_x$.
Hence $Mv=v$ $\forall v\in\Gamma^*$. Since $\Gamma^*$ spans $\mathbb{R}^n$, $Mx=x$ $\forall x\in\mathbb{R}^n$,
so $M=I$. Hence $(s,M)=(s,I)\in\mathbb{R}^n\Rightarrow \rho\subset\mathbb{R}^n$. 
\end{proof}

\begin{thm}\label{T:6}
%\cite{C85} 
(Second Bieberbach).
Suppose that $f:\Gamma_1\longrightarrow \Gamma_2$ is an isomorphism
between crystallographic group. Then there exist a diffeomorphism $h:\mathbb{R}^n \longrightarrow \mathbb{R}^n$
such that $h(\gamma x)=f(\gamma)h(x)$. Moreover, such a diffeomorphism can be taken as an affine transformation,
i.e., $h=(a,A)\in \mathbb{A}(n)$ and so $f(\gamma)=h\gamma h^{-1}$ $(\gamma \in \Gamma_1)$.
\end{thm}
\begin{proof}
$(\Rightarrow )$. 
Let $\mathbb{Z}^n$ be a maximal abelian subgroup of $\Gamma_1$ such that 
\[
1\longrightarrow \mathbb{Z}^n\longrightarrow \Gamma_1\longrightarrow Q_1\longrightarrow 1
\]
is a group extension. Since $\mathbb{Z}^n$ is characteristics, f maps isomorphism onto
$\mathbb{Z}^n$ of $\Gamma_2$ so that we have a commutative diagram
\begin{equation}\label{Pe1}
\begin{CD}
1 @>>> \mathbb{Z}^n @>Z_1>> \Gamma_1  @>p_1>>  Q_1   @>>> 1 \\
  @.                @VVgV             @VVfV             @VV\Phi V              \\
1 @>>> \mathbb{Z}^n @>Z_2>> \Gamma_2  @>p_2>> Q_2  @>>> 1 \\
\end{CD}
\end{equation}
where $g=f|_{\mathbb{Z}^n}$, $\Phi $ is an induced isomorphism.
We write $\Gamma_i$ as the set $\mathbb{Z}^n \times Q_i$ with group law
\begin{equation}\label{Pe2}
(n,\alpha)(m,\beta)=(n\h_i(\al)(m)f_i(\al,\be),\al\be) \,\,\,\,\,\,\, (i=1,2).
\end{equation}
Here we write $\ZZ^n$ multiplicatively, and homomorphism 
\[
\h_i:Q_i\lra Aut(\ZZ^n)=GL(n,\ZZ)
\]
is defined as follows; choose a section $s_i:Q_i\lra\GA_i$ so that $p_i\circ s_i=id$.
$\GA_i$ normalizes $\ZZ^n$, so 
\[
\h_i(\al):\ZZ^n\lra\ZZ^n
\]
is defined by
\[
\h_i(\al)(n)=s_i(\al)ns_i(\al)^{-1}.
\]
Then defined $f_i:Q_i\times Q_i\longrightarrow \ZZ^n$ by 
\[
f_i(\al,\be)s_i(\al\be)=s_i(\al)s_i(\be).
\] 
As $p_2f(1,\al)=\Phi p_1(1,\al)=\Phi (\al)$ and $p_2f(n,1)=\Phi p_1(n,1)=\Phi (1)=1$,
we write 
\begin{equation}\label{Pe3}
f(1,\al)=(\lam(\al),\Phi (\al))\\ \,\,\,\,\, ,
f(n,1)=(g(n),1)
\end{equation}
for some function $\lam:Q_1\lra\ZZ^n$.
Then
\begin{align*}
f(n,\al)&=f((n,1)(1,\al))=f(n,1)f(1,\al)\\
        &=(g(n),1)(\lam(\al),\Phi(\al)). 
\end{align*}
Note we choose a normalized 2-cocycle $f_i$, that is
\[
f_i(1,\al)=f_i(\be,1)=1 \,\,\,\,\,(\forall \al,\be\in Q_i).
\]
Therefore
\begin{equation}\label{Pe4}
f(n,\al)=(g(n)\h_i(1)(\lam(\al))f_i(1,\Phi(\al)),\Phi (\al))\\
=(g(n)\lam(\al),\Phi (\al)).
\end{equation}

Next consider the equalities
\begin{align*}
f(1,\al)f(g^{-1}(n),1)&=f(\h_1(\al)(g^{-1}(n)),\al)\\
(\lam(\al),\Phi (\al))(n,1)&=(g(\h_1(\al)(g^{-1}(n)))\lam(\al),\Phi (\al))\\
(\lam(\al)\h_2(\Phi(\al))(n),\Phi (\al))&=(g(\h_1(\al)(g^{-1}(n)))\lam(\al),\Phi (\al))
\end{align*}
imply that
\begin{equation}\label{Pe5}
g\h_1(\al)g^{-1}(n)=\lam(\al)\h_2(\Phi (\al))(n)\lam(\al)^{-1}.
\end{equation}
The equalities
\begin{align*}
f(1,\al)f(1,\be)&=f((1,\al)(1,\be))\\
                &=f(\h_i(\al)(1)f_1(\al,\be),\al\be)\\
		    &=f(f_1(\al,\be),\al\be)\\
(\lam(\al),\Phi (\al))(\lam(\be),\Phi (\be))&=(g(f_1(\al,\be))\lam(\al\be),\Phi (\al\be)) \,\,\,\,\ \text {by}\,\eqref {Pe3},\, \eqref {Pe4}
\end{align*}
and
\begin{align*}
&(\lam(\al),\Phi (\al))(\lam(\be),\Phi (\be))\\
&=(\lam(\al)\h_2(\Phi(\al))(\lam(\be))f_2(\Phi(\al),\Phi(\be)),\Phi (\al)\Phi(\be))\\
&=(\lam(\al)\h_2(\Phi(\al))(\lam(\be))\lam(\al)^{-1}\lam(\al)f_2(\Phi(\al),\Phi(\be)),\Phi (\al)\Phi(\be))\\
&=(g\h_1(\al)g^{-1}(\lam(\be))\lam(\al)f_2(\Phi(\al),\Phi(\be)),\Phi (\al)\Phi(\be))\,\,\,\, \text {by}\,\, \eqref {Pe5},
\end{align*}
then we have
\begin{equation}\label{Pe6}
g(f_1(\al,\be))=g\h_1(\al)g^{-1}(\lam(\be))\lam(\al)f_2(\Phi (\al),\Phi (\be))\lam(\al\be)^{-1}.
\end{equation}

Let us consider a group cohomology $H_\h^2(Q;\ZZ^n)$, where cocycle $[f_i]\in H_{\h_i}^2(Q;\ZZ^n)$
$(i=1,2)$. The exact sequence
\[
\begin{CD}
1 @>>> \ZZ^n @>i>> \RR^n @>j>> T^n @>>> 1
\end{CD}
\]
gives a cohomology exact sequence
\[
\begin{CD}
... @>>> H_\h^1(Q;\RR^n) @>j_*>> H_\h^1(Q;T^n) @>\delta >> H_\h^2(Q;\ZZ^n) \\
@>i_*>> H_\h^2(Q;\RR^n) @>j_*>> H_\h^2(Q;T^n) @>>> ...
\end{CD}
\]
If $Q$ is finite group, and $\RR^n$ is the $Q-$module, then $H_{\h}^i(Q;\RR^n)=0$ $(i\geq 1)$.
Using this, we have 
\[
\begin{CD}
H_{\h_i}^1(Q_i;T^n) @>\delta >> H_{\h_i}^2(Q_i;\ZZ^n)
\end{CD}
\]
which follows that $\delta [\hat \chi_i ]=[f_i]$, $[\hat \chi_i ]\in H_{\h_i}^1(Q_i;T^n)$.
The meaning of $\delta [\hat \chi_i ]=[f_i]$ is explained as follows;
\begin{align*}
\begin{CD}
0 @>>> C_\h^1(Q,\mathbb{Z}^n) @>i>> C_\h^1(Q,\RR^n)  @>j>>  C_\h^1(Q,T^n)   @>>> 0 \\
  @.                @VV\delta'V             @VV\delta'V             @VV\delta'V              \\
0 @>>> C_\h^2(Q,\mathbb{Z}^n) @>i>> C_\h^2(Q,\RR^n)  @>j>>  C_\h^2(Q,T^n)   @>>> 0 \\
  @.                          @.    \chi_i           @>j>>  \hat \chi_i     @>>> 0 \\
  @.                          @.    @VV\delta'V             @VV\delta'V              \\
  @.          f_i             @>i>> \delta' \chi_i   @>j>>  0       @.  
\end{CD}
\end{align*}
$f_i=if_i=\delta'\chi_i$ $(i=1,2)$. Thus
\begin{equation}\label{Pe21}
f_1(\al,\be)=\delta'_{\bar\h_1 (\al)}\chi_1(\al,\be)=\bar \h_1(\al)(\chi_1(\be))\chi_1(\al)\chi_1(\al\be)^{-1}
\end{equation}
where $\chi_i:Q_i\lra\RR^n$ are functions, and $\delta'_{\bar\h_i (\al)}\chi_i:Q_i\times Q_i\lra\RR^n$
the coboundary of 1-cochains.

We note that $\h_i(\al):\ZZ^n\lra\ZZ^n$ is represented by (an integral) matrix so that $\h_i(\al)$
extends to $\RR^n$ onto itself. We write it as $\bar\h_i(\al)$ (= the same matrix).
So $\h_i:Q_i\lra Aut(\ZZ^n)$ extends uniquely to $\bar\h_i:Q_i\lra Aut(\RR^n)=GL(n,\RR)$.

Put $g(f_1(\al,\be))=g_*f_1(\al,\be)$ and $(\bar g_*\chi_1)(\al)=\bar g(\chi_1(\al))$.
Here the isomorphism $g:\ZZ^n\lra \ZZ^n$ extends uniquely to an isomorphism
$\bar g:\RR^n\lra \RR^n$ by the same reason as before.
Also we put $\Phi^*f_2(\al,\be)=f_2(\Phi(\al),\Phi(\be))$ and $\Phi^*\chi_2(\al)=\chi_2(\Phi(\al))$.
Then \eqref{Pe21} implies
\begin{align*}
gf_1(\al,\be)&=\bar g(\bar \h_1(\al)(\chi_1(\be))\chi_1(\al)\chi_1(\al\be)^{-1})\\
             &=\bar g\bar \h_1(\al)\bar g^{-1}(\bar g\chi_1(\be))\bar g\chi_1(\al)\bar g\chi_1(\al\be)^{-1}\\
             &=\bar g\bar \h_1(\al)\bar g^{-1}(\bar g_*\chi_1(\be))\bar g_*\chi_1(\al)\bar g_*\chi_1(\al\be)^{-1}\\
             &=\delta'_{\bar g\bar \h_1\bar g^{-1}}\bar g_*\chi_1(\al,\be),
\end{align*}
hence
\begin{equation}\label{Pe22}
g_*f_1=\delta'_{\bar g\bar \h_1\bar g^{-1}}\bar g_*\chi_1.
\end{equation}
Here we can write $\overline {g\h_1g^{-1}}=\bar g\bar \h_1(\al)\bar g^{-1}$.
Also, $f_2=\delta'\chi_2$ shows 
\begin{equation}\label{Pe23}
\Phi ^*f_2=\delta'\Phi^*\chi_2.
\end{equation}

Next we calculate
\begin{align*}
&\delta'_{\bar g\bar \h_1\bar g^{-1}}(\lam.\Phi^*\chi_2)(\al,\be)\\
&= \bar g\bar \h_1(\al)\bar g^{-1}(\lam.\Phi^*\chi_2(\be)).(\lam.\Phi^*\chi_2(\al)).(\lam.\Phi^*\chi_2(\al\be))^{-1}\,\,\, \text {by}\,\, \eqref{Pe21}\\
&= \bar g\bar \h_1(\al)\bar g^{-1}(\lam(\be).\chi_2(\Phi(\be))).\lam(\al)\chi_2(\Phi(\al)).\chi_2(\Phi(\al\be))^{-1}\lam(\al\be)^{-1}\\
&= \bar g\bar \h_1(\al)(\bar g^{-1}(\lam(\be)))  \,\,\, \bar g\bar \h_1(\al)\bar g^{-1}(\chi_2(\Phi(\be))) \lam(\al)\,\,\\
& \,\,\,\,\,\,\, \chi_2(\Phi(\al)) \chi_2(\Phi(\al\be))^{-1}\lam(\al\be)^{-1}\\
&= \bar g\bar \h_1(\al)(\bar g^{-1}(\lam(\be)))  \,\,\, \lam(\al) \bar\h_2(\Phi(\al))(\chi_2(\Phi(\be)))\,\,\\
& \,\,\,\,\,\,\, \chi_2(\Phi(\al)) \chi_2(\Phi(\al)\Phi(\be))^{-1}\lam(\al\be)^{-1} \,\,\,\,\ \text {by}\,\, \eqref{Pe5}\\
&= \bar g\bar \h_1(\al)(\bar g^{-1}(\lam(\be)) ) \,\,\, \lam(\al) \,\,\delta'\chi_2(\Phi(\al),\Phi(\be))\lam(\al\be)^{-1} \,\,\,\,\ \text {by}\,\, \eqref{Pe21}\\
&= \bar g\bar \h_1(\al)(\bar g^{-1}(\lam(\be)) ) \,\,\, \lam(\al) \,\,f_2(\Phi(\al),\Phi(\be))\lam(\al\be)^{-1} \,\,\,\,\,\,\, \text {by}\,\, \eqref{Pe21}\\
&=g(f_1(\al,\be)) \,\,\, \text {by}\,\, \eqref{Pe6}.
\end{align*}
Hence 
\begin{equation}\label{Pe24}
\delta'_{\bar g\bar \h_1\bar g^{-1}}(\lam.\Phi^*\chi_2)=g_*f_1.
\end{equation}
Since $g_*f_1=\delta'_{\bar g\bar \h_1\bar g^{-1}}\bar g_*\chi_1$ by \eqref{Pe22}, 
\begin{align*}
\delta'_{\bar g\bar \h_1\bar g^{-1}}(\lam.\Phi^*\chi_2)&=g_*f_1\\
                                                            &=\delta'_{\bar g\bar \h_1\bar g^{-1}}(\bar g_*\chi_1)\\
 [\bar g_*\chi_1(\lam \Phi^*\chi_2)^{-1}]&\in H_{\bar g\bar \h_1\bar g^{-1}}^1(Q;\RR^n)
\end{align*}
On the other hand, using the fact that  $H_{\bar g\bar \h_1\bar g^{-1}}^1(Q;\RR^n)=0$, so there exist
$\mu \in \RR^n$ such that 
\begin{align*}
\delta'_{\bar g\bar \h_1\bar g^{-1}}\mu&=\bar g_*\chi_1.(\lam \Phi^*\chi_2)^{-1}\\
\delta'_{\bar g\bar \h_1\bar g^{-1}}\mu.(\lam \Phi^*\chi_2)&=\bar g_*\chi_1\\
(\delta'_{\bar g\bar \h_1\bar g^{-1}}\mu.\lam \Phi^*\chi_2)(\al)&=\bar g_*\chi_1(\al)\\
(\delta'_{\bar g\bar \h_1\bar g^{-1}}\mu(\al)).(\lam \Phi^*\chi_2)(\al)&=\bar g\chi_1(\al)\\
\bar g\bar \h_1(\al)\bar g^{-1}(\mu)\mu^{-1}\,\,((\lam. \Phi^*\chi_2)(\al))&=\\
\bar g\bar \h_1(\al)\bar g^{-1}(\mu)\,((\lam. \Phi^*\chi_2)(\al))\,\mu^{-1}&=\\
\bar g\bar \h_1(\al)\bar g^{-1}(\mu)\,(\lam(\al). \chi_2(\Phi(\al)))\,\mu^{-1}&=\\
\bar g\bar \h_1(\al) \bar g^{-1}(\mu)\lam(\al)\chi_2(\Phi(\al))\,\mu^{-1}&=
\end{align*}
Hence by \eqref{Pe5}
\begin{equation}\label{Pe25}
\bar g(\chi_1(\al))=\lam(\al)\bar \h_2(\Phi(\al))(\mu)\chi_2(\Phi(\al))\,\mu^{-1}.
\end{equation}

Next we consider  a map
\begin{equation}\label{Pe31}
h:\RR^n\lra\RR^n \,\,\,\, \text {by} \,\,\, h(x)=\bar g(x)\mu.
\end{equation}
Since $\bar g$ is an isomorphism, so is $h$, then $h$ is a diffeomorphism. Now we show that $h$
is $(\GA_1,\GA_2)-$equivariant. 
\\

We first assume the rigid motion of $\GA_i$ on $\RR^n$ is the Seifert action. 
Note that $\GA_1$ acts on $\RR^n$ as follows.
\begin{equation}\label{Pe32}
(n,\al)x=n\bar\h_1(\al)(x)\chi_1(\al),\,\,\,\,
(n,\al)x=n\bar\h_2(\al)(x)\chi_2(\al).
\end{equation}
Then
\begin{align*}
&h((n,\al)x)\\
&=h(n\bar\h_1(\al)(x)\chi_1(\al))\\
           &=\bar g(n\bar\h_1(\al)(x)\chi_1(\al))\mu \,\,\,\,\, \text {by} \, \eqref{Pe31}\\
           &=g(n)\bar g \bar\h_1(\al)(x) \,\,\bar g \chi_1(\al)\,\,\mu \,\,\,\,\, (\bar g \,\text {is an isomorphism)}\\
           &=g(n)\bar g \bar\h_1(\al)\bar g^{-1}(\bar g(x)) 
             \lam(\al)\bar\h_2(\Phi(\al))(\mu) \chi_2(\Phi(\al))\mu^{-1}\mu \,\,\,\,\,\text {by} \, \eqref{Pe25}\\
           &=g(n) \lam(\al) \bar\h_2(\Phi(\al))(\bar g(x)) 
             \bar\h_2(\Phi(\al))(\mu) \chi_2(\Phi(\al)) \,\,\,\,\,\text {by} \, \eqref{Pe5}\\
           &=g(n) \lam(\al). \,\,\bar\h_2(\Phi(\al))(\bar g(x).\mu) \,\,
             \chi_2(\Phi(\al)) \\
           &=(g(n) \lam(\al),\Phi(\al)) \bar g(x)\mu  \,\,\,\,\,\text {(by \eqref{Pe32}\,\,and note $g(n) \lam(\al)\in \ZZ^n$)}\\
           &=f(n,\al)\bar g(x)\mu \,\,\,\,\,\text {by} \, \eqref{Pe4}.\\
           &=f(n,\al)h(x).
\end{align*}
Thus
\begin{equation}\label{Pe33}
h((n,\al)x)=f(n,\al)h(x).
\end{equation}

Now, we shall see that the Seifert action coincides with the rigid motions.

Given a crystallographic group, we have the following extension:

\begin{equation*}
1\ra \ZZ^n\ra \GA\ra Q\ra 1.
\end{equation*}
As usual an element $\gamma$ of $\GA$ is described as
$(n,\al)$.
We wrote the Seifert action on $\RR^n$ additively as follows.
\begin{equation}\label{Seifert}
(n,\al)x=n+\bar\h(\al)(x)+\chi(\al) \  \ ((n,\al)\in\GA).
\end{equation} We may check it is a group action, namely
it satisfies that
\[
(n,\al)((m,\be)x)=((n,\al)\cdot (m,\be))x.
\]

Suppose that $\GA$ is a crystallographic group of ${\EE}(n)$.
Note that $(n,1)x=n+x$, that is, $\ZZ^n$ acts as translations.
So its span $\RR^n$, viewed as a group, acts on $\RR^n$ as translations;
\[
(y,1)x=y+x\ (\forall\ (y,1)\in \RR^n).
\]
Since $\GA$ normalizes $\ZZ^n$, it  normalizes its span, $\RR^n$.
Noting that $(1,\al)$ $(n,1)(1,\al)^{-1}=(\h(\al)(n),1)$,
we can write
\[
(1,\al)(x,1)(1,\al)^{-1}=(\bar\h(\al)(x),1).
\]
For the origin $0\in \RR^n$,
we put
\[
(1,\al)0=\chi(\al)\in\RR^n.
\]

Noting $(x,1)0=x$ as above,
we can interpret the above Seifert action \eqref{Seifert},
\begin{equation*}\begin{split}
(n,\al)x&=((n,1)(1,\al))x=(n,1)((1,\al)x)\\
        &=(n,1)((1,\al)(x,1)0)\\
        &=(n,1)(((1,\al)\cdot (x,1)\cdot(1,\al)^{-1})\cdot(1,\al)0)\\
        &=(n,1)((\bar\h(\al)(x),1)(1,\al)0)\\
        &=(n,1)((\bar\h(\al)(x),1) \chi(\al))\\
        &=(n,1)(\bar\h(\al)(x)+\chi(\al))\\
        &=n+\bar\h(\al)(x)+\chi(\al).
\end{split}\end{equation*}

On the other hand, when $\GA\subset\EE(n)$,
let 
\[(n,\al)=\ga=(a,A)\in \EE(n).\]

As $\ga x=a+Ax$ $(\forall\ x\in \RR^n)$,
taking $x=0$, compare the above equation so that
\[
n+\chi(\al)=a.\]
Since $(n,\al)x=\ga x=a+Ax$,
it follows that $\bar\h(\al)(x)=Ax$.
This holds for all $x\in \RR^n$, it implies that
$\bar\h(\al)=A$.

Hence identifying $(n+\chi(\al),\bar\h(\al))=
(a,A)$, the Seifert action of $\GA$ coincides with the rigid motions.
\[
(n,\al)=(n+\chi(\al),\bar\h(\al))\in \EE(n) \ (\forall (n,\al)\in \GA).
\]
%$(\Leftarrow )$ Define a homomorphism $f:\Gamma_1\longrightarrow \Gamma_2$ by $f(\beta)=\alpha\beta\alpha^{-1}$ 
%$\forall \beta\in\Gamma_1$ and for some $\alpha\in\mathbb{A}(n)$. We show that $f$ is an isomorphism.
%Suppose $\beta=\beta'\in\Gamma_1$. Then $f(\beta)=\alpha\beta\alpha^{-1}=\alpha\beta'\alpha^{-1}=f(\beta')$.
%So f is well defined. Let the identity element of $\Gamma_2$\, $e'=f(\beta)=\alpha\beta\alpha^{-1}$, which implies
%$\beta=e$ the identity element of $\Gamma_1$. Hence $f$ is injective. It is obvious that $f$ is surjective.
\end{proof}

\begin{thm}(Third Bieberbach).\label{Bie3}
For any given n, there are only a finite number of n-dimensional crystallographic groups, up to affine equivalence.
\end{thm}
\noindent This implies that there are only finitely many n-dimensional flat Riemannian manifolds.

To prove the third Bieberbach theorem, first of all, we assume the theorem in \cite{Th97} that

\begin{thm}\label{Thurs}
The number of non-isomorphic finite groups in \\$GL(n,\ZZ)$
up to conjugation is finite.
\end{thm}

\begin{proof} \textit {of Theorem} \ref{Bie3}.
Given a finite group $F$
 and a faithful representation $\phi:F\ra GL(n,\ZZ)$,
$\ZZ^n$ is viewed as a $F$-module through $\phi$.
Then each equivalence class of cocycle of
the second cohomology $H^2_{\phi}(F,\ZZ^n)$ defines
a group extension
\begin{equation}\label{group-ex1}
1\ra \ZZ^n\ra \Gamma\stackrel{L}\lra F\ra 1.
\end{equation}
Then we have proved (see the proof of Theorem \ref{T:6})
that
$\Gamma$ acts on $\RR^n$ by
\[
(n,\al)x=n+\phi(\al)(x)+\chi(\al).
\]
Then note that $\ZZ^n$ is a maximal abelian group of $\Gamma$.
For this, let $\ZZ^n\subset \Delta$ be an abelian subgroup of $\Gamma$.
Choose $(n,\al)\in \Delta$ so that
\[
(n,\al)(m,1)(n,\al)^{-1}=(\phi(\al)(m),1).
\]
Note 
\[(n,\al)^{-1}=(\phi(\al^{-1})(n^{-1}\cdot f(\al,\al^{-1})^{-1}),\al^{-1}).\]
Since $\Delta$ is abelian, $(n,\al)(m,1)(n,\al)^{-1}=(m,1)$,
$\phi(\al)(m)=m$ $(\forall \ m\in \ZZ^n)$.
It follows that $\phi(\al)={\rm id}$.
As our assumption is that
$\phi:$ $F\ra $ $GL(n,\ZZ)$ is a faithful representation,
this implies that $\al=1$ which shows that
$\Delta\ni(n,\al)=(n,1)$ so that $\Delta\subset \ZZ^n$.

Identify $(n,\al)=(n+\chi(\al),\phi(\al))\in \mathbb{A}(n)$.
Moreover since $\phi(F)$ is a finite group in 
$GL(n,\ZZ)\subset GL(n,\RR)$, there is an element 
$B\in GL(n,\RR)$ such that
$B\cdot\phi(F)\cdot B^{-1}\subset O(n)$.
Then the correspondence
\[
(n,\al)\ra (B(n+\chi(\al)),B\cdot\phi(\al)\cdot B^{-1})
\] is the injective homomorphism from $\Gamma$ to
$\EE(n)$.
Now $\Gamma\subset \EE(n)$ where $\ZZ^n$ is a maximal abelian subgroup,
by the definition $\Gamma$ is a crystallographic group.

So, each equivalence class $((F,\phi),[f])$ of
$\{(F,\phi), H^2_{\phi}(F,\ZZ^n)\}$ corresponds to the equivalence class of
crystallographic groups.

Conversely,
given a crystallographic group $\Gamma$ of $\EE(n)$, then there is a 
group extension as in \eqref{group-ex1} such that
$\ZZ^n$ is the maximal free abelian group
and $F$ is a finite group. A group extension $\Gamma$
defines
a unique homomorphism $\phi:F\ra {\rm Aut}(\ZZ^n)=GL(n,\ZZ)$
induced by the conjugation of $\Gamma$.
Note that this is a faithful homomorphism, that is, if $\phi(\al)={\rm id}$,
letting $\ga=(a,A)$ such that $L(\ga)=\al=A$,
\[
An=\phi(\al)(n)={\rm id}(n)=n.
\]Since $n\in \ZZ^n$ is the (integral) basis of $\RR^n$, this implies that
$A={\rm I}$. So $\al=1$.
In particular, we have a pair $(F,\phi)$.
Since $\ZZ^n$ is viewed as a $F$-module through $\phi$,
$\Gamma$ defines a $2$-cocycle $[f]\in H^2_{\phi}(F,\ZZ^n)$.
An equivalence class of crystallographic group $\Gamma$ defines
 an element $((F,\phi),[f])$ which implies
that 
$\{(F,\phi), H^2_{\phi}(F,\ZZ^n)\}$ is in one-to-one
correspondence with the equivalence classes of
crystallographic groups.

As $H^2_{\phi}(F,\ZZ^n)$ is of finite order (since $F$ is finite),
the finiteness of group extensions comes from that
of $(F,\phi)$ where $\phi:F\ra GL(n,\ZZ)$ is a faithful representation.
For each $n$ fixed, this is the number of non-isomorphic
conjugacy classes of a finite group $\phi(F)$ in
$GL(n,\ZZ)$. By Theorem \ref{Thurs}, such 
$(F,\phi)$ is finite.
This proves the theorem.
\end{proof}

Bieberbach theorems imply many results in the theory of compact flat manifolds some of which we state below.

\begin {thm}\label{T:8}
%\cite{C85}
Let $X$ and $Y$ be compact flat manifolds. Then $\pi_1(X)$ is isomorphic to $\pi_1(Y)$
if and only if $X$ and $Y$ are diffeomorphic.
\end {thm}
\begin{proof}
$(\Rightarrow )$ We consider $\pi_1(X)$ and $\pi_1(Y)$ as Bieberbach subgroup of $\mathbb{E}(n)$. 
By Theorem \ref{T:6}, there is $\alpha\in\mathbb{R}^n$ s.t if $F:\pi_1(X)\longrightarrow \pi_1(Y)$
is an isomorphism, $F(\beta)=\alpha\beta\alpha^{-1}$ \,\,$\forall \beta\in\pi_1(X)$.
Let
\[
p_x:\mathbb{R}^n\longrightarrow X=\mathbb{R}^n/\pi_1(X) \,\,\, \text {and} \,\,\,
p_y:\mathbb{R}^n\longrightarrow Y=\mathbb{R}^n/\pi_1(Y)
\]
be the projection (or covering maps). Define $f:X\longrightarrow Y$ by 
\[
f(x)=p_y\circ \alpha\circ p^{-1}_x(x), \,\,\, \text{for} \,\,\, x\in X.
\]
To see that this is well-defined, let $\bar x\in p^{-1}_x(x)$ and $\beta\in\pi_1(X)$. 
We must show $p_y\circ \alpha(\beta\bar x)=p_y\circ \alpha(\bar x)$. 
But $\alpha\beta\alpha^{-1}=\gamma\in\pi_1(Y)$, so $\alpha\beta=\gamma\alpha$
and 
\begin{align*}
p_y\circ \alpha(\beta\bar x)&=p_y(\alpha\beta\bar x)=p_y(\gamma\alpha\bar x)\\
                             &=p_y(\alpha\bar x)\\
				     &=p_y\circ \alpha(\bar x),
\end{align*}
so f is well-defined.

To show that $f$ is diffeomorphism, we need merely to show $f$ is bijective.
%$f$ is surjective, because $f(x)=p_y\circ \alpha(\bar x)=p_y(\alpha.\bar x)$ and $p_y$ is surjective.
%Let the identity element of $Y$, 
%\begin{align*}
%e&=p_y\circ \alpha\circ p^{-1}_x(x)\\
% &=p_y\circ \alpha(\bar x)\\
% &=p_y(\alpha.\bar x)\\
% &=[\alpha.\bar x]
%\end{align*}
Define $g:Y\longrightarrow X$ by
\begin{align*}
g(y)=p_x\circ \alpha^{-1}\circ p^{-1}_y(y), \,\,\, \text{for} \,\,\, y\in Y.
\end{align*}
Similarly, $g$ is also well-defined. For $\bar y\in p^{-1}_y(y)$ and $\gamma\in\pi_1(Y)$. 
We show $p_x\circ \alpha^{-1}(\gamma\bar y)=p_x\circ \alpha^{-1}(\bar y)$. 
But $\alpha\beta\alpha^{-1}=\gamma\in\pi_1(Y)$, so $\beta\alpha^{-1}=\alpha^{-1}\gamma$
and 
\begin{align*}
p_x\circ \alpha^{-1}(\gamma\bar y)&=p_x(\alpha^{-1}\gamma\bar y)=p_x(\beta\alpha^{-1}\bar y)\\
                             &=p_x(\alpha^{-1}\bar y)\\
				     &=p_x\circ \alpha^{-1}(\bar y).
\end{align*}
Now, 
\begin{align*}
f\circ g(y)&=f(p_x\circ \alpha^{-1}\circ p^{-1}_y(y))\\
          &=p_y\circ \alpha\circ p^{-1}_x(p_x\circ \alpha^{-1}\circ p^{-1}_y(y))\\
	    &=y.
\end{align*}
Similarly, $g\circ f(x)=x$. So, $g=f^{-1}$, f is bijective.

$(\Leftarrow )$. Let $f:X\longrightarrow Y$ is a diffeomorphism. So f is a homeomorphism. We show that
$F:\pi_1(X)\longrightarrow \pi_1(Y)$
is an isomorphism.

Since $f$ is a homeomorphism, there exists the inverse map $f^{-1}:Y\longrightarrow X$ has the property that
\[
f^{-1}\circ f=I_X, \,\,\,\, f\circ f^{-1}=I_Y.
\]
Let us consider the induced homomorphisms
\[
f_*:\pi_1(X)\longrightarrow \pi_1(Y), \,\,\,\, f^{-1}_*:\pi_1(Y)\longrightarrow \pi_1(X)
\]
in the fundamental groups. Since $(f^{-1}\circ f)_*=f^{-1}_*\circ f_*$, and $(I_X)_*:\pi_1(X)\longrightarrow \pi_1(X)$,
then $f^{-1}_*\circ f_*$ is the identity map on $\pi_1(X)$. Similarly, we can see that $f_*\circ f^{-1}_*$
is the identity map on $\pi_1(Y)$. Therefore, $f_*$ and $f^{-1}_*$ are inverses of each other, i.e., 
$F=f_*:\pi_1(X)\longrightarrow \pi_1(Y)$ is an isomorphism.
\end{proof}

\begin{thm}\label{T:9}
Let $\Gamma$ be a subgroup of $\mathbb{E}(n)$. Then the orbit space $\mathbb{R}^n/\Gamma$ is 
a compact flat n-dimensional manifold if and only if $\Gamma$ is torsionfree, discrete and irreducible.
\end{thm}
\begin{proof}
$(\Rightarrow )$ If $M=\mathbb{R}^n/\Gamma$ is a compact flat n-manifold, then $\pi_1(M)=\Gamma$ and
hence $\Gamma$ acts freely on $\mathbb{R}^n$, so $\Gamma$ must be torsionfree (by Proposition \ref{P:1}). Since $M$ is an n-manifold,
$\Gamma$ must be discrete. Since $\mathbb{R}^n/\Gamma$ is compact, by Lemma 7 \cite{K11}, $\Gamma$ is irreducible\\
$(\Leftarrow )$ Since $\Gamma$ is a torsionfree discrete subgroup of $\mathbb{E}(n)$, then by Proposition \ref{P:1}
$\Gamma$ acts freely on $\mathbb{R}^n$ and discrete subgroup of $\mathbb{E}(n)$ acts properly discontinuously
on $\mathbb{R}^n$, so $\mathbb{R}^n/\Gamma$ is an n-manifold. By first Bieberbach theorem, $\mathbb{R}^n/\Gamma$ 
is covered by $\mathbb{R}^n/(\Gamma\cap \mathbb{R}^n)$ which is the n-torus. Hence $\mathbb{R}^n/\Gamma$
is compact. 
\end{proof}

%\begin{thm}\label{T:7}

According to some theorems above, Bieberbach proved that, $\Gamma $ and $\Gamma '$ are two isomorphic 
discrete compact subgroups (crystallographic  subgroups)
of $\mathbb{E}(n)$ if and only if  $\exists$ $\gamma \in \mathbb{A}(n)$ such that $\gamma \Gamma\gamma ^{-1}=\Gamma'$.
This implies that, if $M$ and $M'$ are compact flat Riemannian manifolds with isomorphic 
fundamental groups $\Gamma $ and $\Gamma'$ respectively then,
$M$ is diffeomorphic to $M'$ if and only if $\exists$ $\gamma \in \mathbb{A}(n)$ such that $\gamma \Gamma\gamma ^{-1}=\Gamma'$.
\\

To show that $M(A)=(S^1)^n/(\mathbb{Z}_2)^n$ is aspherical, by Theorems \ref{T:1}
then $p: (S^1)^n \rightarrow (S^1)^n/(\mathbb{Z}_2)^n$ is covering. Since $q:\mathbb{R}^n \rightarrow (S^1)^n$
is a covering with $\mathbb{R}^n$ its universal covering space, then $p\circ q:\mathbb{R}^n \rightarrow (S^1)^n/(\mathbb{Z}_2)^n$
is covering. Hence $(S^1)^n/(\mathbb{Z}_2)^n$ is aspherical.

%\newpage
\section{Three Dimensional Real Bott Tower}

We begin this section with defining the action of cyclic group $\mathbb {(Z}_2)^3$ on $(S^1)^3$.

Let 
\begin{equation}\label {matrix}
A=\left(\begin{array}{lcr}
1& a_{12} & a_{13}\\
0& 1 & a_{23}\\
0& 0 & 1
\end{array}\right)
\end{equation}
be one of the 3rd Bott matrices.
Then each $a_{ij}$ represents either $0$ or $1$. 
We use the following notation for $a\in\{0,1\}$
\[
{\bar z}^a=\left\{\begin{array}{lr} \bar z & \mbox{if}\ a=1\\
z& \mbox{if}\ a=0
\end{array}\right.\]where $\bar z$ is the conjugate of the complex number $z$.
If $(g_1,g_2,g_3)$ $\in (\mathbb{Z}_2)^3$ are
generators, then $(\mathbb{Z}_2)^3$ acts on $T^3$ as
\begin{align*}
\begin{split}
g_1(z_1,z_2,z_3)&=(-z_1,{\bar z_2}^{a_{12}},{\bar z_3}^{a_{13}})\\
g_2(z_1,z_2,z_3)&=(z_1,-z_2,{\bar z_3}^{a_{23}})\\
g_3(z_1,z_2,z_3)&=(z_1,z_2,-z_3).
\end{split}
\end{align*}It is easy to see that 
$(\mathbb{Z}_2)^3$ acts freely on $T^3$ such that
the orbit space $M(A)=T^3/(\mathbb{Z}_2)^3$ is a smooth compact manifold.
In this way, taking an $n$-th Bott matrix $A$, 
we obtain a free action of $(\mathbb{Z}_2)^n$ on $T^n$.

%For $g_i^2=e$ the identity element of $\mathbb{Z}_2$, 
%we get 
%\begin{align*}
%\begin{split}
%g_1^2(z_1,z_2,z_3)&=g_1(-z_1,{\bar z_2}^{a_{12}},{\bar z_3}^{a_{13}})=(z_1,z_2,z_3).\\
%g_2^2(z_1,z_2,z_3)&=g_2(z_1,-z_2,{\bar z_3}^{a_{23}})=(z_1,z_2,z_3).\\
%g_3^2(z_1,z_2,z_3)&=g_3(z_1,z_2,-z_3)=(z_1,z_2,z_3).
%\end{split}
%\end{align*}

%According to above definition of action, we can write map 
%$g_i:S^1 \times S^1 \times S^1=T^3 \longrightarrow  T^3$,
%and for any element $z\in S^1$ can be expressed by $z=e^{2\pi ix},$ $x\in \mathbb{R}$. 
%Then we define lifting map $\tilde {g}_i:\mathbb{R}^3\rightarrow \mathbb{R}^3$ such that
%this diagram is commutative
%\[
%\begin{CD}
%\mathbb{R}^3 @>\tilde {g}_i>> \mathbb{R}^3\\
%@VPrVV  @VVPrV\\
%T^3 @>g_i>> T^3
%\end{CD}
%\]
%where $Pr(x,y,z)=(e^{2\pi ix},e^{2\pi iy},e^{2\pi iz})=(z_1,z_2,z_3)$.
%\\

Let $Pr(x,y,z)=(e^{2\pi {\bf i}x},e^{2\pi {\bf i}y},
e^{2\pi {\bf i}z})=(z_1,z_2,z_3)$ be the canonical covering  map of $\RR^3$ onto $T^3$. Given $A$ as before, 
each generator $g_i$ of $(\mathbb{Z}_2)^3$ acting on $T^3$ lifts to a map
$\tilde {g}_i:\mathbb{R}^3\rightarrow \mathbb{R}^3$
so that this diagram is commutative
\[
\begin{CD}
\mathbb{R}^3 @>\tilde {g}_i>> \mathbb{R}^3\\
@VPrVV  @VVPrV\\
T^3 @>g_i>> T^3.
\end{CD}
\]
Then the fundamental group of $M(A)$ is
isomorphic to the group $\Gamma(A)$
generated by $\{\tilde g_1,\tilde g_2,\tilde g_3\}$.
Moreover, it forms a torsion-free discrete
uniform subgroup of the Euclidean group 
$\mathbb {E}(n)=\mathbb {R}^n\rtimes {\rm O}(n)$.
$\Gamma(A)$ is called {\em Bieberbach group} 
(torsion free crystallographic group).
We call $\Gamma(A)$ a Bott group associated to the Bott matrix
$A$, for brevity.

For example, let
\begin{equation} \label{ex}
A=\left(\begin{array}{lcr}
1& 0 & 1\\
0& 1 & 1\\
0& 0 & 1
\end{array}\right),
\end{equation}
and by definition we have
\begin{align*}\label{klein}
\begin{split}
g_1(z_1,z_2,z_3)&=(-z_1, z_2,\bar z_3)\\
g_2(z_1,z_2,z_3)&=(z_1,-z_2,\bar z_3)\\
g_3(z_1,z_2,z_3)&=(z_1,z_2,-z_3).
\end{split}
\end{align*}
Thus we get $\tilde {g}_1 \fz
\left(\begin{array}{c}
x\\
y\\
z
\end{array}\right)=
\left(\begin{array}{ccc}
1& 0 & 0\\
0& 1 & 0\\
0& 0 & -1
\end{array}\right)
\left(\begin{array}{c}
x\\
y\\
z
\end{array}\right)+
\left(\begin{array}{c}
\frac{1}{2}\\
0\\
0
\end{array}\right)$. \nz
Therefore we may write the generator $\tilde {g}_1= \fz
\Bigl(
\left(\begin{array}{c}
\frac{1}{2}\\
0\\
0
\end{array}\right),
\left(\begin{array}{ccc}
1& 0 & 0\\
0& 1 & 0\\
0& 0 & -1
\end{array}\right)\Bigr)=(s_1, M_1)\in \Gamma(A)$ \nz
where $s_1\in\RR^3, \, M_1\in O(3)$.
Similarly, we have \\
$\tilde {g}_2= $ $\fz
\Bigl(
\left(\begin{array}{c}
0\\
\frac{1}{2}\\
0
\end{array}\right),
\left(\begin{array}{ccc}
1& 0 & 0\\
0& 1 & 0\\
0& 0 & -1
\end{array}\right)\Bigr)=$ $(s_2, M_2)$, \nz
$\tilde {g}_3= \fz
\Bigl(
\left(\begin{array}{c}
0\\
0\\
\frac{1}{2}
\end{array}\right),
\left(\begin{array}{ccc}
1& 0 & 0\\
0& 1 & 0\\
0& 0 & 1
\end{array}\right)\Bigr)$  $=(s_3, I_3)$, where $s_2, s_3 \in \RR^3$ and 
$M_2\in O(3)$. \nz
Hence  \small
$\Gamma(A) =<\ti  g_1, \ti  g_2, \ti  g_3>$. \normalsize
\\
By this explanation we see that action $\mathbb{(Z}_2)^3$ on $(S^1)^3$ coincides with action $\Gamma $ on $\mathbb{R}^3$.

Given a Bott matrix $A$ as (\ref{matrix}), $M(A)$ admits a maximal $T^k$-action ($k=1,2,3$) which is obtained as follows.\\
Let $t_i\in T^3$, and define $T^3$-action on $(S^1)^3$ by
\begin{align*}
t_1(z_1,z_2,z_3)=(e^{2\pi i\theta  }z_1,z_2,z_3)\\
t_2(z_1,z_2,z_3)=(z_1,e^{2\pi i\theta  }z_2,z_3)\\
t_3(z_1,z_2,z_3)=(z_1,z_2,e^{2\pi i\theta  }z_3)
\end{align*}
where $0\leq \theta\leq 1 $.\\
If $t_ig(z_1,z_2,z_3)=gt_i(z_1,z_2,z_3)$ for  $t_i\in T^k$ and  $g\in(\ZZ_2)^3$
then $M(A)$ admits a $T^k$-action ($k=1,2,3$) which is maximal..

The above example, it is easy to check that $t_1g_i=g_it_1$, $t_2g_i=g_it_2$ for $i=1,2,3$, but $t_3g_1\not=g_1t_3$.
Therefore $M(A)$ with representative matrix (\ref{ex}) admits $T^2$-action which is maximal.

According to Bieberbach theorem, fundamental group $\pi _1(M(A))=\Gamma=\langle \tilde{g}_1,\tilde{g}_2,\tilde{g}_3 \rangle$ 
is isomorphic to 
$\pi _1(M(A'))=\Gamma '=\langle \tilde{h}_1,\tilde{h}_2,\tilde{h}_3 \rangle$ if and only if 
$\exists$ $\gamma \in \mathbb{A}(3)=\mathbb{R}^3\rtimes GL(3,\mathbb{R})$ such that $\gamma \Gamma\gamma ^{-1}=\Gamma'$. 
So, if we can find $\tilde \varphi =\gamma \in \mathbb{A}(3)$ such that
$\gamma \Gamma\gamma ^{-1}=\Gamma'$, then $M(A)$ is diffeomorphic to  $M(A')$.
This means, we have to define map $\tilde \varphi$ or $\varphi$ such that the diagrams below are commutative
\[
\begin{CD}
\mathbb{R}^3 @>\tilde {\varphi }>> \mathbb{R}^3\\
@V\tilde {g}_iVV  @VV\tilde {h}_jV\\
\mathbb{R}^3 @>\tilde {\varphi }>> \mathbb{R}^3
\end{CD} \hspace{1cm}
\begin{CD}
T^3 @>\varphi>> T^3\\
@Vg_iVV  @VVh_jV\\
T^3 @>\varphi>> T^3
\end{CD}
\] 
for $i=j=1,2,3.$
This is a tool that we will use later to determine the diffeomorphism classes of $M(A)$.
\\

If $M(A)$ is a $3$-dimensional real Bott tower, then it is
classified as follows.

\begin{thm} \label {T:3} Let $M(A)$ be a $3$-dimensional real Bott tower.
Then $M(A)$ admits a maximal $T^k$-action $(k=1,2,3)$.
There exist exactly $4$ distinct diffeomorphism classes of $M(A)$
in which all $8$ Bott matrices fall into the following classes:
\begin{itemize}
\item[(i)](orientable)($T^3$-action.) The Identity matrix $I_3.$
\item [(ii)](nonorientable)($T^2$-action.)\\
\fz $A_1=\left(\begin{array}{lcr}
1& 1 & 0\\
0& 1 & 0\\
0& 0 & 1
\end{array}\right)$,
$A_2=\left(\begin{array}{lcr}
1& 0 & 1\\
0& 1 & 1\\
0& 0 & 1
\end{array}\right)$,
%\item[\ ] \hskip3cm 
$A_3=\left(\begin{array}{lcr}
1& 0 & 0\\
0& 1 & 1\\
0& 0 & 1
\end{array}\right)$,\\
\item[\ ] %\hskip3cm
$A_4=\left(\begin{array}{lcr}
1& 0 & 1\\
0& 1 & 0\\
0& 0 & 1
\end{array}\right)$. \nz
\item[(iii)](nonorientable)($S^1$-action.) \\
\fz $A_5=\left(\begin{array}{lcr}
1& 1 & 0\\
0& 1 & 1\\
0& 0 & 1
\end{array}\right)$,
$A_6=\left(\begin{array}{lcr}
1& 1 & 1\\
0& 1 & 1\\
0& 0 & 1
\end{array}\right)$. \nz
\item[(iv)](orientable)($S^1$-action.)\, \fz  $A_7=\left(\begin{array}{lcr}
1& 1& 1\\
0& 1 & 0\\
0& 0 & 1
\end{array}\right)$.
\end{itemize} \nz
Moreover, Bott matrices in each class are conjugate.
\end{thm}
%=================================================
\begin{proof} For simplicity, notation "$\approx$" indicates diffeomorphic. 
The proof is organized by the following steps.
           \begin{enumerate}
	     \item Determine conjugacy classes of 8 Bott matrices 
			(we used Maple to determine this).
           \item For each conjugacy class, we check whether 
			$M(A), M(A')$ with representative matrices $A,A'$ are diffeomorphic or not
			by using Bieberbach theorems.
           \end{enumerate}
A detailed proof is as follows:
\begin{itemize}
\item[a)] $M(I_3)$ is not diffeomorphic to any $M(A_i)$ for $i=1,2,3,4,5,6,7.$ 
\item[b)] $M(A_1)\approx M(A_2)\approx M(A_3)\approx M(A_4)$
\item[c)] $M(A_5)\approx M(A_6)$
\item[d)] $M(A_1)$ is not diffeomorphic to $M(A_5)$
\item[e)] $M(A_1)$ is not diffeomorphic to $M(A_7)$
\item[f)] $M(A_5)$ is not diffeomorphic to $M(A_7)$
\end{itemize}
%\begin{itemize}
%\item [a)]\vspace{.5cm}
\noindent {\bf a)}. This is clear, because holonomy group of $I_3$ is trivial.
\\

\noindent {\bf b)}.\\
\noindent $\bullet$ $ M(A_{1})\approx  M(A_{2})$.\\
For  
\footnotesize
\begin{align*}
&A_{1}=
\left(\begin{array}{ccc}
1& 1 & 0  \\
0& 1 & 0 \\
0& 0 & 1 
\end{array}\right)
\hspace{.1cm}
&A_{2}=
\left(\begin{array}{ccc}
1& 0 & 1 \\
0& 1 & 1 \\
0& 0 & 1 \\
\end{array}\right)
\\
&g_1(z_1,z_2,z_3)=(-z_1,\bar z_2,z_3) 			 &h_1h_2(z_1,z_2,z_3)=(-z_1,-z_2,z_3)\\
&g_2(z_1,z_2,z_3)=(z_1,-z_2,z_3)          		 &h_2(z_1,z_2,z_3)=(z_1,-z_2,\bar z_3)\\
&g_3(z_1,z_2,z_3)=(z_1,z_2,-z_3)                       &h_3(z_1,z_2,z_3)=(z_1, z_2,-z_3)
\end{align*}
\normalsize then \footnotesize
\begin{align*}
\Gamma_{1}=
\Bigl < \left(\begin{array}{ccc}
\frac{1}{2}\\
0\\
0
\end{array}\right)
\left(\begin{array}{ccc}
1& 0 & 0 \\
0& -1 & 0\\
0& 0 & 1 
\end{array}\right),
\left(\begin{array}{ccc}
0\\
\frac{1}{2}\\
0
\end{array}\right)
\left(\begin{array}{ccc}
1& 0 & 0 \\
0& 1 & 0 \\
0& 0 & 1
\end{array}\right),
\left(\begin{array}{ccc}
0\\
0\\
\frac{1}{2}
\end{array}\right)
\left(\begin{array}{ccc}
1& 0 & 0 \\
0& 1 & 0 \\
0& 0 & 1 
\end{array}\right)
\Bigr >
\end{align*}

\begin{align*}
\Gamma_{2}=
\Bigl < \left(\begin{array}{ccc}
\frac{1}{2}\\
\frac{1}{2}\\
0
\end{array}\right)
\left(\begin{array}{ccc}
1& 0 & 0 \\
0& 1 & 0\\
0& 0 & 1 
\end{array}\right),
\left(\begin{array}{ccc}
0\\
\frac{1}{2}\\
0
\end{array}\right)
\left(\begin{array}{ccc}
1& 0 & 0 \\
0& 1 & 0 \\
0& 0 & -1
\end{array}\right),
\left(\begin{array}{ccc}
0\\
0\\
\frac{1}{2}
\end{array}\right)
\left(\begin{array}{ccc}
1& 0 & 0 \\
0& 1 & 0 \\
0& 0 & 1 
\end{array}\right)
\Bigr >.
\end{align*}

\normalsize
\noindent
Let $\varphi(z_1,z_2,z_3)=(z_3,z_1z_3,z_2)$, we get these commutative diagrams

\footnotesize
\[
\begin{CD}
(z_1,z_2,z_3) @>\varphi>>(z_3,z_1z_3,z_2)\\
@Vg_1VV @Vh_2 VV\\
(-z_1,\bar z_2,z_3) @>\varphi>> (z_3,-z_1z_3,\bar z_2)
\end{CD}
\hspace{1cm}
\begin{CD}
(z_1,z_2,z_3) @>\varphi>>(z_3,z_1z_3,z_2)\\
@Vg_3VV @Vh_1h_2 VV\\
(z_1,z_2,-z_3) @>\varphi>> (-z_3,-z_1z_3,z_2)
\end{CD}
\]
\[ 
\begin{CD}
(z_1,z_2,z_3) @>\varphi>>(z_3,z_1z_3,z_2)\\
@Vg_2VV @Vh_3 VV\\
(z_1,-z_2,z_3) @>\varphi>> (z_3,z_1z_3,-z_2).
\end{CD}
\]

\normalsize
\noindent 
Therefore $\exists$ 
$\gamma 
\footnotesize
=
\Bigl( \left(\begin{array}{ccc}
0\\
0\\
0
\end{array}\right)
\left(\begin{array}{ccc}
0& 0 & 1 \\
1& 0 & 1 \\
0& 1 & 0
\end{array}\right)
\Bigr)
\normalsize
\in \mathbb{A}(n)$ 
s.t. $\gamma \Gamma_{1}\gamma ^{-1}=\Gamma_{2}.$\\
%--------------------------------------------------------------------

\noindent
$\bullet  M(A_{1})\approx  M(A_{3})$.\\
For
\footnotesize
\begin{align*}
&A_{1}=
\left(\begin{array}{ccc}
1& 1 & 0  \\
0& 1 & 0 \\
0& 0 & 1 
\end{array}\right)
\hspace{.1cm}
&A_{3}=
\left(\begin{array}{ccc}
1& 0 & 0 \\
0& 1 & 1 \\
0& 0 & 1 \\
\end{array}\right)
\\
&g_1(z_1,z_2,z_3)=(-z_1,\bar z_2,z_3) 			 &h_1(z_1,z_2,z_3)=(-z_1,z_2,z_3)\\
&g_2(z_1,z_2,z_3)=(z_1,-z_2,z_3)          		 &h_2(z_1,z_2,z_3)=(z_1,-z_2,\bar z_3)\\
&g_3(z_1,z_2,z_3)=(z_1,z_2,-z_3)                       &h_3(z_1,z_2,z_3)=(z_1, z_2,-z_3)
\end{align*}
\normalsize then \footnotesize
\begin{align*}
\Gamma_{1}=
\Bigl < \left(\begin{array}{ccc}
\frac{1}{2}\\
0\\
0
\end{array}\right)
\left(\begin{array}{ccc}
1& 0 & 0 \\
0& -1 & 0\\
0& 0 & 1 
\end{array}\right),
\left(\begin{array}{ccc}
0\\
\frac{1}{2}\\
0
\end{array}\right)
\left(\begin{array}{ccc}
1& 0 & 0 \\
0& 1 & 0 \\
0& 0 & 1
\end{array}\right),
\left(\begin{array}{ccc}
0\\
0\\
\frac{1}{2}
\end{array}\right)
\left(\begin{array}{ccc}
1& 0 & 0 \\
0& 1 & 0 \\
0& 0 & 1 
\end{array}\right)
\Bigr >
\end{align*}

\begin{align*}
\Gamma_{3}=
\Bigl < \left(\begin{array}{ccc}
\frac{1}{2}\\
0\\
0
\end{array}\right)
\left(\begin{array}{ccc}
1& 0 & 0 \\
0& 1 & 0\\
0& 0 & 1 
\end{array}\right),
\left(\begin{array}{ccc}
0\\
\frac{1}{2}\\
0
\end{array}\right)
\left(\begin{array}{ccc}
1& 0 & 0 \\
0& 1 & 0 \\
0& 0 & -1
\end{array}\right),
\left(\begin{array}{ccc}
0\\
0\\
\frac{1}{2}
\end{array}\right)
\left(\begin{array}{ccc}
1& 0 & 0 \\
0& 1 & 0 \\
0& 0 & 1 
\end{array}\right)
\Bigr >.
\end{align*}

\normalsize
\noindent
Let $\varphi(z_1,z_2,z_3)=(z_3,z_1,z_2)$, we get these commutative diagrams

\footnotesize
\[
\begin{CD}
(z_1,z_2,z_3) @>\varphi>>(z_3,z_1,z_2)\\
@Vg_1VV @Vh_2 VV\\
(-z_1,\bar z_2,z_3) @>\varphi>> (z_3,-z_1,\bar z_2)
\end{CD}
\hspace{1cm}
\begin{CD}
(z_1,z_2,z_3) @>\varphi>>(z_3,z_1,z_2)\\
@Vg_3VV @Vh_1 VV\\
(z_1,z_2,-z_3) @>\varphi>> (-z_3,z_1,z_2)
\end{CD}
\]
\[ 
\begin{CD}
(z_1,z_2,z_3) @>\varphi>>(z_3,z_1,z_2)\\
@Vg_2VV @Vh_3 VV\\
(z_1,-z_2,z_3) @>\varphi>> (z_3,z_1,-z_2).
\end{CD}
\]

\normalsize
\noindent 
Therefore $\exists$ 
$\gamma 
\footnotesize
=
\Bigl( \left(\begin{array}{ccc}
0\\
0\\
0
\end{array}\right)
\left(\begin{array}{ccc}
0& 0 & 1 \\
1& 0 & 0 \\
0& 1 & 0
\end{array}\right)
\Bigr)
\normalsize
\in \mathbb{A}(n)$ 
s.t. $\gamma \Gamma_{1}\gamma ^{-1}=\Gamma_{3}$.\\
%--------------------------------------------------------------------

\noindent
$\bullet  M(A_{1})\approx  M(A_{4})$. \\
For
\footnotesize
\begin{align*}
&A_{1}=
\left(\begin{array}{ccc}
1& 1 & 0  \\
0& 1 & 0 \\
0& 0 & 1 
\end{array}\right)
\hspace{.1cm}
&A_{4}=
\left(\begin{array}{ccc}
1& 0 & 1 \\
0& 1 & 0 \\
0& 0 & 1 \\
\end{array}\right)
\\
&g_1(z_1,z_2,z_3)=(-z_1,\bar z_2,z_3) 			 &h_1(z_1,z_2,z_3)=(-z_1,z_2,\bar z_3)\\
&g_2(z_1,z_2,z_3)=(z_1,-z_2,z_3)          		 &h_2(z_1,z_2,z_3)=(z_1,-z_2,z_3)\\
&g_3(z_1,z_2,z_3)=(z_1,z_2,-z_3)                       &h_3(z_1,z_2,z_3)=(z_1, z_2,-z_3)
\end{align*}
\normalsize then \footnotesize
\begin{align*}
\Gamma_{1}=
\Bigl < \left(\begin{array}{ccc}
\frac{1}{2}\\
0\\
0
\end{array}\right)
\left(\begin{array}{ccc}
1& 0 & 0 \\
0& -1 & 0\\
0& 0 & 1 
\end{array}\right),
\left(\begin{array}{ccc}
0\\
\frac{1}{2}\\
0
\end{array}\right)
\left(\begin{array}{ccc}
1& 0 & 0 \\
0& 1 & 0 \\
0& 0 & 1
\end{array}\right),
\left(\begin{array}{ccc}
0\\
0\\
\frac{1}{2}
\end{array}\right)
\left(\begin{array}{ccc}
1& 0 & 0 \\
0& 1 & 0 \\
0& 0 & 1 
\end{array}\right)
\Bigr >
\end{align*}
\begin{align*}
\Gamma_{4}=
\Bigl < \left(\begin{array}{ccc}
\frac{1}{2}\\
0\\
0
\end{array}\right)
\left(\begin{array}{ccc}
1& 0 & 0 \\
0& 1 & 0\\
0& 0 & -1 
\end{array}\right),
\left(\begin{array}{ccc}
0\\
\frac{1}{2}\\
0
\end{array}\right)
\left(\begin{array}{ccc}
1& 0 & 0 \\
0& 1 & 0 \\
0& 0 & 1
\end{array}\right),
\left(\begin{array}{ccc}
0\\
0\\
\frac{1}{2}
\end{array}\right)
\left(\begin{array}{ccc}
1& 0 & 0 \\
0& 1 & 0 \\
0& 0 & 1 
\end{array}\right)
\Bigr >.
\end{align*}

\normalsize
\noindent
Let $\varphi(z_1,z_2,z_3)=(z_1,z_3,z_2)$, we get these commutative diagrams

\footnotesize
\[
\begin{CD}
(z_1,z_2,z_3) @>\varphi>>(z_1,z_3,z_2)\\
@Vg_1VV @Vh_1 VV\\
(-z_1,\bar z_2,z_3) @>\varphi>> (-z_1,z_3,\bar z_2)
\end{CD}
\hspace{1cm}
\begin{CD}
(z_1,z_2,z_3) @>\varphi>>(z_1,z_3,z_2)\\
@Vg_3VV @Vh_2 VV\\
(z_1,z_2,-z_3) @>\varphi>> (z_1,-z_3,z_2)
\end{CD}
\]
\[ 
\begin{CD}
(z_1,z_2,z_3) @>\varphi>>(z_1,z_3,z_2)\\
@Vg_2VV @Vh_3 VV\\
(z_1,-z_2,z_3) @>\varphi>> (z_1,z_3,-z_2).
\end{CD}
\]

\normalsize
\noindent 
Therefore $\exists$ 
$\gamma 
\footnotesize
=
\Bigl( \left(\begin{array}{ccc}
0\\
0\\
0
\end{array}\right)
\left(\begin{array}{ccc}
1& 0 & 0 \\
0& 0 & 1 \\
0& 1 & 0
\end{array}\right)
\Bigr)
\normalsize
\in \mathbb{A}(n)$ 
s.t. $\gamma \Gamma_{1}\gamma ^{-1}=\Gamma_{4}$.\\
%--------------------------------------------------------------------
%\item[b)]

\noindent {\bf c)}. \\
\noindent $\bullet  M(A_{5})\approx  M(A_{6})$.\\
For
\footnotesize
\begin{align*}
&A_{5}=
\left(\begin{array}{ccc}
1& 1 & 0  \\
0& 1 & 1 \\
0& 0 & 1 
\end{array}\right)
\hspace{.1cm}
&A_{6}=
\left(\begin{array}{ccc}
1& 1 & 1 \\
0& 1 & 1 \\
0& 0 & 1 \\
\end{array}\right)
\\
&g_1(z_1,z_2,z_3)=(-z_1,\bar z_2,z_3) 			 &h_1h_2(z_1,z_2,z_3)=(-z_1,-\bar z_2,z_3)\\
&g_2(z_1,z_2,z_3)=(z_1,-z_2,\bar z_3)          		 &h_2(z_1,z_2,z_3)=(z_1,-z_2,\bar z_3)\\
&g_3(z_1,z_2,z_3)=(z_1,z_2,-z_3)                       &h_3(z_1,z_2,z_3)=(z_1, z_2,-z_3)
\end{align*}
\normalsize then \footnotesize
\begin{align*}
\Gamma_{5}=
\Bigl < \left(\begin{array}{ccc}
\frac{1}{2}\\
0\\
0
\end{array}\right)
\left(\begin{array}{ccc}
1& 0 & 0 \\
0& -1 & 0\\
0& 0 & 1 
\end{array}\right),
\left(\begin{array}{ccc}
0\\
\frac{1}{2}\\
0
\end{array}\right)
\left(\begin{array}{ccc}
1& 0 & 0 \\
0& 1 & 0 \\
0& 0 & -1
\end{array}\right),
\left(\begin{array}{ccc}
0\\
0\\
\frac{1}{2}
\end{array}\right)
\left(\begin{array}{ccc}
1& 0 & 0 \\
0& 1 & 0 \\
0& 0 & 1 
\end{array}\right)
\Bigr >
\end{align*}
\begin{align*}
\Gamma_{6}=
\Bigl < \left(\begin{array}{ccc}
\frac{1}{2}\\
\frac{1}{2}\\
0
\end{array}\right)
\left(\begin{array}{ccc}
1& 0 & 0 \\
0& -1 & 0\\
0& 0 & 1 
\end{array}\right),
\left(\begin{array}{ccc}
0\\
\frac{1}{2}\\
0
\end{array}\right)
\left(\begin{array}{ccc}
1& 0 & 0 \\
0& 1 & 0 \\
0& 0 & -1
\end{array}\right),
\left(\begin{array}{ccc}
0\\
0\\
\frac{1}{2}
\end{array}\right)
\left(\begin{array}{ccc}
1& 0 & 0 \\
0& 1 & 0 \\
0& 0 & 1 
\end{array}\right)
\Bigr >.
\end{align*}

\normalsize
\noindent
Let $\varphi(z_1,z_2,z_3)=(z_1,iz_2,z_3)$, we get these commutative diagrams

\footnotesize
\[
\begin{CD}
(z_1,z_2,z_3) @>\varphi>>(z_1,iz_2,z_3)\\
@Vg_1VV @Vh_1h_2 VV\\
(-z_1,\bar z_2,z_3) @>\varphi>> (-z_1,-\bar {iz}_2,z_3)
\end{CD}
\hspace{1cm}
\begin{CD}
(z_1,z_2,z_3) @>\varphi>>(z_1,iz_2,z_3)\\
@Vg_3VV @Vh_3 VV\\
(z_1,z_2,-z_3) @>\varphi>> (z_1,iz_2,-z_3)
\end{CD}
\]
\[ 
\begin{CD}
(z_1,z_2,z_3) @>\varphi>>(z_1,iz_2,z_3)\\
@Vg_2VV @Vh_2 VV\\
(z_1,-z_2,\bar z_3) @>\varphi>> (z_1,-iz_2,\bar z_3).
\end{CD}
\]

\normalsize
\noindent 
Therefore $\exists$ 
$\gamma 
\footnotesize
=
\Bigl( \left(\begin{array}{ccc}
0\\
\frac{1}{4}\\
0
\end{array}\right)
\left(\begin{array}{ccc}
1& 0 & 0 \\
0& 1 & 0 \\
0& 0 & 1
\end{array}\right)
\Bigr)
\normalsize
\in \mathbb{A}(n)$ 
s.t. $\gamma \Gamma_{5}\gamma ^{-1}=\Gamma_{6}$.
%\end{itemize}
%--------------------------------------------------------------------
\\

\noindent {\bf d)}. 
First, we determine the holonomy groups of $\Gamma_1$ and $\Gamma_5$. If $L:\mathbb{E(}3)\longrightarrow O(3) $ is the projection,
then $L(\Gamma )=\Phi $ is called the holonomy group of $\Gamma $. Hence 
%===----- 
\footnotesize
\begin{align*}
\Phi_1=&
\Bigl <I,
\left(\begin{array}{ccc}
1& 0 & 0 \\
0& -1 & 0\\
0& 0 & 1 
\end{array}\right)\Bigr>=\mathbb{Z}_2\\
\Phi_5=&
\Bigl<I,
\left(\begin{array}{ccc}
1& 0 & 0 \\
0& -1 & 0 \\
0& 0 & 1
\end{array}\right),
\left(\begin{array}{ccc}
1& 0 & 0 \\
0& 1 & 0 \\
0& 0 & -1 
\end{array}\right)
\Bigr >=\mathbb{Z}_2 \times \mathbb{Z}_2.
\end{align*} \normalsize
If $M(A_1)$ is diffeomorphic to $M(A_5)$, by Bieberbach's theorem, then $\exists$ $B\in GL(3,\mathbb{R})$
with $(b,B)=\gamma \in \mathbb{A}(3)$ such that $L(\gamma \Gamma_1 \gamma^{-1} )=BL(\Gamma_1)B^{-1}=L(\Gamma_5)$. 
This is impossible because $L(\Gamma_1)=\mathbb{Z}_2$ and $L(\Gamma_5)=\mathbb{Z}_2 \times \mathbb{Z}_2$.
\\

\noindent {\bf e)}. Let \footnotesize
$\Phi_1=
\Bigl <I,
\left(\begin{array}{ccc}
1& 0 & 0 \\
0& -1 & 0\\
0& 0 & 1 
\end{array}\right)=P \Bigr> \hspace{0.5cm}
\Phi_7=
\Bigl<I,
\left(\begin{array}{ccc}
1& 0 & 0 \\
0& -1 & 0 \\
0& 0 & -1
\end{array}\right)=Q \Bigr>$. \normalsize
If $M(A_1)$ is difffeomorphic to $M(A_7)$ then $\exists $ $B$ such that $BPB^{-1}=Q$. 
Consider $|BPB^{-1}|=|Q|$ or $|B||P-\lambda I||B^{-1}|=|Q-\lambda I|$. It means the eigenvalues of $P$ and $Q$
are the same. This is a contradiction, because the eigenvalues of $P$ are 1,1 and -1, but the eigenvalues of $Q$
are -1,-1 and 1.
\\

\noindent {\bf f}). The argument is similar with d), because the holonomy of $\Gamma_7$ is $\mathbb{Z}_2$.
\\

We observed that the sufficient condition of conjecture is true for $n=3$.
%=-------
\end{proof}

%===========================================================================
%\newpage
\section{Four Dimensional Real Bott Tower}

In this case there are 64 Bott matrices $A$ that correspond to action of $\mathbb {(Z}_2)^4$ on $(S^1)^4$.
Definition of the action is similar to the three dimensional real Bott tower.

\begin{thm}
Let $M(A)$ be a $4$-dimensional real Bott tower.
Then $M(A)$ admits a maximal $T^k$-action $(k=1,2,3,4)$.
There exists exactly 12-diffeomorphism classes of $M(A)$. The Bott
matrices among $64$ fall into the
following classes
\begin{itemize}
\item[{\bf (a)}](orientable) ($T^4$-action.) The Identity matrix $I_4.$
\item[{\bf (b)}](orientable) ($T^2$-action.)  \sz
\begin{align*}
\hspace{1cm}
A_{11}=&
\left(\begin{array}{cccc}
1& 0 & 0 & 0 \\
0& 1 & 1 & 1\\
0& 0 & 1 & 0\\
0& 0 & 0 & 1
\end{array}\right),
A_{21}=
\left(\begin{array}{cccc}
1& 0 & 1 & 1 \\
0& 1 & 0 & 0\\
0& 0 & 1 & 0\\
0& 0 & 0 & 1
\end{array}\right),
A_{31}=
\left(\begin{array}{cccc}
1& 0 & 1 & 1 \\
0& 1 & 1 & 1\\
0& 0 & 1 & 0\\
0& 0 & 0 & 1
\end{array}\right)\\
A_{a5}=&
\left(\begin{array}{cccc}
1& 1 & 0 & 1 \\
0& 1 & 0 & 0\\
0& 0 & 1 & 0\\
0& 0 & 0 & 1
\end{array}\right),
A_{a17}=
\left(\begin{array}{cccc}
1& 1 & 1 & 0 \\
0& 1 & 0 & 0\\
0& 0 & 1 & 0\\
0& 0 & 0 & 1
\end{array}\right) .
\end{align*} \normalsize
\item [{\bf (c)}](orientable) ($S^1$-action.)  \sz
\begin{align*}
A_{a15}=
\left(\begin{array}{cccc}
1& 1 & 0 & 1 \\
0& 1 & 1 & 1\\
0& 0 & 1 & 0\\
0& 0 & 0 & 1
\end{array}\right),
A_{a27}=
\left(\begin{array}{cccc}
1& 1 & 1 & 0 \\
0& 1 & 1 & 1\\
0& 0 & 1 & 0\\
0& 0 & 0 & 1
\end{array}\right) .
\end{align*} \normalsize
\item[{\bf (d)}](nonorientable) ($T^3$-action.)  \sz
\begin{align*}
\hspace{1cm}
A_{2}=&
\left(\begin{array}{cccc}
1& 0 & 0 & 0 \\
0& 1 & 0 & 0\\
0& 0 & 1 & 1\\
0& 0 & 0 & 1
\end{array}\right),
A_{3}=
\left(\begin{array}{cccc}
1& 0 & 0 & 0 \\
0& 1 & 0 & 1\\
0& 0 & 1 & 0\\
0& 0 & 0 & 1
\end{array}\right),
A_{4}=
\left(\begin{array}{cccc}
1& 0 & 0 & 0 \\
0& 1 & 0 & 1\\
0& 0 & 1 & 1\\
0& 0 & 0 & 1
\end{array}\right)\\
A_{5}=&
\left(\begin{array}{cccc}
1& 0 & 0 & 1 \\
0& 1 & 0 & 0\\
0& 0 & 1 & 0\\
0& 0 & 0 & 1
\end{array}\right),
A_{6}=
\left(\begin{array}{cccc}
1& 0 & 0 & 1 \\
0& 1 & 0 & 0\\
0& 0 & 1 & 1\\
0& 0 & 0 & 1
\end{array}\right), 
A_{7}=
\left(\begin{array}{cccc}
1& 0 & 0 & 1 \\
0& 1 & 0 & 1\\
0& 0 & 1 & 0\\
0& 0 & 0 & 1
\end{array}\right)\\
A_{8}=&
\left(\begin{array}{cccc}
1& 0 & 0 & 1 \\
0& 1 & 0 & 1\\
0& 0 & 1 & 1\\
0& 0 & 0 & 1
\end{array}\right), 
A_{9}=
\left(\begin{array}{cccc}
1& 0 & 0 & 0 \\
0& 1 & 1 & 0\\
0& 0 & 1 & 0\\
0& 0 & 0 & 1
\end{array}\right),
A_{17}=
\left(\begin{array}{cccc}
1& 0 & 1 & 0 \\
0& 1 & 0 & 0\\
0& 0 & 1 & 0\\
0& 0 & 0 & 1
\end{array}\right)\\
A_{25}=&
\left(\begin{array}{cccc}
1& 0 & 1 & 0 \\
0& 1 & 1 & 0\\
0& 0 & 1 & 0\\
0& 0 & 0 & 1
\end{array}\right),
A_{a1}=
\left(\begin{array}{cccc}
1& 1 & 0 & 0 \\
0& 1 & 0 & 0\\
0& 0 & 1 & 0\\
0& 0 & 0 & 1
\end{array}\right).
\end{align*} \normalsize
\item[{\bf (e)}](nonorientable) ($S^1$-action.)  \sz
$A_{a21}=\left(\begin{array}{cccc}
1& 1 & 1 & 1 \\
0& 1 & 0 & 0\\
0& 0 & 1 & 0\\
0& 0 & 0 & 1
\end{array}\right)$. \normalsize
\item[{\bf (f)}](nonorientable) ($T^2$-action.) \sz
\begin{align*}
\hspace{1cm}
A_{a3}=&
\left(\begin{array}{cccc}
1& 1 & 0 & 0 \\
0& 1 & 0 & 1\\
0& 0 & 1 & 0\\
0& 0 & 0 & 1
\end{array}\right),
A_{a7}=
\left(\begin{array}{cccc}
1& 1 & 0 & 1 \\
0& 1 & 0 & 1\\
0& 0 & 1 & 0\\
0& 0 & 0 & 1
\end{array}\right),
A_{10}=
\left(\begin{array}{cccc}
1& 0 & 0 & 0 \\
0& 1 & 1 & 0\\
0& 0 & 1 & 1\\
0& 0 & 0 & 1
\end{array}\right)\\
A_{12}=&
\left(\begin{array}{cccc}
1& 0 & 0 & 0 \\
0& 1 & 1 & 1\\
0& 0 & 1 & 1\\
0& 0 & 0 & 1
\end{array}\right),
A_{18}=
\left(\begin{array}{cccc}
1& 0 & 1 & 0 \\
0& 1 & 0 & 0\\
0& 0 & 1 & 1\\
0& 0 & 0 & 1
\end{array}\right), 
A_{22}=
\left(\begin{array}{cccc}
1& 0 & 1 & 1 \\
0& 1 & 0 & 0\\
0& 0 & 1 & 1\\
0& 0 & 0 & 1
\end{array}\right)\\
A_{26}=&
\left(\begin{array}{cccc}
1& 0 & 1 & 0 \\
0& 1 & 1 & 0\\
0& 0 & 1 & 1\\
0& 0 & 0 & 1
\end{array}\right), 
A_{32}=
\left(\begin{array}{cccc}
1& 0 & 1 & 1 \\
0& 1 & 1 & 1\\
0& 0 & 1 & 1\\
0& 0 & 0 & 1
\end{array}\right),
A_{a9}=
\left(\begin{array}{cccc}
1& 1 & 0 & 0 \\
0& 1 & 1 & 0\\
0& 0 & 1 & 0\\
0& 0 & 0 & 1
\end{array}\right)\\
A_{a25}=&
\left(\begin{array}{cccc}
1& 1 & 1 & 0 \\
0& 1 & 1 & 0\\
0& 0 & 1 & 0\\
0& 0 & 0 & 1
\end{array}\right).
\end{align*} \normalsize
\item[{\bf (g)}](nonorientable)($T^2$-action.) \sz
\begin{align*}
\hspace{1cm}
A_{a4}=&
\left(\begin{array}{cccc}
1& 1 & 0 & 0 \\
0& 1 & 0 & 1\\
0& 0 & 1 & 1\\
0& 0 & 0 & 1
\end{array}\right),
A_{a8}=
\left(\begin{array}{cccc}
1& 1 & 0 & 1 \\
0& 1 & 0 & 1\\
0& 0 & 1 & 1\\
0& 0 & 0 & 1
\end{array}\right),
A_{14}=
\left(\begin{array}{cccc}
1& 0 & 0 & 1 \\
0& 1 & 1 & 0\\
0& 0 & 1 & 1\\
0& 0 & 0 & 1
\end{array}\right)\\
A_{16}=&
\left(\begin{array}{cccc}
1& 0 & 0 & 1 \\
0& 1 & 1 & 1\\
0& 0 & 1 & 1\\
0& 0 & 0 & 1
\end{array}\right),
A_{20}=
\left(\begin{array}{cccc}
1& 0 & 1 & 0 \\
0& 1 & 0 & 1\\
0& 0 & 1 & 1\\
0& 0 & 0 & 1
\end{array}\right), 
A_{24}=
\left(\begin{array}{cccc}
1& 0 & 1 & 1 \\
0& 1 & 0 & 1\\
0& 0 & 1 & 1\\
0& 0 & 0 & 1
\end{array}\right)\\
A_{28}=&
\left(\begin{array}{cccc}
1& 0 & 1 & 0 \\
0& 1 & 1 & 1\\
0& 0 & 1 & 1\\
0& 0 & 0 & 1
\end{array}\right), 
A_{30}=
\left(\begin{array}{cccc}
1& 0 & 1 & 1 \\
0& 1 & 1 & 0\\
0& 0 & 1 & 1\\
0& 0 & 0 & 1
\end{array}\right).
\end{align*} \normalsize
\item[{\bf (h)}](nonorientable) ($S^1$-action.) \sz
\begin{align*}
\hspace{1.3cm}
A_{a13}=&
\left(\begin{array}{cccc}
1& 1 & 0 & 1 \\
0& 1 & 1 & 0\\
0& 0 & 1 & 0\\
0& 0 & 0 & 1
\end{array}\right),
A_{a29}=
\left(\begin{array}{cccc}
1& 1 & 1 & 1 \\
0& 1 & 1 & 0\\
0& 0 & 1 & 0\\
0& 0 & 0 & 1
\end{array}\right),
A_{a18}=
\left(\begin{array}{cccc}
1& 1 & 1 & 0 \\
0& 1 & 0 & 0\\
0& 0 & 1 & 1\\
0& 0 & 0 & 1
\end{array}\right)\\
A_{a19}=&
\left(\begin{array}{cccc}
1& 1 & 1 & 0 \\
0& 1 & 0 & 1\\
0& 0 & 1 & 0\\
0& 0 & 0 & 1
\end{array}\right),
A_{a20}=
\left(\begin{array}{cccc}
1& 1 & 1 & 0 \\
0& 1 & 0 & 1\\
0& 0 & 1 & 1\\
0& 0 & 0 & 1
\end{array}\right), 
A_{a22}=
\left(\begin{array}{cccc}
1& 1 & 1 & 1 \\
0& 1 & 0 & 0\\
0& 0 & 1 & 1\\
0& 0 & 0 & 1
\end{array}\right)\\
A_{a23}=&
\left(\begin{array}{cccc}
1& 1 & 1 & 1 \\
0& 1 & 0 & 1\\
0& 0 & 1 & 0\\
0& 0 & 0 & 1
\end{array}\right), 
A_{a24}=
\left(\begin{array}{cccc}
1& 1 & 1 & 1 \\
0& 1 & 0 & 1\\
0& 0 & 1 & 1\\
0& 0 & 0 & 1
\end{array}\right).
\end{align*} \normalsize
\item[{\bf (i)}] (nonorientable)($S^1$-action.) \\ 
\sz
$A_{a11}=
\left(\begin{array}{cccc}
1& 1 & 0 & 0\\
0& 1 & 1 & 1\\
0& 0 & 1 & 0\\
0& 0 & 0 & 1
\end{array}\right),
A_{a31}=
\left(\begin{array}{cccc}
1& 1 & 1 & 1\\
0& 1 & 1 & 1\\
0& 0 & 1 & 0\\
0& 0 & 0 & 1
\end{array}\right).$ \vspace{0.2cm}
\normalsize
\item[{\bf (j)}](nonorientable) ($T^2$-action.) \sz
\begin{align*}
\hspace{1cm}
A_{13}=&
\left(\begin{array}{cccc}
1& 0 & 0 & 1 \\
0& 1 & 1 & 0\\
0& 0 & 1 & 0\\
0& 0 & 0 & 1
\end{array}\right),
A_{15}=
\left(\begin{array}{cccc}
1& 0 & 0 & 1 \\
0& 1 & 1 & 1\\
0& 0 & 1 & 0\\
0& 0 & 0 & 1
\end{array}\right),
A_{19}=
\left(\begin{array}{cccc}
1& 0 & 1 & 0 \\
0& 1 & 0 & 1\\
0& 0 & 1 & 0\\
0& 0 & 0 & 1
\end{array}\right)\\
A_{23}=&
\left(\begin{array}{cccc}
1& 0 & 1 & 1 \\
0& 1 & 0 & 1\\
0& 0 & 1 & 0\\
0& 0 & 0 & 1
\end{array}\right),
A_{27}=
\left(\begin{array}{cccc}
1& 0 & 1 & 0 \\
0& 1 & 1 & 1\\
0& 0 & 1 & 0\\
0& 0 & 0 & 1
\end{array}\right), 
A_{29}=
\left(\begin{array}{cccc}
1& 0 & 1 & 1 \\
0& 1 & 1 & 0\\
0& 0 & 1 & 0\\
0& 0 & 0 & 1
\end{array}\right)\\
A_{a2}=&
\left(\begin{array}{cccc}
1& 1 & 0 & 0 \\
0& 1 & 0 & 0\\
0& 0 & 1 & 1\\
0& 0 & 0 & 1
\end{array}\right), 
A_{a6}=
\left(\begin{array}{cccc}
1& 1 & 0 & 1 \\
0& 1 & 0 & 0\\
0& 0 & 1 & 1\\
0& 0 & 0 & 1
\end{array}\right).
\end{align*} \normalsize
\item[{\bf (k)}](nonorientable) ($S^1$-action.) \sz
\begin{align*}
\hspace{1.3cm}
A_{a10}=&
\left(\begin{array}{cccc}
1& 1 & 0 & 0 \\
0& 1 & 1 & 0\\
0& 0 & 1 & 1\\
0& 0 & 0 & 1
\end{array}\right),
A_{a12}=
\left(\begin{array}{cccc}
1& 1 & 0 & 0 \\
0& 1 & 1 & 1\\
0& 0 & 1 & 1\\
0& 0 & 0 & 1
\end{array}\right),
A_{a26}=
\left(\begin{array}{cccc}
1& 1 & 1 & 0 \\
0& 1 & 1 & 0\\
0& 0 & 1 & 1\\
0& 0 & 0 & 1
\end{array}\right)\\
A_{a32}=&
\left(\begin{array}{cccc}
1& 1 & 1 & 1 \\
0& 1 & 1 & 1\\
0& 0 & 1 & 1\\
0& 0 & 0 & 1
\end{array}\right).
\end{align*} \normalsize
\item[{\bf (l)}](nonorientable)($S^1$-action.) \sz
\begin{align*} 
\hspace{1.3cm}
A_{a14}=&
\left(\begin{array}{cccc}
1& 1 & 0 & 1 \\
0& 1 & 1 & 0\\
0& 0 & 1 & 1\\
0& 0 & 0 & 1
\end{array}\right),
A_{a16}=
\left(\begin{array}{cccc}
1& 1 & 0 & 1 \\
0& 1 & 1 & 1\\
0& 0 & 1 & 1\\
0& 0 & 0 & 1
\end{array}\right),
A_{a28}=
\left(\begin{array}{cccc}
1& 1 & 1 & 0 \\
0& 1 & 1 & 1\\
0& 0 & 1 & 1\\
0& 0 & 0 & 1
\end{array}\right)\\
A_{a30}=&
\left(\begin{array}{cccc}
1& 1 & 1 & 1 \\
0& 1 & 1 & 0\\
0& 0 & 1 & 1\\
0& 0 & 0 & 1
\end{array}\right).
\end{align*} \normalsize
\end{itemize}
Moreover, Bott matrices in each class are conjugate.
\end{thm}
%===============================================================
\begin{proof} Similar with 3 dimensional real Bott tower, the proof is organized by the following steps.
           \begin{enumerate}
	     \item Determine conjugacy classes of 64 Bott matrices 
			(we used Maple to determine this).
           \item For each conjugacy class, we check whether 
			$M(A), M(A')$ with representative matrices $A,A'$ are diffeomorphic or not
			by using Bieberbach theorems.
           \end{enumerate}
A detailed proof is as follows:
\begin{itemize}
\item[(1).] $I_4$ is not diffeomorphic to any $M(A)$.
\item[(2).] $M(A_{11})\approx M(A_{21})\approx M(A_{31})\approx M(A_{a5})\approx M(A_{a17})$.
\item[(3).] $M(A_{a15})\approx M(A_{a27})$.
\item[(4).] $M(A_{11})$ is not diffeomorphic to $M(A_{a15})$.
\item[(5).] $M(A_{2})\approx M(A_{3})\approx M(A_{4})\approx M(A_{5})\approx M(A_{6})
             \approx M(A_{7})\approx M(A_{8})\approx M(A_{9})\approx M(A_{17})\approx M(A_{25})\approx M(A_{a1})$.
\item[(6).] $M(A_{11})$ is not diffeomorphic to $M(A_{2})$ and $M(A_{a21})$.
\item[(7).] $M(A_{a15})$ is not diffeomorphic to $M(A_{2})$ and $M(A_{a21})$.
\item[(8).] $M(A_{2})$ is not diffeomorphic to $M(A_{a21})$.
\item[(9).] $M(A_{a3})\approx M(A_{a7})\approx M(A_{10})\approx M(A_{12})\approx M(A_{18})
             \approx M(A_{22})\approx M(A_{26})\approx M(A_{32})\approx M(A_{a9})\approx M(A_{a25})$.
\item[(10).] $M(A_{a3})$ is not diffeomorphic to $M(A_{11})$, $M(A_{2})$ and $M(A_{a21})$.
\item[(11).] $M(A_{a3})$ is not diffeomorphic to $M(A_{a15})$.
\item[(12).] $M(A_{a4})\approx M(A_{a8})\approx M(A_{14})\approx M(A_{16})\approx M(A_{20})
             \approx M(A_{24})\approx M(A_{28})\approx M(A_{30})$.
\item[(13).] $M(A_{a4})$ is not diffeomorphic to $M(A_{11})$, $M(A_{2})$ and $M(A_{a21})$.
\item[(14).] $M(A_{a4})$ is not diffeomorphic to $M(A_{a15})$.
\item[(15).] $M(A_{14})$ is not diffeomorphic to $M(A_{a3})$.
\item[(16).] $M(A_{a13})\approx M(A_{a29})\approx M(A_{a18})\approx M(A_{a19})\approx M(A_{a20})
             \approx M(A_{a22})\approx M(A_{a23})\approx M(A_{a24})$.
\item[(17).] $M(A_{a13})$ is not diffeomorphic to $M(A_{11})$, $M(A_{2})$ and $M(A_{a21})$.
\item[(18).] $M(A_{a13})$ is not diffeomorphic to $M(A_{a15})$, $M(A_{a3})$ and $M(A_{a4})$.
\item[(19).] $M(A_{a11})\approx M(A_{a31})$.
\item[(20).] $M(A_{a11})$ is not diffeomorphic to $M(A_{11})$, $M(A_{2})$ and $M(A_{a21})$.
\item[(21).] $M(A_{a11})$ is not diffeomorphic to $M(A_{a15})$, $M(A_{a3})$ and $M(A_{a4})$ .
\item[(22).] $M(A_{a11})$ is not diffeomorphic to $M(A_{a13})$.
\item[(23).] $M(A_{13})\approx M(A_{15})\approx M(A_{19})\approx M(A_{23})\approx M(A_{27})
             \approx M(A_{29})\approx M(A_{a2})\approx M(A_{a6})$.
\item[(24).] $M(A_{13})$ is not diffeomorphic to $M(A_{11})$, $M(A_{2})$ and $M(A_{a21})$.
\item[(25).] $M(A_{13})$ is not diffeomorphic to $M(A_{a15})$, $M(A_{a13})$ and $M(A_{a11})$.
\item[(26).] $M(A_{13})$ is not diffeomorphic to $M(A_{a3})$.
\item[(27).] $M(A_{13})$ is not diffeomorphic to $M(A_{a4})$.
\item[(28).] $M(A_{a10})\approx M(A_{a12})\approx M(A_{a26})\approx M(A_{a32})$. and \\
             $M(A_{a14})\approx M(A_{a16})\approx M(A_{a28})\approx M(A_{a30})$.
\item[(29).] $M(A_{a10})$ and $M(A_{a14})$ are not diffeomorphic to any $M(A)$ in classes {\bf a)-j)} as written in theorem.
\item[(30).] $M(A_{a10})$ is not diffeomorphic to $M(A_{a14})$.
\end{itemize}
\noindent {\bf (1).} This is clear because the holonomy of $I_4$ is trivial.\\

\noindent {\bf (2)}.\\
$\bullet$ $M(A_{11})\approx  M(A_{21})$.\\
For \fz
\begin{align*}
&A_{11}=
\left(\begin{array}{cccc}
1& 0 & 0 & 0 \\
0& 1 & 1 & 1\\
0& 0 & 1 & 0\\
0& 0 & 0 & 1
\end{array}\right)
\hspace{.1cm}
&A_{21}=
\left(\begin{array}{cccc}
1& 0 & 1 & 1 \\
0& 1 & 0 & 0\\
0& 0 & 1 & 0\\
0& 0 & 0 & 1
\end{array}\right)
\end{align*} \fz
\begin{align*}
&g_1(z_1,z_2,z_3,z_4)=(-z_1, z_2,z_3,z_4). 			 &h_1(z_1,z_2,z_3,z_4)=(-z_1, z_2,\bar z_3,\bar z_4)\\
&g_2(z_1,z_2,z_3,z_4)=(z_1,-z_2,\bar z_3,\bar z_4).          &h_2(z_1,z_2,z_3,z_4)=(z_1, -z_2,z_3,z_4)\\
&g_3(z_1,z_2,z_3,z_4)=(z_1,z_2,-z_3,z_4).                    &h_3(z_1,z_2,z_3,z_4)=(z_1, z_2,-z_3,z_4)\\
&g_4(z_1,z_2,z_3,z_4)=(z_1,z_2,z_3,-z_4).                    &h_4(z_1,z_2,z_3,z_4)=(z_1, z_2,z_3,-z_4)
\end{align*}
\normalsize then \fz
\begin{align*}
\Gamma_{11}=
\Bigl < &\left(\begin{array}{cccc}
\frac{1}{2}\\
0\\
0\\
0
\end{array}\right)
\left(\begin{array}{cccc}
1& 0 & 0 & 0\\
0& 1 & 0 & 0\\
0& 0 & 1 & 0\\
0& 0 & 0 & 1
\end{array}\right),
\left(\begin{array}{cccc}
0\\
\frac{1}{2}\\
0\\
0
\end{array}\right)
\left(\begin{array}{cccc}
1& 0 & 0 & 0\\
0& 1 & 0 & 0\\
0& 0 & -1 & 0\\
0& 0 & 0 & -1
\end{array}\right),
\\
&\left(\begin{array}{cccc}
0\\
0\\
\frac{1}{2}\\
0
\end{array}\right)
\left(\begin{array}{cccc}
1& 0 & 0 & 0\\
0& 1 & 0 & 0\\
0& 0 & 1 & 0\\
0& 0 & 0 & 1
\end{array}\right),
\left(\begin{array}{cccc}
0\\
0\\
0\\
\frac{1}{2}
\end{array}\right)
\left(\begin{array}{cccc}
1& 0 & 0 & 0\\
0& 1 & 0 & 0\\
0& 0 & 1 & 0\\
0& 0 & 0 & 1
\end{array}\right)
\Bigr >
\end{align*}

\begin{align*}
\Gamma_{21}= 
\Bigl < & \left(\begin{array}{cccc}
\frac{1}{2}\\
0\\
0\\
0
\end{array}\right)
\left(\begin{array}{cccc}
1& 0 & 0 & 0\\
0& 1 & 0 & 0\\
0& 0 & -1 & 0\\
0& 0 & 0 & -1
\end{array}\right),
\left(\begin{array}{cccc}
0\\
\frac{1}{2}\\
0\\
0
\end{array}\right)
\left(\begin{array}{cccc}
1& 0 & 0 & 0\\
0& 1 & 0 & 0\\
0& 0 & 1 & 0\\
0& 0 & 0 & 1
\end{array}\right),
\\
&\left(\begin{array}{cccc}
0\\
0\\
\frac{1}{2}\\
0
\end{array}\right)
\left(\begin{array}{cccc}
1& 0 & 0 & 0\\
0& 1 & 0 & 0\\
0& 0 & 1 & 0\\
0& 0 & 0 & 1
\end{array}\right),
\left(\begin{array}{cccc}
0\\
0\\
0\\
\frac{1}{2}
\end{array}\right)
\left(\begin{array}{cccc}
1& 0 & 0 & 0\\
0& 1 & 0 & 0\\
0& 0 & 1 & 0\\
0& 0 & 0 & 1
\end{array}\right)
\Bigr >.
\end{align*}

\normalsize
\noindent
Let $\varphi(z_1,z_2,z_3,z_4)=(z_2,z_1,z_3,z_4)$, we get these commutative diagrams

\footnotesize
\[
\begin{CD}
(z_1,z_2,z_3,z_4) @>\varphi>>(z_2,z_1,z_3,z_4)\\
@Vg_1VV @Vh_2 VV\\
(-z_1,z_2,z_3,z_4) @>\varphi>> (z_2,-z_1,z_3,z_4)
\end{CD}
\hspace{.5cm}
\begin{CD}
(z_1,z_2,z_3,z_4) @>\varphi>>(z_2,z_1,z_3,z_4)\\
@Vg_3VV @Vh_3 VV\\
(z_1,z_2,-z_3,z_4) @>\varphi>> (z_2,z_1,-z_3,z_4)
\end{CD}
\]
\[ 
\begin{CD}
(z_1,z_2,z_3,z_4) @>\varphi>>(z_2,z_1,z_3,z_4)\\
@Vg_2VV @Vh_1 VV\\
(z_1,-z_2,\bar z_3,\bar z_4) @>\varphi>> (-z_2,z_1,\bar z_3,\bar z_4)
\end{CD}
\hspace{.5cm}
\begin{CD}
(z_1,z_2,z_3,z_4) @>\varphi>>(z_2,z_1,z_3,z_4)\\
@Vg_4VV @Vh_4 VV\\
(z_1,z_2,z_3,-z_4) @>\varphi>> (z_2,z_1,z_3,-z_4).
\end{CD}
\]

\normalsize
\noindent 
Therefore $\exists$ 
$\gamma 
\fz
=
\Bigl( \left(\begin{array}{cccc}
0\\
0\\
0\\
0
\end{array}\right)
\left(\begin{array}{cccc}
0& 1 & 0 & 0\\
1& 0 & 0 & 0\\
0& 0 & 1 & 0\\
0& 0 & 0 & 1
\end{array}\right)
\Bigr)
\normalsize
\in \mathbb{A}(n)$ 
s.t. $\gamma \Gamma_{11}\gamma ^{-1}=\Gamma_{21}$.\\
%--------------------------------------------------------------------

\noindent 
$\bullet M(A_{11}) \approx M(A_{31})$.\\
For
\footnotesize
\begin{align*}
&A_{11}=
\left(\begin{array}{cccc}
1& 0 & 0 & 0 \\
0& 1 & 1 & 1\\
0& 0 & 1 & 0\\
0& 0 & 0 & 1
\end{array}\right)
\hspace{.1cm}
&A_{31}&=
\left(\begin{array}{cccc}
1& 0 & 1 & 1 \\
0& 1 & 1 & 1\\
0& 0 & 1 & 0\\
0& 0 & 0 & 1
\end{array}\right)&
\\
&g_1(z_1,z_2,z_3,z_4)=(-z_1, z_2,z_3,z_4). \hspace{.1cm}	 &h_1(z_1,z_2,z_3,z_4)&=(-z_1, z_2,\bar z_3,\bar z_4)\\
&g_2(z_1,z_2,z_3,z_4)=(z_1,-z_2,\bar z_3,\bar z_4).          &h_2h_1(z_1,z_2,z_3,z_4)&=(-z_1, -z_2,z_3,z_4)\\
&g_3(z_1,z_2,z_3,z_4)=(z_1,z_2,-z_3,z_4).                    &h_3(z_1,z_2,z_3,z_4)&=(z_1, z_2,-z_3,z_4)\\
&g_4(z_1,z_2,z_3,z_4)=(z_1,z_2,z_3,-z_4).                    &h_4(z_1,z_2,z_3,z_4)&=(z_1, z_2,z_3,-z_4)
\end{align*}
\normalsize then \footnotesize
\begin{align*}
\Gamma_{11}=
\Bigl < &\left(\begin{array}{cccc}
\frac{1}{2}\\
0\\
0\\
0
\end{array}\right)
\left(\begin{array}{cccc}
1& 0 & 0 & 0\\
0& 1 & 0 & 0\\
0& 0 & 1 & 0\\
0& 0 & 0 & 1
\end{array}\right),
\left(\begin{array}{cccc}
0\\
\frac{1}{2}\\
0\\
0
\end{array}\right)
\left(\begin{array}{cccc}
1& 0 & 0 & 0\\
0& 1 & 0 & 0\\
0& 0 & -1 & 0\\
0& 0 & 0 & -1
\end{array}\right),
\\
&\left(\begin{array}{cccc}
0\\
0\\
\frac{1}{2}\\
0
\end{array}\right)
\left(\begin{array}{cccc}
1& 0 & 0 & 0\\
0& 1 & 0 & 0\\
0& 0 & 1 & 0\\
0& 0 & 0 & 1
\end{array}\right),
\left(\begin{array}{cccc}
0\\
0\\
0\\
\frac{1}{2}
\end{array}\right)
\left(\begin{array}{cccc}
1& 0 & 0 & 0\\
0& 1 & 0 & 0\\
0& 0 & 1 & 0\\
0& 0 & 0 & 1
\end{array}\right)
\Bigr >
\end{align*}

\begin{align*}
\Gamma_{31}= 
\Bigl < & \left(\begin{array}{cccc}
\frac{1}{2}\\
0\\
0\\
0
\end{array}\right)
\left(\begin{array}{cccc}
1& 0 & 0 & 0\\
0& 1 & 0 & 0\\
0& 0 & -1 & 0\\
0& 0 & 0 & -1
\end{array}\right),
\left(\begin{array}{cccc}
\frac{1}{2}\\
\frac{1}{2}\\
0\\
0
\end{array}\right)
\left(\begin{array}{cccc}
1& 0 & 0 & 0\\
0& 1 & 0 & 0\\
0& 0 & 1 & 0\\
0& 0 & 0 & 1
\end{array}\right),
\\
&\left(\begin{array}{cccc}
0\\
0\\
\frac{1}{2}\\
0
\end{array}\right)
\left(\begin{array}{cccc}
1& 0 & 0 & 0\\
0& 1 & 0 & 0\\
0& 0 & 1 & 0\\
0& 0 & 0 & 1
\end{array}\right),
\left(\begin{array}{cccc}
0\\
0\\
0\\
\frac{1}{2}
\end{array}\right)
\left(\begin{array}{cccc}
1& 0 & 0 & 0\\
0& 1 & 0 & 0\\
0& 0 & 1 & 0\\
0& 0 & 0 & 1
\end{array}\right)
\Bigr >.
\end{align*}

\normalsize
\noindent
Let $\varphi(z_1,z_2,z_3,z_4)=(z_1z_2,z_1,z_3,z_4)$, we get these commutative diagrams

\footnotesize
\[
\begin{CD}
(z_1,z_2,z_3,z_4) @>\varphi>>(z_1z_2,z_1,z_3,z_4)\\
@Vg_1VV @Vh_2h_1 VV\\
(-z_1,z_2,z_3,z_4) @>\varphi>> (-z_1z_2,-z_1,z_3,z_4)
\end{CD}
\hspace{.1cm}
\begin{CD}
(z_1,z_2,z_3,z_4) @>\varphi>>(z_1z_2,z_1,z_3,z_4)\\
@Vg_3VV @Vh_3 VV\\
(z_1,z_2,-z_3,z_4) @>\varphi>> (z_1z_2,z_1,-z_3,z_4)
\end{CD}
\]
\[ 
\begin{CD}
(z_1,z_2,z_3,z_4) @>\varphi>>(z_1z_2,z_1,z_3,z_4)\\
@Vg_2VV @Vh_1 VV\\
(z_1,-z_2,\bar z_3,\bar z_4) @>\varphi>> (-z_1z_2,z_1,\bar z_3,\bar z_4)
\end{CD}
\hspace{.1cm}
\begin{CD}
(z_1,z_2,z_3,z_4) @>\varphi>>(z_1z_2,z_1,z_3,z_4)\\
@Vg_4VV @Vh_4 VV\\
(z_1,z_2,z_3,-z_4) @>\varphi>> (z_1z_2,z_1,z_3,-z_4).
\end{CD}
\]

\normalsize
\noindent 
Therefore $\exists$ 
$\gamma 
\footnotesize
=
\Bigl( \left(\begin{array}{cccc}
0\\
0\\
0\\
0
\end{array}\right)
\left(\begin{array}{cccc}
1& 1 & 0 & 0\\
1& 0 & 0 & 0\\
0& 0 & 1 & 0\\
0& 0 & 0 & 1
\end{array}\right)
\Bigr)
\normalsize
\in \mathbb{A}(n)$ 
s.t. $\gamma \Gamma_{11}\gamma ^{-1}=\Gamma_{31}$.\\
%--------------------------------------------------------------------

\noindent 
$\bullet M(A_{11})\approx  M(A_{a5})$.\\
For
\footnotesize
\begin{align*}
&A_{11}=
\left(\begin{array}{cccc}
1& 0 & 0 & 0 \\
0& 1 & 1 & 1\\
0& 0 & 1 & 0\\
0& 0 & 0 & 1
\end{array}\right)
\hspace{.1cm}
&A_{a5}=
\left(\begin{array}{cccc}
1& 1 & 0 & 1 \\
0& 1 & 0 & 0\\
0& 0 & 1 & 0\\
0& 0 & 0 & 1
\end{array}\right)
\\
&g_1(z_1,z_2,z_3,z_4)=(-z_1, z_2,z_3,z_4). \hspace{.1cm}	 &h_1(z_1,z_2,z_3,z_4)=(-z_1,\bar  z_2,z_3,\bar z_4)\\
&g_2(z_1,z_2,z_3,z_4)=(z_1,-z_2,\bar z_3,\bar z_4).          &h_2(z_1,z_2,z_3,z_4)=(z_1, -z_2,z_3,z_4)\\
&g_3(z_1,z_2,z_3,z_4)=(z_1,z_2,-z_3,z_4).                    &h_3(z_1,z_2,z_3,z_4)=(z_1, z_2,-z_3,z_4)\\
&g_4(z_1,z_2,z_3,z_4)=(z_1,z_2,z_3,-z_4).                    &h_4(z_1,z_2,z_3,z_4)=(z_1, z_2,z_3,-z_4)
\end{align*}
\normalsize then \footnotesize
\begin{align*}
\Gamma_{11}=
\Bigl < &\left(\begin{array}{cccc}
\frac{1}{2}\\
0\\
0\\
0
\end{array}\right)
\left(\begin{array}{cccc}
1& 0 & 0 & 0\\
0& 1 & 0 & 0\\
0& 0 & 1 & 0\\
0& 0 & 0 & 1
\end{array}\right),
\left(\begin{array}{cccc}
0\\
\frac{1}{2}\\
0\\
0
\end{array}\right)
\left(\begin{array}{cccc}
1& 0 & 0 & 0\\
0& 1 & 0 & 0\\
0& 0 & -1 & 0\\
0& 0 & 0 & -1
\end{array}\right),
\\
&\left(\begin{array}{cccc}
0\\
0\\
\frac{1}{2}\\
0
\end{array}\right)
\left(\begin{array}{cccc}
1& 0 & 0 & 0\\
0& 1 & 0 & 0\\
0& 0 & 1 & 0\\
0& 0 & 0 & 1
\end{array}\right),
\left(\begin{array}{cccc}
0\\
0\\
0\\
\frac{1}{2}
\end{array}\right)
\left(\begin{array}{cccc}
1& 0 & 0 & 0\\
0& 1 & 0 & 0\\
0& 0 & 1 & 0\\
0& 0 & 0 & 1
\end{array}\right)
\Bigr >
\end{align*}

\begin{align*}
\Gamma_{a5}= 
\Bigl < & \left(\begin{array}{cccc}
\frac{1}{2}\\
0\\
0\\
0
\end{array}\right)
\left(\begin{array}{cccc}
1& 0 & 0 & 0\\
0& -1 & 0 & 0\\
0& 0 & 1 & 0\\
0& 0 & 0 & -1
\end{array}\right),
\left(\begin{array}{cccc}
0\\
\frac{1}{2}\\
0\\
0
\end{array}\right)
\left(\begin{array}{cccc}
1& 0 & 0 & 0\\
0& 1 & 0 & 0\\
0& 0 & 1 & 0\\
0& 0 & 0 & 1
\end{array}\right),
\\
&\left(\begin{array}{cccc}
0\\
0\\
\frac{1}{2}\\
0
\end{array}\right)
\left(\begin{array}{cccc}
1& 0 & 0 & 0\\
0& 1 & 0 & 0\\
0& 0 & 1 & 0\\
0& 0 & 0 & 1
\end{array}\right),
\left(\begin{array}{cccc}
0\\
0\\
0\\
\frac{1}{2}
\end{array}\right)
\left(\begin{array}{cccc}
1& 0 & 0 & 0\\
0& 1 & 0 & 0\\
0& 0 & 1 & 0\\
0& 0 & 0 & 1
\end{array}\right)
\Bigr >.
\end{align*}

\normalsize
\noindent
Let $\varphi(z_1,z_2,z_3,z_4)=(z_2,z_3,z_1,z_4)$, we get these commutative diagrams

\footnotesize
\[
\begin{CD}
(z_1,z_2,z_3,z_4) @>\varphi>>(z_2,z_3,z_1,z_4)\\
@Vg_1VV @Vh_3 VV\\
(-z_1,z_2,z_3,z_4) @>\varphi>> (z_2,z_3,-z_1,z_4)
\end{CD}
\hspace{.1cm}
\begin{CD}
(z_1,z_2,z_3,z_4) @>\varphi>>(z_2,z_3,z_1,z_4)\\
@Vg_3VV @Vh_2 VV\\
(z_1,z_2,-z_3,z_4) @>\varphi>> (z_2,-z_3,z_1,z_4)
\end{CD}
\]
\[ 
\begin{CD}
(z_1,z_2,z_3,z_4) @>\varphi>>(z_2,z_3,z_1,z_4)\\
@Vg_2VV @Vh_1 VV\\
(z_1,-z_2,\bar z_3,\bar z_4) @>\varphi>> (-z_2,\bar z_3,z_1,\bar z_4)
\end{CD}
\hspace{.1cm}
\begin{CD}
(z_1,z_2,z_3,z_4) @>\varphi>>(z_2,z_3,z_1,z_4)\\
@Vg_4VV @Vh_4 VV\\
(z_1,z_2,z_3,-z_4) @>\varphi>> (z_2,z_3,z_1,-z_4).
\end{CD}
\]

\normalsize
\noindent 
Therefore $\exists$ 
$\gamma 
\footnotesize
=
\Bigl( \left(\begin{array}{cccc}
0\\
0\\
0\\
0
\end{array}\right)
\left(\begin{array}{cccc}
0& 1 & 0 & 0\\
0& 0 & 1 & 0\\
1& 0 & 0 & 0\\
0& 0 & 0 & 1
\end{array}\right)
\Bigr)
\normalsize
\in \mathbb{A}(n)$ 
s.t. $\gamma \Gamma_{11}\gamma ^{-1}=\Gamma_{a5}$.\\
%--------------------------------------------------------------------

\noindent 
$\bullet M(A_{11})\approx  M(A_{a17})$.\\
For
\footnotesize
\begin{align*}
&A_{11}=
\left(\begin{array}{cccc}
1& 0 & 0 & 0 \\
0& 1 & 1 & 1\\
0& 0 & 1 & 0\\
0& 0 & 0 & 1
\end{array}\right)
\hspace{.1cm}
&A_{a17}=
\left(\begin{array}{cccc}
1& 1 & 1 & 0 \\
0& 1 & 0 & 0\\
0& 0 & 1 & 0\\
0& 0 & 0 & 1
\end{array}\right)
\\
&g_1(z_1,z_2,z_3,z_4)=(-z_1, z_2,z_3,z_4). \hspace{.1cm}	 &h_1(z_1,z_2,z_3,z_4)=(-z_1,\bar  z_2,\bar z_3,z_4)\\
&g_2(z_1,z_2,z_3,z_4)=(z_1,-z_2,\bar z_3,\bar z_4).          &h_2(z_1,z_2,z_3,z_4)=(z_1, -z_2,z_3,z_4)\\
&g_3(z_1,z_2,z_3,z_4)=(z_1,z_2,-z_3,z_4).                    &h_3(z_1,z_2,z_3,z_4)=(z_1, z_2,-z_3,z_4)\\
&g_4(z_1,z_2,z_3,z_4)=(z_1,z_2,z_3,-z_4).                    &h_4(z_1,z_2,z_3,z_4)=(z_1, z_2,z_3,-z_4)
\end{align*}
\normalsize then \footnotesize
\begin{align*}
\Gamma_{11}=
\Bigl < &\left(\begin{array}{cccc}
\frac{1}{2}\\
0\\
0\\
0
\end{array}\right)
\left(\begin{array}{cccc}
1& 0 & 0 & 0\\
0& 1 & 0 & 0\\
0& 0 & 1 & 0\\
0& 0 & 0 & 1
\end{array}\right),
\left(\begin{array}{cccc}
0\\
\frac{1}{2}\\
0\\
0
\end{array}\right)
\left(\begin{array}{cccc}
1& 0 & 0 & 0\\
0& 1 & 0 & 0\\
0& 0 & -1 & 0\\
0& 0 & 0 & -1
\end{array}\right),
\\
&\left(\begin{array}{cccc}
0\\
0\\
\frac{1}{2}\\
0
\end{array}\right)
\left(\begin{array}{cccc}
1& 0 & 0 & 0\\
0& 1 & 0 & 0\\
0& 0 & 1 & 0\\
0& 0 & 0 & 1
\end{array}\right),
\left(\begin{array}{cccc}
0\\
0\\
0\\
\frac{1}{2}
\end{array}\right)
\left(\begin{array}{cccc}
1& 0 & 0 & 0\\
0& 1 & 0 & 0\\
0& 0 & 1 & 0\\
0& 0 & 0 & 1
\end{array}\right)
\Bigr >
\end{align*}

\begin{align*}
\Gamma_{a17}= 
\Bigl < & \left(\begin{array}{cccc}
\frac{1}{2}\\
0\\
0\\
0
\end{array}\right)
\left(\begin{array}{cccc}
1& 0 & 0 & 0\\
0& -1 & 0 & 0\\
0& 0 & -1 & 0\\
0& 0 & 0 & 1
\end{array}\right),
\left(\begin{array}{cccc}
0\\
\frac{1}{2}\\
0\\
0
\end{array}\right)
\left(\begin{array}{cccc}
1& 0 & 0 & 0\\
0& 1 & 0 & 0\\
0& 0 & 1 & 0\\
0& 0 & 0 & 1
\end{array}\right),
\\
&\left(\begin{array}{cccc}
0\\
0\\
\frac{1}{2}\\
0
\end{array}\right)
\left(\begin{array}{cccc}
1& 0 & 0 & 0\\
0& 1 & 0 & 0\\
0& 0 & 1 & 0\\
0& 0 & 0 & 1
\end{array}\right),
\left(\begin{array}{cccc}
0\\
0\\
0\\
\frac{1}{2}
\end{array}\right)
\left(\begin{array}{cccc}
1& 0 & 0 & 0\\
0& 1 & 0 & 0\\
0& 0 & 1 & 0\\
0& 0 & 0 & 1
\end{array}\right)
\Bigr >.
\end{align*}

\normalsize
\noindent
Let $\varphi(z_1,z_2,z_3,z_4)=(z_2,z_3,z_4,z_1)$, we get these commutative diagrams

\footnotesize
\[
\begin{CD}
(z_1,z_2,z_3,z_4) @>\varphi>>(z_2,z_3,z_4,z_1)\\
@Vg_1VV @Vh_4 VV\\
(-z_1,z_2,z_3,z_4) @>\varphi>> (z_2,z_3,z_4,-z_1)
\end{CD}
\hspace{.1cm}
\begin{CD}
(z_1,z_2,z_3,z_4) @>\varphi>>(z_2,z_3,z_4,z_1)\\
@Vg_3VV @Vh_2 VV\\
(z_1,z_2,-z_3,z_4) @>\varphi>> (z_2,-z_3,z_4,z_1)
\end{CD}
\]
\[ 
\begin{CD}
(z_1,z_2,z_3,z_4) @>\varphi>>(z_2,z_3,z_4,z_1)\\
@Vg_2VV @Vh_1 VV\\
(z_1,-z_2,\bar z_3,\bar z_4) @>\varphi>> (-z_2,\bar z_3,\bar z_4,z_1)
\end{CD}
\hspace{.1cm}
\begin{CD}
(z_1,z_2,z_3,z_4) @>\varphi>>(z_2,z_3,z_4,z_1)\\
@Vg_4VV @Vh_3 VV\\
(z_1,z_2,z_3,-z_4) @>\varphi>> (z_2,z_3,-z_4,z_1).
\end{CD}
\]

\normalsize
\noindent 
Therefore $\exists$ 
$\gamma 
\footnotesize
=
\Bigl( \left(\begin{array}{cccc}
0\\
0\\
0\\
0
\end{array}\right)
\left(\begin{array}{cccc}
0& 1 & 0 & 0\\
0& 0 & 1 & 0\\
0& 0 & 0 & 1\\
1& 0 & 0 & 0
\end{array}\right)
\Bigr)
\normalsize
\in \mathbb{A}(n)$ 
s.t. $\gamma \Gamma_{11}\gamma ^{-1}=\Gamma_{a17}$.\\
%--------------------------------------------------------------------
%=============================================================================

\noindent 
{\bf (3).}\\
\noindent 
$\bullet M(A_{a15})\approx  M(A_{a27})$.\\
For
\footnotesize
\begin{align*}
&A_{a15}=
\left(\begin{array}{cccc}
1& 1 & 0 & 1 \\
0& 1 & 1 & 1\\
0& 0 & 1 & 0\\
0& 0 & 0 & 1
\end{array}\right)
\hspace{.1cm}
&A_{a27}=
\left(\begin{array}{cccc}
1& 1 & 1 & 0 \\
0& 1 & 1 & 1\\
0& 0 & 1 & 0\\
0& 0 & 0 & 1
\end{array}\right)
\\
&g_1(z_1,z_2,z_3,z_4)=(-z_1,\bar z_2,z_3,\bar z_4).\hspace{.1cm}	 &h_1(z_1,z_2,z_3,z_4)=(-z_1,\bar  z_2,\bar z_3,z_4)\\
&g_2(z_1,z_2,z_3,z_4)=(z_1,-z_2,\bar z_3,\bar z_4).                &h_2(z_1,z_2,z_3,z_4)=(z_1,-z_2,\bar z_3,\bar z_4)\\
&g_3(z_1,z_2,z_3,z_4)=(z_1,z_2,-z_3,z_4).                          &h_3(z_1,z_2,z_3,z_4)=(z_1, z_2,-z_3,z_4)\\
&g_4(z_1,z_2,z_3,z_4)=(z_1,z_2,z_3,-z_4).                          &h_4(z_1,z_2,z_3,z_4)=(z_1, z_2,z_3,-z_4)
\end{align*}
\normalsize then \footnotesize
\begin{align*}
\Gamma_{a15}=
\Bigl < &\left(\begin{array}{cccc}
\frac{1}{2}\\
0\\
0\\
0
\end{array}\right)
\left(\begin{array}{cccc}
1& 0 & 0 & 0\\
0& -1 & 0 & 0\\
0& 0 & 1 & 0\\
0& 0 & 0 & -1
\end{array}\right),
\left(\begin{array}{cccc}
0\\
\frac{1}{2}\\
0\\
0
\end{array}\right)
\left(\begin{array}{cccc}
1& 0 & 0 & 0\\
0& 1 & 0 & 0\\
0& 0 & -1 & 0\\
0& 0 & 0 & -1
\end{array}\right),
\\
&\left(\begin{array}{cccc}
0\\
0\\
\frac{1}{2}\\
0
\end{array}\right)
\left(\begin{array}{cccc}
1& 0 & 0 & 0\\
0& 1 & 0 & 0\\
0& 0 & 1 & 0\\
0& 0 & 0 & 1
\end{array}\right),
\left(\begin{array}{cccc}
0\\
0\\
0\\
\frac{1}{2}
\end{array}\right)
\left(\begin{array}{cccc}
1& 0 & 0 & 0\\
0& 1 & 0 & 0\\
0& 0 & 1 & 0\\
0& 0 & 0 & 1
\end{array}\right)
\Bigr >
\end{align*}

\begin{align*}
\Gamma_{a27}= 
\Bigl < & \left(\begin{array}{cccc}
\frac{1}{2}\\
0\\
0\\
0
\end{array}\right)
\left(\begin{array}{cccc}
1& 0 & 0 & 0\\
0& -1 & 0 & 0\\
0& 0 & -1 & 0\\
0& 0 & 0 & 1
\end{array}\right),
\left(\begin{array}{cccc}
0\\
\frac{1}{2}\\
0\\
0
\end{array}\right)
\left(\begin{array}{cccc}
1& 0 & 0 & 0\\
0& 1 & 0 & 0\\
0& 0 & -1 & 0\\
0& 0 & 0 & -1
\end{array}\right),
\\
&\left(\begin{array}{cccc}
0\\
0\\
\frac{1}{2}\\
0
\end{array}\right)
\left(\begin{array}{cccc}
1& 0 & 0 & 0\\
0& 1 & 0 & 0\\
0& 0 & 1 & 0\\
0& 0 & 0 & 1
\end{array}\right),
\left(\begin{array}{cccc}
0\\
0\\
0\\
\frac{1}{2}
\end{array}\right)
\left(\begin{array}{cccc}
1& 0 & 0 & 0\\
0& 1 & 0 & 0\\
0& 0 & 1 & 0\\
0& 0 & 0 & 1
\end{array}\right)
\Bigr >.
\end{align*}

\normalsize
\noindent
Let $\varphi(z_1,z_2,z_3,z_4)=(z_1,z_2,z_4,z_3)$, we get these commutative diagrams

\footnotesize
\[
\begin{CD}
(z_1,z_2,z_3,z_4) @>\varphi>>(z_1,z_2,z_4,z_3)\\
@Vg_1VV @Vh_1 VV\\
(-z_1,\bar z_2,z_3,\bar z_4) @>\varphi>> (-z_1,\bar z_2,\bar z_4,z_3)
\end{CD}
\hspace{.1cm}
\begin{CD}
(z_1,z_2,z_3,z_4) @>\varphi>>(z_1,z_2,z_4,z_3)\\
@Vg_3VV @Vh_4 VV\\
(z_1,z_2,-z_3,z_4) @>\varphi>> (z_1,z_2,z_4,-z_3)
\end{CD}
\]
\[ 
\begin{CD}
(z_1,z_2,z_3,z_4) @>\varphi>>(z_1,z_2,z_4,z_3)\\
@Vg_2VV @Vh_2 VV\\
(z_1,-z_2,\bar z_3,\bar z_4) @>\varphi>> (z_1,-z_2,\bar z_4,\bar z_3)
\end{CD}
\hspace{.1cm}
\begin{CD}
(z_1,z_2,z_3,z_4) @>\varphi>>(z_1,z_2,z_4,z_3)\\
@Vg_4VV @Vh_3 VV\\
(z_1,z_2,z_3,-z_4) @>\varphi>> (z_1,z_2,-z_4,z_3).
\end{CD}
\]

\normalsize
\noindent 
Therefore $\exists$ 
$\gamma 
\footnotesize
=
\Bigl( \left(\begin{array}{cccc}
0\\
0\\
0\\
0
\end{array}\right)
\left(\begin{array}{cccc}
1& 0 & 0 & 0\\
0& 1 & 0 & 0\\
0& 0 & 0 & 1\\
0& 0 & 1 & 0
\end{array}\right)
\Bigr)
\normalsize
\in \mathbb{A}(n)$ 
s.t. $\gamma \Gamma_{a15}\gamma ^{-1}=\Gamma_{a27}$.\\

%#######################################################################
\noindent {\bf (4).} They are not diffeomorphic, because the holonomy of $\Gamma _{11}$ and
                      $\Gamma _{a15}$ are $\mathbb{Z}_2$ and $\mathbb{Z}_2 \times \mathbb{Z}_2$ respectively.\\
%===============================================================================================

\noindent {\bf (5).}\\
\noindent $\bullet$ $M(A_{3})\approx  M(A_{2})$.\\
For
\footnotesize
\begin{align*}
&A_{3}=
\left(\begin{array}{cccc}
1& 0 & 0 & 0 \\
0& 1 & 0 & 1\\
0& 0 & 1 & 0\\
0& 0 & 0 & 1
\end{array}\right)
\hspace{.1cm}
&A_{2}=
\left(\begin{array}{cccc}
1& 0 & 0 & 0 \\
0& 1 & 0 & 0\\
0& 0 & 1 & 1\\
0& 0 & 0 & 1
\end{array}\right)
\\
&g_1(z_1,z_2,z_3,z_4)=(-z_1,z_2,z_3,z_4).\hspace{.1cm}   &h_1(z_1,z_2,z_3,z_4)=(-z_1,z_2,z_3,z_4)\\
&g_2(z_1,z_2,z_3,z_4)=(z_1,-z_2,z_3,\bar z_4).           &h_2(z_1,z_2,z_3,z_4)=(z_1,-z_2,z_3,z_4)\\
&g_3(z_1,z_2,z_3,z_4)=(z_1,z_2,-z_3,z_4).                &h_3(z_1,z_2,z_3,z_4)=(z_1, z_2,-z_3,\bar z_4)\\
&g_4(z_1,z_2,z_3,z_4)=(z_1,z_2,z_3,-z_4).                &h_4(z_1,z_2,z_3,z_4)=(z_1, z_2,z_3,-z_4)
\end{align*}
\normalsize then \footnotesize
\begin{align*}
\Gamma_{3}=
\Bigl < &\left(\begin{array}{cccc}
\frac{1}{2}\\
0\\
0\\
0
\end{array}\right)
\left(\begin{array}{cccc}
1& 0 & 0 & 0\\
0& 1 & 0 & 0\\
0& 0 & 1 & 0\\
0& 0 & 0 & 1
\end{array}\right),
\left(\begin{array}{cccc}
0\\
\frac{1}{2}\\
0\\
0
\end{array}\right)
\left(\begin{array}{cccc}
1& 0 & 0 & 0\\
0& 1 & 0 & 0\\
0& 0 & 1 & 0\\
0& 0 & 0 & -1
\end{array}\right),
\\
&\left(\begin{array}{cccc}
0\\
0\\
\frac{1}{2}\\
0
\end{array}\right)
\left(\begin{array}{cccc}
1& 0 & 0 & 0\\
0& 1 & 0 & 0\\
0& 0 & 1 & 0\\
0& 0 & 0 & 1
\end{array}\right),
\left(\begin{array}{cccc}
0\\
0\\
0\\
\frac{1}{2}
\end{array}\right)
\left(\begin{array}{cccc}
1& 0 & 0 & 0\\
0& 1 & 0 & 0\\
0& 0 & 1 & 0\\
0& 0 & 0 & 1
\end{array}\right)
\Bigr >
\end{align*}

\begin{align*}
\Gamma_{2}= 
\Bigl < & \left(\begin{array}{cccc}
\frac{1}{2}\\
0\\
0\\
0
\end{array}\right)
\left(\begin{array}{cccc}
1& 0 & 0 & 0\\
0& 1 & 0 & 0\\
0& 0 & 1 & 0\\
0& 0 & 0 & 1
\end{array}\right),
\left(\begin{array}{cccc}
0\\
\frac{1}{2}\\
0\\
0
\end{array}\right)
\left(\begin{array}{cccc}
1& 0 & 0 & 0\\
0& 1 & 0 & 0\\
0& 0 & 1 & 0\\
0& 0 & 0 & 1
\end{array}\right),
\\
&\left(\begin{array}{cccc}
0\\
0\\
\frac{1}{2}\\
0
\end{array}\right)
\left(\begin{array}{cccc}
1& 0 & 0 & 0\\
0& 1 & 0 & 0\\
0& 0 & 1 & 0\\
0& 0 & 0 & -1
\end{array}\right),
\left(\begin{array}{cccc}
0\\
0\\
0\\
\frac{1}{2}
\end{array}\right)
\left(\begin{array}{cccc}
1& 0 & 0 & 0\\
0& 1 & 0 & 0\\
0& 0 & 1 & 0\\
0& 0 & 0 & 1
\end{array}\right)
\Bigr >.
\end{align*}

\normalsize
\noindent
Let $\varphi(z_1,z_2,z_3,z_4)=(z_1,z_3,z_2,z_4)$, we get these commutative diagrams

\footnotesize
\[
\begin{CD}
(z_1,z_2,z_3,z_4) @>\varphi>>(z_1,z_3,z_2,z_4)\\
@Vg_1VV @Vh_1 VV\\
(-z_1,z_2,z_3,z_4) @>\varphi>> (-z_1,z_3,z_2,z_4)
\end{CD}
\hspace{.1cm}
\begin{CD}
(z_1,z_2,z_3,z_4) @>\varphi>>(z_1,z_3,z_2,z_4)\\
@Vg_3VV @Vh_2 VV\\
(z_1,z_2,-z_3,z_4) @>\varphi>> (z_1,-z_3,z_2,z_4)
\end{CD}
\]
\[ 
\begin{CD}
(z_1,z_2,z_3,z_4) @>\varphi>>(z_1,z_3,z_2,z_4)\\
@Vg_2VV @Vh_3 VV\\
(z_1,-z_2,z_3,\bar z_4) @>\varphi>> (z_1,z_3,-z_2,\bar z_4)
\end{CD}
\hspace{.1cm}
\begin{CD}
(z_1,z_2,z_3,z_4) @>\varphi>>(z_1,z_3,z_2,z_4)\\
@Vg_4VV @Vh_4 VV\\
(z_1,z_2,z_3,-z_4) @>\varphi>> (z_1,z_3,z_2,-z_4).
\end{CD}
\]

\normalsize
\noindent 
Therefore $\exists$ 
$\gamma 
\footnotesize
=
\Bigl( \left(\begin{array}{cccc}
0\\
0\\
0\\
0
\end{array}\right)
\left(\begin{array}{cccc}
1& 0 & 0 & 0\\
0& 0 & 1 & 0\\
0& 1 & 0 & 0\\
0& 0 & 0 & 1
\end{array}\right)
\Bigr)
\normalsize
\in \mathbb{A}(n)$ 
s.t. $\gamma \Gamma_{3}\gamma ^{-1}=\Gamma_{2}$.\\
%----------------------------------------------------------------------------

\noindent 
$\bullet$ $M(A_{2})\approx  M(A_{4})$.\\
For
\footnotesize
\begin{align*}
&A_{2}=
\left(\begin{array}{cccc}
1& 0 & 0 & 0 \\
0& 1 & 0 & 0\\
0& 0 & 1 & 1\\
0& 0 & 0 & 1
\end{array}\right)
\hspace{.1cm}
&A_{4}=
\left(\begin{array}{cccc}
1& 0 & 0 & 0 \\
0& 1 & 0 & 1\\
0& 0 & 1 & 1\\
0& 0 & 0 & 1
\end{array}\right)
\\
&g_1(z_1,z_2,z_3,z_4)=(-z_1,z_2,z_3,z_4).\hspace{.1cm}   &h_1(z_1,z_2,z_3,z_4)=(-z_1,z_2,z_3,z_4)\\
&g_2(z_1,z_2,z_3,z_4)=(z_1,-z_2,z_3,z_4).                &h_2h_3(z_1,z_2,z_3,z_4)=(z_1,-z_2,-z_3,z_4)\\
&g_3(z_1,z_2,z_3,z_4)=(z_1,z_2,-z_3,\bar z_4).           &h_3(z_1,z_2,z_3,z_4)=(z_1, z_2,-z_3,\bar z_4)\\
&g_4(z_1,z_2,z_3,z_4)=(z_1,z_2,z_3,-z_4).                &h_4(z_1,z_2,z_3,z_4)=(z_1, z_2,z_3,-z_4)
\end{align*}
\normalsize then \footnotesize
\begin{align*}
\Gamma_{2}=
\Bigl < &\left(\begin{array}{cccc}
\frac{1}{2}\\
0\\
0\\
0
\end{array}\right)
\left(\begin{array}{cccc}
1& 0 & 0 & 0\\
0& 1 & 0 & 0\\
0& 0 & 1 & 0\\
0& 0 & 0 & 1
\end{array}\right),
\left(\begin{array}{cccc}
0\\
\frac{1}{2}\\
0\\
0
\end{array}\right)
\left(\begin{array}{cccc}
1& 0 & 0 & 0\\
0& 1 & 0 & 0\\
0& 0 & 1 & 0\\
0& 0 & 0 & 1
\end{array}\right),
\\
&\left(\begin{array}{cccc}
0\\
0\\
\frac{1}{2}\\
0
\end{array}\right)
\left(\begin{array}{cccc}
1& 0 & 0 & 0\\
0& 1 & 0 & 0\\
0& 0 & 1 & 0\\
0& 0 & 0 & -1
\end{array}\right),
\left(\begin{array}{cccc}
0\\
0\\
0\\
\frac{1}{2}
\end{array}\right)
\left(\begin{array}{cccc}
1& 0 & 0 & 0\\
0& 1 & 0 & 0\\
0& 0 & 1 & 0\\
0& 0 & 0 & 1
\end{array}\right)
\Bigr >
\end{align*}

\begin{align*}
\Gamma_{4}= 
\Bigl < & \left(\begin{array}{cccc}
\frac{1}{2}\\
0\\
0\\
0
\end{array}\right)
\left(\begin{array}{cccc}
1& 0 & 0 & 0\\
0& 1 & 0 & 0\\
0& 0 & 1 & 0\\
0& 0 & 0 & 1
\end{array}\right),
\left(\begin{array}{cccc}
0\\
\frac{1}{2}\\
\frac{1}{2}\\
0
\end{array}\right)
\left(\begin{array}{cccc}
1& 0 & 0 & 0\\
0& 1 & 0 & 0\\
0& 0 & 1 & 0\\
0& 0 & 0 & 1
\end{array}\right),
\\
&\left(\begin{array}{cccc}
0\\
0\\
\frac{1}{2}\\
0
\end{array}\right)
\left(\begin{array}{cccc}
1& 0 & 0 & 0\\
0& 1 & 0 & 0\\
0& 0 & 1 & 0\\
0& 0 & 0 & -1
\end{array}\right),
\left(\begin{array}{cccc}
0\\
0\\
0\\
\frac{1}{2}
\end{array}\right)
\left(\begin{array}{cccc}
1& 0 & 0 & 0\\
0& 1 & 0 & 0\\
0& 0 & 1 & 0\\
0& 0 & 0 & 1
\end{array}\right)
\Bigr >.
\end{align*}

\normalsize
\noindent
Let $\varphi(z_1,z_2,z_3,z_4)=(z_1,z_2,z_2z_3,z_4)$, we get these commutative diagrams

\footnotesize
\[
\begin{CD}
(z_1,z_2,z_3,z_4) @>\varphi>>(z_1,z_2,z_2z_3,z_4)\\
@Vg_1VV @Vh_1 VV\\
(-z_1,z_2,z_3,z_4) @>\varphi>> (-z_1,z_2,z_2z_3,z_4)
\end{CD}
\hspace{.1cm}
\begin{CD}
(z_1,z_2,z_3,z_4) @>\varphi>>(z_1,z_2,z_2z_3,z_4)\\
@Vg_3VV @Vh_3 VV\\
(z_1,z_2,-z_3,\bar z_4) @>\varphi>> (z_1,z_2,-z_2z_3,\bar z_4)
\end{CD}
\]
\[ 
\begin{CD}
(z_1,z_2,z_3,z_4) @>\varphi>>(z_1,z_2,z_2z_3,z_4)\\
@Vg_2VV @Vh_2h_3 VV\\
(z_1,-z_2,z_3,z_4) @>\varphi>> (z_1,-z_2,-z_2z_3,z_4)
\end{CD}
\hspace{.1cm}
\begin{CD}
(z_1,z_2,z_3,z_4) @>\varphi>>(z_1,z_2,z_2z_3,z_4)\\
@Vg_4VV @Vh_4 VV\\
(z_1,z_2,z_3,-z_4) @>\varphi>> (z_1,z_2,z_2z_3,-z_4).
\end{CD}
\]

\normalsize
\noindent 
Therefore $\exists$ 
$\gamma 
\footnotesize
=
\Bigl( \left(\begin{array}{cccc}
0\\
0\\
0\\
0
\end{array}\right)
\left(\begin{array}{cccc}
1& 0 & 0 & 0\\
0& 1 & 0 & 0\\
0& 1 & 1 & 0\\
0& 0 & 0 & 1
\end{array}\right)
\Bigr)
\normalsize
\in \mathbb{A}(n)$ 
s.t. $\gamma \Gamma_{2}\gamma ^{-1}=\Gamma_{4}$.\\
%----------------------------------------------------------------------------

\noindent 
$\bullet$ $M(A_{2})\approx  M(A_{5})$.\\
For
\footnotesize
\begin{align*}
&A_{2}=
\left(\begin{array}{cccc}
1& 0 & 0 & 0 \\
0& 1 & 0 & 0\\
0& 0 & 1 & 1\\
0& 0 & 0 & 1
\end{array}\right)
\hspace{.1cm}
&A_{5}=
\left(\begin{array}{cccc}
1& 0 & 0 & 1 \\
0& 1 & 0 & 0\\
0& 0 & 1 & 0\\
0& 0 & 0 & 1
\end{array}\right)
\\
&g_1(z_1,z_2,z_3,z_4)=(-z_1,z_2,z_3,z_4).\hspace{.1cm}   &h_1(z_1,z_2,z_3,z_4)=(-z_1,z_2,z_3,\bar z_4)\\
&g_2(z_1,z_2,z_3,z_4)=(z_1,-z_2,z_3,z_4).                &h_2(z_1,z_2,z_3,z_4)=(z_1,-z_2,z_3,z_4)\\
&g_3(z_1,z_2,z_3,z_4)=(z_1,z_2,-z_3,\bar z_4).           &h_3(z_1,z_2,z_3,z_4)=(z_1, z_2,-z_3,z_4)\\
&g_4(z_1,z_2,z_3,z_4)=(z_1,z_2,z_3,-z_4).                &h_4(z_1,z_2,z_3,z_4)=(z_1, z_2,z_3,-z_4)
\end{align*}
\normalsize then \footnotesize
\begin{align*}
\Gamma_{2}=
\Bigl < &\left(\begin{array}{cccc}
\frac{1}{2}\\
0\\
0\\
0
\end{array}\right)
\left(\begin{array}{cccc}
1& 0 & 0 & 0\\
0& 1 & 0 & 0\\
0& 0 & 1 & 0\\
0& 0 & 0 & 1
\end{array}\right),
\left(\begin{array}{cccc}
0\\
\frac{1}{2}\\
0\\
0
\end{array}\right)
\left(\begin{array}{cccc}
1& 0 & 0 & 0\\
0& 1 & 0 & 0\\
0& 0 & 1 & 0\\
0& 0 & 0 & 1
\end{array}\right),
\\
&\left(\begin{array}{cccc}
0\\
0\\
\frac{1}{2}\\
0
\end{array}\right)
\left(\begin{array}{cccc}
1& 0 & 0 & 0\\
0& 1 & 0 & 0\\
0& 0 & 1 & 0\\
0& 0 & 0 & -1
\end{array}\right),
\left(\begin{array}{cccc}
0\\
0\\
0\\
\frac{1}{2}
\end{array}\right)
\left(\begin{array}{cccc}
1& 0 & 0 & 0\\
0& 1 & 0 & 0\\
0& 0 & 1 & 0\\
0& 0 & 0 & 1
\end{array}\right)
\Bigr >
\end{align*}

\begin{align*}
\Gamma_{5}= 
\Bigl < & \left(\begin{array}{cccc}
\frac{1}{2}\\
0\\
0\\
0
\end{array}\right)
\left(\begin{array}{cccc}
1& 0 & 0 & 0\\
0& 1 & 0 & 0\\
0& 0 & 1 & 0\\
0& 0 & 0 & -1
\end{array}\right),
\left(\begin{array}{cccc}
0\\
\frac{1}{2}\\
0\\
0
\end{array}\right)
\left(\begin{array}{cccc}
1& 0 & 0 & 0\\
0& 1 & 0 & 0\\
0& 0 & 1 & 0\\
0& 0 & 0 & 1
\end{array}\right),
\\
&\left(\begin{array}{cccc}
0\\
0\\
\frac{1}{2}\\
0
\end{array}\right)
\left(\begin{array}{cccc}
1& 0 & 0 & 0\\
0& 1 & 0 & 0\\
0& 0 & 1 & 0\\
0& 0 & 0 & 1
\end{array}\right),
\left(\begin{array}{cccc}
0\\
0\\
0\\
\frac{1}{2}
\end{array}\right)
\left(\begin{array}{cccc}
1& 0 & 0 & 0\\
0& 1 & 0 & 0\\
0& 0 & 1 & 0\\
0& 0 & 0 & 1
\end{array}\right)
\Bigr >.
\end{align*}

\normalsize
\noindent
Let $\varphi(z_1,z_2,z_3,z_4)=(z_3,z_2,z_1,z_4)$, we get these commutative diagrams

\footnotesize
\[
\begin{CD}
(z_1,z_2,z_3,z_4) @>\varphi>>(z_3,z_2,z_1,z_4)\\
@Vg_1VV @Vh_3 VV\\
(-z_1,z_2,z_3,z_4) @>\varphi>> (z_3,z_2,-z_1,z_4)
\end{CD}
\hspace{.1cm}
\begin{CD}
(z_1,z_2,z_3,z_4) @>\varphi>>(z_3,z_2,z_1,z_4)\\
@Vg_3VV @Vh_1 VV\\
(z_1,z_2,-z_3,\bar z_4) @>\varphi>> (-z_3,z_2,z_1,\bar z_4)
\end{CD}
\]
\[ 
\begin{CD}
(z_1,z_2,z_3,z_4) @>\varphi>>(z_3,z_2,z_1,z_4)\\
@Vg_2VV @Vh_2 VV\\
(z_1,-z_2,z_3,z_4) @>\varphi>> (z_3,-z_2,z_1,z_4)
\end{CD}
\hspace{.1cm}
\begin{CD}
(z_1,z_2,z_3,z_4) @>\varphi>>(z_3,z_2,z_1,z_4)\\
@Vg_4VV @Vh_4 VV\\
(z_1,z_2,z_3,-z_4) @>\varphi>> (z_3,z_2,z_1,-z_4).
\end{CD}
\]

\normalsize
\noindent 
Therefore $\exists$ 
$\gamma 
\footnotesize
=
\Bigl( \left(\begin{array}{cccc}
0\\
0\\
0\\
0
\end{array}\right)
\left(\begin{array}{cccc}
0& 0 & 1 & 0\\
0& 1 & 0 & 0\\
1& 0 & 0 & 0\\
0& 0 & 0 & 1
\end{array}\right)
\Bigr)
\normalsize
\in \mathbb{A}(n)$ 
s.t. $\gamma \Gamma_{2}\gamma ^{-1}=\Gamma_{5}$.\\
%----------------------------------------------------------------------------

\noindent 
$\bullet$ $M(A_{2})\approx  M(A_{6})$.\\
For 
\footnotesize
\begin{align*}
&A_{2}=
\left(\begin{array}{cccc}
1& 0 & 0 & 0 \\
0& 1 & 0 & 0\\
0& 0 & 1 & 1\\
0& 0 & 0 & 1
\end{array}\right)
\hspace{.1cm}
&A_{6}=
\left(\begin{array}{cccc}
1& 0 & 0 & 1 \\
0& 1 & 0 & 0\\
0& 0 & 1 & 1\\
0& 0 & 0 & 1
\end{array}\right)
\\
&g_1(z_1,z_2,z_3,z_4)=(-z_1,z_2,z_3,z_4).\hspace{.1cm}   &h_1h_3(z_1,z_2,z_3,z_4)=(-z_1,z_2,-z_3,z_4)\\
&g_2(z_1,z_2,z_3,z_4)=(z_1,-z_2,z_3,z_4).                &h_2(z_1,z_2,z_3,z_4)=(z_1,-z_2,z_3,z_4)\\
&g_3(z_1,z_2,z_3,z_4)=(z_1,z_2,-z_3,\bar z_4).           &h_3(z_1,z_2,z_3,z_4)=(z_1, z_2,-z_3,\bar z_4)\\
&g_4(z_1,z_2,z_3,z_4)=(z_1,z_2,z_3,-z_4).                &h_4(z_1,z_2,z_3,z_4)=(z_1, z_2,z_3,-z_4)
\end{align*}
\normalsize then \footnotesize
\begin{align*}
\Gamma_{2}=
\Bigl < &\left(\begin{array}{cccc}
\frac{1}{2}\\
0\\
0\\
0
\end{array}\right)
\left(\begin{array}{cccc}
1& 0 & 0 & 0\\
0& 1 & 0 & 0\\
0& 0 & 1 & 0\\
0& 0 & 0 & 1
\end{array}\right),
\left(\begin{array}{cccc}
0\\
\frac{1}{2}\\
0\\
0
\end{array}\right)
\left(\begin{array}{cccc}
1& 0 & 0 & 0\\
0& 1 & 0 & 0\\
0& 0 & 1 & 0\\
0& 0 & 0 & 1
\end{array}\right),
\\
&\left(\begin{array}{cccc}
0\\
0\\
\frac{1}{2}\\
0
\end{array}\right)
\left(\begin{array}{cccc}
1& 0 & 0 & 0\\
0& 1 & 0 & 0\\
0& 0 & 1 & 0\\
0& 0 & 0 & -1
\end{array}\right),
\left(\begin{array}{cccc}
0\\
0\\
0\\
\frac{1}{2}
\end{array}\right)
\left(\begin{array}{cccc}
1& 0 & 0 & 0\\
0& 1 & 0 & 0\\
0& 0 & 1 & 0\\
0& 0 & 0 & 1
\end{array}\right)
\Bigr >
\end{align*}

\begin{align*}
\Gamma_{6}= 
\Bigl < & \left(\begin{array}{cccc}
\frac{1}{2}\\
0\\
\frac{1}{2}\\
0
\end{array}\right)
\left(\begin{array}{cccc}
1& 0 & 0 & 0\\
0& 1 & 0 & 0\\
0& 0 & 1 & 0\\
0& 0 & 0 & 1
\end{array}\right),
\left(\begin{array}{cccc}
0\\
\frac{1}{2}\\
0\\
0
\end{array}\right)
\left(\begin{array}{cccc}
1& 0 & 0 & 0\\
0& 1 & 0 & 0\\
0& 0 & 1 & 0\\
0& 0 & 0 & 1
\end{array}\right),
\\
&\left(\begin{array}{cccc}
0\\
0\\
\frac{1}{2}\\
0
\end{array}\right)
\left(\begin{array}{cccc}
1& 0 & 0 & 0\\
0& 1 & 0 & 0\\
0& 0 & 1 & 0\\
0& 0 & 0 & -1
\end{array}\right),
\left(\begin{array}{cccc}
0\\
0\\
0\\
\frac{1}{2}
\end{array}\right)
\left(\begin{array}{cccc}
1& 0 & 0 & 0\\
0& 1 & 0 & 0\\
0& 0 & 1 & 0\\
0& 0 & 0 & 1
\end{array}\right)
\Bigr >.
\end{align*}

\normalsize
\noindent
Let $\varphi(z_1,z_2,z_3,z_4)=(z_1,z_2,z_1z_3,z_4)$, we get these commutative diagrams

\footnotesize
\[
\begin{CD}
(z_1,z_2,z_3,z_4) @>\varphi>>(z_1,z_2,z_1z_3,z_4)\\
@Vg_1VV @Vh_1h_3 VV\\
(-z_1,z_2,z_3,z_4) @>\varphi>> (-z_1,z_2,-z_1z_3,z_4)
\end{CD}
\hspace{.1cm}
\begin{CD}
(z_1,z_2,z_3,z_4) @>\varphi>>(z_1,z_2,z_1z_3,\bar z_4)\\
@Vg_3VV @Vh_3 VV\\
(z_1,z_2,-z_3,\bar z_4) @>\varphi>> (z_1,z_2,-z_1z_3,\bar z_4)
\end{CD}
\]
\[ 
\begin{CD}
(z_1,z_2,z_3,z_4) @>\varphi>>(z_1,z_2,z_1z_3,z_4)\\
@Vg_2VV @Vh_2 VV\\
(z_1,-z_2,z_3,z_4) @>\varphi>> (z_1,-z_2,z_1z_3,z_4)
\end{CD}
\hspace{.1cm}
\begin{CD}
(z_1,z_2,z_3,z_4) @>\varphi>>(z_1,z_2,z_1z_3,z_4)\\
@Vg_4VV @Vh_4 VV\\
(z_1,z_2,z_3,-z_4) @>\varphi>> (z_1,z_2,z_1z_3,-z_4).
\end{CD}
\]

\normalsize
\noindent 
Therefore $\exists$ 
$\gamma 
\footnotesize
=
\Bigl( \left(\begin{array}{cccc}
0\\
0\\
0\\
0
\end{array}\right)
\left(\begin{array}{cccc}
1& 0 & 0 & 0\\
0& 1 & 0 & 0\\
1& 0 & 1 & 0\\
0& 0 & 0 & 1
\end{array}\right)
\Bigr)
\normalsize
\in \mathbb{A}(n)$ 
s.t. $\gamma \Gamma_{2}\gamma ^{-1}=\Gamma_{6}$.\\
%--------------------------------------------------------------------

\noindent 
$\bullet$ $M(A_{3})\approx  M(A_{7})$.\\
For
\footnotesize
\begin{align*}
&A_{3}=
\left(\begin{array}{cccc}
1& 0 & 0 & 0 \\
0& 1 & 0 & 1\\
0& 0 & 1 & 0\\
0& 0 & 0 & 1
\end{array}\right)
\hspace{.1cm}
&A_{7}=
\left(\begin{array}{cccc}
1& 0 & 0 & 1 \\
0& 1 & 0 & 1\\
0& 0 & 1 & 0\\
0& 0 & 0 & 1
\end{array}\right)
\\
&g_1(z_1,z_2,z_3,z_4)=(-z_1,z_2,z_3,z_4).\hspace{.1cm}   &h_1h_2(z_1,z_2,z_3,z_4)=(-z_1,-z_2,z_3,z_4)\\
&g_2(z_1,z_2,z_3,z_4)=(z_1,-z_2,z_3,\bar z_4).           &h_2(z_1,z_2,z_3,z_4)=(z_1,-z_2,z_3,\bar z_4)\\
&g_3(z_1,z_2,z_3,z_4)=(z_1,z_2,-z_3,z_4).                &h_3(z_1,z_2,z_3,z_4)=(z_1, z_2,-z_3,z_4)\\
&g_4(z_1,z_2,z_3,z_4)=(z_1,z_2,z_3,-z_4).                &h_4(z_1,z_2,z_3,z_4)=(z_1, z_2,z_3,-z_4)
\end{align*}
\normalsize then \footnotesize
\begin{align*}
\Gamma_{3}=
\Bigl < &\left(\begin{array}{cccc}
\frac{1}{2}\\
0\\
0\\
0
\end{array}\right)
\left(\begin{array}{cccc}
1& 0 & 0 & 0\\
0& 1 & 0 & 0\\
0& 0 & 1 & 0\\
0& 0 & 0 & 1
\end{array}\right),
\left(\begin{array}{cccc}
0\\
\frac{1}{2}\\
0\\
0
\end{array}\right)
\left(\begin{array}{cccc}
1& 0 & 0 & 0\\
0& 1 & 0 & 0\\
0& 0 & 1 & 0\\
0& 0 & 0 & -1
\end{array}\right),
\\
&\left(\begin{array}{cccc}
0\\
0\\
\frac{1}{2}\\
0
\end{array}\right)
\left(\begin{array}{cccc}
1& 0 & 0 & 0\\
0& 1 & 0 & 0\\
0& 0 & 1 & 0\\
0& 0 & 0 & 1
\end{array}\right),
\left(\begin{array}{cccc}
0\\
0\\
0\\
\frac{1}{2}
\end{array}\right)
\left(\begin{array}{cccc}
1& 0 & 0 & 0\\
0& 1 & 0 & 0\\
0& 0 & 1 & 0\\
0& 0 & 0 & 1
\end{array}\right)
\Bigr >
\end{align*}

\begin{align*}
\Gamma_{7}= 
\Bigl < & \left(\begin{array}{cccc}
\frac{1}{2}\\
\frac{1}{2}\\
0\\
0
\end{array}\right)
\left(\begin{array}{cccc}
1& 0 & 0 & 0\\
0& 1 & 0 & 0\\
0& 0 & 1 & 0\\
0& 0 & 0 & 1
\end{array}\right),
\left(\begin{array}{cccc}
0\\
\frac{1}{2}\\
0\\
0
\end{array}\right)
\left(\begin{array}{cccc}
1& 0 & 0 & 0\\
0& 1 & 0 & 0\\
0& 0 & 1 & 0\\
0& 0 & 0 & -1
\end{array}\right),
\\
&\left(\begin{array}{cccc}
0\\
0\\
\frac{1}{2}\\
0
\end{array}\right)
\left(\begin{array}{cccc}
1& 0 & 0 & 0\\
0& 1 & 0 & 0\\
0& 0 & 1 & 0\\
0& 0 & 0 & 1
\end{array}\right),
\left(\begin{array}{cccc}
0\\
0\\
0\\
\frac{1}{2}
\end{array}\right)
\left(\begin{array}{cccc}
1& 0 & 0 & 0\\
0& 1 & 0 & 0\\
0& 0 & 1 & 0\\
0& 0 & 0 & 1
\end{array}\right)
\Bigr >.
\end{align*}

\normalsize
\noindent
Let $\varphi(z_1,z_2,z_3,z_4)=(z_1,z_1z_2,z_3,z_4)$, we get these commutative diagrams

\footnotesize
\[
\begin{CD}
(z_1,z_2,z_3,z_4) @>\varphi>>(z_1,z_1z_2,z_3,z_4)\\
@Vg_1VV @Vh_1h_2 VV\\
(-z_1,z_2,z_3,z_4) @>\varphi>> (-z_1,-z_1z_2,z_3,z_4)
\end{CD}
\hspace{.1cm}
\begin{CD}
(z_1,z_2,z_3,z_4) @>\varphi>>(z_1,z_1z_2,z_3,z_4)\\
@Vg_3VV @Vh_3 VV\\
(z_1,z_2,-z_3,z_4) @>\varphi>> (z_1,z_1z_2,-z_3,z_4)
\end{CD}
\]
\[ 
\begin{CD}
(z_1,z_2,z_3,z_4) @>\varphi>>(z_1,z_1z_2,z_3,z_4)\\
@Vg_2VV @Vh_2 VV\\
(z_1,-z_2,z_3,\bar z_4) @>\varphi>> (z_1,-z_1z_2,z_3,\bar z_4)
\end{CD}
\hspace{.1cm}
\begin{CD}
(z_1,z_2,z_3,z_4) @>\varphi>>(z_1,z_1z_2,z_3,z_4)\\
@Vg_4VV @Vh_4 VV\\
(z_1,z_2,z_3,-z_4) @>\varphi>> (z_1,z_1z_2,z_3,-z_4).
\end{CD}
\]

\normalsize
\noindent 
Therefore $\exists$ 
$\gamma 
\footnotesize
=
\Bigl( \left(\begin{array}{cccc}
0\\
0\\
0\\
0
\end{array}\right)
\left(\begin{array}{cccc}
1& 0 & 0 & 0\\
1& 1 & 0 & 0\\
0& 0 & 1 & 0\\
0& 0 & 0 & 1
\end{array}\right)
\Bigr)
\normalsize
\in \mathbb{A}(n)$ 
s.t. $\gamma \Gamma_{3}\gamma ^{-1}=\Gamma_{7}$.\\
%----------------------------------------------------------------------------

\noindent 
$\bullet$ $M(A_{2})\approx M(A_{8})$.\\
For
\footnotesize
\begin{align*}
&A_{2}=
\left(\begin{array}{cccc}
1& 0 & 0 & 0 \\
0& 1 & 0 & 0\\
0& 0 & 1 & 1\\
0& 0 & 0 & 1
\end{array}\right)
\hspace{.1cm}
&A_{8}=
\left(\begin{array}{cccc}
1& 0 & 0 & 1 \\
0& 1 & 0 & 1\\
0& 0 & 1 & 1\\
0& 0 & 0 & 1
\end{array}\right)
\\
&g_1(z_1,z_2,z_3,z_4)=(-z_1,z_2,z_3,z_4).\hspace{.1cm}   &h_1h_2(z_1,z_2,z_3,z_4)=(-z_1,-z_2,z_3,z_4)\\
&g_2(z_1,z_2,z_3,z_4)=(z_1,-z_2,z_3,z_4).                &h_2h_3(z_1,z_2,z_3,z_4)=(z_1,-z_2,-z_3,z_4)\\
&g_3(z_1,z_2,z_3,z_4)=(z_1,z_2,-z_3,\bar z_4).           &h_3(z_1,z_2,z_3,z_4)=(z_1, z_2,-z_3,\bar z_4)\\
&g_4(z_1,z_2,z_3,z_4)=(z_1,z_2,z_3,-z_4).                &h_4(z_1,z_2,z_3,z_4)=(z_1, z_2,z_3,-z_4)
\end{align*}
\normalsize then \footnotesize
\begin{align*}
\Gamma_{2}=
\Bigl < &\left(\begin{array}{cccc}
\frac{1}{2}\\
0\\
0\\
0
\end{array}\right)
\left(\begin{array}{cccc}
1& 0 & 0 & 0\\
0& 1 & 0 & 0\\
0& 0 & 1 & 0\\
0& 0 & 0 & 1
\end{array}\right),
\left(\begin{array}{cccc}
0\\
\frac{1}{2}\\
0\\
0
\end{array}\right)
\left(\begin{array}{cccc}
1& 0 & 0 & 0\\
0& 1 & 0 & 0\\
0& 0 & 1 & 0\\
0& 0 & 0 & 1
\end{array}\right),
\\
&\left(\begin{array}{cccc}
0\\
0\\
\frac{1}{2}\\
0
\end{array}\right)
\left(\begin{array}{cccc}
1& 0 & 0 & 0\\
0& 1 & 0 & 0\\
0& 0 & 1 & 0\\
0& 0 & 0 & -1
\end{array}\right),
\left(\begin{array}{cccc}
0\\
0\\
0\\
\frac{1}{2}
\end{array}\right)
\left(\begin{array}{cccc}
1& 0 & 0 & 0\\
0& 1 & 0 & 0\\
0& 0 & 1 & 0\\
0& 0 & 0 & 1
\end{array}\right)
\Bigr >
\end{align*}

\begin{align*}
\Gamma_{8}= 
\Bigl < & \left(\begin{array}{cccc}
\frac{1}{2}\\
\frac{1}{2}\\
0\\
0
\end{array}\right)
\left(\begin{array}{cccc}
1& 0 & 0 & 0\\
0& 1 & 0 & 0\\
0& 0 & 1 & 0\\
0& 0 & 0 & 1
\end{array}\right),
\left(\begin{array}{cccc}
0\\
\frac{1}{2}\\
\frac{1}{2}\\
0
\end{array}\right)
\left(\begin{array}{cccc}
1& 0 & 0 & 0\\
0& 1 & 0 & 0\\
0& 0 & 1 & 0\\
0& 0 & 0 & 1
\end{array}\right),
\\
&\left(\begin{array}{cccc}
0\\
0\\
\frac{1}{2}\\
0
\end{array}\right)
\left(\begin{array}{cccc}
1& 0 & 0 & 0\\
0& 1 & 0 & 0\\
0& 0 & 1 & 0\\
0& 0 & 0 & -1
\end{array}\right),
\left(\begin{array}{cccc}
0\\
0\\
0\\
\frac{1}{2}
\end{array}\right)
\left(\begin{array}{cccc}
1& 0 & 0 & 0\\
0& 1 & 0 & 0\\
0& 0 & 1 & 0\\
0& 0 & 0 & 1
\end{array}\right)
\Bigr >.
\end{align*}

\normalsize
\noindent
Let $\varphi(z_1,z_2,z_3,z_4)=(z_1,z_1z_2,z_2z_3,z_4)$, we get these commutative diagrams

\footnotesize
\[
\begin{CD}
(z_1,z_2,z_3,z_4) @>\varphi>>(z_1,z_1z_2,z_2z_3,z_4)\\
@Vg_1VV @Vh_1h_2 VV\\
(-z_1,z_2,z_3,z_4) @>\varphi>> (-z_1,-z_1z_2,z_2z_3,z_4)
\end{CD}
\hspace{.1cm}
\begin{CD}
(z_1,z_2,z_3,z_4) @>\varphi>>(z_1,z_1z_2,z_2z_3,z_4)\\
@Vg_3VV @Vh_3 VV\\
(z_1,z_2,-z_3,\bar z_4) @>\varphi>> (z_1,z_1z_2,-z_2z_3,\bar z_4)
\end{CD}
\]
\[ 
\begin{CD}
(z_1,z_2,z_3,z_4) @>\varphi>>(z_1,z_1z_2,z_2z_3,z_4)\\
@Vg_2VV @Vh_2h_3 VV\\
(z_1,-z_2,z_3,z_4) @>\varphi>> (z_1,-z_1z_2,-z_2z_3,z_4)
\end{CD}
\hspace{.1cm}
\begin{CD}
(z_1,z_2,z_3,z_4) @>\varphi>>(z_1,z_1z_2,z_2z_3,z_4)\\
@Vg_4VV @Vh_4 VV\\
(z_1,z_2,z_3,-z_4) @>\varphi>> (z_1,z_1z_2,z_2z_3,-z_4).
\end{CD}
\]

\normalsize
\noindent 
Therefore $\exists$ 
$\gamma 
\footnotesize
=
\Bigl( \left(\begin{array}{cccc}
0\\
0\\
0\\
0
\end{array}\right)
\left(\begin{array}{cccc}
1& 0 & 0 & 0\\
1& 1 & 0 & 0\\
0& 1 & 1 & 0\\
0& 0 & 0 & 1
\end{array}\right)
\Bigr)
\normalsize
\in \mathbb{A}(n)$ 
s.t. $\gamma \Gamma_{2}\gamma ^{-1}=\Gamma_{8}$.\\
%--------------------------------------------------------------------

\noindent 
$\bullet$ $M(A_{2})\approx  M(A_{9})$.\\
For
\footnotesize
\begin{align*}
&A_{2}=
\left(\begin{array}{cccc}
1& 0 & 0 & 0 \\
0& 1 & 0 & 0\\
0& 0 & 1 & 1\\
0& 0 & 0 & 1
\end{array}\right)
\hspace{.1cm}
&A_{9}=
\left(\begin{array}{cccc}
1& 0 & 0 & 0 \\
0& 1 & 1 & 0\\
0& 0 & 1 & 0\\
0& 0 & 0 & 1
\end{array}\right)
\\
&g_1(z_1,z_2,z_3,z_4)=(-z_1,z_2,z_3,z_4).\hspace{.1cm}   &h_1(z_1,z_2,z_3,z_4)=(-z_1,z_2,z_3,z_4)\\
&g_2(z_1,z_2,z_3,z_4)=(z_1,-z_2,z_3,z_4).                &h_2(z_1,z_2,z_3,z_4)=(z_1,-z_2,\bar z_3,z_4)\\
&g_3(z_1,z_2,z_3,z_4)=(z_1,z_2,-z_3,\bar z_4).           &h_3(z_1,z_2,z_3,z_4)=(z_1, z_2,-z_3,z_4)\\
&g_4(z_1,z_2,z_3,z_4)=(z_1,z_2,z_3,-z_4).                &h_4(z_1,z_2,z_3,z_4)=(z_1, z_2,z_3,-z_4)
\end{align*}
\normalsize then \footnotesize
\begin{align*}
\Gamma_{2}=
\Bigl < &\left(\begin{array}{cccc}
\frac{1}{2}\\
0\\
0\\
0
\end{array}\right)
\left(\begin{array}{cccc}
1& 0 & 0 & 0\\
0& 1 & 0 & 0\\
0& 0 & 1 & 0\\
0& 0 & 0 & 1
\end{array}\right),
\left(\begin{array}{cccc}
0\\
\frac{1}{2}\\
0\\
0
\end{array}\right)
\left(\begin{array}{cccc}
1& 0 & 0 & 0\\
0& 1 & 0 & 0\\
0& 0 & 1 & 0\\
0& 0 & 0 & 1
\end{array}\right),
\\
&\left(\begin{array}{cccc}
0\\
0\\
\frac{1}{2}\\
0
\end{array}\right)
\left(\begin{array}{cccc}
1& 0 & 0 & 0\\
0& 1 & 0 & 0\\
0& 0 & 1 & 0\\
0& 0 & 0 & -1
\end{array}\right),
\left(\begin{array}{cccc}
0\\
0\\
0\\
\frac{1}{2}
\end{array}\right)
\left(\begin{array}{cccc}
1& 0 & 0 & 0\\
0& 1 & 0 & 0\\
0& 0 & 1 & 0\\
0& 0 & 0 & 1
\end{array}\right)
\Bigr >
\end{align*}

\begin{align*}
\Gamma_{9}= 
\Bigl < & \left(\begin{array}{cccc}
\frac{1}{2}\\
0\\
0\\
0
\end{array}\right)
\left(\begin{array}{cccc}
1& 0 & 0 & 0\\
0& 1 & 0 & 0\\
0& 0 & 1 & 0\\
0& 0 & 0 & 1
\end{array}\right),
\left(\begin{array}{cccc}
0\\
\frac{1}{2}\\
0\\
0
\end{array}\right)
\left(\begin{array}{cccc}
1& 0 & 0 & 0\\
0& 1 & 0 & 0\\
0& 0 & -1 & 0\\
0& 0 & 0 & 1
\end{array}\right),
\\
&\left(\begin{array}{cccc}
0\\
0\\
\frac{1}{2}\\
0
\end{array}\right)
\left(\begin{array}{cccc}
1& 0 & 0 & 0\\
0& 1 & 0 & 0\\
0& 0 & 1 & 0\\
0& 0 & 0 & 1
\end{array}\right),
\left(\begin{array}{cccc}
0\\
0\\
0\\
\frac{1}{2}
\end{array}\right)
\left(\begin{array}{cccc}
1& 0 & 0 & 0\\
0& 1 & 0 & 0\\
0& 0 & 1 & 0\\
0& 0 & 0 & 1
\end{array}\right)
\Bigr >.
\end{align*}

\normalsize
\noindent
Let $\varphi(z_1,z_2,z_3,z_4)=(z_1,z_3,z_4,z_2)$, we get these commutative diagrams

\footnotesize
\[
\begin{CD}
(z_1,z_2,z_3,z_4) @>\varphi>>(z_1,z_3,z_4,z_2)\\
@Vg_1VV @Vh_1 VV\\
(-z_1,z_2,z_3,z_4) @>\varphi>> (-z_1,z_3,z_4,z_2)
\end{CD}
\hspace{.1cm}
\begin{CD}
(z_1,z_2,z_3,z_4) @>\varphi>>(z_1,z_3,z_4,z_2)\\
@Vg_3VV @Vh_2 VV\\
(z_1,z_2,-z_3,\bar z_4) @>\varphi>> (z_1,-z_3,\bar z_4,z_2)
\end{CD}
\]
\[ 
\begin{CD}
(z_1,z_2,z_3,z_4) @>\varphi>>(z_1,z_3,z_4,z_2)\\
@Vg_2VV @Vh_4 VV\\
(z_1,-z_2,z_3,z_4) @>\varphi>> (z_1,z_3,z_4,-z_2)
\end{CD}
\hspace{.1cm}
\begin{CD}
(z_1,z_2,z_3,z_4) @>\varphi>>(z_1,z_3,z_4,z_2)\\
@Vg_4VV @Vh_3 VV\\
(z_1,z_2,z_3,-z_4) @>\varphi>> (z_1,z_3,-z_4,z_2).
\end{CD}
\]

\normalsize
\noindent 
Therefore $\exists$ 
$\gamma 
\footnotesize
=
\Bigl( \left(\begin{array}{cccc}
0\\
0\\
0\\
0
\end{array}\right)
\left(\begin{array}{cccc}
1& 0 & 0 & 0\\
0& 0 & 1 & 0\\
0& 0 & 0 & 1\\
0& 1 & 0 & 0
\end{array}\right)
\Bigr)
\normalsize
\in \mathbb{A}(n)$ 
s.t. $\gamma \Gamma_{2}\gamma ^{-1}=\Gamma_{9}$.\\
%--------------------------------------------------------------------

\noindent 
$\bullet$ $M(A_{2})\approx  M(A_{17})$.\\
For
\footnotesize
\begin{align*}
&A_{2}=
\left(\begin{array}{cccc}
1& 0 & 0 & 0 \\
0& 1 & 0 & 0\\
0& 0 & 1 & 1\\
0& 0 & 0 & 1
\end{array}\right)
\hspace{.1cm}
&A_{17}=
\left(\begin{array}{cccc}
1& 0 & 1 & 0 \\
0& 1 & 0 & 0\\
0& 0 & 1 & 0\\
0& 0 & 0 & 1
\end{array}\right)
\\
&g_1(z_1,z_2,z_3,z_4)=(-z_1,z_2,z_3,z_4).\hspace{.1cm}   &h_1(z_1,z_2,z_3,z_4)=(-z_1,z_2,\bar z_3,z_4)\\
&g_2(z_1,z_2,z_3,z_4)=(z_1,-z_2,z_3,z_4).                &h_2(z_1,z_2,z_3,z_4)=(z_1,-z_2,z_3,z_4)\\
&g_3(z_1,z_2,z_3,z_4)=(z_1,z_2,-z_3,\bar z_4).           &h_3(z_1,z_2,z_3,z_4)=(z_1, z_2,-z_3,z_4)\\
&g_4(z_1,z_2,z_3,z_4)=(z_1,z_2,z_3,-z_4).                &h_4(z_1,z_2,z_3,z_4)=(z_1, z_2,z_3,-z_4)
\end{align*}
\normalsize then \footnotesize
\begin{align*}
\Gamma_{2}=
\Bigl < &\left(\begin{array}{cccc}
\frac{1}{2}\\
0\\
0\\
0
\end{array}\right)
\left(\begin{array}{cccc}
1& 0 & 0 & 0\\
0& 1 & 0 & 0\\
0& 0 & 1 & 0\\
0& 0 & 0 & 1
\end{array}\right),
\left(\begin{array}{cccc}
0\\
\frac{1}{2}\\
0\\
0
\end{array}\right)
\left(\begin{array}{cccc}
1& 0 & 0 & 0\\
0& 1 & 0 & 0\\
0& 0 & 1 & 0\\
0& 0 & 0 & 1
\end{array}\right),
\\
&\left(\begin{array}{cccc}
0\\
0\\
\frac{1}{2}\\
0
\end{array}\right)
\left(\begin{array}{cccc}
1& 0 & 0 & 0\\
0& 1 & 0 & 0\\
0& 0 & 1 & 0\\
0& 0 & 0 & -1
\end{array}\right),
\left(\begin{array}{cccc}
0\\
0\\
0\\
\frac{1}{2}
\end{array}\right)
\left(\begin{array}{cccc}
1& 0 & 0 & 0\\
0& 1 & 0 & 0\\
0& 0 & 1 & 0\\
0& 0 & 0 & 1
\end{array}\right)
\Bigr >
\end{align*}

\begin{align*}
\Gamma_{17}= 
\Bigl < & \left(\begin{array}{cccc}
\frac{1}{2}\\
0\\
0\\
0
\end{array}\right)
\left(\begin{array}{cccc}
1& 0 & 0 & 0\\
0& 1 & 0 & 0\\
0& 0 & -1 & 0\\
0& 0 & 0 & 1
\end{array}\right),
\left(\begin{array}{cccc}
0\\
\frac{1}{2}\\
0\\
0
\end{array}\right)
\left(\begin{array}{cccc}
1& 0 & 0 & 0\\
0& 1 & 0 & 0\\
0& 0 & 1 & 0\\
0& 0 & 0 & 1
\end{array}\right),
\\
&\left(\begin{array}{cccc}
0\\
0\\
\frac{1}{2}\\
0
\end{array}\right)
\left(\begin{array}{cccc}
1& 0 & 0 & 0\\
0& 1 & 0 & 0\\
0& 0 & 1 & 0\\
0& 0 & 0 & 1
\end{array}\right),
\left(\begin{array}{cccc}
0\\
0\\
0\\
\frac{1}{2}
\end{array}\right)
\left(\begin{array}{cccc}
1& 0 & 0 & 0\\
0& 1 & 0 & 0\\
0& 0 & 1 & 0\\
0& 0 & 0 & 1
\end{array}\right)
\Bigr >.
\end{align*}

\normalsize
\noindent
Let $\varphi(z_1,z_2,z_3,z_4)=(z_3,z_1,z_4,z_2)$, we get these commutative diagrams

\footnotesize
\[
\begin{CD}
(z_1,z_2,z_3,z_4) @>\varphi>>(z_3,z_1,z_4,z_2)\\
@Vg_1VV @Vh_2 VV\\
(-z_1,z_2,z_3,z_4) @>\varphi>> (z_3,-z_1,z_4,z_2)
\end{CD}
\hspace{.1cm}
\begin{CD}
(z_1,z_2,z_3,z_4) @>\varphi>>(z_3,z_1,z_4,z_2)\\
@Vg_3VV @Vh_1 VV\\
(z_1,z_2,-z_3,\bar z_4) @>\varphi>> (-z_3,z_1,\bar z_4,z_2)
\end{CD}
\]
\[ 
\begin{CD}
(z_1,z_2,z_3,z_4) @>\varphi>>(z_3,z_1,z_4,z_2)\\
@Vg_2VV @Vh_4 VV\\
(z_1,-z_2,z_3,z_4) @>\varphi>> (z_3,z_1,z_4,-z_2)
\end{CD}
\hspace{.1cm}
\begin{CD}
(z_1,z_2,z_3,z_4) @>\varphi>>(z_3,z_1,z_4,z_2)\\
@Vg_4VV @Vh_3 VV\\
(z_1,z_2,z_3,-z_4) @>\varphi>> (z_3,z_1,-z_4,z_2).
\end{CD}
\]

\normalsize
\noindent 
Therefore $\exists$ 
$\gamma 
\footnotesize
=
\Bigl( \left(\begin{array}{cccc}
0\\
0\\
0\\
0
\end{array}\right)
\left(\begin{array}{cccc}
0& 0 & 1 & 0\\
1& 0 & 0 & 0\\
0& 0 & 0 & 1\\
0& 1 & 0 & 0
\end{array}\right)
\Bigr)
\normalsize
\in \mathbb{A}(n)$ 
s.t. $\gamma \Gamma_{2}\gamma ^{-1}=\Gamma_{17}$.\\
%--------------------------------------------------------------------

\noindent 
$\bullet$ $M(A_{2})\approx  M(A_{25})$.\\
For
\footnotesize
\begin{align*}
&A_{2}=
\left(\begin{array}{cccc}
1& 0 & 0 & 0 \\
0& 1 & 0 & 0\\
0& 0 & 1 & 1\\
0& 0 & 0 & 1
\end{array}\right)
\hspace{.1cm}
&A_{25}=
\left(\begin{array}{cccc}
1& 0 & 1 & 0 \\
0& 1 & 1 & 0\\
0& 0 & 1 & 0\\
0& 0 & 0 & 1
\end{array}\right)
\\
&g_1(z_1,z_2,z_3,z_4)=(-z_1,z_2,z_3,z_4).\hspace{.1cm}   &h_1h_2(z_1,z_2,z_3,z_4)=(-z_1,-z_2,z_3,z_4)\\
&g_2(z_1,z_2,z_3,z_4)=(z_1,-z_2,z_3,z_4).                &h_2(z_1,z_2,z_3,z_4)=(z_1,-z_2,\bar z_3,z_4)\\
&g_3(z_1,z_2,z_3,z_4)=(z_1,z_2,-z_3,\bar z_4).           &h_3(z_1,z_2,z_3,z_4)=(z_1, z_2,-z_3,z_4)\\
&g_4(z_1,z_2,z_3,z_4)=(z_1,z_2,z_3,-z_4).                &h_4(z_1,z_2,z_3,z_4)=(z_1, z_2,z_3,-z_4)
\end{align*}
\normalsize then \footnotesize
\begin{align*}
\Gamma_{2}=
\Bigl < &\left(\begin{array}{cccc}
\frac{1}{2}\\
0\\
0\\
0
\end{array}\right)
\left(\begin{array}{cccc}
1& 0 & 0 & 0\\
0& 1 & 0 & 0\\
0& 0 & 1 & 0\\
0& 0 & 0 & 1
\end{array}\right),
\left(\begin{array}{cccc}
0\\
\frac{1}{2}\\
0\\
0
\end{array}\right)
\left(\begin{array}{cccc}
1& 0 & 0 & 0\\
0& 1 & 0 & 0\\
0& 0 & 1 & 0\\
0& 0 & 0 & 1
\end{array}\right),
\\
&\left(\begin{array}{cccc}
0\\
0\\
\frac{1}{2}\\
0
\end{array}\right)
\left(\begin{array}{cccc}
1& 0 & 0 & 0\\
0& 1 & 0 & 0\\
0& 0 & 1 & 0\\
0& 0 & 0 & -1
\end{array}\right),
\left(\begin{array}{cccc}
0\\
0\\
0\\
\frac{1}{2}
\end{array}\right)
\left(\begin{array}{cccc}
1& 0 & 0 & 0\\
0& 1 & 0 & 0\\
0& 0 & 1 & 0\\
0& 0 & 0 & 1
\end{array}\right)
\Bigr >
\end{align*}

\begin{align*}
\Gamma_{25}= 
\Bigl < & \left(\begin{array}{cccc}
\frac{1}{2}\\
\frac{1}{2}\\
0\\
0
\end{array}\right)
\left(\begin{array}{cccc}
1& 0 & 0 & 0\\
0& 1 & 0 & 0\\
0& 0 & 1 & 0\\
0& 0 & 0 & 1
\end{array}\right),
\left(\begin{array}{cccc}
0\\
\frac{1}{2}\\
0\\
0
\end{array}\right)
\left(\begin{array}{cccc}
1& 0 & 0 & 0\\
0& 1 & 0 & 0\\
0& 0 & -1 & 0\\
0& 0 & 0 & 1
\end{array}\right),
\\
&\left(\begin{array}{cccc}
0\\
0\\
\frac{1}{2}\\
0
\end{array}\right)
\left(\begin{array}{cccc}
1& 0 & 0 & 0\\
0& 1 & 0 & 0\\
0& 0 & 1 & 0\\
0& 0 & 0 & 1
\end{array}\right),
\left(\begin{array}{cccc}
0\\
0\\
0\\
\frac{1}{2}
\end{array}\right)
\left(\begin{array}{cccc}
1& 0 & 0 & 0\\
0& 1 & 0 & 0\\
0& 0 & 1 & 0\\
0& 0 & 0 & 1
\end{array}\right)
\Bigr >.
\end{align*}

\normalsize
\noindent
Let $\varphi(z_1,z_2,z_3,z_4)=(z_2,z_2z_3,z_4,z_1)$, we get these commutative diagrams

\footnotesize
\[
\begin{CD}
(z_1,z_2,z_3,z_4) @>\varphi>>(z_2,z_2z_3,z_4,z_1)\\
@Vg_1VV @Vh_4 VV\\
(-z_1,z_2,z_3,z_4) @>\varphi>> (z_2,z_2z_3,z_4,-z_1)
\end{CD}
\hspace{.1cm}
\begin{CD}
(z_1,z_2,z_3,z_4) @>\varphi>>(z_2,z_2z_3,z_4,z_1)\\
@Vg_3VV @Vh_2 VV\\
(z_1,z_2,-z_3,\bar z_4) @>\varphi>> (z_2,-z_2z_3,\bar z_4,z_1)
\end{CD}
\]
\[ 
\begin{CD}
(z_1,z_2,z_3,z_4) @>\varphi>>(z_2,z_2z_3,z_4,z_1)\\
@Vg_2VV @Vh_1h_2 VV\\
(z_1,-z_2,z_3,z_4) @>\varphi>> (-z_2,-z_2z_3,z_4,z_1)
\end{CD}
\hspace{.1cm}
\begin{CD}
(z_1,z_2,z_3,z_4) @>\varphi>>(-z_2,-z_2z_3,z_4,z_1)\\
@Vg_4VV @Vh_3 VV\\
(z_1,z_2,z_3,-z_4) @>\varphi>> (z_2,z_2z_3,-z_4,z_1).
\end{CD}
\]

\normalsize
\noindent 
Therefore $\exists$ 
$\gamma 
\footnotesize
=
\Bigl( \left(\begin{array}{cccc}
0\\
0\\
0\\
0
\end{array}\right)
\left(\begin{array}{cccc}
0& 1 & 0 & 0\\
0& 1 & 1 & 0\\
0& 0 & 0 & 1\\
1& 0 & 0 & 0
\end{array}\right)
\Bigr)
\normalsize
\in \mathbb{A}(n)$ 
s.t. $\gamma \Gamma_{2}\gamma ^{-1}=\Gamma_{25}$.\\
%--------------------------------------------------------------------

\noindent 
$\bullet$ $M(A_{2})\approx  M(A_{a1})$.\\
For
\footnotesize
\begin{align*}
&A_{2}=
\left(\begin{array}{cccc}
1& 0 & 0 & 0 \\
0& 1 & 0 & 0\\
0& 0 & 1 & 1\\
0& 0 & 0 & 1
\end{array}\right)
\hspace{.1cm}
&A_{a1}=
\left(\begin{array}{cccc}
1& 1 & 0 & 0 \\
0& 1 & 0 & 0\\
0& 0 & 1 & 0\\
0& 0 & 0 & 1
\end{array}\right)
\\
&g_1(z_1,z_2,z_3,z_4)=(-z_1,z_2,z_3,z_4).\hspace{.1cm}   &h_1(z_1,z_2,z_3,z_4)=(-z_1,\bar z_2,z_3,z_4)\\
&g_2(z_1,z_2,z_3,z_4)=(z_1,-z_2,z_3,z_4).                &h_2(z_1,z_2,z_3,z_4)=(z_1,-z_2,z_3,z_4)\\
&g_3(z_1,z_2,z_3,z_4)=(z_1,z_2,-z_3,\bar z_4).           &h_3(z_1,z_2,z_3,z_4)=(z_1, z_2,-z_3,z_4)\\
&g_4(z_1,z_2,z_3,z_4)=(z_1,z_2,z_3,-z_4).                &h_4(z_1,z_2,z_3,z_4)=(z_1, z_2,z_3,-z_4)
\end{align*}
\normalsize then \footnotesize
\begin{align*}
\Gamma_{2}=
\Bigl < &\left(\begin{array}{cccc}
\frac{1}{2}\\
0\\
0\\
0
\end{array}\right)
\left(\begin{array}{cccc}
1& 0 & 0 & 0\\
0& 1 & 0 & 0\\
0& 0 & 1 & 0\\
0& 0 & 0 & 1
\end{array}\right),
\left(\begin{array}{cccc}
0\\
\frac{1}{2}\\
0\\
0
\end{array}\right)
\left(\begin{array}{cccc}
1& 0 & 0 & 0\\
0& 1 & 0 & 0\\
0& 0 & 1 & 0\\
0& 0 & 0 & 1
\end{array}\right),
\\
&\left(\begin{array}{cccc}
0\\
0\\
\frac{1}{2}\\
0
\end{array}\right)
\left(\begin{array}{cccc}
1& 0 & 0 & 0\\
0& 1 & 0 & 0\\
0& 0 & 1 & 0\\
0& 0 & 0 & -1
\end{array}\right),
\left(\begin{array}{cccc}
0\\
0\\
0\\
\frac{1}{2}
\end{array}\right)
\left(\begin{array}{cccc}
1& 0 & 0 & 0\\
0& 1 & 0 & 0\\
0& 0 & 1 & 0\\
0& 0 & 0 & 1
\end{array}\right)
\Bigr >
\end{align*}

\begin{align*}
\Gamma_{a1}= 
\Bigl < & \left(\begin{array}{cccc}
\frac{1}{2}\\
0\\
0\\
0
\end{array}\right)
\left(\begin{array}{cccc}
1& 0 & 0 & 0\\
0& -1 & 0 & 0\\
0& 0 & 1 & 0\\
0& 0 & 0 & 1
\end{array}\right),
\left(\begin{array}{cccc}
0\\
\frac{1}{2}\\
0\\
0
\end{array}\right)
\left(\begin{array}{cccc}
1& 0 & 0 & 0\\
0& 1 & 0 & 0\\
0& 0 & 1 & 0\\
0& 0 & 0 & 1
\end{array}\right),
\\
&\left(\begin{array}{cccc}
0\\
0\\
\frac{1}{2}\\
0
\end{array}\right)
\left(\begin{array}{cccc}
1& 0 & 0 & 0\\
0& 1 & 0 & 0\\
0& 0 & 1 & 0\\
0& 0 & 0 & 1
\end{array}\right),
\left(\begin{array}{cccc}
0\\
0\\
0\\
\frac{1}{2}
\end{array}\right)
\left(\begin{array}{cccc}
1& 0 & 0 & 0\\
0& 1 & 0 & 0\\
0& 0 & 1 & 0\\
0& 0 & 0 & 1
\end{array}\right)
\Bigr >.
\end{align*}

\normalsize
\noindent
Let $\varphi(z_1,z_2,z_3,z_4)=(z_3,z_4,z_1,z_2)$, we get these commutative diagrams

\footnotesize
\[
\begin{CD}
(z_1,z_2,z_3,z_4) @>\varphi>>(z_3,z_4,z_1,z_2)\\
@Vg_1VV @Vh_3 VV\\
(-z_1,z_2,z_3,z_4) @>\varphi>> (z_3,z_4,-z_1,z_2)
\end{CD}
\hspace{.1cm}
\begin{CD}
(z_1,z_2,z_3,z_4) @>\varphi>>(z_3,z_4,z_1,z_2)\\
@Vg_3VV @Vh_1 VV\\
(z_1,z_2,-z_3,\bar z_4) @>\varphi>> (-z_3,\bar z_4,z_1,z_2)
\end{CD}
\]
\[ 
\begin{CD}
(z_1,z_2,z_3,z_4) @>\varphi>>(z_3,z_4,z_1,z_2)\\
@Vg_2VV @Vh_4 VV\\
(z_1,-z_2,z_3,z_4) @>\varphi>> (z_3,z_4,z_1,-z_2)
\end{CD}
\hspace{.1cm}
\begin{CD}
(z_1,z_2,z_3,z_4) @>\varphi>>(z_3,z_4,z_1,z_2)\\
@Vg_4VV @Vh_2 VV\\
(z_1,z_2,z_3,-z_4) @>\varphi>> (z_3,-z_4,z_1,z_2).
\end{CD}
\]

\normalsize
\noindent 
Therefore $\exists$ 
$\gamma 
\footnotesize
=
\Bigl( \left(\begin{array}{cccc}
0\\
0\\
0\\
0
\end{array}\right)
\left(\begin{array}{cccc}
0& 0 & 1 & 0\\
0& 0 & 0 & 1\\
1& 0 & 0 & 0\\
0& 1 & 0 & 0
\end{array}\right)
\Bigr)
\normalsize
\in \mathbb{A}(n)$ 
s.t. $\gamma \Gamma_{2}\gamma ^{-1}=\Gamma_{a1}$.\\
%--------------------------------------------------------------------
%=====================================================================================

\noindent {\bf (6).} 
If $M(A_{11})$ is  diffeomorphic to $M(A_{2})$ and $M(A_{a21})$ respectively, then by Bieberbach's theorem $\exists B\in GL(4,\mathbb{R})$ 
such that $BPB^{-1}=Q$ and $BPB^{-1}=R$ respectively, where 
$\Phi_{11}=\footnotesize   
\Bigl <I, P=\left(\begin{array}{cccc}
1& 0 & 0 & 0\\
0& 1 & 0 & 0\\
0& 0 & -1 & 0\\
0& 0 & 0 & -1
\end{array}\right)\Bigr>$, \normalsize
$\Phi_{2}=\footnotesize 
\Bigl <I, Q=\left(\begin{array}{cccc}
1& 0 & 0 & 0\\
0& 1 & 0 & 0\\
0& 0 & 1 & 0\\
0& 0 & 0 & -1
\end{array}\right)\Bigr>$, \normalsize
$\Phi_{a21}=\footnotesize 
\Bigl <I, R=\left(\begin{array}{cccc}
1& 0 & 0 & 0\\
0& -1 & 0 & 0\\
0& 0 & -1 & 0\\
0& 0 & 0 & -1
\end{array}\right)\Bigr>$. \normalsize
Similar to the proof of Theorem \ref{T:3}(e), by argument of eigenvalue, 
there is a contradiction, where the eigenvalues of $P$ are 1,1,-1 and -1, but the eigenvalues of $Q$
are 1,1,1 and -1, and the eigenvalues of $R$ are 1,-1,-1 and -1.\\
%===============================================================================

\noindent {\bf (7).} Similar to the proof of (4) above.\\ 

\noindent {\bf (8).} Similar to the proof of (6) above.\\
%============================================================================

\noindent {\bf (9).}\\
\noindent 
$\bullet$ $M(A_{a3})\approx  M(A_{a7})$.\\
For
\footnotesize
\begin{align*}
&A_{a3}=
\left(\begin{array}{cccc}
1& 1 & 0 & 0 \\
0& 1 & 0 & 1\\
0& 0 & 1 & 0\\
0& 0 & 0 & 1
\end{array}\right)
\hspace{.1cm}
&A_{a7}=
\left(\begin{array}{cccc}
1& 1 & 0 & 1 \\
0& 1 & 0 & 1\\
0& 0 & 1 & 0\\
0& 0 & 0 & 1
\end{array}\right)
\\
&g_1(z_1,z_2,z_3,z_4)=(-z_1,\bar z_2,z_3,z_4).\hspace{.1cm}   &h_1h_2(z_1,z_2,z_3,z_4)=(-z_1,-\bar z_2,z_3,z_4)\\
&g_2(z_1,z_2,z_3,z_4)=(z_1,-z_2,z_3,\bar z_4).                &h_2(z_1,z_2,z_3,z_4)=(z_1,-z_2,z_3,\bar z_4)\\
&g_3(z_1,z_2,z_3,z_4)=(z_1,z_2,-z_3,z_4).                     &h_3(z_1,z_2,z_3,z_4)=(z_1, z_2,-z_3,z_4)\\
&g_4(z_1,z_2,z_3,z_4)=(z_1,z_2,z_3,-z_4).                     &h_4(z_1,z_2,z_3,z_4)=(z_1, z_2,z_3,-z_4)
\end{align*}
\normalsize then \footnotesize
\begin{align*}
\Gamma_{a3}=
\Bigl < &\left(\begin{array}{cccc}
\frac{1}{2}\\
0\\
0\\
0
\end{array}\right)
\left(\begin{array}{cccc}
1& 0 & 0 & 0\\
0& -1 & 0 & 0\\
0& 0 & 1 & 0\\
0& 0 & 0 & 1
\end{array}\right),
\left(\begin{array}{cccc}
0\\
\frac{1}{2}\\
0\\
0
\end{array}\right)
\left(\begin{array}{cccc}
1& 0 & 0 & 0\\
0& 1 & 0 & 0\\
0& 0 & 1 & 0\\
0& 0 & 0 & -1
\end{array}\right),
\\
&\left(\begin{array}{cccc}
0\\
0\\
\frac{1}{2}\\
0
\end{array}\right)
\left(\begin{array}{cccc}
1& 0 & 0 & 0\\
0& 1 & 0 & 0\\
0& 0 & 1 & 0\\
0& 0 & 0 & 1
\end{array}\right),
\left(\begin{array}{cccc}
0\\
0\\
0\\
\frac{1}{2}
\end{array}\right)
\left(\begin{array}{cccc}
1& 0 & 0 & 0\\
0& 1 & 0 & 0\\
0& 0 & 1 & 0\\
0& 0 & 0 & 1
\end{array}\right)
\Bigr >
\end{align*}

\begin{align*}
\Gamma_{a7}= 
\Bigl < & \left(\begin{array}{cccc}
\frac{1}{2}\\
\frac{1}{2}\\
0\\
0
\end{array}\right)
\left(\begin{array}{cccc}
1& 0 & 0 & 0\\
0& -1 & 0 & 0\\
0& 0 & 1 & 0\\
0& 0 & 0 & 1
\end{array}\right),
\left(\begin{array}{cccc}
0\\
\frac{1}{2}\\
0\\
0
\end{array}\right)
\left(\begin{array}{cccc}
1& 0 & 0 & 0\\
0& 1 & 0 & 0\\
0& 0 & 1 & 0\\
0& 0 & 0 & -1
\end{array}\right),
\\
&\left(\begin{array}{cccc}
0\\
0\\
\frac{1}{2}\\
0
\end{array}\right)
\left(\begin{array}{cccc}
1& 0 & 0 & 0\\
0& 1 & 0 & 0\\
0& 0 & 1 & 0\\
0& 0 & 0 & 1
\end{array}\right),
\left(\begin{array}{cccc}
0\\
0\\
0\\
\frac{1}{2}
\end{array}\right)
\left(\begin{array}{cccc}
1& 0 & 0 & 0\\
0& 1 & 0 & 0\\
0& 0 & 1 & 0\\
0& 0 & 0 & 1
\end{array}\right)
\Bigr >.
\end{align*}

\normalsize
\noindent
Let $\varphi(z_1,z_2,z_3,z_4)=(z_1,iz_2,z_3,z_4)$, we get these commutative diagrams

\footnotesize
\[
\begin{CD}
(z_1,z_2,z_3,z_4) @>\varphi>>(z_1,iz_2,z_3,z_4)\\
@Vg_1VV @Vh_1h_2 VV\\
(-z_1,\bar z_2,z_3,z_4) @>\varphi>> (-z_1,-\bar {iz}_2,z_3,z_4)
\end{CD}
\hspace{.1cm}
\begin{CD}
(z_1,z_2,z_3,z_4) @>\varphi>>(z_1,iz_2,z_3,z_4)\\
@Vg_3VV @Vh_3 VV\\
(z_1,z_2,-z_3,z_4) @>\varphi>> (z_1,iz_2,-z_3,z_4)
\end{CD}
\]
\[ 
\begin{CD}
(z_1,z_2,z_3,z_4) @>\varphi>>(z_1,iz_2,z_3,z_4)\\
@Vg_2VV @Vh_2 VV\\
(z_1,-z_2,z_3,\bar z_4) @>\varphi>> (z_1,-iz_2,z_3,\bar z_4)
\end{CD}
\hspace{.1cm}
\begin{CD}
(z_1,z_2,z_3,z_4) @>\varphi>>(z_1,iz_2,z_3,z_4)\\
@Vg_4VV @Vh_4 VV\\
(z_1,z_2,z_3,-z_4) @>\varphi>> (z_1,iz_2,z_3,-z_4).
\end{CD}
\]

\normalsize
\noindent 
Therefore $\exists$ 
$\gamma 
\footnotesize
=
\Bigl( \left(\begin{array}{cccc}
0\\
\frac{1}{4}\\
0\\
0
\end{array}\right)
\left(\begin{array}{cccc}
1& 0 & 0 & 0\\
0& 1 & 0 & 0\\
0& 0 & 1 & 0\\
0& 0 & 0 & 1
\end{array}\right)
\Bigr)
\normalsize
\in \mathbb{A}(n)$ 
s.t. $\gamma \Gamma_{a3}\gamma ^{-1}=\Gamma_{a7}.$\\
%--------------------------------------------------------------------

\noindent 
$\bullet$ $M(A_{10})\approx  M(A_{a3})$.\\
For
\footnotesize
\begin{align*}
&A_{10}=
\left(\begin{array}{cccc}
1& 0 & 0 & 0 \\
0& 1 & 1 & 0\\
0& 0 & 1 & 1\\
0& 0 & 0 & 1
\end{array}\right)
\hspace{.1cm}
&A_{a3}=
\left(\begin{array}{cccc}
1& 1 & 0 & 0 \\
0& 1 & 0 & 1\\
0& 0 & 1 & 0\\
0& 0 & 0 & 1
\end{array}\right)
\\
&g_1(z_1,z_2,z_3,z_4)=(-z_1,z_2,z_3,z_4).\hspace{.1cm}   &h_1(z_1,z_2,z_3,z_4)=(-z_1,\bar z_2,z_3,z_4)\\
&g_2(z_1,z_2,z_3,z_4)=(z_1,-z_2,\bar z_3,z_4).           &h_2(z_1,z_2,z_3,z_4)=(z_1,-z_2,z_3,\bar z_4)\\
&g_3(z_1,z_2,z_3,z_4)=(z_1,z_2,-z_3,\bar z_4).           &h_3(z_1,z_2,z_3,z_4)=(z_1,z_2,-z_3,z_4)\\
&g_4(z_1,z_2,z_3,z_4)=(z_1,z_2,z_3,-z_4).                &h_4(z_1,z_2,z_3,z_4)=(z_1,z_2,z_3,-z_4)
\end{align*}
\normalsize then \footnotesize
\begin{align*}
\Gamma_{10}=
\Bigl < &\left(\begin{array}{cccc}
\frac{1}{2}\\
0\\
0\\
0
\end{array}\right)
\left(\begin{array}{cccc}
1& 0 & 0 & 0\\
0& 1 & 0 & 0\\
0& 0 & 1 & 0\\
0& 0 & 0 & 1
\end{array}\right),
\left(\begin{array}{cccc}
0\\
\frac{1}{2}\\
0\\
0
\end{array}\right)
\left(\begin{array}{cccc}
1& 0 & 0 & 0\\
0& 1 & 0 & 0\\
0& 0 & -1 & 0\\
0& 0 & 0 & 1
\end{array}\right),
\\
&\left(\begin{array}{cccc}
0\\
0\\
\frac{1}{2}\\
0
\end{array}\right)
\left(\begin{array}{cccc}
1& 0 & 0 & 0\\
0& 1 & 0 & 0\\
0& 0 & 1 & 0\\
0& 0 & 0 & -1
\end{array}\right),
\left(\begin{array}{cccc}
0\\
0\\
0\\
\frac{1}{2}
\end{array}\right)
\left(\begin{array}{cccc}
1& 0 & 0 & 0\\
0& 1 & 0 & 0\\
0& 0 & 1 & 0\\
0& 0 & 0 & 1
\end{array}\right)
\Bigr >
\end{align*}

\begin{align*}
\Gamma_{a3}= 
\Bigl < & \left(\begin{array}{cccc}
\frac{1}{2}\\
0\\
0\\
0
\end{array}\right)
\left(\begin{array}{cccc}
1& 0 & 0 & 0\\
0& -1 & 0 & 0\\
0& 0 & 1 & 0\\
0& 0 & 0 & 1
\end{array}\right),
\left(\begin{array}{cccc}
0\\
\frac{1}{2}\\
0\\
0
\end{array}\right)
\left(\begin{array}{cccc}
1& 0 & 0 & 0\\
0& 1 & 0 & 0\\
0& 0 & 1 & 0\\
0& 0 & 0 & -1
\end{array}\right),
\\
&\left(\begin{array}{cccc}
0\\
0\\
\frac{1}{2}\\
0
\end{array}\right)
\left(\begin{array}{cccc}
1& 0 & 0 & 0\\
0& 1 & 0 & 0\\
0& 0 & 1 & 0\\
0& 0 & 0 & 1
\end{array}\right),
\left(\begin{array}{cccc}
0\\
0\\
0\\
\frac{1}{2}
\end{array}\right)
\left(\begin{array}{cccc}
1& 0 & 0 & 0\\
0& 1 & 0 & 0\\
0& 0 & 1 & 0\\
0& 0 & 0 & 1
\end{array}\right)
\Bigr >.
\end{align*}

\normalsize
\noindent
Let $\varphi(z_1,z_2,z_3,z_4)=(z_2,z_3,z_1,z_4)$, we get these commutative diagrams

\footnotesize
\[
\begin{CD}
(z_1,z_2,z_3,z_4) @>\varphi>>(z_2,z_3,z_1,z_4)\\
@Vg_1VV @Vh_3 VV\\
(-z_1,z_2,z_3,z_4) @>\varphi>> (z_2,z_3,-z_1,z_4)
\end{CD}
\hspace{.1cm}
\begin{CD}
(z_1,z_2,z_3,z_4) @>\varphi>>(z_2,z_3,z_1,z_4)\\
@Vg_3VV @Vh_2 VV\\
(z_1,z_2,-z_3,\bar z_4) @>\varphi>> (z_2,-z_3,z_1,\bar z_4)
\end{CD}
\]
\[ 
\begin{CD}
(z_1,z_2,z_3,z_4) @>\varphi>>(z_2,z_3,z_1,z_4)\\
@Vg_2VV @Vh_1 VV\\
(z_1,-z_2,\bar z_3,z_4) @>\varphi>> (-z_2,\bar z_3,z_1,z_4)
\end{CD}
\hspace{.1cm}
\begin{CD}
(z_1,z_2,z_3,z_4) @>\varphi>>(z_2,z_3,z_1,z_4)\\
@Vg_4VV @Vh_4 VV\\
(z_1,z_2,z_3,-z_4) @>\varphi>> (z_2,z_3,z_1,-z_4).
\end{CD}
\]

\normalsize
\noindent 
Therefore $\exists$ 
$\gamma 
\footnotesize
=
\Bigl( \left(\begin{array}{cccc}
0\\
0\\
0\\
0
\end{array}\right)
\left(\begin{array}{cccc}
0& 1 & 0 & 0\\
0& 0 & 1 & 0\\
1& 0 & 0 & 0\\
0& 0 & 0 & 1
\end{array}\right)
\Bigr)
\normalsize
\in \mathbb{A}(n)$ 
s.t. $\gamma \Gamma_{10}\gamma ^{-1}=\Gamma_{a3}$.\\
%--------------------------------------------------------------------

\noindent 
$\bullet$ $M(A_{a9})\approx  M(A_{a7})$.\\
For
\footnotesize
\begin{align*}
&A_{a9}=
\left(\begin{array}{cccc}
1& 1 & 0 & 0 \\
0& 1 & 1 & 0\\
0& 0 & 1 & 0\\
0& 0 & 0 & 1
\end{array}\right)
\hspace{.1cm}
&A_{a7}=
\left(\begin{array}{cccc}
1& 1 & 0 & 1 \\
0& 1 & 0 & 1\\
0& 0 & 1 & 0\\
0& 0 & 0 & 1
\end{array}\right)
\\
&g_1(z_1,z_2,z_3,z_4)=(-z_1,\bar z_2,z_3,z_4).\hspace{.1cm}   &h_1h_2(z_1,z_2,z_3,z_4)=(-z_1,-\bar z_2,z_3,z_4)\\
&g_2(z_1,z_2,z_3,z_4)=(z_1,-z_2,\bar z_3,z_4).                &h_2(z_1,z_2,z_3,z_4)=(z_1,-z_2,z_3,\bar z_4)\\
&g_3(z_1,z_2,z_3,z_4)=(z_1,z_2,-z_3,z_4).                     &h_3(z_1,z_2,z_3,z_4)=(z_1, z_2,-z_3,z_4)\\
&g_4(z_1,z_2,z_3,z_4)=(z_1,z_2,z_3,-z_4).                     &h_4(z_1,z_2,z_3,z_4)=(z_1, z_2,z_3,-z_4)
\end{align*}
\normalsize then \footnotesize
\begin{align*}
\Gamma_{a9}=
\Bigl < &\left(\begin{array}{cccc}
\frac{1}{2}\\
0\\
0\\
0
\end{array}\right)
\left(\begin{array}{cccc}
1& 0 & 0 & 0\\
0& -1 & 0 & 0\\
0& 0 & 1 & 0\\
0& 0 & 0 & 1
\end{array}\right),
\left(\begin{array}{cccc}
0\\
\frac{1}{2}\\
0\\
0
\end{array}\right)
\left(\begin{array}{cccc}
1& 0 & 0 & 0\\
0& 1 & 0 & 0\\
0& 0 & -1 & 0\\
0& 0 & 0 & 1
\end{array}\right),
\\
&\left(\begin{array}{cccc}
0\\
0\\
\frac{1}{2}\\
0
\end{array}\right)
\left(\begin{array}{cccc}
1& 0 & 0 & 0\\
0& 1 & 0 & 0\\
0& 0 & 1 & 0\\
0& 0 & 0 & 1
\end{array}\right),
\left(\begin{array}{cccc}
0\\
0\\
0\\
\frac{1}{2}
\end{array}\right)
\left(\begin{array}{cccc}
1& 0 & 0 & 0\\
0& 1 & 0 & 0\\
0& 0 & 1 & 0\\
0& 0 & 0 & 1
\end{array}\right)
\Bigr >
\end{align*}

\begin{align*}
\Gamma_{a7}= 
\Bigl < & \left(\begin{array}{cccc}
\frac{1}{2}\\
\frac{1}{2}\\
0\\
0
\end{array}\right)
\left(\begin{array}{cccc}
1& 0 & 0 & 0\\
0& -1 & 0 & 0\\
0& 0 & 1 & 0\\
0& 0 & 0 & 1
\end{array}\right),
\left(\begin{array}{cccc}
0\\
\frac{1}{2}\\
0\\
0
\end{array}\right)
\left(\begin{array}{cccc}
1& 0 & 0 & 0\\
0& 1 & 0 & 0\\
0& 0 & 1 & 0\\
0& 0 & 0 & -1
\end{array}\right),
\\
&\left(\begin{array}{cccc}
0\\
0\\
\frac{1}{2}\\
0
\end{array}\right)
\left(\begin{array}{cccc}
1& 0 & 0 & 0\\
0& 1 & 0 & 0\\
0& 0 & 1 & 0\\
0& 0 & 0 & 1
\end{array}\right),
\left(\begin{array}{cccc}
0\\
0\\
0\\
\frac{1}{2}
\end{array}\right)
\left(\begin{array}{cccc}
1& 0 & 0 & 0\\
0& 1 & 0 & 0\\
0& 0 & 1 & 0\\
0& 0 & 0 & 1
\end{array}\right)
\Bigr >.
\end{align*}

\normalsize
\noindent
Let $\varphi(z_1,z_2,z_3,z_4)=(z_1,iz_2,z_4,z_3)$, we get these commutative diagrams

\footnotesize
\[
\begin{CD}
(z_1,z_2,z_3,z_4) @>\varphi>>(z_1,iz_2,z_4,z_3)\\
@Vg_1VV @Vh_1h_2 VV\\
(-z_1,\bar z_2,z_3,z_4) @>\varphi>> (-z_1,-\bar {iz}_2,z_4,z_3)
\end{CD}
\hspace{.1cm}
\begin{CD}
(z_1,z_2,z_3,z_4) @>\varphi>>(z_1,iz_2,z_4,z_3)\\
@Vg_3VV @Vh_4 VV\\
(z_1,z_2,-z_3,z_4) @>\varphi>> (z_1,iz_2,z_4,-z_3)
\end{CD}
\]
\[ 
\begin{CD}
(z_1,z_2,z_3,z_4) @>\varphi>>(z_1,iz_2,z_4,z_3)\\
@Vg_2VV @Vh_2 VV\\
(z_1,-z_2,\bar z_3,z_4) @>\varphi>> (z_1,-iz_2,z_4,\bar z_3)
\end{CD}
\hspace{.1cm}
\begin{CD}
(z_1,z_2,z_3,z_4) @>\varphi>>(z_1,iz_2,z_4,z_3)\\
@Vg_4VV @Vh_3 VV\\
(z_1,z_2,z_3,-z_4) @>\varphi>> (z_1,iz_2,-z_4,z_3).
\end{CD}
\]

\normalsize
\noindent 
Therefore $\exists$ 
$\gamma 
\footnotesize
=
\Bigl( \left(\begin{array}{cccc}
0\\
\frac{1}{4}\\
0\\
0
\end{array}\right)
\left(\begin{array}{cccc}
1& 0 & 0 & 0\\
0& 1 & 0 & 0\\
0& 0 & 0 & 1\\
0& 0 & 1 & 0
\end{array}\right)
\Bigr)
\normalsize
\in \mathbb{A}(n)$ 
s.t. $\gamma \Gamma_{a9}\gamma ^{-1}=\Gamma_{a7}.$\\
%--------------------------------------------------------------------

\noindent 
$\bullet$ $M(A_{a9})\approx  M(A_{a25})$.\\
For
\footnotesize
\begin{align*}
&A_{a9}=
\left(\begin{array}{cccc}
1& 1 & 0 & 0 \\
0& 1 & 1 & 0\\
0& 0 & 1 & 0\\
0& 0 & 0 & 1
\end{array}\right)
\hspace{.1cm}
&A_{a25}=
\left(\begin{array}{cccc}
1& 1 & 1 & 0 \\
0& 1 & 1 & 0\\
0& 0 & 1 & 0\\
0& 0 & 0 & 1
\end{array}\right)
\\
&g_1(z_1,z_2,z_3,z_4)=(-z_1,\bar z_2,z_3,z_4).\hspace{.1cm}   &h_1h_2(z_1,z_2,z_3,z_4)=(-z_1,-\bar z_2,z_3,z_4)\\
&g_2(z_1,z_2,z_3,z_4)=(z_1,-z_2,\bar z_3,z_4).                &h_2(z_1,z_2,z_3,z_4)=(z_1,-z_2,\bar z_3,z_4)\\
&g_3(z_1,z_2,z_3,z_4)=(z_1,z_2,-z_3,z_4).                     &h_3(z_1,z_2,z_3,z_4)=(z_1, z_2,-z_3,z_4)\\
&g_4(z_1,z_2,z_3,z_4)=(z_1,z_2,z_3,-z_4).                     &h_4(z_1,z_2,z_3,z_4)=(z_1, z_2,z_3,-z_4)
\end{align*}
\normalsize then \footnotesize
\begin{align*}
\Gamma_{a9}=
\Bigl < &\left(\begin{array}{cccc}
\frac{1}{2}\\
0\\
0\\
0
\end{array}\right)
\left(\begin{array}{cccc}
1& 0 & 0 & 0\\
0& -1 & 0 & 0\\
0& 0 & 1 & 0\\
0& 0 & 0 & 1
\end{array}\right),
\left(\begin{array}{cccc}
0\\
\frac{1}{2}\\
0\\
0
\end{array}\right)
\left(\begin{array}{cccc}
1& 0 & 0 & 0\\
0& 1 & 0 & 0\\
0& 0 & -1 & 0\\
0& 0 & 0 & 1
\end{array}\right),
\\
&\left(\begin{array}{cccc}
0\\
0\\
\frac{1}{2}\\
0
\end{array}\right)
\left(\begin{array}{cccc}
1& 0 & 0 & 0\\
0& 1 & 0 & 0\\
0& 0 & 1 & 0\\
0& 0 & 0 & 1
\end{array}\right),
\left(\begin{array}{cccc}
0\\
0\\
0\\
\frac{1}{2}
\end{array}\right)
\left(\begin{array}{cccc}
1& 0 & 0 & 0\\
0& 1 & 0 & 0\\
0& 0 & 1 & 0\\
0& 0 & 0 & 1
\end{array}\right)
\Bigr >
\end{align*}

\begin{align*}
\Gamma_{a25}= 
\Bigl < & \left(\begin{array}{cccc}
\frac{1}{2}\\
\frac{1}{2}\\
0\\
0
\end{array}\right)
\left(\begin{array}{cccc}
1& 0 & 0 & 0\\
0& -1 & 0 & 0\\
0& 0 & 1 & 0\\
0& 0 & 0 & 1
\end{array}\right),
\left(\begin{array}{cccc}
0\\
\frac{1}{2}\\
0\\
0
\end{array}\right)
\left(\begin{array}{cccc}
1& 0 & 0 & 0\\
0& 1 & 0 & 0\\
0& 0 & -1 & 0\\
0& 0 & 0 & 1
\end{array}\right),
\\
&\left(\begin{array}{cccc}
0\\
0\\
\frac{1}{2}\\
0
\end{array}\right)
\left(\begin{array}{cccc}
1& 0 & 0 & 0\\
0& 1 & 0 & 0\\
0& 0 & 1 & 0\\
0& 0 & 0 & 1
\end{array}\right),
\left(\begin{array}{cccc}
0\\
0\\
0\\
\frac{1}{2}
\end{array}\right)
\left(\begin{array}{cccc}
1& 0 & 0 & 0\\
0& 1 & 0 & 0\\
0& 0 & 1 & 0\\
0& 0 & 0 & 1
\end{array}\right)
\Bigr >.
\end{align*}

\normalsize
\noindent
Let $\varphi(z_1,z_2,z_3,z_4)=(z_1,iz_2,z_3,z_4)$, we get these commutative diagrams

\footnotesize
\[
\begin{CD}
(z_1,z_2,z_3,z_4) @>\varphi>>(z_1,iz_2,z_3,z_4)\\
@Vg_1VV @Vh_1h_2 VV\\
(-z_1,\bar z_2,z_3,z_4) @>\varphi>> (-z_1,-\bar {iz}_2,z_3,z_4)
\end{CD}
\hspace{.1cm}
\begin{CD}
(z_1,z_2,z_3,z_4) @>\varphi>>(z_1,iz_2,z_3,z_4)\\
@Vg_3VV @Vh_3 VV\\
(z_1,z_2,-z_3,z_4) @>\varphi>> (z_1,iz_2,-z_3,z_4)
\end{CD}
\]
\[ 
\begin{CD}
(z_1,z_2,z_3,z_4) @>\varphi>>(z_1,iz_2,z_3,z_4)\\
@Vg_2VV @Vh_2 VV\\
(z_1,-z_2,\bar z_3,z_4) @>\varphi>> (z_1,-iz_2,\bar z_3,z_4)
\end{CD}
\hspace{.1cm}
\begin{CD}
(z_1,z_2,z_3,z_4) @>\varphi>>(z_1,iz_2,z_3,z_4)\\
@Vg_4VV @Vh_4 VV\\
(z_1,z_2,z_3,-z_4) @>\varphi>> (z_1,iz_2,z_3,-z_4).
\end{CD}
\]

\normalsize
\noindent 
Therefore $\exists$ 
$\gamma 
\footnotesize
=
\Bigl( \left(\begin{array}{cccc}
0\\
\frac{1}{4}\\
0\\
0
\end{array}\right)
\left(\begin{array}{cccc}
1& 0 & 0 & 0\\
0& 1 & 0 & 0\\
0& 0 & 1 & 0\\
0& 0 & 0 & 1
\end{array}\right)
\Bigr)
\normalsize
\in \mathbb{A}(n)$ 
s.t. $\gamma \Gamma_{a9}\gamma ^{-1}=\Gamma_{a25}.$\\
%-----------------------------------------------------------------------

\noindent 
$\bullet$ $M(A_{10})\approx  M(A_{12})$.\\
For 
\footnotesize
\begin{align*}
&A_{10}=
\left(\begin{array}{cccc}
1& 0 & 0 & 0 \\
0& 1 & 1 & 0\\
0& 0 & 1 & 1\\
0& 0 & 0 & 1
\end{array}\right)
\hspace{.1cm}
&A_{12}=
\left(\begin{array}{cccc}
1& 0 & 0 & 0 \\
0& 1 & 1 & 1\\
0& 0 & 1 & 1\\
0& 0 & 0 & 1
\end{array}\right)
\\
&g_1(z_1,z_2,z_3,z_4)=(-z_1,z_2,z_3,z_4).\hspace{.1cm}   &h_1(z_1,z_2,z_3,z_4)=(-z_1,z_2,z_3,z_4)\\
&g_2(z_1,z_2,z_3,z_4)=(z_1,-z_2,\bar z_3,z_4).           &h_2h_3(z_1,z_2,z_3,z_4)=(z_1,-z_2,-\bar z_3,z_4)\\
&g_3(z_1,z_2,z_3,z_4)=(z_1,z_2,-z_3,\bar z_4).           &h_3(z_1,z_2,z_3,z_4)=(z_1,z_2,-z_3,\bar z_4)\\
&g_4(z_1,z_2,z_3,z_4)=(z_1,z_2,z_3,-z_4).                &h_4(z_1,z_2,z_3,z_4)=(z_1,z_2,z_3,-z_4)
\end{align*}
\normalsize then \footnotesize
\begin{align*}
\Gamma_{10}=
\Bigl < &\left(\begin{array}{cccc}
\frac{1}{2}\\
0\\
0\\
0
\end{array}\right)
\left(\begin{array}{cccc}
1& 0 & 0 & 0\\
0& 1 & 0 & 0\\
0& 0 & 1 & 0\\
0& 0 & 0 & 1
\end{array}\right),
\left(\begin{array}{cccc}
0\\
\frac{1}{2}\\
0\\
0
\end{array}\right)
\left(\begin{array}{cccc}
1& 0 & 0 & 0\\
0& 1 & 0 & 0\\
0& 0 & -1 & 0\\
0& 0 & 0 & 1
\end{array}\right),
\\
&\left(\begin{array}{cccc}
0\\
0\\
\frac{1}{2}\\
0
\end{array}\right)
\left(\begin{array}{cccc}
1& 0 & 0 & 0\\
0& 1 & 0 & 0\\
0& 0 & 1 & 0\\
0& 0 & 0 & -1
\end{array}\right),
\left(\begin{array}{cccc}
0\\
0\\
0\\
\frac{1}{2}
\end{array}\right)
\left(\begin{array}{cccc}
1& 0 & 0 & 0\\
0& 1 & 0 & 0\\
0& 0 & 1 & 0\\
0& 0 & 0 & 1
\end{array}\right)
\Bigr >
\end{align*}

\begin{align*}
\Gamma_{12}= 
\Bigl < & \left(\begin{array}{cccc}
\frac{1}{2}\\
0\\
0\\
0
\end{array}\right)
\left(\begin{array}{cccc}
1& 0 & 0 & 0\\
0& 1 & 0 & 0\\
0& 0 & 1 & 0\\
0& 0 & 0 & 1
\end{array}\right),
\left(\begin{array}{cccc}
0\\
\frac{1}{2}\\
\frac{1}{2}\\
0
\end{array}\right)
\left(\begin{array}{cccc}
1& 0 & 0 & 0\\
0& 1 & 0 & 0\\
0& 0 & -1 & 0\\
0& 0 & 0 & 1
\end{array}\right),
\\
&\left(\begin{array}{cccc}
0\\
0\\
\frac{1}{2}\\
0
\end{array}\right)
\left(\begin{array}{cccc}
1& 0 & 0 & 0\\
0& 1 & 0 & 0\\
0& 0 & 1 & 0\\
0& 0 & 0 & -1
\end{array}\right),
\left(\begin{array}{cccc}
0\\
0\\
0\\
\frac{1}{2}
\end{array}\right)
\left(\begin{array}{cccc}
1& 0 & 0 & 0\\
0& 1 & 0 & 0\\
0& 0 & 1 & 0\\
0& 0 & 0 & 1
\end{array}\right)
\Bigr >.
\end{align*}

\normalsize
\noindent
Let $\varphi(z_1,z_2,z_3,z_4)=(z_1,z_2,iz_3,z_4)$, we get these commutative diagrams

\footnotesize
\[
\begin{CD}
(z_1,z_2,z_3,z_4) @>\varphi>>(z_1,z_2,iz_3,z_4)\\
@Vg_1VV @Vh_1 VV\\
(-z_1,z_2,z_3,z_4) @>\varphi>> (-z_1,z_2,iz_3,z_4)
\end{CD}
\hspace{.1cm}
\begin{CD}
(z_1,z_2,z_3,z_4) @>\varphi>>(z_1,z_2,iz_3,z_4)\\
@Vg_3VV @Vh_3 VV\\
(z_1,z_2,-z_3,\bar z_4) @>\varphi>> (z_1,z_2,-iz_3,\bar z_4)
\end{CD}
\]
\[ 
\begin{CD}
(z_1,z_2,z_3,z_4) @>\varphi>>(z_1,z_2,iz_3,z_4)\\
@Vg_2VV @Vh_2h_3 VV\\
(z_1,-z_2,\bar z_3,z_4) @>\varphi>> (z_1,-z_2,-\bar {iz}_3,z_4)
\end{CD}
\hspace{.1cm}
\begin{CD}
(z_1,z_2,z_3,z_4) @>\varphi>>(z_1,z_2,iz_3,z_4)\\
@Vg_4VV @Vh_4 VV\\
(z_1,z_2,z_3,-z_4) @>\varphi>> (z_1,z_2,iz_3,-z_4).
\end{CD}
\]

\normalsize
\noindent 
Therefore $\exists$ 
$\gamma 
\footnotesize
=
\Bigl( \left(\begin{array}{cccc}
0\\
0\\
\frac{1}{4}\\
0
\end{array}\right)
\left(\begin{array}{cccc}
1& 0 & 0 & 0\\
0& 1 & 0 & 0\\
0& 0 & 1 & 0\\
0& 0 & 0 & 1
\end{array}\right)
\Bigr)
\normalsize
\in \mathbb{A}(n)$ 
s.t. $\gamma \Gamma_{10}\gamma ^{-1}=\Gamma_{12}.$\\
%--------------------------------------------------------------------

\noindent 
$\bullet$ $M(A_{10})\approx  M(A_{18})$.\\
For
\footnotesize
\begin{align*}
&A_{10}=
\left(\begin{array}{cccc}
1& 0 & 0 & 0 \\
0& 1 & 1 & 0\\
0& 0 & 1 & 1\\
0& 0 & 0 & 1
\end{array}\right)
\hspace{.1cm}
&A_{18}=
\left(\begin{array}{cccc}
1& 0 & 1 & 0 \\
0& 1 & 0 & 0\\
0& 0 & 1 & 1\\
0& 0 & 0 & 1
\end{array}\right)
\\
&g_1(z_1,z_2,z_3,z_4)=(-z_1,z_2,z_3,z_4).\hspace{.1cm}   &h_1(z_1,z_2,z_3,z_4)=(-z_1,z_2,\bar z_3,z_4)\\
&g_2(z_1,z_2,z_3,z_4)=(z_1,-z_2,\bar z_3,z_4).           &h_2(z_1,z_2,z_3,z_4)=(z_1,-z_2,z_3,z_4)\\
&g_3(z_1,z_2,z_3,z_4)=(z_1,z_2,-z_3,\bar z_4).           &h_3(z_1,z_2,z_3,z_4)=(z_1,z_2,-z_3,\bar z_4)\\
&g_4(z_1,z_2,z_3,z_4)=(z_1,z_2,z_3,-z_4).                &h_4(z_1,z_2,z_3,z_4)=(z_1,z_2,z_3,-z_4)
\end{align*}
\normalsize then \footnotesize
\begin{align*}
\Gamma_{10}=
\Bigl < &\left(\begin{array}{cccc}
\frac{1}{2}\\
0\\
0\\
0
\end{array}\right)
\left(\begin{array}{cccc}
1& 0 & 0 & 0\\
0& 1 & 0 & 0\\
0& 0 & 1 & 0\\
0& 0 & 0 & 1
\end{array}\right),
\left(\begin{array}{cccc}
0\\
\frac{1}{2}\\
0\\
0
\end{array}\right)
\left(\begin{array}{cccc}
1& 0 & 0 & 0\\
0& 1 & 0 & 0\\
0& 0 & -1 & 0\\
0& 0 & 0 & 1
\end{array}\right),
\\
&\left(\begin{array}{cccc}
0\\
0\\
\frac{1}{2}\\
0
\end{array}\right)
\left(\begin{array}{cccc}
1& 0 & 0 & 0\\
0& 1 & 0 & 0\\
0& 0 & 1 & 0\\
0& 0 & 0 & -1
\end{array}\right),
\left(\begin{array}{cccc}
0\\
0\\
0\\
\frac{1}{2}
\end{array}\right)
\left(\begin{array}{cccc}
1& 0 & 0 & 0\\
0& 1 & 0 & 0\\
0& 0 & 1 & 0\\
0& 0 & 0 & 1
\end{array}\right)
\Bigr >
\end{align*}

\begin{align*}
\Gamma_{18}= 
\Bigl < & \left(\begin{array}{cccc}
\frac{1}{2}\\
0\\
0\\
0
\end{array}\right)
\left(\begin{array}{cccc}
1& 0 & 0 & 0\\
0& 1 & 0 & 0\\
0& 0 & -1 & 0\\
0& 0 & 0 & 1
\end{array}\right),
\left(\begin{array}{cccc}
0\\
\frac{1}{2}\\
0\\
0
\end{array}\right)
\left(\begin{array}{cccc}
1& 0 & 0 & 0\\
0& 1 & 0 & 0\\
0& 0 & 1 & 0\\
0& 0 & 0 & 1
\end{array}\right),
\\
&\left(\begin{array}{cccc}
0\\
0\\
\frac{1}{2}\\
0
\end{array}\right)
\left(\begin{array}{cccc}
1& 0 & 0 & 0\\
0& 1 & 0 & 0\\
0& 0 & 1 & 0\\
0& 0 & 0 & -1
\end{array}\right),
\left(\begin{array}{cccc}
0\\
0\\
0\\
\frac{1}{2}
\end{array}\right)
\left(\begin{array}{cccc}
1& 0 & 0 & 0\\
0& 1 & 0 & 0\\
0& 0 & 1 & 0\\
0& 0 & 0 & 1
\end{array}\right)
\Bigr >.
\end{align*}

\normalsize
\noindent
Let $\varphi(z_1,z_2,z_3,z_4)=(z_2,z_1,z_3,z_4)$, we get these commutative diagrams

\footnotesize
\[
\begin{CD}
(z_1,z_2,z_3,z_4) @>\varphi>>(z_2,z_1,z_3,z_4)\\
@Vg_1VV @Vh_2 VV\\
(-z_1,z_2,z_3,z_4) @>\varphi>> (z_2,-z_1,z_3,z_4)
\end{CD}
\hspace{.1cm}
\begin{CD}
(z_1,z_2,z_3,z_4) @>\varphi>>(z_2,z_1,z_3,z_4)\\
@Vg_3VV @Vh_3 VV\\
(z_1,z_2,-z_3,\bar z_4) @>\varphi>> (z_2,z_1,-z_3,\bar z_4)
\end{CD}
\]
\[ 
\begin{CD}
(z_1,z_2,z_3,z_4) @>\varphi>>(z_2,z_1,z_3,z_4)\\
@Vg_2VV @Vh_1 VV\\
(z_1,-z_2,\bar z_3,z_4) @>\varphi>> (-z_2,z_1,\bar z_3,z_4)
\end{CD}
\hspace{.1cm}
\begin{CD}
(z_1,z_2,z_3,z_4) @>\varphi>>(z_2,z_1,z_3,z_4)\\
@Vg_4VV @Vh_4 VV\\
(z_1,z_2,z_3,-z_4) @>\varphi>> (z_2,z_1,z_3,-z_4).
\end{CD}
\]

\normalsize
\noindent 
Therefore $\exists$ 
$\gamma 
\footnotesize
=
\Bigl( \left(\begin{array}{cccc}
0\\
0\\
0\\
0
\end{array}\right)
\left(\begin{array}{cccc}
0& 1 & 0 & 0\\
1& 0 & 0 & 0\\
0& 0 & 1 & 0\\
0& 0 & 0 & 1
\end{array}\right)
\Bigr)
\normalsize
\in \mathbb{A}(n)$ 
s.t. $\gamma \Gamma_{10}\gamma ^{-1}=\Gamma_{18}$.\\
%--------------------------------------------------------------------

\noindent 
$\bullet$ $M(A_{10})\approx  M(A_{22})$.\\
For
\footnotesize
\begin{align*}
&A_{10}=
\left(\begin{array}{cccc}
1& 0 & 0 & 0 \\
0& 1 & 1 & 0\\
0& 0 & 1 & 1\\
0& 0 & 0 & 1
\end{array}\right)
\hspace{.1cm}
&A_{22}=
\left(\begin{array}{cccc}
1& 0 & 1 & 1 \\
0& 1 & 0 & 0\\
0& 0 & 1 & 1\\
0& 0 & 0 & 1
\end{array}\right)
\\
&g_1(z_1,z_2,z_3,z_4)=(-z_1,z_2,z_3,z_4).\hspace{.1cm}   &h_1h_3(z_1,z_2,z_3,z_4)=(-z_1,z_2,-\bar z_3,z_4)\\
&g_2(z_1,z_2,z_3,z_4)=(z_1,-z_2,\bar z_3,z_4).           &h_2(z_1,z_2,z_3,z_4)=(z_1,-z_2,z_3,z_4)\\
&g_3(z_1,z_2,z_3,z_4)=(z_1,z_2,-z_3,\bar z_4).           &h_3(z_1,z_2,z_3,z_4)=(z_1,z_2,-z_3,\bar z_4)\\
&g_4(z_1,z_2,z_3,z_4)=(z_1,z_2,z_3,-z_4).                &h_4(z_1,z_2,z_3,z_4)=(z_1,z_2,z_3,-z_4)
\end{align*}
\normalsize then \footnotesize
\begin{align*}
\Gamma_{10}=
\Bigl < &\left(\begin{array}{cccc}
\frac{1}{2}\\
0\\
0\\
0
\end{array}\right)
\left(\begin{array}{cccc}
1& 0 & 0 & 0\\
0& 1 & 0 & 0\\
0& 0 & 1 & 0\\
0& 0 & 0 & 1
\end{array}\right),
\left(\begin{array}{cccc}
0\\
\frac{1}{2}\\
0\\
0
\end{array}\right)
\left(\begin{array}{cccc}
1& 0 & 0 & 0\\
0& 1 & 0 & 0\\
0& 0 & -1 & 0\\
0& 0 & 0 & 1
\end{array}\right),
\\
&\left(\begin{array}{cccc}
0\\
0\\
\frac{1}{2}\\
0
\end{array}\right)
\left(\begin{array}{cccc}
1& 0 & 0 & 0\\
0& 1 & 0 & 0\\
0& 0 & 1 & 0\\
0& 0 & 0 & -1
\end{array}\right),
\left(\begin{array}{cccc}
0\\
0\\
0\\
\frac{1}{2}
\end{array}\right)
\left(\begin{array}{cccc}
1& 0 & 0 & 0\\
0& 1 & 0 & 0\\
0& 0 & 1 & 0\\
0& 0 & 0 & 1
\end{array}\right)
\Bigr >
\end{align*}

\begin{align*}
\Gamma_{22}= 
\Bigl < & \left(\begin{array}{cccc}
\frac{1}{2}\\
0\\
\frac{1}{2}\\
0
\end{array}\right)
\left(\begin{array}{cccc}
1& 0 & 0 & 0\\
0& 1 & 0 & 0\\
0& 0 & -1 & 0\\
0& 0 & 0 & 1
\end{array}\right),
\left(\begin{array}{cccc}
0\\
\frac{1}{2}\\
0\\
0
\end{array}\right)
\left(\begin{array}{cccc}
1& 0 & 0 & 0\\
0& 1 & 0 & 0\\
0& 0 & 1 & 0\\
0& 0 & 0 & 1
\end{array}\right),
\\
&\left(\begin{array}{cccc}
0\\
0\\
\frac{1}{2}\\
0
\end{array}\right)
\left(\begin{array}{cccc}
1& 0 & 0 & 0\\
0& 1 & 0 & 0\\
0& 0 & 1 & 0\\
0& 0 & 0 & -1
\end{array}\right),
\left(\begin{array}{cccc}
0\\
0\\
0\\
\frac{1}{2}
\end{array}\right)
\left(\begin{array}{cccc}
1& 0 & 0 & 0\\
0& 1 & 0 & 0\\
0& 0 & 1 & 0\\
0& 0 & 0 & 1
\end{array}\right)
\Bigr >.
\end{align*}

\normalsize
\noindent
Let $\varphi(z_1,z_2,z_3,z_4)=(z_2,z_1,iz_3,z_4)$, we get these commutative diagrams

\footnotesize
\[
\begin{CD}
(z_1,z_2,z_3,z_4) @>\varphi>>(z_2,z_1,iz_3,z_4)\\
@Vg_1VV @Vh_2 VV\\
(-z_1,z_2,z_3,z_4) @>\varphi>> (z_2,-z_1,iz_3,z_4)
\end{CD}
\hspace{.1cm}
\begin{CD}
(z_1,z_2,z_3,z_4) @>\varphi>>(z_2,z_1,iz_3,z_4)\\
@Vg_3VV @Vh_3 VV\\
(z_1,z_2,-z_3,\bar z_4) @>\varphi>> (z_2,z_1,-iz_3,\bar z_4)
\end{CD}
\]
\[ 
\begin{CD}
(z_1,z_2,z_3,z_4) @>\varphi>>(z_2,z_1,iz_3,z_4)\\
@Vg_2VV @Vh_1h_3 VV\\
(z_1,-z_2,\bar z_3,z_4) @>\varphi>> (-z_2,z_1,-\bar {iz}_3,z_4)
\end{CD}
\hspace{.1cm}
\begin{CD}
(z_1,z_2,z_3,z_4) @>\varphi>>(z_2,z_1,iz_3,z_4)\\
@Vg_4VV @Vh_4 VV\\
(z_1,z_2,z_3,-z_4) @>\varphi>> (z_2,z_1,iz_3,-z_4).
\end{CD}
\]

\normalsize
\noindent 
Therefore $\exists$ 
$\gamma 
\footnotesize
=
\Bigl( \left(\begin{array}{cccc}
0\\
0\\
\frac{1}{4}\\
0
\end{array}\right)
\left(\begin{array}{cccc}
0& 1 & 0 & 0\\
1& 0 & 0 & 0\\
0& 0 & 1 & 0\\
0& 0 & 0 & 1
\end{array}\right)
\Bigr)
\normalsize
\in \mathbb{A}(n)$ 
s.t. $\gamma \Gamma_{10}\gamma ^{-1}=\Gamma_{22}$.\\
%--------------------------------------------------------------------

\noindent 
$\bullet$ $M(A_{10})\approx  M(A_{26})$.\\
For
\footnotesize
\begin{align*}
&A_{10}=
\left(\begin{array}{cccc}
1& 0 & 0 & 0 \\
0& 1 & 1 & 0\\
0& 0 & 1 & 1\\
0& 0 & 0 & 1
\end{array}\right)
\hspace{.1cm}
&A_{26}=
\left(\begin{array}{cccc}
1& 0 & 1 & 0 \\
0& 1 & 1 & 0\\
0& 0 & 1 & 1\\
0& 0 & 0 & 1
\end{array}\right)
\\
&g_1(z_1,z_2,z_3,z_4)=(-z_1,z_2,z_3,z_4).\hspace{.1cm}   &h_1h_2(z_1,z_2,z_3,z_4)=(-z_1,-z_2,z_3,z_4)\\
&g_2(z_1,z_2,z_3,z_4)=(z_1,-z_2,\bar z_3,z_4).           &h_2(z_1,z_2,z_3,z_4)=(z_1,-z_2,\bar z_3,z_4)\\
&g_3(z_1,z_2,z_3,z_4)=(z_1,z_2,-z_3,\bar z_4).           &h_3(z_1,z_2,z_3,z_4)=(z_1,z_2,-z_3,\bar z_4)\\
&g_4(z_1,z_2,z_3,z_4)=(z_1,z_2,z_3,-z_4).                &h_4(z_1,z_2,z_3,z_4)=(z_1,z_2,z_3,-z_4)
\end{align*}
\normalsize then \footnotesize
\begin{align*}
\Gamma_{10}=
\Bigl < &\left(\begin{array}{cccc}
\frac{1}{2}\\
0\\
0\\
0
\end{array}\right)
\left(\begin{array}{cccc}
1& 0 & 0 & 0\\
0& 1 & 0 & 0\\
0& 0 & 1 & 0\\
0& 0 & 0 & 1
\end{array}\right),
\left(\begin{array}{cccc}
0\\
\frac{1}{2}\\
0\\
0
\end{array}\right)
\left(\begin{array}{cccc}
1& 0 & 0 & 0\\
0& 1 & 0 & 0\\
0& 0 & -1 & 0\\
0& 0 & 0 & 1
\end{array}\right),
\\
&\left(\begin{array}{cccc}
0\\
0\\
\frac{1}{2}\\
0
\end{array}\right)
\left(\begin{array}{cccc}
1& 0 & 0 & 0\\
0& 1 & 0 & 0\\
0& 0 & 1 & 0\\
0& 0 & 0 & -1
\end{array}\right),
\left(\begin{array}{cccc}
0\\
0\\
0\\
\frac{1}{2}
\end{array}\right)
\left(\begin{array}{cccc}
1& 0 & 0 & 0\\
0& 1 & 0 & 0\\
0& 0 & 1 & 0\\
0& 0 & 0 & 1
\end{array}\right)
\Bigr >
\end{align*}

\begin{align*}
\Gamma_{26}= 
\Bigl < & \left(\begin{array}{cccc}
\frac{1}{2}\\
\frac{1}{2}\\
0\\
0
\end{array}\right)
\left(\begin{array}{cccc}
1& 0 & 0 & 0\\
0& 1 & 0 & 0\\
0& 0 & 1 & 0\\
0& 0 & 0 & 1
\end{array}\right),
\left(\begin{array}{cccc}
0\\
\frac{1}{2}\\
0\\
0
\end{array}\right)
\left(\begin{array}{cccc}
1& 0 & 0 & 0\\
0& 1 & 0 & 0\\
0& 0 & -1 & 0\\
0& 0 & 0 & 1
\end{array}\right),
\\
&\left(\begin{array}{cccc}
0\\
0\\
\frac{1}{2}\\
0
\end{array}\right)
\left(\begin{array}{cccc}
1& 0 & 0 & 0\\
0& 1 & 0 & 0\\
0& 0 & 1 & 0\\
0& 0 & 0 & -1
\end{array}\right),
\left(\begin{array}{cccc}
0\\
0\\
0\\
\frac{1}{2}
\end{array}\right)
\left(\begin{array}{cccc}
1& 0 & 0 & 0\\
0& 1 & 0 & 0\\
0& 0 & 1 & 0\\
0& 0 & 0 & 1
\end{array}\right)
\Bigr >.
\end{align*}

\normalsize
\noindent
Let $\varphi(z_1,z_2,z_3,z_4)=(z_1,z_1z_2,z_3,z_4)$, we get these commtative diagrams

\footnotesize
\[
\begin{CD}
(z_1,z_2,z_3,z_4) @>\varphi>>(z_1,z_1z_2,z_3,z_4)\\
@Vg_1VV @Vh_1h_2 VV\\
(-z_1,z_2,z_3,z_4) @>\varphi>> (-z_1,-z_1z_2,z_3,z_4)
\end{CD}
\hspace{.1cm}
\begin{CD}
(z_1,z_2,z_3,z_4) @>\varphi>>(z_1,z_1z_2,z_3,z_4)\\
@Vg_3VV @Vh_3 VV\\
(z_1,z_2,-z_3,\bar z_4) @>\varphi>> (z_1,z_1z_2,-z_3,\bar z_4)
\end{CD}
\]
\[ 
\begin{CD}
(z_1,z_2,z_3,z_4) @>\varphi>>(z_1,z_1z_2,z_3,z_4)\\
@Vg_2VV @Vh_2 VV\\
(z_1,-z_2,\bar z_3,z_4) @>\varphi>> (z_1,-z_1z_2,\bar z_3,z_4)
\end{CD}
\hspace{.1cm}
\begin{CD}
(z_1,z_2,z_3,z_4) @>\varphi>>(z_1,z_1z_2,z_3,z_4)\\
@Vg_4VV @Vh_4 VV\\
(z_1,z_2,z_3,-z_4) @>\varphi>> (z_1,z_1z_2,z_3,-z_4).
\end{CD}
\]

\normalsize
\noindent 
Therefore $\exists$ 
$\gamma 
\footnotesize
=
\Bigl( \left(\begin{array}{cccc}
0\\
0\\
0\\
0
\end{array}\right)
\left(\begin{array}{cccc}
1& 0 & 0 & 0\\
1& 1 & 0 & 0\\
0& 0 & 1 & 0\\
0& 0 & 0 & 1
\end{array}\right)
\Bigr)
\normalsize
\in \mathbb{A}(n)$ 
s.t. $\gamma \Gamma_{10}\gamma ^{-1}=\Gamma_{26}$.\\
%--------------------------------------------------------------------

\noindent 
$\bullet$ $M(A_{12})\approx  M(A_{32})$.\\
For
\footnotesize
\begin{align*}
&A_{12}=
\left(\begin{array}{cccc}
1& 0 & 0 & 0 \\
0& 1 & 1 & 1\\
0& 0 & 1 & 1\\
0& 0 & 0 & 1
\end{array}\right)
\hspace{.1cm}
&A_{32}=
\left(\begin{array}{cccc}
1& 0 & 1 & 1 \\
0& 1 & 1 & 1\\
0& 0 & 1 & 1\\
0& 0 & 0 & 1
\end{array}\right)
\\
&g_1(z_1,z_2,z_3,z_4)=(-z_1,z_2,z_3,z_4).\hspace{.1cm}   &h_1h_2(z_1,z_2,z_3,z_4)=(-z_1,-z_2,z_3,z_4)\\
&g_2(z_1,z_2,z_3,z_4)=(z_1,-z_2,\bar z_3,\bar z_4).      &h_2(z_1,z_2,z_3,z_4)=(z_1,-z_2,\bar z_3,\bar z_4)\\
&g_3(z_1,z_2,z_3,z_4)=(z_1,z_2,-z_3,\bar z_4).           &h_3(z_1,z_2,z_3,z_4)=(z_1,z_2,-z_3,\bar z_4)\\
&g_4(z_1,z_2,z_3,z_4)=(z_1,z_2,z_3,-z_4).                &h_4(z_1,z_2,z_3,z_4)=(z_1,z_2,z_3,-z_4)
\end{align*}
\normalsize then \footnotesize
\begin{align*}
\Gamma_{12}= 
\Bigl < & \left(\begin{array}{cccc}
\frac{1}{2}\\
0\\
0\\
0
\end{array}\right)
\left(\begin{array}{cccc}
1& 0 & 0 & 0\\
0& 1 & 0 & 0\\
0& 0 & 1 & 0\\
0& 0 & 0 & 1
\end{array}\right),
\left(\begin{array}{cccc}
0\\
\frac{1}{2}\\
0\\
0
\end{array}\right)
\left(\begin{array}{cccc}
1& 0 & 0 & 0\\
0& 1 & 0 & 0\\
0& 0 & -1 & 0\\
0& 0 & 0 & -1
\end{array}\right),
\\
&\left(\begin{array}{cccc}
0\\
0\\
\frac{1}{2}\\
0
\end{array}\right)
\left(\begin{array}{cccc}
1& 0 & 0 & 0\\
0& 1 & 0 & 0\\
0& 0 & 1 & 0\\
0& 0 & 0 & -1
\end{array}\right),
\left(\begin{array}{cccc}
0\\
0\\
0\\
\frac{1}{2}
\end{array}\right)
\left(\begin{array}{cccc}
1& 0 & 0 & 0\\
0& 1 & 0 & 0\\
0& 0 & 1 & 0\\
0& 0 & 0 & 1
\end{array}\right)
\Bigr >
\end{align*}

\begin{align*}
\Gamma_{32}=
\Bigl < &\left(\begin{array}{cccc}
\frac{1}{2}\\
\frac{1}{2}\\
0\\
0
\end{array}\right)
\left(\begin{array}{cccc}
1& 0 & 0 & 0\\
0& 1 & 0 & 0\\
0& 0 & 1 & 0\\
0& 0 & 0 & 1
\end{array}\right),
\left(\begin{array}{cccc}
0\\
\frac{1}{2}\\
0\\
0
\end{array}\right)
\left(\begin{array}{cccc}
1& 0 & 0 & 0\\
0& 1 & 0 & 0\\
0& 0 & -1 & 0\\
0& 0 & 0 & -1
\end{array}\right),
\\
&\left(\begin{array}{cccc}
0\\
0\\
\frac{1}{2}\\
0
\end{array}\right)
\left(\begin{array}{cccc}
1& 0 & 0 & 0\\
0& 1 & 0 & 0\\
0& 0 & 1 & 0\\
0& 0 & 0 & -1
\end{array}\right),
\left(\begin{array}{cccc}
0\\
0\\
0\\
\frac{1}{2}
\end{array}\right)
\left(\begin{array}{cccc}
1& 0 & 0 & 0\\
0& 1 & 0 & 0\\
0& 0 & 1 & 0\\
0& 0 & 0 & 1
\end{array}\right)
\Bigr >.
\end{align*}

\normalsize
\noindent
Let $\varphi(z_1,z_2,z_3,z_4)=(z_1,z_1z_2,z_3,z_4)$, we get these commutative diagrams

\footnotesize
\[
\begin{CD}
(z_1,z_2,z_3,z_4) @>\varphi>>(z_1,z_1z_2,z_3,z_4)\\
@Vg_1VV @Vh_1h_2 VV\\
(-z_1,z_2,z_3,z_4) @>\varphi>> (-z_1,-z_1z_2,z_3,z_4)
\end{CD}
\hspace{.1cm}
\begin{CD}
(z_1,z_2,z_3,z_4) @>\varphi>>(z_1,z_1z_2,z_3,z_4)\\
@Vg_3VV @Vh_3 VV\\
(z_1,z_2,-z_3,\bar z_4) @>\varphi>> (z_1,z_1z_2,-z_3,\bar z_4)
\end{CD}
\]
\[ 
\begin{CD}
(z_1,z_2,z_3,z_4) @>\varphi>>(z_1,z_1z_2,z_3,z_4)\\
@Vg_2VV @Vh_2 VV\\
(z_1,-z_2,\bar z_3,\bar z_4) @>\varphi>> (z_1,-z_1z_2,\bar z_3,\bar z_4)
\end{CD}
\hspace{.1cm}
\begin{CD}
(z_1,z_2,z_3,z_4) @>\varphi>>(z_1,z_1z_2,z_3,z_4)\\
@Vg_4VV @Vh_4 VV\\
(z_1,z_2,z_3,-z_4) @>\varphi>> (z_1,z_1z_2,z_3,-z_4).
\end{CD}
\]

\normalsize
\noindent 
Therefore $\exists$ 
$\gamma 
\footnotesize
=
\Bigl( \left(\begin{array}{cccc}
0\\
0\\
0\\
0
\end{array}\right)
\left(\begin{array}{cccc}
1& 0 & 0 & 0\\
1& 1 & 0 & 0\\
0& 0 & 1 & 0\\
0& 0 & 0 & 1
\end{array}\right)
\Bigr)
\normalsize
\in \mathbb{A}(n)$ 
s.t. $\gamma \Gamma_{12}\gamma ^{-1}=\Gamma_{32}$.\\
%--------------------------------------------------------------------
%===============================================================

\noindent{\bf (10).} Similar to (4).\\

\noindent{\bf (11).} We know that  \\
$\Phi_{a15}=\sz \left<I,\left(\begin{array}{cccc}
1& 0 & 0 & 0\\
0& -1 & 0 & 0\\
0& 0 & 1 & 0\\
0& 0 & 0 & -1
\end{array}\right)=P,\left(\begin{array}{cccc}
1& 0 & 0 & 0\\
0& 1 & 0 & 0\\
0& 0 & -1 & 0\\
0& 0 & 0 & -1
\end{array}\right)=Q \right>$ \normalsize 
and \sz
$\Phi_{a3}=\left<I,\left(\begin{array}{cccc}
1& 0 & 0 & 0\\
0& -1 & 0 & 0\\
0& 0 & 1 & 0\\
0& 0 & 0 & 1
\end{array}\right)=R, \left(\begin{array}{cccc}
1& 0 & 0 & 0\\
0& 1 & 0 & 0\\
0& 0 & 1 & 0\\
0& 0 & 0 & -1
\end{array}\right)=S\right>$. \normalsize If $M(A_{a15})$ is diffeomorphic to $M(A_{a3})$ then 
$\exists $ $B\in GL(4,\mathbb{R})$ such that 
\[
BPB^{-1}=
\begin{cases}
R  \hspace{1cm} (i)\\
S \hspace{1cm} (ii)\\
RS  \hspace{.71cm} (iii)
\end{cases}
\] 
and
\[
BQB^{-1}=
\begin{cases}
R  \hspace{1cm} (iv)\\
S \hspace{1cm} (v)\\
RS  \hspace{.71cm} (vi).
\end{cases}
\] 
By calculating the eigenvalues of cases (i), (ii),(iv) and (v), then we get a contradiction,
and from the cases (iii) and (vi) we get $P=Q$ which is also a contradiction. 
\\
%===============================================================

\noindent {\bf (12)}.\\
\noindent $\bullet$ $M(A_{a4})\approx  M(A_{a8})$.\\
For
\footnotesize
\begin{align*}
&A_{a4}=
\left(\begin{array}{cccc}
1& 1 & 0 & 0 \\
0& 1 & 0 & 1\\
0& 0 & 1 & 1\\
0& 0 & 0 & 1
\end{array}\right)
\hspace{.1cm}
&A_{a8}=
\left(\begin{array}{cccc}
1& 1 & 0 & 1 \\
0& 1 & 0 & 1\\
0& 0 & 1 & 1\\
0& 0 & 0 & 1
\end{array}\right)
\\
&g_1(z_1,z_2,z_3,z_4)=(-z_1,\bar z_2,z_3,z_4).\hspace{.1cm}   &h_1h_2(z_1,z_2,z_3,z_4)=(-z_1,-\bar z_2,z_3,z_4)\\
&g_2(z_1,z_2,z_3,z_4)=(z_1,-z_2,z_3,\bar z_4).                &h_2(z_1,z_2,z_3,z_4)=(z_1,-z_2,z_3,\bar z_4)\\
&g_3(z_1,z_2,z_3,z_4)=(z_1,z_2,-z_3,\bar z_4).                &h_3(z_1,z_2,z_3,z_4)=(z_1, z_2,-z_3,\bar z_4)\\
&g_4(z_1,z_2,z_3,z_4)=(z_1,z_2,z_3,-z_4).                     &h_4(z_1,z_2,z_3,z_4)=(z_1, z_2,z_3,-z_4)
\end{align*}
\normalsize then \footnotesize
\begin{align*}
\Gamma_{a4}=
\Bigl < &\left(\begin{array}{cccc}
\frac{1}{2}\\
0\\
0\\
0
\end{array}\right)
\left(\begin{array}{cccc}
1& 0 & 0 & 0\\
0& -1 & 0 & 0\\
0& 0 & 1 & 0\\
0& 0 & 0 & 1
\end{array}\right),
\left(\begin{array}{cccc}
0\\
\frac{1}{2}\\
0\\
0
\end{array}\right)
\left(\begin{array}{cccc}
1& 0 & 0 & 0\\
0& 1 & 0 & 0\\
0& 0 & 1 & 0\\
0& 0 & 0 & -1
\end{array}\right),
\\
&\left(\begin{array}{cccc}
0\\
0\\
\frac{1}{2}\\
0
\end{array}\right)
\left(\begin{array}{cccc}
1& 0 & 0 & 0\\
0& 1 & 0 & 0\\
0& 0 & 1 & 0\\
0& 0 & 0 & -1
\end{array}\right),
\left(\begin{array}{cccc}
0\\
0\\
0\\
\frac{1}{2}
\end{array}\right)
\left(\begin{array}{cccc}
1& 0 & 0 & 0\\
0& 1 & 0 & 0\\
0& 0 & 1 & 0\\
0& 0 & 0 & 1
\end{array}\right)
\Bigr >
\end{align*}

\begin{align*}
\Gamma_{a8}= 
\Bigl < & \left(\begin{array}{cccc}
\frac{1}{2}\\
\frac{1}{2}\\
0\\
0
\end{array}\right)
\left(\begin{array}{cccc}
1& 0 & 0 & 0\\
0& -1 & 0 & 0\\
0& 0 & 1 & 0\\
0& 0 & 0 & 1
\end{array}\right),
\left(\begin{array}{cccc}
0\\
\frac{1}{2}\\
0\\
0
\end{array}\right)
\left(\begin{array}{cccc}
1& 0 & 0 & 0\\
0& 1 & 0 & 0\\
0& 0 & 1 & 0\\
0& 0 & 0 & -1
\end{array}\right),
\\
&\left(\begin{array}{cccc}
0\\
0\\
\frac{1}{2}\\
0
\end{array}\right)
\left(\begin{array}{cccc}
1& 0 & 0 & 0\\
0& 1 & 0 & 0\\
0& 0 & 1 & 0\\
0& 0 & 0 & -1
\end{array}\right),
\left(\begin{array}{cccc}
0\\
0\\
0\\
\frac{1}{2}
\end{array}\right)
\left(\begin{array}{cccc}
1& 0 & 0 & 0\\
0& 1 & 0 & 0\\
0& 0 & 1 & 0\\
0& 0 & 0 & 1
\end{array}\right)
\Bigr >.
\end{align*}

\normalsize
\noindent
Let $\varphi(z_1,z_2,z_3,z_4)=(z_1,iz_2,z_3,z_4)$, we get these commutative diagrams

\footnotesize
\[
\begin{CD}
(z_1,z_2,z_3,z_4) @>\varphi>>(z_1,iz_2,z_3,z_4)\\
@Vg_1VV @Vh_1h_2 VV\\
(-z_1,\bar z_2,z_3,z_4) @>\varphi>> (-z_1,-\bar {iz}_2,z_3,z_4)
\end{CD}
\hspace{.1cm}
\begin{CD}
(z_1,z_2,z_3,z_4) @>\varphi>>(z_1,iz_2,z_3,z_4)\\
@Vg_3VV @Vh_3 VV\\
(z_1,z_2,-z_3,\bar z_4) @>\varphi>> (z_1,iz_2,-z_3,\bar z_4)
\end{CD}
\]
\[ 
\begin{CD}
(z_1,z_2,z_3,z_4) @>\varphi>>(z_1,iz_2,z_3,z_4)\\
@Vg_2VV @Vh_2 VV\\
(z_1,-z_2,z_3,\bar z_4) @>\varphi>> (z_1,-iz_2,z_3,\bar z_4)
\end{CD}
\hspace{.1cm}
\begin{CD}
(z_1,z_2,z_3,z_4) @>\varphi>>(z_1,iz_2,z_3,z_4)\\
@Vg_4VV @Vh_4 VV\\
(z_1,z_2,z_3,-z_4) @>\varphi>> (z_1,iz_2,z_3,-z_4).
\end{CD}
\]

\normalsize
\noindent 
Therefore $\exists$ 
$\gamma 
\footnotesize
=
\Bigl( \left(\begin{array}{cccc}
0\\
\frac{1}{4}\\
0\\
0
\end{array}\right)
\left(\begin{array}{cccc}
1& 0 & 0 & 0\\
0& 1 & 0 & 0\\
0& 0 & 1 & 0\\
0& 0 & 0 & 1
\end{array}\right)
\Bigr)
\normalsize
\in \mathbb{A}(n)$ 
s.t. $\gamma \Gamma_{a4}\gamma ^{-1}=\Gamma_{a8}.$\\
%--------------------------------------------------------------------

\noindent 
$\bullet$ $M(A_{a4})\approx M(A_{14})$.\\
For 
\footnotesize
\begin{align*}
&A_{a4}=
\left(\begin{array}{cccc}
1& 1 & 0 & 0 \\
0& 1 & 0 & 1\\
0& 0 & 1 & 1\\
0& 0 & 0 & 1
\end{array}\right)
\hspace{.1cm}
&A_{14}=
\left(\begin{array}{cccc}
1& 0 & 0 & 1 \\
0& 1 & 1 & 0\\
0& 0 & 1 & 1\\
0& 0 & 0 & 1
\end{array}\right)
\\
&g_1(z_1,z_2,z_3,z_4)=(-z_1,\bar z_2,z_3,z_4).\hspace{.1cm}   &h_1(z_1,z_2,z_3,z_4)=(-z_1,z_2,z_3,\bar z_4)\\
&g_2(z_1,z_2,z_3,z_4)=(z_1,-z_2,z_3,\bar z_4).                &h_2(z_1,z_2,z_3,z_4)=(z_1,-z_2,\bar z_3,z_4)\\
&g_3(z_1,z_2,z_3,z_4)=(z_1,z_2,-z_3,\bar z_4).                &h_3(z_1,z_2,z_3,z_4)=(z_1, z_2,-z_3,\bar z_4)\\
&g_4(z_1,z_2,z_3,z_4)=(z_1,z_2,z_3,-z_4).                     &h_4(z_1,z_2,z_3,z_4)=(z_1, z_2,z_3,-z_4)
\end{align*}
\normalsize then \footnotesize
\begin{align*}
\Gamma_{a4}=
\Bigl < &\left(\begin{array}{cccc}
\frac{1}{2}\\
0\\
0\\
0
\end{array}\right)
\left(\begin{array}{cccc}
1& 0 & 0 & 0\\
0& -1 & 0 & 0\\
0& 0 & 1 & 0\\
0& 0 & 0 & 1
\end{array}\right),
\left(\begin{array}{cccc}
0\\
\frac{1}{2}\\
0\\
0
\end{array}\right)
\left(\begin{array}{cccc}
1& 0 & 0 & 0\\
0& 1 & 0 & 0\\
0& 0 & 1 & 0\\
0& 0 & 0 & -1
\end{array}\right),
\\
&\left(\begin{array}{cccc}
0\\
0\\
\frac{1}{2}\\
0
\end{array}\right)
\left(\begin{array}{cccc}
1& 0 & 0 & 0\\
0& 1 & 0 & 0\\
0& 0 & 1 & 0\\
0& 0 & 0 & -1
\end{array}\right),
\left(\begin{array}{cccc}
0\\
0\\
0\\
\frac{1}{2}
\end{array}\right)
\left(\begin{array}{cccc}
1& 0 & 0 & 0\\
0& 1 & 0 & 0\\
0& 0 & 1 & 0\\
0& 0 & 0 & 1
\end{array}\right)
\Bigr >
\end{align*}

\begin{align*}
\Gamma_{14}= 
\Bigl < & \left(\begin{array}{cccc}
\frac{1}{2}\\
0\\
0\\
0
\end{array}\right)
\left(\begin{array}{cccc}
1& 0 & 0 & 0\\
0& 1 & 0 & 0\\
0& 0 & 1 & 0\\
0& 0 & 0 & -1
\end{array}\right),
\left(\begin{array}{cccc}
0\\
\frac{1}{2}\\
0\\
0
\end{array}\right)
\left(\begin{array}{cccc}
1& 0 & 0 & 0\\
0& 1 & 0 & 0\\
0& 0 & -1 & 0\\
0& 0 & 0 & 1
\end{array}\right),
\\
&\left(\begin{array}{cccc}
0\\
0\\
\frac{1}{2}\\
0
\end{array}\right)
\left(\begin{array}{cccc}
1& 0 & 0 & 0\\
0& 1 & 0 & 0\\
0& 0 & 1 & 0\\
0& 0 & 0 & -1
\end{array}\right),
\left(\begin{array}{cccc}
0\\
0\\
0\\
\frac{1}{2}
\end{array}\right)
\left(\begin{array}{cccc}
1& 0 & 0 & 0\\
0& 1 & 0 & 0\\
0& 0 & 1 & 0\\
0& 0 & 0 & 1
\end{array}\right)
\Bigr >.
\end{align*}

\normalsize
\noindent
Let $\varphi(z_1,z_2,z_3,z_4)=(z_3,z_1,z_2,z_4)$, we get these commutative diagrams

\footnotesize
\[
\begin{CD}
(z_1,z_2,z_3,z_4) @>\varphi>>(z_3,z_1,z_2,z_4)\\
@Vg_1VV @Vh_2 VV\\
(-z_1,\bar z_2,z_3,z_4) @>\varphi>> (z_3,-z_1,\bar z_2,z_4)
\end{CD}
\hspace{.1cm}
\begin{CD}
(z_1,z_2,z_3,z_4) @>\varphi>>(z_3,z_1,z_2,z_4)\\
@Vg_3VV @Vh_1 VV\\
(z_1,z_2,-z_3,\bar z_4) @>\varphi>> (-z_3,z_1,z_2,\bar z_4)
\end{CD}
\]
\[ 
\begin{CD}
(z_1,z_2,z_3,z_4) @>\varphi>>(z_3,z_1,z_2,z_4)\\
@Vg_2VV @Vh_3 VV\\
(z_1,-z_2,z_3,\bar z_4) @>\varphi>> (z_3,z_1,-z_2,\bar z_4)
\end{CD}
\hspace{.1cm}
\begin{CD}
(z_1,z_2,z_3,z_4) @>\varphi>>(z_3,z_1,z_2,z_4)\\
@Vg_4VV @Vh_4 VV\\
(z_1,z_2,z_3,-z_4) @>\varphi>> (z_3,z_1,z_2,-z_4).
\end{CD}
\]

\normalsize
\noindent 
Therefore $\exists$ 
$\gamma 
\footnotesize
=
\Bigl( \left(\begin{array}{cccc}
0\\
0\\
0\\
0
\end{array}\right)
\left(\begin{array}{cccc}
0& 0 & 1 & 0\\
1& 0 & 0 & 0\\
0& 1 & 0 & 0\\
0& 0 & 0 & 1
\end{array}\right)
\Bigr)
\normalsize
\in \mathbb{A}(n)$ 
s.t. $\gamma \Gamma_{a4}\gamma ^{-1}=\Gamma_{14}$.\\
%----------------------------------------------------------------------

\noindent 
$\bullet$ $M(A_{14})\approx  M(A_{16})$.\\
For
\footnotesize
\begin{align*}
&A_{14}=
\left(\begin{array}{cccc}
1& 0 & 0 & 1 \\
0& 1 & 1 & 0\\
0& 0 & 1 & 1\\
0& 0 & 0 & 1
\end{array}\right)
\hspace{.1cm}
&A_{16}=
\left(\begin{array}{cccc}
1& 0 & 0 & 1 \\
0& 1 & 1 & 1\\
0& 0 & 1 & 1\\
0& 0 & 0 & 1
\end{array}\right)
\\
&g_1(z_1,z_2,z_3,z_4)=(-z_1,z_2,z_3,\bar z_4).\hspace{.1cm}   &h_1(z_1,z_2,z_3,z_4)=(-z_1,z_2,z_3,\bar z_4)\\
&g_2(z_1,z_2,z_3,z_4)=(z_1,-z_2,\bar z_3,z_4).            &h_2h_3(z_1,z_2,z_3,z_4)=(z_1,-z_2,-\bar z_3,z_4)\\
&g_3(z_1,z_2,z_3,z_4)=(z_1,z_2,-z_3,\bar z_4).           &h_3(z_1,z_2,z_3,z_4)=(z_1,z_2,-z_3,\bar z_4)\\
&g_4(z_1,z_2,z_3,z_4)=(z_1,z_2,z_3,-z_4).                &h_4(z_1,z_2,z_3,z_4)=(z_1,z_2,z_3,-z_4)
\end{align*}
\normalsize then \footnotesize
\begin{align*}
\Gamma_{14}= 
\Bigl < & \left(\begin{array}{cccc}
\frac{1}{2}\\
0\\
0\\
0
\end{array}\right)
\left(\begin{array}{cccc}
1& 0 & 0 & 0\\
0& 1 & 0 & 0\\
0& 0 & 1 & 0\\
0& 0 & 0 & -1
\end{array}\right),
\left(\begin{array}{cccc}
0\\
\frac{1}{2}\\
0\\
0
\end{array}\right)
\left(\begin{array}{cccc}
1& 0 & 0 & 0\\
0& 1 & 0 & 0\\
0& 0 & -1 & 0\\
0& 0 & 0 & 1
\end{array}\right),
\\
&\left(\begin{array}{cccc}
0\\
0\\
\frac{1}{2}\\
0
\end{array}\right)
\left(\begin{array}{cccc}
1& 0 & 0 & 0\\
0& 1 & 0 & 0\\
0& 0 & 1 & 0\\
0& 0 & 0 & -1
\end{array}\right),
\left(\begin{array}{cccc}
0\\
0\\
0\\
\frac{1}{2}
\end{array}\right)
\left(\begin{array}{cccc}
1& 0 & 0 & 0\\
0& 1 & 0 & 0\\
0& 0 & 1 & 0\\
0& 0 & 0 & 1
\end{array}\right)
\Bigr >
\end{align*}

\begin{align*}
\Gamma_{16}=
\Bigl < &\left(\begin{array}{cccc}
\frac{1}{2}\\
0\\
0\\
0
\end{array}\right)
\left(\begin{array}{cccc}
1& 0 & 0 & 0\\
0& 1 & 0 & 0\\
0& 0 & 1 & 0\\
0& 0 & 0 & -1
\end{array}\right),
\left(\begin{array}{cccc}
0\\
\frac{1}{2}\\
\frac{1}{2}\\
0
\end{array}\right)
\left(\begin{array}{cccc}
1& 0 & 0 & 0\\
0& 1 & 0 & 0\\
0& 0 & -1 & 0\\
0& 0 & 0 & 1
\end{array}\right),
\\
&\left(\begin{array}{cccc}
0\\
0\\
\frac{1}{2}\\
0
\end{array}\right)
\left(\begin{array}{cccc}
1& 0 & 0 & 0\\
0& 1 & 0 & 0\\
0& 0 & 1 & 0\\
0& 0 & 0 & -1
\end{array}\right),
\left(\begin{array}{cccc}
0\\
0\\
0\\
\frac{1}{2}
\end{array}\right)
\left(\begin{array}{cccc}
1& 0 & 0 & 0\\
0& 1 & 0 & 0\\
0& 0 & 1 & 0\\
0& 0 & 0 & 1
\end{array}\right)
\Bigr >.
\end{align*}

\normalsize
\noindent
Let $\varphi(z_1,z_2,z_3,z_4)=(z_1,z_2,iz_3,z_4)$, we get these commutative diagrams

\footnotesize
\[
\begin{CD}
(z_1,z_2,z_3,z_4) @>\varphi>>(z_1,z_2,iz_3,z_4)\\
@Vg_1VV @Vh_1 VV\\
(-z_1,z_2,z_3,\bar z_4) @>\varphi>> (-z_1,z_2,iz_3,\bar z_4)
\end{CD}
\hspace{.1cm}
\begin{CD}
(z_1,z_2,z_3,z_4) @>\varphi>>(z_1,z_2,iz_3,z_4)\\
@Vg_3VV @Vh_3 VV\\
(z_1,z_2,-z_3,\bar z_4) @>\varphi>> (z_1,z_2,-iz_3,\bar z_4)
\end{CD}
\]
\[ 
\begin{CD}
(z_1,z_2,z_3,z_4) @>\varphi>>(z_1,z_2,iz_3,z_4)\\
@Vg_2VV @Vh_2h_3 VV\\
(z_1,-z_2,\bar z_3,z_4) @>\varphi>> (z_1,-z_2,-\bar {iz}_3,z_4)
\end{CD}
\hspace{.1cm}
\begin{CD}
(z_1,z_2,z_3,z_4) @>\varphi>>(z_1,z_2,iz_3,z_4)\\
@Vg_4VV @Vh_4 VV\\
(z_1,z_2,z_3,-z_4) @>\varphi>> (z_1,z_2,iz_3,-z_4).
\end{CD}
\]

\normalsize
\noindent 
Therefore $\exists$ 
$\gamma 
\footnotesize
=
\Bigl( \left(\begin{array}{cccc}
0\\
0\\
\frac{1}{4}\\
0
\end{array}\right)
\left(\begin{array}{cccc}
1& 0 & 0 & 0\\
0& 1 & 0 & 0\\
0& 0 & 1 & 0\\
0& 0 & 0 & 1
\end{array}\right)
\Bigr)
\normalsize
\in \mathbb{A}(n)$ 
s.t. $\gamma \Gamma_{14}\gamma ^{-1}=\Gamma_{16}$.\\
%--------------------------------------------------------------------

\noindent 
$\bullet$ $M(A_{14})\approx  M(A_{24})$.\\
For 
\footnotesize
\begin{align*}
&A_{14}=
\left(\begin{array}{cccc}
1& 0 & 0 & 1 \\
0& 1 & 1 & 0\\
0& 0 & 1 & 1\\
0& 0 & 0 & 1
\end{array}\right)
\hspace{.1cm}
&A_{24}=
\left(\begin{array}{cccc}
1& 0 & 1 & 1 \\
0& 1 & 0 & 1\\
0& 0 & 1 & 1\\
0& 0 & 0 & 1
\end{array}\right)
\\
&g_1(z_1,z_2,z_3,z_4)=(-z_1,z_2,z_3,\bar z_4).\hspace{.1cm}   &h_1h_2(z_1,z_2,z_3,z_4)=(-z_1,-z_2,\bar z_3,z_4)\\
&g_2(z_1,z_2,z_3,z_4)=(z_1,-z_2,\bar z_3,z_4).            &h_2(z_1,z_2,z_3,z_4)=(z_1,-z_2,z_3,\bar z_4)\\
&g_3(z_1,z_2,z_3,z_4)=(z_1,z_2,-z_3,\bar z_4).           &h_3(z_1,z_2,z_3,z_4)=(z_1,z_2,-z_3,\bar z_4)\\
&g_4(z_1,z_2,z_3,z_4)=(z_1,z_2,z_3,-z_4).                &h_4(z_1,z_2,z_3,z_4)=(z_1,z_2,z_3,-z_4)
\end{align*}
\normalsize then \footnotesize
\begin{align*}
\Gamma_{14}= 
\Bigl < & \left(\begin{array}{cccc}
\frac{1}{2}\\
0\\
0\\
0
\end{array}\right)
\left(\begin{array}{cccc}
1& 0 & 0 & 0\\
0& 1 & 0 & 0\\
0& 0 & 1 & 0\\
0& 0 & 0 & -1
\end{array}\right),
\left(\begin{array}{cccc}
0\\
\frac{1}{2}\\
0\\
0
\end{array}\right)
\left(\begin{array}{cccc}
1& 0 & 0 & 0\\
0& 1 & 0 & 0\\
0& 0 & -1 & 0\\
0& 0 & 0 & 1
\end{array}\right),
\\
&\left(\begin{array}{cccc}
0\\
0\\
\frac{1}{2}\\
0
\end{array}\right)
\left(\begin{array}{cccc}
1& 0 & 0 & 0\\
0& 1 & 0 & 0\\
0& 0 & 1 & 0\\
0& 0 & 0 & -1
\end{array}\right),
\left(\begin{array}{cccc}
0\\
0\\
0\\
\frac{1}{2}
\end{array}\right)
\left(\begin{array}{cccc}
1& 0 & 0 & 0\\
0& 1 & 0 & 0\\
0& 0 & 1 & 0\\
0& 0 & 0 & 1
\end{array}\right)
\Bigr >
\end{align*}

\begin{align*}
\Gamma_{24}=
\Bigl < &\left(\begin{array}{cccc}
\frac{1}{2}\\
\frac{1}{2}\\
0\\
0
\end{array}\right)
\left(\begin{array}{cccc}
1& 0 & 0 & 0\\
0& 1 & 0 & 0\\
0& 0 & -1 & 0\\
0& 0 & 0 & 1
\end{array}\right),
\left(\begin{array}{cccc}
0\\
\frac{1}{2}\\
0\\
0
\end{array}\right)
\left(\begin{array}{cccc}
1& 0 & 0 & 0\\
0& 1 & 0 & 0\\
0& 0 & 1 & 0\\
0& 0 & 0 & -1
\end{array}\right),
\\
&\left(\begin{array}{cccc}
0\\
0\\
\frac{1}{2}\\
0
\end{array}\right)
\left(\begin{array}{cccc}
1& 0 & 0 & 0\\
0& 1 & 0 & 0\\
0& 0 & 1 & 0\\
0& 0 & 0 & -1
\end{array}\right),
\left(\begin{array}{cccc}
0\\
0\\
0\\
\frac{1}{2}
\end{array}\right)
\left(\begin{array}{cccc}
1& 0 & 0 & 0\\
0& 1 & 0 & 0\\
0& 0 & 1 & 0\\
0& 0 & 0 & 1
\end{array}\right)
\Bigr >.
\end{align*}

\normalsize
\noindent
Let $\varphi(z_1,z_2,z_3,z_4)=(z_2,z_1z_2,z_3,z_4)$, we get these commutative diagrams

\footnotesize
\[
\begin{CD}
(z_1,z_2,z_3,z_4) @>\varphi>>(z_2,z_1z_2,z_3,z_4)\\
@Vg_1VV @Vh_2 VV\\
(-z_1,z_2,z_3,\bar z_4) @>\varphi>> (z_2,-z_1z_2,z_3,\bar z_4)
\end{CD}
\hspace{.1cm}
\begin{CD}
(z_1,z_2,z_3,z_4) @>\varphi>>(z_2,z_1z_2,z_3,z_4)\\
@Vg_3VV @Vh_3 VV\\
(z_1,z_2,-z_3,\bar z_4) @>\varphi>> (z_2,z_1z_2,-z_3,\bar z_4)
\end{CD}
\]
\[ 
\begin{CD}
(z_1,z_2,z_3,z_4) @>\varphi>>(z_2,z_1z_2,z_3,z_4)\\
@Vg_2VV @Vh_1h_2 VV\\
(z_1,-z_2,\bar z_3,z_4) @>\varphi>> (-z_2,-z_1z_2,\bar z_3,z_4)
\end{CD}
\hspace{.1cm}
\begin{CD}
(z_1,z_2,z_3,z_4) @>\varphi>>(z_2,z_1z_2,z_3,z_4)\\
@Vg_4VV @Vh_4 VV\\
(z_1,z_2,z_3,-z_4) @>\varphi>> (z_2,z_1z_2,z_3,-z_4).
\end{CD}
\]

\normalsize
\noindent 
Therefore $\exists$ 
$\gamma 
\footnotesize
=
\Bigl( \left(\begin{array}{cccc}
0\\
0\\
0\\
0
\end{array}\right)
\left(\begin{array}{cccc}
0& 1 & 0 & 0\\
1& 1 & 0 & 0\\
0& 0 & 1 & 0\\
0& 0 & 0 & 1
\end{array}\right)
\Bigr)
\normalsize
\in \mathbb{A}(n)$ 
s.t. $\gamma \Gamma_{14}\gamma ^{-1}=\Gamma_{24}$.\\
%--------------------------------------------------------------------

\noindent 
$\bullet$ $M(A_{20})\approx  M(A_{24})$.\\
For
\footnotesize
\begin{align*}
&A_{20}=
\left(\begin{array}{cccc}
1& 0 & 1 & 0 \\
0& 1 & 0 & 1\\
0& 0 & 1 & 1\\
0& 0 & 0 & 1
\end{array}\right)
\hspace{.1cm}
&A_{24}=
\left(\begin{array}{cccc}
1& 0 & 1 & 1 \\
0& 1 & 0 & 1\\
0& 0 & 1 & 1\\
0& 0 & 0 & 1
\end{array}\right)
\\
&g_1(z_1,z_2,z_3,z_4)=(-z_1,z_2,\bar z_3,z_4).\hspace{.1cm}   &h_1h_3(z_1,z_2,z_3,z_4)=(-z_1,z_2,-\bar z_3,z_4)\\
&g_2(z_1,z_2,z_3,z_4)=(z_1,-z_2,z_3,\bar z_4).            &h_2(z_1,z_2,z_3,z_4)=(z_1,-z_2,z_3,\bar z_4)\\
&g_3(z_1,z_2,z_3,z_4)=(z_1,z_2,-z_3,\bar z_4).           &h_3(z_1,z_2,z_3,z_4)=(z_1,z_2,-z_3,\bar z_4)\\
&g_4(z_1,z_2,z_3,z_4)=(z_1,z_2,z_3,-z_4).                &h_4(z_1,z_2,z_3,z_4)=(z_1,z_2,z_3,-z_4)
\end{align*}
\normalsize then \footnotesize
\begin{align*}
\Gamma_{20}= 
\Bigl < & \left(\begin{array}{cccc}
\frac{1}{2}\\
0\\
0\\
0
\end{array}\right)
\left(\begin{array}{cccc}
1& 0 & 0 & 0\\
0& 1 & 0 & 0\\
0& 0 & -1 & 0\\
0& 0 & 0 & 1
\end{array}\right),
\left(\begin{array}{cccc}
0\\
\frac{1}{2}\\
0\\
0
\end{array}\right)
\left(\begin{array}{cccc}
1& 0 & 0 & 0\\
0& 1 & 0 & 0\\
0& 0 & 1 & 0\\
0& 0 & 0 & -1
\end{array}\right),
\\
&\left(\begin{array}{cccc}
0\\
0\\
\frac{1}{2}\\
0
\end{array}\right)
\left(\begin{array}{cccc}
1& 0 & 0 & 0\\
0& 1 & 0 & 0\\
0& 0 & 1 & 0\\
0& 0 & 0 & -1
\end{array}\right),
\left(\begin{array}{cccc}
0\\
0\\
0\\
\frac{1}{2}
\end{array}\right)
\left(\begin{array}{cccc}
1& 0 & 0 & 0\\
0& 1 & 0 & 0\\
0& 0 & 1 & 0\\
0& 0 & 0 & 1
\end{array}\right)
\Bigr >
\end{align*}

\begin{align*}
\Gamma_{24}=
\Bigl < &\left(\begin{array}{cccc}
\frac{1}{2}\\
0\\
\frac{1}{2}\\
0
\end{array}\right)
\left(\begin{array}{cccc}
1& 0 & 0 & 0\\
0& 1 & 0 & 0\\
0& 0 & -1 & 0\\
0& 0 & 0 & 1
\end{array}\right),
\left(\begin{array}{cccc}
0\\
\frac{1}{2}\\
0\\
0
\end{array}\right)
\left(\begin{array}{cccc}
1& 0 & 0 & 0\\
0& 1 & 0 & 0\\
0& 0 & 1 & 0\\
0& 0 & 0 & -1
\end{array}\right),
\\
&\left(\begin{array}{cccc}
0\\
0\\
\frac{1}{2}\\
0
\end{array}\right)
\left(\begin{array}{cccc}
1& 0 & 0 & 0\\
0& 1 & 0 & 0\\
0& 0 & 1 & 0\\
0& 0 & 0 & -1
\end{array}\right),
\left(\begin{array}{cccc}
0\\
0\\
0\\
\frac{1}{2}
\end{array}\right)
\left(\begin{array}{cccc}
1& 0 & 0 & 0\\
0& 1 & 0 & 0\\
0& 0 & 1 & 0\\
0& 0 & 0 & 1
\end{array}\right)
\Bigr >.
\end{align*}

\normalsize
\noindent
Let $\varphi(z_1,z_2,z_3,z_4)=(z_1,z_2,iz_3,z_4)$, we get these commutative diagrams

\footnotesize
\[
\begin{CD}
(z_1,z_2,z_3,z_4) @>\varphi>>(z_1,z_2,iz_3,z_4)\\
@Vg_1VV @Vh_1h_3 VV\\
(-z_1,z_2,\bar z_3,z_4) @>\varphi>> (-z_1,z_2,-\bar {iz}_3,z_4)
\end{CD}
\hspace{.1cm}
\begin{CD}
(z_1,z_2,z_3,z_4) @>\varphi>>(z_1,z_2,iz_3,z_4)\\
@Vg_3VV @Vh_3 VV\\
(z_1,z_2,-z_3,\bar z_4) @>\varphi>> (z_1,z_2,-iz_3,\bar z_4)
\end{CD}
\]
\[ 
\begin{CD}
(z_1,z_2,z_3,z_4) @>\varphi>>(z_1,z_2,iz_3,z_4)\\
@Vg_2VV @Vh_2 VV\\
(z_1,-z_2,z_3,\bar z_4) @>\varphi>> (z_1,-z_2,iz_3,\bar z_4)
\end{CD}
\hspace{.1cm}
\begin{CD}
(z_1,z_2,z_3,z_4) @>\varphi>>(z_1,z_2,iz_3,z_4)\\
@Vg_4VV @Vh_4 VV\\
(z_1,z_2,z_3,-z_4) @>\varphi>> (z_1,z_2,iz_3,-z_4).
\end{CD}
\]

\normalsize
\noindent 
Therefore $\exists$ 
$\gamma 
\footnotesize
=
\Bigl( \left(\begin{array}{cccc}
0\\
0\\
\frac{1}{4}\\
0
\end{array}\right)
\left(\begin{array}{cccc}
1& 0 & 0 & 0\\
0& 1 & 0 & 0\\
0& 0 & 1 & 0\\
0& 0 & 0 & 1
\end{array}\right)
\Bigr)
\normalsize
\in \mathbb{A}(n)$ 
s.t. $\gamma \Gamma_{20}\gamma ^{-1}=\Gamma_{24}$.\\
%--------------------------------------------------------------------

\noindent 
$\bullet$ $M(A_{14})\approx  M(A_{30})$.\\
For
\footnotesize
\begin{align*}
&A_{14}=
\left(\begin{array}{cccc}
1& 0 & 0 & 1 \\
0& 1 & 1 & 0\\
0& 0 & 1 & 1\\
0& 0 & 0 & 1
\end{array}\right)
\hspace{.1cm}
&A_{30}=
\left(\begin{array}{cccc}
1& 0 & 1 & 1 \\
0& 1 & 1 & 0\\
0& 0 & 1 & 1\\
0& 0 & 0 & 1
\end{array}\right)
\\
&g_1(z_1,z_2,z_3,z_4)=(-z_1,z_2,z_3,\bar z_4).\hspace{.1cm}   &h_1h_2(z_1,z_2,z_3,z_4)=(-z_1,-z_2,z_3,\bar z_4)\\
&g_2(z_1,z_2,z_3,z_4)=(z_1,-z_2,\bar z_3,z_4).            &h_2(z_1,z_2,z_3,z_4)=(z_1,-z_2,\bar z_3,z_4)\\
&g_3(z_1,z_2,z_3,z_4)=(z_1,z_2,-z_3,\bar z_4).           &h_3(z_1,z_2,z_3,z_4)=(z_1,z_2,-z_3,\bar z_4)\\
&g_4(z_1,z_2,z_3,z_4)=(z_1,z_2,z_3,-z_4).                &h_4(z_1,z_2,z_3,z_4)=(z_1,z_2,z_3,-z_4)
\end{align*}
\normalsize then \footnotesize
\begin{align*}
\Gamma_{14}= 
\Bigl < & \left(\begin{array}{cccc}
\frac{1}{2}\\
0\\
0\\
0
\end{array}\right)
\left(\begin{array}{cccc}
1& 0 & 0 & 0\\
0& 1 & 0 & 0\\
0& 0 & 1 & 0\\
0& 0 & 0 & -1
\end{array}\right),
\left(\begin{array}{cccc}
0\\
\frac{1}{2}\\
0\\
0
\end{array}\right)
\left(\begin{array}{cccc}
1& 0 & 0 & 0\\
0& 1 & 0 & 0\\
0& 0 & -1 & 0\\
0& 0 & 0 & 1
\end{array}\right),
\\
&\left(\begin{array}{cccc}
0\\
0\\
\frac{1}{2}\\
0
\end{array}\right)
\left(\begin{array}{cccc}
1& 0 & 0 & 0\\
0& 1 & 0 & 0\\
0& 0 & 1 & 0\\
0& 0 & 0 & -1
\end{array}\right),
\left(\begin{array}{cccc}
0\\
0\\
0\\
\frac{1}{2}
\end{array}\right)
\left(\begin{array}{cccc}
1& 0 & 0 & 0\\
0& 1 & 0 & 0\\
0& 0 & 1 & 0\\
0& 0 & 0 & 1
\end{array}\right)
\Bigr >
\end{align*}

\begin{align*}
\Gamma_{30}=
\Bigl < &\left(\begin{array}{cccc}
\frac{1}{2}\\
\frac{1}{2}\\
0\\
0
\end{array}\right)
\left(\begin{array}{cccc}
1& 0 & 0 & 0\\
0& 1 & 0 & 0\\
0& 0 & 1 & 0\\
0& 0 & 0 & -1
\end{array}\right),
\left(\begin{array}{cccc}
0\\
\frac{1}{2}\\
0\\
0
\end{array}\right)
\left(\begin{array}{cccc}
1& 0 & 0 & 0\\
0& 1 & 0 & 0\\
0& 0 & -1 & 0\\
0& 0 & 0 & 1
\end{array}\right),
\\
&\left(\begin{array}{cccc}
0\\
0\\
\frac{1}{2}\\
0
\end{array}\right)
\left(\begin{array}{cccc}
1& 0 & 0 & 0\\
0& 1 & 0 & 0\\
0& 0 & 1 & 0\\
0& 0 & 0 & -1
\end{array}\right),
\left(\begin{array}{cccc}
0\\
0\\
0\\
\frac{1}{2}
\end{array}\right)
\left(\begin{array}{cccc}
1& 0 & 0 & 0\\
0& 1 & 0 & 0\\
0& 0 & 1 & 0\\
0& 0 & 0 & 1
\end{array}\right)
\Bigr >.
\end{align*}

\normalsize
\noindent
Let $\varphi(z_1,z_2,z_3,z_4)=(z_1,z_1z_2,z_3,z_4)$, we get these commutative diagrams

\footnotesize
\[
\begin{CD}
(z_1,z_2,z_3,z_4) @>\varphi>>(z_1,z_1z_2,z_3,z_4)\\
@Vg_1VV @Vh_1h_2 VV\\
(-z_1,z_2,z_3,\bar z_4) @>\varphi>> (-z_1,-z_1z_2,z_3,\bar z_4)
\end{CD}
\hspace{.1cm}
\begin{CD}
(z_1,z_2,z_3,z_4) @>\varphi>>(z_1,z_1z_2,z_3,z_4)\\
@Vg_3VV @Vh_3 VV\\
(z_1,z_2,-z_3,\bar z_4) @>\varphi>> (z_1,z_1z_2,-z_3,\bar z_4)
\end{CD}
\]
\[ 
\begin{CD}
(z_1,z_2,z_3,z_4) @>\varphi>>(z_1,z_1z_2,z_3,z_4)\\
@Vg_2VV @Vh_2 VV\\
(z_1,-z_2,\bar z_3,z_4) @>\varphi>> (z_1,-z_1z_2,\bar z_3,z_4)
\end{CD}
\hspace{.1cm}
\begin{CD}
(z_1,z_2,z_3,z_4) @>\varphi>>(z_1,z_1z_2,z_3,z_4)\\
@Vg_4VV @Vh_4 VV\\
(z_1,z_2,z_3,-z_4) @>\varphi>> (z_1,z_1z_2,z_3,-z_4).
\end{CD}
\]

\normalsize
\noindent 
Therefore $\exists$ 
$\gamma 
\footnotesize
=
\Bigl( \left(\begin{array}{cccc}
0\\
0\\
0\\
0
\end{array}\right)
\left(\begin{array}{cccc}
1& 0 & 0 & 0\\
1& 1 & 0 & 0\\
0& 0 & 1 & 0\\
0& 0 & 0 & 1
\end{array}\right)
\Bigr)
\normalsize
\in \mathbb{A}(n)$ 
s.t. $\gamma \Gamma_{14}\gamma ^{-1}=\Gamma_{30}$.\\
%---------------------------------------------------------------------------

\noindent 
$\bullet$ $M(A_{30})\approx  M(A_{28})$.\\
For
\footnotesize
\begin{align*}
&A_{30}=
\left(\begin{array}{cccc}
1& 0 & 1 & 1 \\
0& 1 & 1 & 0\\
0& 0 & 1 & 1\\
0& 0 & 0 & 1
\end{array}\right)
\hspace{.1cm}
&A_{28}=
\left(\begin{array}{cccc}
1& 0 & 1 & 0 \\
0& 1 & 1 & 1\\
0& 0 & 1 & 1\\
0& 0 & 0 & 1
\end{array}\right)
\\
&g_1(z_1,z_2,z_3,z_4)=(-z_1,z_2,\bar z_3,\bar z_4).\hspace{.1cm}   &h_1(z_1,z_2,z_3,z_4)=(-z_1,z_2,\bar z_3,z_4)\\
&g_2(z_1,z_2,z_3,z_4)=(z_1,-z_2,\bar z_3,z_4).            &h_2(z_1,z_2,z_3,z_4)=(z_1,-z_2,\bar z_3,\bar z_4)\\
&g_3(z_1,z_2,z_3,z_4)=(z_1,z_2,-z_3,\bar z_4).           &h_3(z_1,z_2,z_3,z_4)=(z_1,z_2,-z_3,\bar z_4)\\
&g_4(z_1,z_2,z_3,z_4)=(z_1,z_2,z_3,-z_4).                &h_4(z_1,z_2,z_3,z_4)=(z_1,z_2,z_3,-z_4)
\end{align*}
\normalsize then \footnotesize
\begin{align*}
\Gamma_{30}=
\Bigl < &\left(\begin{array}{cccc}
\frac{1}{2}\\
0\\
0\\
0
\end{array}\right)
\left(\begin{array}{cccc}
1& 0 & 0 & 0\\
0& 1 & 0 & 0\\
0& 0 & -1 & 0\\
0& 0 & 0 & -1
\end{array}\right),
\left(\begin{array}{cccc}
0\\
\frac{1}{2}\\
0\\
0
\end{array}\right)
\left(\begin{array}{cccc}
1& 0 & 0 & 0\\
0& 1 & 0 & 0\\
0& 0 & -1 & 0\\
0& 0 & 0 & 1
\end{array}\right),
\\
&\left(\begin{array}{cccc}
0\\
0\\
\frac{1}{2}\\
0
\end{array}\right)
\left(\begin{array}{cccc}
1& 0 & 0 & 0\\
0& 1 & 0 & 0\\
0& 0 & 1 & 0\\
0& 0 & 0 & -1
\end{array}\right),
\left(\begin{array}{cccc}
0\\
0\\
0\\
\frac{1}{2}
\end{array}\right)
\left(\begin{array}{cccc}
1& 0 & 0 & 0\\
0& 1 & 0 & 0\\
0& 0 & 1 & 0\\
0& 0 & 0 & 1
\end{array}\right)
\Bigr >
\end{align*}

\begin{align*}
\Gamma_{28}= 
\Bigl < & \left(\begin{array}{cccc}
\frac{1}{2}\\
0\\
0\\
0
\end{array}\right)
\left(\begin{array}{cccc}
1& 0 & 0 & 0\\
0& 1 & 0 & 0\\
0& 0 & -1 & 0\\
0& 0 & 0 & 1
\end{array}\right),
\left(\begin{array}{cccc}
0\\
\frac{1}{2}\\
0\\
0
\end{array}\right)
\left(\begin{array}{cccc}
1& 0 & 0 & 0\\
0& 1 & 0 & 0\\
0& 0 & -1 & 0\\
0& 0 & 0 & -1
\end{array}\right),
\\
&\left(\begin{array}{cccc}
0\\
0\\
\frac{1}{2}\\
0
\end{array}\right)
\left(\begin{array}{cccc}
1& 0 & 0 & 0\\
0& 1 & 0 & 0\\
0& 0 & 1 & 0\\
0& 0 & 0 & -1
\end{array}\right),
\left(\begin{array}{cccc}
0\\
0\\
0\\
\frac{1}{2}
\end{array}\right)
\left(\begin{array}{cccc}
1& 0 & 0 & 0\\
0& 1 & 0 & 0\\
0& 0 & 1 & 0\\
0& 0 & 0 & 1
\end{array}\right)
\Bigr >.
\end{align*}

\normalsize
\noindent
Let $\varphi(z_1,z_2,z_3,z_4)=(z_2,z_1,z_3,z_4)$, we get these commutative diagrams

\footnotesize
\[
\begin{CD}
(z_1,z_2,z_3,z_4) @>\varphi>>(z_2,z_1,z_3,z_4)\\
@Vg_1VV @Vh_2 VV\\
(-z_1,z_2,\bar z_3,\bar z_4) @>\varphi>> (z_2,-z_1,\bar z_3,\bar z_4)
\end{CD}
\hspace{.1cm}
\begin{CD}
(z_1,z_2,z_3,z_4) @>\varphi>>(z_2,z_1,z_3,z_4)\\
@Vg_3VV @Vh_3 VV\\
(z_1,z_2,-z_3,\bar z_4) @>\varphi>> (z_2,z_1,-z_3,\bar z_4)
\end{CD}
\]
\[ 
\begin{CD}
(z_1,z_2,z_3,z_4) @>\varphi>>(z_2,z_1,z_3,z_4)\\
@Vg_2VV @Vh_1 VV\\
(z_1,-z_2,\bar z_3,z_4) @>\varphi>> (-z_2,z_1,\bar z_3,z_4)
\end{CD}
\hspace{.1cm}
\begin{CD}
(z_1,z_2,z_3,z_4) @>\varphi>>(z_2,z_1,z_3,z_4)\\
@Vg_4VV @Vh_4 VV\\
(z_1,z_2,z_3,-z_4) @>\varphi>> (z_2,z_1,z_3,-z_4).
\end{CD}
\]

\normalsize
\noindent 
Therefore $\exists$ 
$\gamma 
\footnotesize
=
\Bigl( \left(\begin{array}{cccc}
0\\
0\\
0\\
0
\end{array}\right)
\left(\begin{array}{cccc}
0& 1 & 0 & 0\\
1& 0 & 0 & 0\\
0& 0 & 1 & 0\\
0& 0 & 0 & 1
\end{array}\right)
\Bigr)
\normalsize
\in \mathbb{A}(n)$ 
s.t. $\gamma \Gamma_{30}\gamma ^{-1}=\Gamma_{28}$.\\
%---------------------------------------------------------------------------
%=================================================================================

\noindent {\bf (13)}. Similar to (4)\\

\noindent {\bf (14)}. Similar to (11)\\

%===============================================================================
\noindent {\bf (15)}. 
\noindent
$M(A_{14})$ is not diffeomorphic to $M(A_{a3})$ .\\

Let \footnotesize
\begin{align*}
\Gamma_{a3}= <&\tilde {g}_1,\tilde {g}_2,t_3,t_4>\\
=\Bigl < &\left(\begin{array}{cccc}
\frac{1}{2}\\
0\\
0\\
0
\end{array}\right)
\left(\begin{array}{cccc}
1& 0 & 0 & 0\\
0& -1 & 0 & 0\\
0& 0 & 1 & 0\\
0& 0 & 0 & 1
\end{array}\right),
\left(\begin{array}{cccc}
0\\
\frac{1}{2}\\
0\\
0
\end{array}\right)
\left(\begin{array}{cccc}
1& 0 & 0 & 0\\
0& 1 & 0 & 0\\
0& 0 & 1 & 0\\
0& 0 & 0 & -1
\end{array}\right),
\\
&\left(\begin{array}{cccc}
0\\
0\\
\frac{1}{2}\\
0
\end{array}\right)
\left(\begin{array}{cccc}
1& 0 & 0 & 0\\
0& 1 & 0 & 0\\
0& 0 & 1 & 0\\
0& 0 & 0 & 1
\end{array}\right),
\left(\begin{array}{cccc}
0\\
0\\
0\\
\frac{1}{2}
\end{array}\right)
\left(\begin{array}{cccc}
1& 0 & 0 & 0\\
0& 1 & 0 & 0\\
0& 0 & 1 & 0\\
0& 0 & 0 & 1
\end{array}\right)
\Bigr >
\\
\Gamma_{14}= <&\tilde {h}_1,\tilde {h}_2,\tilde {h}_3,t_4>\\
=\Bigl < & \left(\begin{array}{cccc}
\frac{1}{2}\\
0\\
0\\
0
\end{array}\right)
\left(\begin{array}{cccc}
1& 0 & 0 & 0\\
0& 1 & 0 & 0\\
0& 0 & 1 & 0\\
0& 0 & 0 & -1
\end{array}\right),
\left(\begin{array}{cccc}
0\\
\frac{1}{2}\\
0\\
0
\end{array}\right)
\left(\begin{array}{cccc}
1& 0 & 0 & 0\\
0& 1 & 0 & 0\\
0& 0 & -1 & 0\\
0& 0 & 0 & 1
\end{array}\right),
\\
&\left(\begin{array}{cccc}
0\\
0\\
\frac{1}{2}\\
0
\end{array}\right)
\left(\begin{array}{cccc}
1& 0 & 0 & 0\\
0& 1 & 0 & 0\\
0& 0 & 1 & 0\\
0& 0 & 0 & -1
\end{array}\right),
\left(\begin{array}{cccc}
0\\
0\\
0\\
\frac{1}{2}
\end{array}\right)
\left(\begin{array}{cccc}
1& 0 & 0 & 0\\
0& 1 & 0 & 0\\
0& 0 & 1 & 0\\
0& 0 & 0 & 1
\end{array}\right)
\Bigr >. 
\end{align*}
\normalsize

Since $\tilde {g}_1\tilde {g}_2\tilde {g}_1^{-1}=\tilde {g}_2^{-1}$,
% $\tilde {g}_1t_3\tilde {g}_1^{-1}=t_3$,
 $\tilde {g}_1t_4\tilde {g}_1^{-1}=t_4$,
% $t_3\tilde {g}_1 t_3^{-1}=\tilde {g}_1$,
$t_3\tilde {g}_2t_3^{-1}=\tilde {g}_2$, and $t_3t_4t_3^{-1}=t_4$,
then $\Gamma _{a3}=\langle \tilde {g}_2,t_4 \rangle \rtimes \langle \tilde {g}_1,t_3 \rangle$.
Then the center $\mathcal C(\Gamma_{a3})=\langle t_1, t_3\rangle$
where $t_1=\tilde g_1^2=\footnotesize \left(\begin{array}{c}
1\\
0\\
0\\
0
\end{array}\right), t_3=\tilde g_3=\left(\begin{array}{c}
0\\
0\\
\frac{1}{2}\\
0
\end{array}\right)$. \normalsize
There is an central group extension:
\begin{align*}
1\rightarrow \langle t_1, t_3 \rangle \rightarrow  \Gamma_{a3}\rightarrow  \Delta_{a3}\rightarrow  1,
\end{align*}
where 
\begin{align*}
\Delta_{a3}&= \Bigl<\gamma_1=\left(\begin{array}{c}
\frac{1}{2}\\
0
\end{array}\right)
\left(\begin{array}{cc}
1& 0 \\
0& -1
\end{array}\right), \gamma _2=\left(\begin{array}{c}
0\\
\frac{1}{2}
\end{array}\right)
\left(\begin{array}{cc}
1& 0 \\
0& 1
\end{array}\right)\Bigr >\rtimes
\langle \beta \rangle 
\end{align*}
with $\beta =\Bigl(\left(\begin{array}{c}
0\\
0
\end{array}\right)
\left(\begin{array}{cc}
-1& 0\\
0& 1
\end{array}\right)\Bigr)$.
\\

For $\Gamma_{14}= \langle \tilde {h}_1,\tilde {h}_2, \tilde{h}_3,t_4\rangle$, 
since 
%$\tilde {h}_1\tilde {h}_2\tilde {h}_1^{-1}=\tilde {h}_2$,
 $\tilde {h}_1\tilde {h}_3\tilde {h}_1^{-1}=\tilde {h}_3$,
$\tilde {h}_1t_4\tilde {h}_1^{-1}=t_4^{-1}$,
%$\tilde {h}_2\tilde {h}_1\tilde {h}_2^{-1}=\tilde {h}_1$,
 $\tilde {h}_2\tilde {h}_3\tilde {h}_2^{-1}=\tilde {h}_3^{-1}$, and 
$\tilde {h}_2t_4\tilde {h}_2^{-1}=t_4$
then $\Gamma _{14}=\langle \tilde {h}_3,t_4 \rangle \rtimes \langle \tilde {h}_1,\tilde {h}_2 \rangle$.
The center $\mathcal
C(\Gamma_{14})=\langle t_1, t_2\rangle$ where $t_1=\tilde h_1^2$,
$t_2=\tilde h_2^2= \footnotesize
\left(\begin{array}{c}
0\\
1\\
0\\
0\end{array}\right)$. \normalsize
It induces an extension
\begin{align*}
1\rightarrow  \langle t_1,t_2\rangle \rightarrow  \Gamma_{14}\rightarrow  \Delta_{14}\rightarrow  1,
\end{align*}
where 
\begin{align*}
\Delta_{14}=\Bigl <s_1=\left(\begin{array}{c}
\frac 12\\
0
\end{array}\right)
\left(\begin{array}{cc}
1& 0  \\
0& -1
\end{array}\right),s_2=\left(\begin{array}{c}
0\\
\frac{1}{2}
\end{array}\right)
\left(\begin{array}{ccc}
1& 0  \\
0& 1
\end{array}\right)\Bigr >\rtimes
\langle \alpha ,\beta  \rangle
\end{align*}
with
$\alpha =\left(\begin{array}{c}
0\\
0
\end{array}\right)
\left(\begin{array}{cc}
1& 0\\
0& -1\end{array}\right), \beta =\left(\begin{array}{c}
0\\
0
\end{array}\right) \left(\begin{array}{cc}
-1& 0\\
0& 1
\end{array}\right)$.

Suppose that $\Gamma_{a3}$ is isomorphic to $\Gamma_{14}$
 with isomorphism $\varphi:\Gamma_{a3}\rightarrow  \Gamma_{14}$.
Then it induces an isomorphism $\hat\varphi:\Delta_{a3}\rightarrow 
\Delta_{14}$
\begin{align*}
\begin{CD}
1 @>>>  \langle t_1,t_3 \rangle @>>> \Gamma_{a3}@>>>  \Delta_{a3}@>>>  1\\
  @.                @.                @V \varphi VV           @V \hat \varphi VV\\
1 @>>>  \langle t_1,t_2 \rangle @>>> \Gamma_{14}@>>>  \Delta_{14}@>>>  1.
\end{CD}
\end{align*}
\\
Consider $\hat \varphi (\gamma_1^{a_1} \gamma_2^{a_2} \beta )=\alpha $ and $\hat \varphi (\gamma_1^{b_1} \gamma_2^{b_2} \beta )=\beta  $ 
where $a_1, a_2, b_1, b_2 \in \mathbb{Z}$. 
Then 
\begin{align*}
\alpha\beta=&\hat \varphi (\gamma_1^{a_1} \gamma_2^{a_2} \beta  \gamma_1^{b_1} \gamma_2^{b_2} \beta)  \\
      =&\hat \varphi (\gamma_1^{a_1} \gamma_2^{a_2} \gamma_1^{-b_1} \beta  \gamma_2^{b_2} \beta)  \hspace{.5cm} (\beta \gamma_1 \beta^{-1}=\gamma_1^{-1})\\
      =&\hat \varphi (\gamma_1^{a_1} \gamma_2^{a_2} \gamma_1^{-b_1} \gamma_2^{b_2}) \hspace{.9cm} (\beta \gamma_2 \beta^{-1}=\gamma_2).
\end{align*}
Since $\alpha \beta $ is a torsion element and $\hat \varphi $ is an isomorphism, then 
$\gamma_1^{a_1} \gamma_2^{a_2} \gamma_1^{-b_1} \gamma_2^{b_2}$ is also a torsion element.
Let
\begin{align*}
g=&\gamma_1^{a_1} \gamma_2^{a_2} \gamma_1^{-b_1} \gamma_2^{b_2} \\
 =&
\Bigl ( \left(\begin{array}{c}
\frac{1}{2}a_1\\
(-1)^{a_1}\frac{1}{2}a_2
\end{array}\right)
\left(\begin{array}{cc}
1& 0\\
0& (-1)^{a_1}\end{array}\right)\Bigr)
\Bigl ( \left(\begin{array}{c}
-\frac{1}{2}b_1\\
(-1)^{-b_1}\frac{1}{2}b_2
\end{array}\right)
\left(\begin{array}{cc}
1& 0\\
0& (-1)^{-b_1}\end{array}\right)\Bigr)\\
=&
\Bigl ( \left(\begin{array}{c}
\frac{1}{2}(a_1 - b_1)\\
\frac{1}{2}((-1)^{a_1}a_2+(-1)^{(a_1-b_1)}b_2)
\end{array}\right)
\left(\begin{array}{cc}
1& 0\\
0& (-1)^{a_1-b_1}\end{array}\right)\Bigr).
\end{align*}
We have
\begin{align*}
g^2=&\Bigl ( \left(\begin{array}{c}
a_1 - b_1\\
(1+(-1)^{(a_1-b_1)})\frac{1}{2}((-1)^{a_1}a_2+(-1)^{(a_1-b_1)}b_2
\end{array}\right)
\left(\begin{array}{cc}
1& 0\\
0& 1 \end{array}\right)\Bigr)\\
=&
\Bigl ( \left(\begin{array}{c}
0\\
0
\end{array}\right)
\left(\begin{array}{cc}
1& 0\\
0& 1 \end{array}\right)\Bigr).
\end{align*}
\\
Therefore we get $a_1=b_1$ and $(-1)^{a_1}a_2+b_2=0$.
Hence $\hat \varphi (g)=\hat \varphi 
\Bigl ( \left(\begin{array}{c}
0\\
0
\end{array}\right)
\left(\begin{array}{cc}
1& 0\\
0& 1 \end{array}\right)\Bigr)=
\Bigl ( \left(\begin{array}{c}
0\\
0
\end{array}\right)
\left(\begin{array}{cc}
-1& 0\\
0& -1 \end{array}\right)\Bigr)$.
This yields a contradiction.\\
%===============================================================================

\noindent {\bf (16)}.\\
%--------------------------------------------------------------------
\noindent 
$\bullet$ $M(A_{a13})\approx M(A_{a29})$.\\
For
\footnotesize
\begin{align*}
&A_{a13}=
\left(\begin{array}{cccc}
1& 1 & 0 & 1 \\
0& 1 & 1 & 0\\
0& 0 & 1 & 0\\
0& 0 & 0 & 1
\end{array}\right)
\hspace{.1cm}
&A_{a29}=
\left(\begin{array}{cccc}
1& 1 & 1 & 1 \\
0& 1 & 1 & 0\\
0& 0 & 1 & 0\\
0& 0 & 0 & 1
\end{array}\right)
\\
&g_1(z_1,z_2,z_3,z_4)=(-z_1,\bar z_2,z_3,\bar z_4).\hspace{.1cm}   &h_1h_2(z_1,z_2,z_3,z_4)=(-z_1,-\bar z_2,z_3,\bar z_4)\\
&g_2(z_1,z_2,z_3,z_4)=(z_1,-z_2,\bar z_3,z_4).                     &h_2(z_1,z_2,z_3,z_4)=(z_1,-z_2,\bar z_3,z_4)\\
&g_3(z_1,z_2,z_3,z_4)=(z_1,z_2,-z_3,z_4).                          &h_3(z_1,z_2,z_3,z_4)=(z_1, z_2,-z_3,z_4)\\
&g_4(z_1,z_2,z_3,z_4)=(z_1,z_2,z_3,-z_4).                          &h_4(z_1,z_2,z_3,z_4)=(z_1, z_2,z_3,-z_4)
\end{align*}
\normalsize then \footnotesize
\begin{align*}
\Gamma_{a13}=
\Bigl < &\left(\begin{array}{cccc}
\frac{1}{2}\\
0\\
0\\
0
\end{array}\right)
\left(\begin{array}{cccc}
1& 0 & 0 & 0\\
0& -1 & 0 & 0\\
0& 0 & 1 & 0\\
0& 0 & 0 & -1
\end{array}\right),
\left(\begin{array}{cccc}
0\\
\frac{1}{2}\\
0\\
0
\end{array}\right)
\left(\begin{array}{cccc}
1& 0 & 0 & 0\\
0& 1 & 0 & 0\\
0& 0 & -1 & 0\\
0& 0 & 0 & 1
\end{array}\right),
\\
&\left(\begin{array}{cccc}
0\\
0\\
\frac{1}{2}\\
0
\end{array}\right)
\left(\begin{array}{cccc}
1& 0 & 0 & 0\\
0& 1 & 0 & 0\\
0& 0 & 1 & 0\\
0& 0 & 0 & 1
\end{array}\right),
\left(\begin{array}{cccc}
0\\
0\\
0\\
\frac{1}{2}
\end{array}\right)
\left(\begin{array}{cccc}
1& 0 & 0 & 0\\
0& 1 & 0 & 0\\
0& 0 & 1 & 0\\
0& 0 & 0 & 1
\end{array}\right)
\Bigr >
\end{align*}

\begin{align*}
\Gamma_{a29}= 
\Bigl < & \left(\begin{array}{cccc}
\frac{1}{2}\\
\frac{1}{2}\\
0\\
0
\end{array}\right)
\left(\begin{array}{cccc}
1& 0 & 0 & 0\\
0& -1 & 0 & 0\\
0& 0 & 1 & 0\\
0& 0 & 0 & -1
\end{array}\right),
\left(\begin{array}{cccc}
0\\
\frac{1}{2}\\
0\\
0
\end{array}\right)
\left(\begin{array}{cccc}
1& 0 & 0 & 0\\
0& 1 & 0 & 0\\
0& 0 & -1 & 0\\
0& 0 & 0 & 1
\end{array}\right),
\\
&\left(\begin{array}{cccc}
0\\
0\\
\frac{1}{2}\\
0
\end{array}\right)
\left(\begin{array}{cccc}
1& 0 & 0 & 0\\
0& 1 & 0 & 0\\
0& 0 & 1 & 0\\
0& 0 & 0 & 1
\end{array}\right),
\left(\begin{array}{cccc}
0\\
0\\
0\\
\frac{1}{2}
\end{array}\right)
\left(\begin{array}{cccc}
1& 0 & 0 & 0\\
0& 1 & 0 & 0\\
0& 0 & 1 & 0\\
0& 0 & 0 & 1
\end{array}\right)
\Bigr >.
\end{align*}

\normalsize
\noindent
Let $\varphi(z_1,z_2,z_3,z_4)=(z_1,iz_2,z_3,z_4)$, we get these commutative diagrams

\footnotesize
\[
\begin{CD}
(z_1,z_2,z_3,z_4) @>\varphi>>(z_1,iz_2,z_3,z_4)\\
@Vg_1VV @Vh_1h_2 VV\\
(-z_1,\bar z_2,z_3,\bar z_4) @>\varphi>> (-z_1,-\bar {iz}_2,z_3,\bar z_4)
\end{CD}
\hspace{.1cm}
\begin{CD}
(z_1,z_2,z_3,z_4) @>\varphi>>(z_1,iz_2,z_3,z_4)\\
@Vg_3VV @Vh_3 VV\\
(z_1,z_2,-z_3,z_4) @>\varphi>> (z_1,iz_2,-z_3,z_4)
\end{CD}
\]
\[ 
\begin{CD}
(z_1,z_2,z_3,z_4) @>\varphi>>(z_1,iz_2,z_3,z_4)\\
@Vg_2VV @Vh_2 VV\\
(z_1,-z_2,\bar z_3,z_4) @>\varphi>> (z_1,-iz_2,\bar z_3,z_4)
\end{CD}
\hspace{.1cm}
\begin{CD}
(z_1,z_2,z_3,z_4) @>\varphi>>(z_1,iz_2,z_3,z_4)\\
@Vg_4VV @Vh_4 VV\\
(z_1,z_2,z_3,-z_4) @>\varphi>> (z_1,iz_2,z_3,-z_4).
\end{CD}
\]

\normalsize
\noindent 
Therefore $\exists$ 
$\gamma 
\footnotesize
=
\Bigl( \left(\begin{array}{cccc}
0\\
\frac{1}{4}\\
0\\
0
\end{array}\right)
\left(\begin{array}{cccc}
1& 0 & 0 & 0\\
0& 1 & 0 & 0\\
0& 0 & 1 & 0\\
0& 0 & 0 & 1
\end{array}\right)
\Bigr)
\normalsize
\in \mathbb{A}(n)$ 
s.t. $\gamma \Gamma_{a13}\gamma ^{-1}=\Gamma_{a29}.$\\
%-------------------------------------------------------------------------------

\noindent 
$\bullet$ $M(A_{a13})\approx  M(A_{a18})$.\\
For 
\footnotesize
\begin{align*}
&A_{a13}=
\left(\begin{array}{cccc}
1& 1 & 0 & 1 \\
0& 1 & 1 & 0\\
0& 0 & 1 & 0\\
0& 0 & 0 & 1
\end{array}\right)
\hspace{.1cm}
&A_{a18}=
\left(\begin{array}{cccc}
1& 1 & 1 & 0 \\
0& 1 & 0 & 0\\
0& 0 & 1 & 1\\
0& 0 & 0 & 1
\end{array}\right)
\\
&g_1(z_1,z_2,z_3,z_4)=(-z_1,\bar z_2,z_3,\bar z_4).\hspace{.1cm}   &h_1(z_1,z_2,z_3,z_4)=(-z_1,\bar z_2,\bar z_3,z_4)\\
&g_2(z_1,z_2,z_3,z_4)=(z_1,-z_2,\bar z_3,z_4).                     &h_2(z_1,z_2,z_3,z_4)=(z_1,-z_2,z_3,z_4)\\
&g_3(z_1,z_2,z_3,z_4)=(z_1,z_2,-z_3,z_4).                          &h_3(z_1,z_2,z_3,z_4)=(z_1, z_2,-z_3,\bar z_4)\\
&g_4(z_1,z_2,z_3,z_4)=(z_1,z_2,z_3,-z_4).                          &h_4(z_1,z_2,z_3,z_4)=(z_1, z_2,z_3,-z_4)
\end{align*}
\normalsize then \footnotesize
\begin{align*}
\Gamma_{a13}=
\Bigl < &\left(\begin{array}{cccc}
\frac{1}{2}\\
0\\
0\\
0
\end{array}\right)
\left(\begin{array}{cccc}
1& 0 & 0 & 0\\
0& -1 & 0 & 0\\
0& 0 & 1 & 0\\
0& 0 & 0 & -1
\end{array}\right),
\left(\begin{array}{cccc}
0\\
\frac{1}{2}\\
0\\
0
\end{array}\right)
\left(\begin{array}{cccc}
1& 0 & 0 & 0\\
0& 1 & 0 & 0\\
0& 0 & -1 & 0\\
0& 0 & 0 & 1
\end{array}\right),
\\
&\left(\begin{array}{cccc}
0\\
0\\
\frac{1}{2}\\
0
\end{array}\right)
\left(\begin{array}{cccc}
1& 0 & 0 & 0\\
0& 1 & 0 & 0\\
0& 0 & 1 & 0\\
0& 0 & 0 & 1
\end{array}\right),
\left(\begin{array}{cccc}
0\\
0\\
0\\
\frac{1}{2}
\end{array}\right)
\left(\begin{array}{cccc}
1& 0 & 0 & 0\\
0& 1 & 0 & 0\\
0& 0 & 1 & 0\\
0& 0 & 0 & 1
\end{array}\right)
\Bigr >
\end{align*}

\begin{align*}
\Gamma_{a18}= 
\Bigl < & \left(\begin{array}{cccc}
\frac{1}{2}\\
0\\
0\\
0
\end{array}\right)
\left(\begin{array}{cccc}
1& 0 & 0 & 0\\
0& -1 & 0 & 0\\
0& 0 & -1 & 0\\
0& 0 & 0 & 1
\end{array}\right),
\left(\begin{array}{cccc}
0\\
\frac{1}{2}\\
0\\
0
\end{array}\right)
\left(\begin{array}{cccc}
1& 0 & 0 & 0\\
0& 1 & 0 & 0\\
0& 0 & 1 & 0\\
0& 0 & 0 & 1
\end{array}\right),
\\
&\left(\begin{array}{cccc}
0\\
0\\
\frac{1}{2}\\
0
\end{array}\right)
\left(\begin{array}{cccc}
1& 0 & 0 & 0\\
0& 1 & 0 & 0\\
0& 0 & 1 & 0\\
0& 0 & 0 & -1
\end{array}\right),
\left(\begin{array}{cccc}
0\\
0\\
0\\
\frac{1}{2}
\end{array}\right)
\left(\begin{array}{cccc}
1& 0 & 0 & 0\\
0& 1 & 0 & 0\\
0& 0 & 1 & 0\\
0& 0 & 0 & 1
\end{array}\right)
\Bigr >.
\end{align*}

\normalsize
\noindent
Let $\varphi(z_1,z_2,z_3,z_4)=(z_1,z_4,z_2,z_3)$, we get these commutative diagrams

\footnotesize
\[
\begin{CD}
(z_1,z_2,z_3,z_4) @>\varphi>>(z_1,z_4,z_2,z_3)\\
@Vg_1VV @Vh_1 VV\\
(-z_1,\bar z_2,z_3,\bar z_4) @>\varphi>> (-z_1,\bar z_4,\bar z_2,z_3)
\end{CD}
\hspace{.1cm}
\begin{CD}
(z_1,z_2,z_3,z_4) @>\varphi>>(z_1,z_4,z_2,z_3)\\
@Vg_3VV @Vh_4 VV\\
(z_1,z_2,-z_3,z_4) @>\varphi>> (z_1,z_4,z_2,-z_3)
\end{CD}
\]
\[ 
\begin{CD}
(z_1,z_2,z_3,z_4) @>\varphi>>(z_1,z_4,z_2,z_3)\\
@Vg_2VV @Vh_3 VV\\
(z_1,-z_2,\bar z_3,z_4) @>\varphi>> (z_1,z_4,-z_2,\bar z_3)
\end{CD}
\hspace{.1cm}
\begin{CD}
(z_1,z_2,z_3,z_4) @>\varphi>>(z_1,z_4,z_2,z_3)\\
@Vg_4VV @Vh_2 VV\\
(z_1,z_2,z_3,-z_4) @>\varphi>> (z_1,-z_4,z_2,z_3).
\end{CD}
\]

\normalsize
\noindent 
Therefore $\exists$ 
$\gamma 
\footnotesize
=
\Bigl( \left(\begin{array}{cccc}
0\\
0\\
0\\
0
\end{array}\right)
\left(\begin{array}{cccc}
1& 0 & 0 & 0\\
0& 0 & 0 & 1\\
0& 1 & 0 & 0\\
0& 0 & 1 & 0
\end{array}\right)
\Bigr)
\normalsize
\in \mathbb{A}(n)$ 
s.t. $\gamma \Gamma_{a13}\gamma ^{-1}=\Gamma_{a18}$.\\
%--------------------------------------------------------------------

\noindent 
$\bullet$ $M(A_{a13})\approx  M(A_{a19})$.\\
For
\footnotesize
\begin{align*}
&A_{a13}=
\left(\begin{array}{cccc}
1& 1 & 0 & 1 \\
0& 1 & 1 & 0\\
0& 0 & 1 & 0\\
0& 0 & 0 & 1
\end{array}\right)
\hspace{.1cm}
&A_{a19}=
\left(\begin{array}{cccc}
1& 1 & 1 & 0 \\
0& 1 & 0 & 1\\
0& 0 & 1 & 0\\
0& 0 & 0 & 1
\end{array}\right)
\\
&g_1(z_1,z_2,z_3,z_4)=(-z_1,\bar z_2,z_3,\bar z_4).\hspace{.1cm}   &h_1(z_1,z_2,z_3,z_4)=(-z_1,\bar z_2,\bar z_3,z_4)\\
&g_2(z_1,z_2,z_3,z_4)=(z_1,-z_2,\bar z_3,z_4).                     &h_2(z_1,z_2,z_3,z_4)=(z_1,-z_2,z_3,\bar z_4)\\
&g_3(z_1,z_2,z_3,z_4)=(z_1,z_2,-z_3,z_4).                          &h_3(z_1,z_2,z_3,z_4)=(z_1, z_2,-z_3,z_4)\\
&g_4(z_1,z_2,z_3,z_4)=(z_1,z_2,z_3,-z_4).                          &h_4(z_1,z_2,z_3,z_4)=(z_1, z_2,z_3,-z_4)
\end{align*}
\normalsize then \footnotesize
\begin{align*}
\Gamma_{a13}=
\Bigl < &\left(\begin{array}{cccc}
\frac{1}{2}\\
0\\
0\\
0
\end{array}\right)
\left(\begin{array}{cccc}
1& 0 & 0 & 0\\
0& -1 & 0 & 0\\
0& 0 & 1 & 0\\
0& 0 & 0 & -1
\end{array}\right),
\left(\begin{array}{cccc}
0\\
\frac{1}{2}\\
0\\
0
\end{array}\right)
\left(\begin{array}{cccc}
1& 0 & 0 & 0\\
0& 1 & 0 & 0\\
0& 0 & -1 & 0\\
0& 0 & 0 & 1
\end{array}\right),
\\
&\left(\begin{array}{cccc}
0\\
0\\
\frac{1}{2}\\
0
\end{array}\right)
\left(\begin{array}{cccc}
1& 0 & 0 & 0\\
0& 1 & 0 & 0\\
0& 0 & 1 & 0\\
0& 0 & 0 & 1
\end{array}\right),
\left(\begin{array}{cccc}
0\\
0\\
0\\
\frac{1}{2}
\end{array}\right)
\left(\begin{array}{cccc}
1& 0 & 0 & 0\\
0& 1 & 0 & 0\\
0& 0 & 1 & 0\\
0& 0 & 0 & 1
\end{array}\right)
\Bigr >
\end{align*}

\begin{align*}
\Gamma_{a19}= 
\Bigl < & \left(\begin{array}{cccc}
\frac{1}{2}\\
0\\
0\\
0
\end{array}\right)
\left(\begin{array}{cccc}
1& 0 & 0 & 0\\
0& -1 & 0 & 0\\
0& 0 & -1 & 0\\
0& 0 & 0 & 1
\end{array}\right),
\left(\begin{array}{cccc}
0\\
\frac{1}{2}\\
0\\
0
\end{array}\right)
\left(\begin{array}{cccc}
1& 0 & 0 & 0\\
0& 1 & 0 & 0\\
0& 0 & 1 & 0\\
0& 0 & 0 & -1
\end{array}\right),
\\
&\left(\begin{array}{cccc}
0\\
0\\
\frac{1}{2}\\
0
\end{array}\right)
\left(\begin{array}{cccc}
1& 0 & 0 & 0\\
0& 1 & 0 & 0\\
0& 0 & 1 & 0\\
0& 0 & 0 & 1
\end{array}\right),
\left(\begin{array}{cccc}
0\\
0\\
0\\
\frac{1}{2}
\end{array}\right)
\left(\begin{array}{cccc}
1& 0 & 0 & 0\\
0& 1 & 0 & 0\\
0& 0 & 1 & 0\\
0& 0 & 0 & 1
\end{array}\right)
\Bigr >.
\end{align*}

\normalsize
\noindent
Let $\varphi(z_1,z_2,z_3,z_4)=(z_1,z_2,z_4,z_3)$, we get these commutative diagrams

\footnotesize
\[
\begin{CD}
(z_1,z_2,z_3,z_4) @>\varphi>>(z_1,z_2,z_4,z_3)\\
@Vg_1VV @Vh_1 VV\\
(-z_1,\bar z_2,z_3,\bar z_4) @>\varphi>> (-z_1,\bar z_2,\bar z_4,z_3)
\end{CD}
\hspace{.1cm}
\begin{CD}
(z_1,z_2,z_3,z_4) @>\varphi>>(z_1,z_2,z_4,z_3)\\
@Vg_3VV @Vh_4 VV\\
(z_1,z_2,-z_3,z_4) @>\varphi>> (z_1,z_2,z_4,-z_3)
\end{CD}
\]
\[ 
\begin{CD}
(z_1,z_2,z_3,z_4) @>\varphi>>(z_1,z_2,z_4,z_3)\\
@Vg_2VV @Vh_2 VV\\
(z_1,-z_2,\bar z_3,z_4) @>\varphi>> (z_1,-z_2,z_4,\bar z_3)
\end{CD}
\hspace{.1cm}
\begin{CD}
(z_1,z_2,z_3,z_4) @>\varphi>>(z_1,z_2,z_4,z_3)\\
@Vg_4VV @Vh_3 VV\\
(z_1,z_2,z_3,-z_4) @>\varphi>> (z_1,z_2,-z_4,z_3).
\end{CD}
\]

\normalsize
\noindent 
Therefore $\exists$ 
$\gamma 
\footnotesize
=
\Bigl( \left(\begin{array}{cccc}
0\\
0\\
0\\
0
\end{array}\right)
\left(\begin{array}{cccc}
1& 0 & 0 & 0\\
0& 1 & 0 & 0\\
0& 0 & 0 & 1\\
0& 0 & 1 & 0
\end{array}\right)
\Bigr)
\normalsize
\in \mathbb{A}(n)$ 
s.t. $\gamma \Gamma_{a13}\gamma ^{-1}=\Gamma_{a19}$.\\
%--------------------------------------------------------------------

\noindent 
$\bullet$ $M(A_{a18})\approx  M(A_{a20})$.\\
For
\footnotesize
\begin{align*}
&A_{a18}=
\left(\begin{array}{cccc}
1& 1 & 1 & 0 \\
0& 1 & 0 & 0\\
0& 0 & 1 & 1\\
0& 0 & 0 & 1
\end{array}\right)
\hspace{.1cm}
&A_{a20}=
\left(\begin{array}{cccc}
1& 1 & 1 & 0 \\
0& 1 & 0 & 1\\
0& 0 & 1 & 1\\
0& 0 & 0 & 1
\end{array}\right)
\\
&g_1(z_1,z_2,z_3,z_4)=(-z_1,\bar z_2,\bar z_3,z_4)\hspace{.1cm}   &h_1(z_1,z_2,z_3,z_4)=(-z_1,\bar z_2,\bar z_3,z_4)\\
&g_2(z_1,z_2,z_3,z_4)=(z_1,-z_2,z_3,z_4)                          &h_2h_3(z_1,z_2,z_3,z_4)=(z_1,-z_2,-z_3,z_4)\\
&g_3(z_1,z_2,z_3,z_4)=(z_1,z_2,-z_3,\bar z_4)                     &h_3(z_1,z_2,z_3,z_4)=(z_1, z_2,-z_3,\bar z_4)\\
&g_4(z_1,z_2,z_3,z_4)=(z_1,z_2,z_3,-z_4)                          &h_4(z_1,z_2,z_3,z_4)=(z_1, z_2,z_3,-z_4)
\end{align*}
\normalsize then \footnotesize
\begin{align*}
\Gamma_{a18}= 
\Bigl < & \left(\begin{array}{cccc}
\frac{1}{2}\\
0\\
0\\
0
\end{array}\right)
\left(\begin{array}{cccc}
1& 0 & 0 & 0\\
0& -1 & 0 & 0\\
0& 0 & -1 & 0\\
0& 0 & 0 & 1
\end{array}\right),
\left(\begin{array}{cccc}
0\\
\frac{1}{2}\\
0\\
0
\end{array}\right)
\left(\begin{array}{cccc}
1& 0 & 0 & 0\\
0& 1 & 0 & 0\\
0& 0 & 1 & 0\\
0& 0 & 0 & 1
\end{array}\right),
\\
&\left(\begin{array}{cccc}
0\\
0\\
\frac{1}{2}\\
0
\end{array}\right)
\left(\begin{array}{cccc}
1& 0 & 0 & 0\\
0& 1 & 0 & 0\\
0& 0 & 1 & 0\\
0& 0 & 0 & -1
\end{array}\right),
\left(\begin{array}{cccc}
0\\
0\\
0\\
\frac{1}{2}
\end{array}\right)
\left(\begin{array}{cccc}
1& 0 & 0 & 0\\
0& 1 & 0 & 0\\
0& 0 & 1 & 0\\
0& 0 & 0 & 1
\end{array}\right)
\Bigr >
\end{align*}

\begin{align*}
\Gamma_{a20}=
\Bigl < &\left(\begin{array}{cccc}
\frac{1}{2}\\
0\\
0\\
0
\end{array}\right)
\left(\begin{array}{cccc}
1& 0 & 0 & 0\\
0& -1 & 0 & 0\\
0& 0 & -1 & 0\\
0& 0 & 0 & 1
\end{array}\right),
\left(\begin{array}{cccc}
0\\
\frac{1}{2}\\
\frac{1}{2}\\
0
\end{array}\right)
\left(\begin{array}{cccc}
1& 0 & 0 & 0\\
0& 1 & 0 & 0\\
0& 0 & 1 & 0\\
0& 0 & 0 & 1
\end{array}\right),
\\
&\left(\begin{array}{cccc}
0\\
0\\
\frac{1}{2}\\
0
\end{array}\right)
\left(\begin{array}{cccc}
1& 0 & 0 & 0\\
0& 1 & 0 & 0\\
0& 0 & 1 & 0\\
0& 0 & 0 & -1
\end{array}\right),
\left(\begin{array}{cccc}
0\\
0\\
0\\
\frac{1}{2}
\end{array}\right)
\left(\begin{array}{cccc}
1& 0 & 0 & 0\\
0& 1 & 0 & 0\\
0& 0 & 1 & 0\\
0& 0 & 0 & 1
\end{array}\right)
\Bigr >.
\end{align*}

\normalsize
\noindent
Let $\varphi(z_1,z_2,z_3,z_4)=(z_1,z_2,z_2z_3,z_4)$, we get these commutative diagrams

\footnotesize
\[
\begin{CD}
(z_1,z_2,z_3,z_4) @>\varphi>>(z_1,z_2,z_2z_3,z_4)\\
@Vg_1VV @Vh_1 VV\\
(-z_1,\bar z_2,\bar z_3,z_4) @>\varphi>> (-z_1,\bar z_2,\bar z_2\bar z_3,z_4)
\end{CD}
\hspace{.1cm}
\begin{CD}
(z_1,z_2,z_3,z_4) @>\varphi>>(z_1,z_2,z_2z_3,z_4)\\
@Vg_3VV @Vh_3 VV\\
(z_1,z_2,-z_3,\bar z_4) @>\varphi>> (z_1,z_2,-z_2z_3,\bar z_4)
\end{CD}
\]
\[ 
\begin{CD}
(z_1,z_2,z_3,z_4) @>\varphi>>(z_1,z_2,z_2z_3,z_4)\\
@Vg_2VV @Vh_2h_3 VV\\
(z_1,-z_2,z_3,z_4) @>\varphi>> (z_1,-z_2,-z_2z_3,z_4)
\end{CD}
\hspace{.1cm}
\begin{CD}
(z_1,z_2,z_3,z_4) @>\varphi>>(z_1,z_2,z_2z_3,z_4)\\
@Vg_4VV @Vh_4 VV\\
(z_1,z_2,z_3,-z_4) @>\varphi>> (z_1,z_2,z_2z_3,-z_4).
\end{CD}
\]

\normalsize
\noindent 
Therefore $\exists$ 
$\gamma 
\footnotesize
=
\Bigl( \left(\begin{array}{cccc}
0\\
0\\
0\\
0
\end{array}\right)
\left(\begin{array}{cccc}
1& 0 & 0 & 0\\
0& 1 & 0 & 0\\
0& 1 & 1 & 0\\
0& 0 & 0 & 1
\end{array}\right)
\Bigr)
\normalsize
\in \mathbb{A}(n)$ 
s.t. $\gamma \Gamma_{a18}\gamma ^{-1}=\Gamma_{a20}$.\\
%--------------------------------------------------------------------

\noindent 
$\bullet$ $M(A_{a22})\approx  M(A_{a23})$.\\
For
\footnotesize
\begin{align*}
&A_{a22}=
\left(\begin{array}{cccc}
1& 1 & 1 & 1 \\
0& 1 & 0 & 0\\
0& 0 & 1 & 1\\
0& 0 & 0 & 1
\end{array}\right)
\hspace{.1cm}
&A_{a23}=
\left(\begin{array}{cccc}
1& 1 & 1 & 1 \\
0& 1 & 0 & 1\\
0& 0 & 1 & 0\\
0& 0 & 0 & 1
\end{array}\right)
\\
&g_1(z_1,z_2,z_3,z_4)=(-z_1,\bar z_2,\bar z_3,\bar z_4)\hspace{.1cm}   &h_1(z_1,z_2,z_3,z_4)=(-z_1,\bar z_2,\bar z_3,\bar z_4)\\
&g_2(z_1,z_2,z_3,z_4)=(z_1,-z_2,z_3,z_4)                          &h_2(z_1,z_2,z_3,z_4)=(z_1,-z_2,z_3,\bar z_4)\\
&g_3(z_1,z_2,z_3,z_4)=(z_1,z_2,-z_3,\bar z_4)                     &h_3(z_1,z_2,z_3,z_4)=(z_1, z_2,-z_3,z_4)\\
&g_4(z_1,z_2,z_3,z_4)=(z_1,z_2,z_3,-z_4)                          &h_4(z_1,z_2,z_3,z_4)=(z_1, z_2,z_3,-z_4)
\end{align*}
\normalsize then \footnotesize
\begin{align*}
\Gamma_{a22}= 
\Bigl < & \left(\begin{array}{cccc}
\frac{1}{2}\\
0\\
0\\
0
\end{array}\right)
\left(\begin{array}{cccc}
1& 0 & 0 & 0\\
0& -1 & 0 & 0\\
0& 0 & -1 & 0\\
0& 0 & 0 & -1
\end{array}\right),
\left(\begin{array}{cccc}
0\\
\frac{1}{2}\\
0\\
0
\end{array}\right)
\left(\begin{array}{cccc}
1& 0 & 0 & 0\\
0& 1 & 0 & 0\\
0& 0 & 1 & 0\\
0& 0 & 0 & 1
\end{array}\right),
\\
&\left(\begin{array}{cccc}
0\\
0\\
\frac{1}{2}\\
0
\end{array}\right)
\left(\begin{array}{cccc}
1& 0 & 0 & 0\\
0& 1 & 0 & 0\\
0& 0 & 1 & 0\\
0& 0 & 0 & -1
\end{array}\right),
\left(\begin{array}{cccc}
0\\
0\\
0\\
\frac{1}{2}
\end{array}\right)
\left(\begin{array}{cccc}
1& 0 & 0 & 0\\
0& 1 & 0 & 0\\
0& 0 & 1 & 0\\
0& 0 & 0 & 1
\end{array}\right)
\Bigr >
\end{align*}

\begin{align*}
\Gamma_{a23}=
\Bigl < &\left(\begin{array}{cccc}
\frac{1}{2}\\
0\\
0\\
0
\end{array}\right)
\left(\begin{array}{cccc}
1& 0 & 0 & 0\\
0& -1 & 0 & 0\\
0& 0 & -1 & 0\\
0& 0 & 0 & -1
\end{array}\right),
\left(\begin{array}{cccc}
0\\
\frac{1}{2}\\
0\\
0
\end{array}\right)
\left(\begin{array}{cccc}
1& 0 & 0 & 0\\
0& 1 & 0 & 0\\
0& 0 & 1 & 0\\
0& 0 & 0 & -1
\end{array}\right),
\\
&\left(\begin{array}{cccc}
0\\
0\\
\frac{1}{2}\\
0
\end{array}\right)
\left(\begin{array}{cccc}
1& 0 & 0 & 0\\
0& 1 & 0 & 0\\
0& 0 & 1 & 0\\
0& 0 & 0 & 1
\end{array}\right),
\left(\begin{array}{cccc}
0\\
0\\
0\\
\frac{1}{2}
\end{array}\right)
\left(\begin{array}{cccc}
1& 0 & 0 & 0\\
0& 1 & 0 & 0\\
0& 0 & 1 & 0\\
0& 0 & 0 & 1
\end{array}\right)
\Bigr >.
\end{align*}

\normalsize
\noindent
Let $\varphi(z_1,z_2,z_3,z_4)=(z_1,z_3,z_2,z_4)$, we get these commutative diagrams

\footnotesize
\[
\begin{CD}
(z_1,z_2,z_3,z_4) @>\varphi>>(z_1,z_3,z_2,z_4)\\
@Vg_1VV @Vh_1 VV\\
(-z_1,\bar z_2,\bar z_3,\bar z_4) @>\varphi>> (-z_1,\bar z_3,\bar z_2,\bar z_4)
\end{CD}
\hspace{.1cm}
\begin{CD}
(z_1,z_2,z_3,z_4) @>\varphi>>(z_1,z_3,z_2,z_4)\\
@Vg_3VV @Vh_2 VV\\
(z_1,z_2,-z_3,\bar z_4) @>\varphi>> (z_1,-z_3,z_2,\bar z_4)
\end{CD}
\]
\[ 
\begin{CD}
(z_1,z_2,z_3,z_4) @>\varphi>>(z_1,z_3,z_2,z_4)\\
@Vg_2VV @Vh_3 VV\\
(z_1,-z_2,z_3,z_4) @>\varphi>> (z_1,z_3,-z_2,z_4)
\end{CD}
\hspace{.1cm}
\begin{CD}
(z_1,z_2,z_3,z_4) @>\varphi>>(z_1,z_3,z_2,z_4)\\
@Vg_4VV @Vh_4 VV\\
(z_1,z_2,z_3,-z_4) @>\varphi>> (z_1,z_3,z_2,-z_4).
\end{CD}
\]

\normalsize
\noindent 
Therefore $\exists$ 
$\gamma 
\footnotesize
=
\Bigl( \left(\begin{array}{cccc}
0\\
0\\
0\\
0
\end{array}\right)
\left(\begin{array}{cccc}
1& 0 & 0 & 0\\
0& 0 & 1 & 0\\
0& 1 & 0 & 0\\
0& 0 & 0 & 1
\end{array}\right)
\Bigr)
\normalsize
\in \mathbb{A}(n)$ 
s.t. $\gamma \Gamma_{a22}\gamma ^{-1}=\Gamma_{a23}$.\\
%--------------------------------------------------------------------

\noindent 
$\bullet$ $M(A_{a23})\approx  M(A_{a24})$.\\
For 
\footnotesize
\begin{align*}
&A_{a23}=
\left(\begin{array}{cccc}
1& 1 & 1 & 1 \\
0& 1 & 0 & 1\\
0& 0 & 1 & 0\\
0& 0 & 0 & 1
\end{array}\right)
\hspace{.1cm}
&A_{a24}=
\left(\begin{array}{cccc}
1& 1 & 1 & 1 \\
0& 1 & 0 & 1\\
0& 0 & 1 & 1\\
0& 0 & 0 & 1
\end{array}\right)
\\
&g_1(z_1,z_2,z_3,z_4)=(-z_1,\bar z_2,\bar z_3,\bar z_4)\hspace{.1cm}   &h_1(z_1,z_2,z_3,z_4)=(-z_1,\bar z_2,\bar z_3,\bar z_4)\\
&g_2(z_1,z_2,z_3,z_4)=(z_1,-z_2,z_3,\bar z_4)                          &h_2h_3(z_1,z_2,z_3,z_4)=(z_1,-z_2,-z_3,z_4)\\
&g_3(z_1,z_2,z_3,z_4)=(z_1,z_2,-z_3,z_4)                               &h_3(z_1,z_2,z_3,z_4)=(z_1, z_2,-z_3,\bar z_4)\\
&g_4(z_1,z_2,z_3,z_4)=(z_1,z_2,z_3,-z_4)                               &h_4(z_1,z_2,z_3,z_4)=(z_1, z_2,z_3,-z_4)
\end{align*}
\normalsize then \footnotesize
\begin{align*}
\Gamma_{a23}=
\Bigl < &\left(\begin{array}{cccc}
\frac{1}{2}\\
0\\
0\\
0
\end{array}\right)
\left(\begin{array}{cccc}
1& 0 & 0 & 0\\
0& -1 & 0 & 0\\
0& 0 & -1 & 0\\
0& 0 & 0 & -1
\end{array}\right),
\left(\begin{array}{cccc}
0\\
\frac{1}{2}\\
0\\
0
\end{array}\right)
\left(\begin{array}{cccc}
1& 0 & 0 & 0\\
0& 1 & 0 & 0\\
0& 0 & 1 & 0\\
0& 0 & 0 & -1
\end{array}\right),
\\
&\left(\begin{array}{cccc}
0\\
0\\
\frac{1}{2}\\
0
\end{array}\right)
\left(\begin{array}{cccc}
1& 0 & 0 & 0\\
0& 1 & 0 & 0\\
0& 0 & 1 & 0\\
0& 0 & 0 & 1
\end{array}\right),
\left(\begin{array}{cccc}
0\\
0\\
0\\
\frac{1}{2}
\end{array}\right)
\left(\begin{array}{cccc}
1& 0 & 0 & 0\\
0& 1 & 0 & 0\\
0& 0 & 1 & 0\\
0& 0 & 0 & 1
\end{array}\right)
\Bigr >
\end{align*}

\begin{align*}
\Gamma_{a24}= 
\Bigl < & \left(\begin{array}{cccc}
\frac{1}{2}\\
0\\
0\\
0
\end{array}\right)
\left(\begin{array}{cccc}
1& 0 & 0 & 0\\
0& -1 & 0 & 0\\
0& 0 & -1 & 0\\
0& 0 & 0 & -1
\end{array}\right),
\left(\begin{array}{cccc}
0\\
\frac{1}{2}\\
\frac{1}{2}\\
0
\end{array}\right)
\left(\begin{array}{cccc}
1& 0 & 0 & 0\\
0& 1 & 0 & 0\\
0& 0 & 1 & 0\\
0& 0 & 0 & 1
\end{array}\right),
\\
&\left(\begin{array}{cccc}
0\\
0\\
\frac{1}{2}\\
0
\end{array}\right)
\left(\begin{array}{cccc}
1& 0 & 0 & 0\\
0& 1 & 0 & 0\\
0& 0 & 1 & 0\\
0& 0 & 0 & -1
\end{array}\right),
\left(\begin{array}{cccc}
0\\
0\\
0\\
\frac{1}{2}
\end{array}\right)
\left(\begin{array}{cccc}
1& 0 & 0 & 0\\
0& 1 & 0 & 0\\
0& 0 & 1 & 0\\
0& 0 & 0 & 1
\end{array}\right)
\Bigr >.
\end{align*}

\normalsize
\noindent
Let $\varphi(z_1,z_2,z_3,z_4)=(z_1,z_3,z_2z_3,z_4)$, we get these commutative diagrams

\footnotesize
\[
\begin{CD}
(z_1,z_2,z_3,z_4) @>\varphi>>(z_1,z_3,z_2z_3,z_4)\\
@Vg_1VV @Vh_1 VV\\
(-z_1,\bar z_2,\bar z_3,\bar z_4) @>\varphi>> (-z_1,\bar z_3,\bar z_2\bar z_3,\bar z_4)
\end{CD}
\hspace{.1cm}
\begin{CD}
(z_1,z_2,z_3,z_4) @>\varphi>>(z_1,z_3,z_2z_3,z_4)\\
@Vg_3VV @Vh_2h_3 VV\\
(z_1,z_2,-z_3,z_4) @>\varphi>> (z_1,-z_3,-z_2z_3,z_4)
\end{CD}
\]
\[ 
\begin{CD}
(z_1,z_2,z_3,z_4) @>\varphi>>(z_1,z_3,z_2z_3,z_4)\\
@Vg_2VV @Vh_3 VV\\
(z_1,-z_2,z_3,\bar z_4) @>\varphi>> (z_1,z_3,-z_2z_3,\bar z_4)
\end{CD}
\hspace{.1cm}
\begin{CD}
(z_1,z_2,z_3,z_4) @>\varphi>>(z_1,z_3,z_2z_3,z_4)\\
@Vg_4VV @Vh_4 VV\\
(z_1,z_2,z_3,-z_4) @>\varphi>> (z_1,z_3,z_2,-z_4).
\end{CD}
\]

\normalsize
\noindent 
Therefore $\exists$ 
$\gamma 
\footnotesize
=
\Bigl( \left(\begin{array}{cccc}
0\\
0\\
0\\
0
\end{array}\right)
\left(\begin{array}{cccc}
1& 0 & 0 & 0\\
0& 0 & 1 & 0\\
0& 1 & 1 & 0\\
0& 0 & 0 & 1
\end{array}\right)
\Bigr)
\normalsize
\in \mathbb{A}(n)$ 
s.t. $\gamma \Gamma_{a23}\gamma ^{-1}=\Gamma_{a24}$.\\
%--------------------------------------------------------------------

\noindent 
$\bullet$ $M(A_{a23})\approx M(A_{a29})$.\\
For
\footnotesize
\begin{align*}
&A_{a23}=
\left(\begin{array}{cccc}
1& 1 & 1 & 1 \\
0& 1 & 0 & 1\\
0& 0 & 1 & 0\\
0& 0 & 0 & 1
\end{array}\right)
\hspace{.1cm}
&A_{a29}=
\left(\begin{array}{cccc}
1& 1 & 1 & 1 \\
0& 1 & 1 & 0\\
0& 0 & 1 & 0\\
0& 0 & 0 & 1
\end{array}\right)
\\
&g_1(z_1,z_2,z_3,z_4)=(-z_1,\bar z_2,\bar z_3,\bar z_4)\hspace{.1cm}   &h_1(z_1,z_2,z_3,z_4)=(-z_1,\bar z_2,\bar z_3,\bar z_4)\\
&g_2(z_1,z_2,z_3,z_4)=(z_1,-z_2,z_3,\bar z_4)                          &h_2(z_1,z_2,z_3,z_4)=(z_1,-z_2,\bar z_3,z_4)\\
&g_3(z_1,z_2,z_3,z_4)=(z_1,z_2,-z_3,z_4)                               &h_3(z_1,z_2,z_3,z_4)=(z_1, z_2,-z_3,z_4)\\
&g_4(z_1,z_2,z_3,z_4)=(z_1,z_2,z_3,-z_4)                               &h_4(z_1,z_2,z_3,z_4)=(z_1, z_2,z_3,-z_4)
\end{align*}
\normalsize then \footnotesize
\begin{align*}
\Gamma_{a23}=
\Bigl < &\left(\begin{array}{cccc}
\frac{1}{2}\\
0\\
0\\
0
\end{array}\right)
\left(\begin{array}{cccc}
1& 0 & 0 & 0\\
0& -1 & 0 & 0\\
0& 0 & -1 & 0\\
0& 0 & 0 & -1
\end{array}\right),
\left(\begin{array}{cccc}
0\\
\frac{1}{2}\\
0\\
0
\end{array}\right)
\left(\begin{array}{cccc}
1& 0 & 0 & 0\\
0& 1 & 0 & 0\\
0& 0 & 1 & 0\\
0& 0 & 0 & -1
\end{array}\right),
\\
&\left(\begin{array}{cccc}
0\\
0\\
\frac{1}{2}\\
0
\end{array}\right)
\left(\begin{array}{cccc}
1& 0 & 0 & 0\\
0& 1 & 0 & 0\\
0& 0 & 1 & 0\\
0& 0 & 0 & 1
\end{array}\right),
\left(\begin{array}{cccc}
0\\
0\\
0\\
\frac{1}{2}
\end{array}\right)
\left(\begin{array}{cccc}
1& 0 & 0 & 0\\
0& 1 & 0 & 0\\
0& 0 & 1 & 0\\
0& 0 & 0 & 1
\end{array}\right)
\Bigr >
\end{align*}

\begin{align*}
\Gamma_{a29}= 
\Bigl < & \left(\begin{array}{cccc}
\frac{1}{2}\\
0\\
0\\
0
\end{array}\right)
\left(\begin{array}{cccc}
1& 0 & 0 & 0\\
0& -1 & 0 & 0\\
0& 0 & -1 & 0\\
0& 0 & 0 & -1
\end{array}\right),
\left(\begin{array}{cccc}
0\\
\frac{1}{2}\\
0\\
0
\end{array}\right)
\left(\begin{array}{cccc}
1& 0 & 0 & 0\\
0& 1 & 0 & 0\\
0& 0 & -1 & 0\\
0& 0 & 0 & 1
\end{array}\right),
\\
&\left(\begin{array}{cccc}
0\\
0\\
\frac{1}{2}\\
0
\end{array}\right)
\left(\begin{array}{cccc}
1& 0 & 0 & 0\\
0& 1 & 0 & 0\\
0& 0 & 1 & 0\\
0& 0 & 0 & 1
\end{array}\right),
\left(\begin{array}{cccc}
0\\
0\\
0\\
\frac{1}{2}
\end{array}\right)
\left(\begin{array}{cccc}
1& 0 & 0 & 0\\
0& 1 & 0 & 0\\
0& 0 & 1 & 0\\
0& 0 & 0 & 1
\end{array}\right)
\Bigr >.
\end{align*}

\normalsize
\noindent
Let $\varphi(z_1,z_2,z_3,z_4)=(z_1,z_2,z_4,z_3)$, we get these commutative diagrams

\footnotesize
\[
\begin{CD}
(z_1,z_2,z_3,z_4) @>\varphi>>(z_1,z_2,z_4,z_3)\\
@Vg_1VV @Vh_1 VV\\
(-z_1,\bar z_2,\bar z_3,\bar z_4) @>\varphi>> (-z_1,\bar z_2,\bar z_4,\bar z_3)
\end{CD}
\hspace{.1cm}
\begin{CD}
(z_1,z_2,z_3,z_4) @>\varphi>>(z_1,z_2,z_4,z_3)\\
@Vg_3VV @Vh_4 VV\\
(z_1,z_2,-z_3,z_4) @>\varphi>> (z_1,z_2,z_4,-z_3)
\end{CD}
\]
\[ 
\begin{CD}
(z_1,z_2,z_3,z_4) @>\varphi>>(z_1,z_2,z_4,z_3)\\
@Vg_2VV @Vh_2 VV\\
(z_1,-z_2,z_3,\bar z_4) @>\varphi>> (z_1,-z_2,\bar z_4,z_3)
\end{CD}
\hspace{.1cm}
\begin{CD}
(z_1,z_2,z_3,z_4) @>\varphi>>(z_1,z_2,z_4,z_3)\\
@Vg_4VV @Vh_3 VV\\
(z_1,z_2,z_3,-z_4) @>\varphi>> (z_1,z_2,-z_4,z_3).
\end{CD}
\]

\normalsize
\noindent 
Therefore $\exists$ 
$\gamma 
\footnotesize
=
\Bigl( \left(\begin{array}{cccc}
0\\
0\\
0\\
0
\end{array}\right)
\left(\begin{array}{cccc}
1& 0 & 0 & 0\\
0& 1 & 0 & 0\\
0& 0 & 0 & 1\\
0& 0 & 1 & 0
\end{array}\right)
\Bigr)
\normalsize
\in \mathbb{A}(n)$ 
s.t. $\gamma \Gamma_{a23}\gamma ^{-1}=\Gamma_{a29}$.\\
%===============================================================================

\noindent {\bf (17)}. Similar to (4).\\

\noindent {\bf (18)}. Similar to (11).\\
%===============================================================================

\noindent {\bf (19)}.\\
\noindent 
$\bullet$ $M(A_{a11})\approx  M(A_{a31})$.\\
For
\footnotesize
\begin{align*}
&A_{a11}=
\left(\begin{array}{cccc}
1& 1 & 0 & 0 \\
0& 1 & 1 & 1\\
0& 0 & 1 & 0\\
0& 0 & 0 & 1
\end{array}\right)
\hspace{.1cm}
&A_{a31}=
\left(\begin{array}{cccc}
1& 1 & 1 & 1 \\
0& 1 & 1 & 1\\
0& 0 & 1 & 0\\
0& 0 & 0 & 1
\end{array}\right)
\\
&g_1(z_1,z_2,z_3,z_4)=(-z_1,\bar z_2,z_3,z_4).\hspace{.1cm}   &h_1h_2(z_1,z_2,z_3,z_4)=(-z_1,-\bar z_2,z_3,z_4)\\
&g_2(z_1,z_2,z_3,z_4)=(z_1,-z_2,\bar z_3,\bar z_4).                &h_2(z_1,z_2,z_3,z_4)=(z_1,-z_2,\bar z_3,\bar z_4)\\
&g_3(z_1,z_2,z_3,z_4)=(z_1,z_2,-z_3,z_4).                     &h_3(z_1,z_2,z_3,z_4)=(z_1, z_2,-z_3,z_4)\\
&g_4(z_1,z_2,z_3,z_4)=(z_1,z_2,z_3,-z_4).                     &h_4(z_1,z_2,z_3,z_4)=(z_1, z_2,z_3,-z_4)
\end{align*}
\normalsize then \footnotesize
\begin{align*}
\Gamma_{a11}=
\Bigl < &\left(\begin{array}{cccc}
\frac{1}{2}\\
0\\
0\\
0
\end{array}\right)
\left(\begin{array}{cccc}
1& 0 & 0 & 0\\
0& -1 & 0 & 0\\
0& 0 & 1 & 0\\
0& 0 & 0 & 1
\end{array}\right),
\left(\begin{array}{cccc}
0\\
\frac{1}{2}\\
0\\
0
\end{array}\right)
\left(\begin{array}{cccc}
1& 0 & 0 & 0\\
0& 1 & 0 & 0\\
0& 0 & -1 & 0\\
0& 0 & 0 & -1
\end{array}\right),
\\
&\left(\begin{array}{cccc}
0\\
0\\
\frac{1}{2}\\
0
\end{array}\right)
\left(\begin{array}{cccc}
1& 0 & 0 & 0\\
0& 1 & 0 & 0\\
0& 0 & 1 & 0\\
0& 0 & 0 & 1
\end{array}\right),
\left(\begin{array}{cccc}
0\\
0\\
0\\
\frac{1}{2}
\end{array}\right)
\left(\begin{array}{cccc}
1& 0 & 0 & 0\\
0& 1 & 0 & 0\\
0& 0 & 1 & 0\\
0& 0 & 0 & 1
\end{array}\right)
\Bigr >
\end{align*}

\begin{align*}
\Gamma_{a31}= 
\Bigl < & \left(\begin{array}{cccc}
\frac{1}{2}\\
\frac{1}{2}\\
0\\
0
\end{array}\right)
\left(\begin{array}{cccc}
1& 0 & 0 & 0\\
0& -1 & 0 & 0\\
0& 0 & 1 & 0\\
0& 0 & 0 & 1
\end{array}\right),
\left(\begin{array}{cccc}
0\\
\frac{1}{2}\\
0\\
0
\end{array}\right)
\left(\begin{array}{cccc}
1& 0 & 0 & 0\\
0& 1 & 0 & 0\\
0& 0 & -1 & 0\\
0& 0 & 0 & -1
\end{array}\right),
\\
&\left(\begin{array}{cccc}
0\\
0\\
\frac{1}{2}\\
0
\end{array}\right)
\left(\begin{array}{cccc}
1& 0 & 0 & 0\\
0& 1 & 0 & 0\\
0& 0 & 1 & 0\\
0& 0 & 0 & 1
\end{array}\right),
\left(\begin{array}{cccc}
0\\
0\\
0\\
\frac{1}{2}
\end{array}\right)
\left(\begin{array}{cccc}
1& 0 & 0 & 0\\
0& 1 & 0 & 0\\
0& 0 & 1 & 0\\
0& 0 & 0 & 1
\end{array}\right)
\Bigr >.
\end{align*}

\normalsize
\noindent
Let $\varphi(z_1,z_2,z_3,z_4)=(z_1,iz_2,z_3,z_4)$, we get these commutative diagrams

\footnotesize
\[
\begin{CD}
(z_1,z_2,z_3,z_4) @>\varphi>>(z_1,iz_2,z_3,z_4)\\
@Vg_1VV @Vh_1h_2 VV\\
(-z_1,\bar z_2,z_3,z_4) @>\varphi>> (-z_1,-\bar {iz}_2,z_3,z_4)
\end{CD}
\hspace{.1cm}
\begin{CD}
(z_1,z_2,z_3,z_4) @>\varphi>>(z_1,iz_2,z_3,z_4)\\
@Vg_3VV @Vh_3 VV\\
(z_1,z_2,-z_3,z_4) @>\varphi>> (z_1,iz_2,-z_3,z_4)
\end{CD}
\]
\[ 
\begin{CD}
(z_1,z_2,z_3,z_4) @>\varphi>>(z_1,iz_2,z_3,z_4)\\
@Vg_2VV @Vh_2 VV\\
(z_1,-z_2,\bar z_3,\bar z_4) @>\varphi>> (z_1,-iz_2,\bar z_3,\bar z_4)
\end{CD}
\hspace{.1cm}
\begin{CD}
(z_1,z_2,z_3,z_4) @>\varphi>>(z_1,iz_2,z_3,z_4)\\
@Vg_4VV @Vh_4 VV\\
(z_1,z_2,z_3,-z_4) @>\varphi>> (z_1,iz_2,z_3,-z_4).
\end{CD}
\]

\normalsize
\noindent 
Therefore $\exists$ 
$\gamma 
\footnotesize
=
\Bigl( \left(\begin{array}{cccc}
0\\
\frac{1}{4}\\
0\\
0
\end{array}\right)
\left(\begin{array}{cccc}
1& 0 & 0 & 0\\
0& 1 & 0 & 0\\
0& 0 & 1 & 0\\
0& 0 & 0 & 1
\end{array}\right)
\Bigr)
\normalsize
\in \mathbb{A}(n)$ 
s.t. $\gamma \Gamma_{a11}\gamma ^{-1}=\Gamma_{a31}.$\\
%============================================================================

\noindent {\bf (20)}. Similar to (4).\\

\noindent {\bf (21)}. Similar to (11).\\

%============================================================================
\noindent {\bf (22)}. $M(A_{a11})$ is not diffeomorphic to $M(A_{a13})$.\\

Let \footnotesize
\begin{align*}
\Gamma_{a13}= \langle &\tilde {g}_1,\tilde {g}_2,t_3,t_4\rangle\\
            = \Bigl <  &\left(\begin{array}{cccc}
\frac{1}{2}\\
0\\
0\\
0
\end{array}\right)
\left(\begin{array}{cccc}
1& 0 & 0 & 0\\
0& -1 & 0 & 0\\
0& 0 & 1 & 0\\
0& 0 & 0 & -1
\end{array}\right),
\left(\begin{array}{cccc}
0\\
\frac{1}{2}\\
0\\
0
\end{array}\right)
\left(\begin{array}{cccc}
1& 0 & 0 & 0\\
0& 1 & 0 & 0\\
0& 0 & -1 & 0\\
0& 0 & 0 & 1
\end{array}\right),
\\
&\left(\begin{array}{cccc}
0\\
0\\
\frac{1}{2}\\
0
\end{array}\right)
\left(\begin{array}{cccc}
1& 0 & 0 & 0\\
0& 1 & 0 & 0\\
0& 0 & 1 & 0\\
0& 0 & 0 & 1
\end{array}\right),
\left(\begin{array}{cccc}
0\\
0\\
0\\
\frac{1}{2}
\end{array}\right)
\left(\begin{array}{cccc}
1& 0 & 0 & 0\\
0& 1 & 0 & 0\\
0& 0 & 1 & 0\\
0& 0 & 0 & 1
\end{array}\right)
\Bigr >
\\ 
\Gamma_{a11}= \langle &\tilde {h}_1,\tilde {h}_2,t_3,t_4\rangle\\
            = \Bigl < & \left(\begin{array}{cccc}
\frac{1}{2}\\
0\\
0\\
0
\end{array}\right)
\left(\begin{array}{cccc}
1& 0 & 0 & 0\\
0& -1 & 0 & 0\\
0& 0 & 1 & 0\\
0& 0 & 0 & 1
\end{array}\right),
\left(\begin{array}{cccc}
0\\
\frac{1}{2}\\
0\\
0
\end{array}\right)
\left(\begin{array}{cccc}
1& 0 & 0 & 0\\
0& 1 & 0 & 0\\
0& 0 & -1 & 0\\
0& 0 & 0 & -1
\end{array}\right),
\\
&\left(\begin{array}{cccc}
0\\
0\\
\frac{1}{2}\\
0
\end{array}\right)
\left(\begin{array}{cccc}
1& 0 & 0 & 0\\
0& 1 & 0 & 0\\
0& 0 & 1 & 0\\
0& 0 & 0 & 1
\end{array}\right),
\left(\begin{array}{cccc}
0\\
0\\
0\\
\frac{1}{2}
\end{array}\right)
\left(\begin{array}{cccc}
1& 0 & 0 & 0\\
0& 1 & 0 & 0\\
0& 0 & 1 & 0\\
0& 0 & 0 & 1
\end{array}\right)
\Bigr >.
\end{align*}

\normalsize
Since $\tilde {g}_1\tilde {g}_2\tilde {g}_1^{-1}=\tilde {g}_2^{-1},\tilde {g}_1 t_3\tilde {g}_1^{-1}=t_3,$ and
$\tilde {g}_1 t_4\tilde {g}_1^{-1}=t_4$
then $\Gamma_{a13}= \langle \tilde {g}_2,t_3,t_4\rangle \rtimes \langle\tilde {g}_1\rangle$
and the center $\displaystyle \mathcal C(\Gamma_{a13})= \langle \tilde {g}_1^2 \rangle 
                                                       = \langle t_1 \rangle$,
where $t_1$= \footnotesize
$\left(\begin{array}{c}
1\\
0\\
0\\
0\end{array}\right)$. 
\normalsize
It induces an extension
\begin{align*}
1\rightarrow  \langle t_1\rangle \rightarrow  \Gamma_{a13}\longrightarrow  \Delta_{a13}\rightarrow  1,
\end{align*}
where $\Delta_{a13}$ is isomorphic to \footnotesize
\begin{align*}
\Bigl <&\left(\begin{array}{c}
\frac{1}{2}\\
0\\
0
\end{array}\right)
\left(\begin{array}{ccc}
1& 0 & 0 \\
0& -1 & 0\\
0& 0 & 1
\end{array}\right),
\left(\begin{array}{c}
0\\
\frac{1}{2}\\
0
\end{array}\right)
\left(\begin{array}{ccc}
1& 0 & 0\\
0& 1 & 0\\
0& 0 & 1
\end{array}\right),
\\
&\left(\begin{array}{c}
0\\
0\\
\frac{1}{2}
\end{array}\right)
\left(\begin{array}{ccc}
1& 0 & 0\\
0& 1 & 0\\
0& 0 & 1
\end{array}\right)
\Bigr >\rtimes \Bigl< \left(\begin{array}{c}
0\\
0\\
0
\end{array}\right)
\left(\begin{array}{ccc}
-1& 0 & 0 \\
0& 1 & 0\\
0& 0 & -1
\end{array}\right)\Bigr>.
\end{align*}
\normalsize
Put
\begin{align*}
\Delta_{a13}^f=&\langle p,q,r \rangle\\
=& \footnotesize \Bigl <\left(\begin{array}{c} 
\frac{1}{2}\\
0\\
0
\end{array}\right)
\left(\begin{array}{ccc}
1& 0 & 0 \\
0& -1 & 0\\
0& 0 & 1
\end{array}\right),
\left(\begin{array}{c}
0\\
\frac{1}{2}\\
0
\end{array}\right)
\left(\begin{array}{ccc}
1& 0 & 0\\
0& 1 & 0\\
0& 0 & 1
\end{array}\right),
\left(\begin{array}{c}
0\\
0\\
\frac{1}{2}
\end{array}\right)
\left(\begin{array}{ccc}
1& 0 & 0\\
0& 1 & 0\\
0& 0 & 1
\end{array}\right)
\Bigr >
\end{align*}
\normalsize
and 
\begin{align*}
\beta = \footnotesize \Bigl< \left(\begin{array}{c}
0\\
0\\
0
\end{array}\right)
\left(\begin{array}{ccc}
-1& 0 & 0 \\
0& 1 & 0\\
0& 0 & -1
\end{array}\right)\Bigr>
\end{align*}
\normalsize 
so that $\Delta_{a13}=\Delta_{a13}^f \rtimes \beta$. 

\normalsize
Next, since $\tilde {h}_1\tilde {h}_2\tilde {h}_1^{-1}=\tilde {h}_2^{-1},\tilde {h}_1 t_3\tilde {h}_1^{-1}=t_3,$ and
$\tilde {h}_1 t_4\tilde {h}_1^{-1}=t_4$ then 
 $\Gamma_{a11}= \langle \tilde{h}_2,t_3,t_4\rangle \rtimes \langle\tilde {h}_1\rangle$ and the center
$\displaystyle \mathcal C(\Gamma_{a11})= \langle \tilde {h}_1^2 \rangle= \langle t_1 \rangle$. 
It induces an extension
\begin{align*}
1\rightarrow  \langle t_1 \rangle \rightarrow \Gamma_{a11}\longrightarrow  \Delta_{a11}\rightarrow  1,
\end{align*}
where $\Delta_{a11}$ is isomorphic to \footnotesize
\begin{align*}
\Bigl <&\left(\begin{array}{c}
\frac{1}{2}\\
0\\
0
\end{array}\right)
\left(\begin{array}{ccc}
1& 0 & 0 \\
0& -1 & 0\\
0& 0 & -1
\end{array}\right),
\left(\begin{array}{c}
0\\
\frac{1}{2}\\
0
\end{array}\right)
\left(\begin{array}{ccc}
1& 0 & 0\\
0& 1 & 0\\
0& 0 & 1
\end{array}\right),
\\
&\left(\begin{array}{c}
0\\
0\\
\frac{1}{2}
\end{array}\right)
\left(\begin{array}{ccc}
1& 0 & 0\\
0& 1 & 0\\
0& 0 & 1
\end{array}\right)
\Bigr >\rtimes \langle \left(\begin{array}{c}
0\\
0\\
0
\end{array}\right)
\left(\begin{array}{ccc}
-1& 0 & 0 \\
0& 1 & 0\\
0& 0 & 1
\end{array}\right)\rangle.
\end{align*}
\normalsize
Put 
\begin{align*}
\Delta_{a11}^f= \footnotesize
\Bigl <\left(\begin{array}{c}
\frac{1}{2}\\
0\\
0
\end{array}\right)
\left(\begin{array}{ccc}
1& 0 & 0 \\
0& -1 & 0\\
0& 0 & -1
\end{array}\right),
\left(\begin{array}{c}
0\\
\frac{1}{2}\\
0
\end{array}\right)
\left(\begin{array}{ccc}
1& 0 & 0\\
0& 1 & 0\\
0& 0 & 1
\end{array}\right),
\left(\begin{array}{c}
0\\
0\\
\frac{1}{2}
\end{array}\right)
\left(\begin{array}{ccc}
1& 0 & 0\\
0& 1 & 0\\
0& 0 & 1
\end{array}\right)
\Bigr >
\end{align*}
\normalsize
and
\begin{align*}
\alpha = \footnotesize
\langle \left(\begin{array}{c}
0\\
0\\
0
\end{array}\right)
\left(\begin{array}{ccc}
-1& 0 & 0 \\
0& 1 & 0\\
0& 0 & 1
\end{array}\right)\rangle
\end{align*}
\normalsize
so that
$\Delta_{a11}=\Delta_{a11}^f\rtimes \alpha $. 
 
Suppose there is an isomorphism $\varphi:\Gamma_{a13}\rightarrow 
\Gamma_{a11}$ which induces an isomorphism
$\hat\varphi:\Delta_{a13}\rightarrow  \Delta_{a11}$
\begin{align*}
\begin{CD}
1 @>>>  \langle t_1 \rangle @>>> \Gamma_{a13}@>>>  \Delta_{a13}@>>>  1\\
  @.                @.                @V \varphi VV               @V \hat \varphi  VV\\
1 @>>>  \langle t_1 \rangle @>>> \Gamma_{a11}@>>>  \Delta_{a11}@>>>  1.
\end{CD}
\end{align*}
Consider $\alpha=\hat \varphi (p^{a_1} q^{a_2} r^{a_3}\beta )=\footnotesize 
\hat \varphi  \bigg (\left(\begin{array}{c}
\frac{1}{2}a_1\\
(-1)^{a_1}\frac{1}{2}a_2\\
\frac{1}{2}a_3
\end{array}\right)
\left(\begin{array}{ccc}
-1& 0 & 0 \\
0& (-1)^{a_1} & 0\\
0& 0 & -1
\end{array}\right)\bigg )$. \normalsize
Since $\alpha$ is a torsion element and $\hat \varphi $ is an isomorphism then
$p^{a_1} q^{a_2} r^{a_3}\beta$ is also a torsion element.
Therefore 
\begin{align*}
(p^{a_1} q^{a_2} r^{a_3}\beta)^2=&\footnotesize
\bigg (\left(\begin{array}{c}
0\\
(1+(-1)^{a_1})(-1)^{a_1}\frac{1}{2}a_2\\
0
\end{array}\right)
\left(\begin{array}{ccc}
1& 0 & 0 \\
0& 1 & 0\\
0& 0 & 1
\end{array}\right)\bigg )\\
=&
\bigg (\left(\begin{array}{c}
0\\
0\\
0
\end{array}\right)
\left(\begin{array}{ccc}
1& 0 & 0 \\
0& 1 & 0\\
0& 0 & 1
\end{array}\right)\bigg ).
\end{align*} \normalsize
Hence $a_1=odd$ or $a_2=0$. 
If $a_1=odd$ then 
\begin{align*}
p^{a_1} q^{a_2} r^{a_3}\beta=
\footnotesize 
\bigg (\left(\begin{array}{c}
\frac{1}{2}(2n+1)\\
-\frac{1}{2}a_2\\
\frac{1}{2}a_3
\end{array}\right)
\left(\begin{array}{ccc}
-1& 0 & 0 \\
0& -1 & 0\\
0& 0 & -1
\end{array}\right)\bigg ).
\end{align*}
By Bieberbach theorem, $\exists \gamma=(b,B) \in \mathbb{A}(4) $ such that 
$BL(p^{a_1} q^{a_2} r^{a_3}\beta)B^{-1}=L(\hat\varphi (p^{a_1} q^{a_2} r^{a_3}\beta))$ or 
\begin{align*}
B\footnotesize \left(\begin{array}{ccc}
-1& 0 & 0 \\
0& -1 & 0\\
0& 0 & -1
\end{array}\right)\normalsize B^{-1}=\footnotesize 
\left(\begin{array}{ccc}
-1& 0 & 0 \\
0& 1 & 0\\
0& 0 & 1
\end{array}\right),
\end{align*} which is impossible. Also if $a_2=0$
then 
\begin{align*}
p^{a_1} q^{a_2} r^{a_3}\beta=
\footnotesize 
\bigg (\left(\begin{array}{c}
\frac{1}{2}a_1\\
0\\
\frac{1}{2}a_3
\end{array}\right)
\left(\begin{array}{ccc}
-1& 0 & 0 \\
0& (-1)^{a_1} & 0\\
0& 0 & -1
\end{array}\right)\bigg ).
\end{align*}
By Bieberbach theorem, $\exists \gamma=(b,B) \in \mathbb{A}(4) $ such that 
$BL(p^{a_1} q^{a_2} r^{a_3}\beta)B^{-1}=L(\hat\varphi (p^{a_1} q^{a_2} r^{a_3}\beta))$ or 
\begin{align*}
B\footnotesize \left(\begin{array}{ccc}
-1& 0 & 0 \\
0& (-1)^{a_1} & 0\\
0& 0 & -1
\end{array}\right)\normalsize B^{-1}=\footnotesize 
\left(\begin{array}{ccc}
-1& 0 & 0 \\
0& 1 & 0\\
0& 0 & 1
\end{array}\right),
\end{align*} which is also impossible.
Therefore $\Gamma_{a11}$ is not isomorphic to $\Gamma_{a13}$.\\

%============================================================================
\noindent {\bf (23)}.\\
\noindent $\bullet$ $M(A_{13})\approx  M(A_{15})$.\\
For
\footnotesize
\begin{align*}
&A_{13}=
\left(\begin{array}{cccc}
1& 0 & 0 & 1 \\
0& 1 & 1 & 0\\
0& 0 & 1 & 0\\
0& 0 & 0 & 1
\end{array}\right)
\hspace{.1cm}
&A_{15}=
\left(\begin{array}{cccc}
1& 0 & 0 & 1 \\
0& 1 & 1 & 1\\
0& 0 & 1 & 0\\
0& 0 & 0 & 1
\end{array}\right)
\\
&g_1(z_1,z_2,z_3,z_4)=(-z_1,z_2,z_3,\bar z_4).\hspace{.1cm}   &h_1(z_1,z_2,z_3,z_4)=(-z_1,z_2,z_3,\bar z_4)\\
&g_2(z_1,z_2,z_3,z_4)=(z_1,-z_2,\bar z_3,z_4).            &h_2h_1(z_1,z_2,z_3,z_4)=(-z_1,-z_2,\bar z_3,z_4)\\
&g_3(z_1,z_2,z_3,z_4)=(z_1,z_2,-z_3,z_4).           &h_3(z_1,z_2,z_3,z_4)=(z_1,z_2,-z_3,z_4)\\
&g_4(z_1,z_2,z_3,z_4)=(z_1,z_2,z_3,-z_4).                &h_4(z_1,z_2,z_3,z_4)=(z_1,z_2,z_3,-z_4)
\end{align*}
\normalsize then \footnotesize
\begin{align*}
\Gamma_{13}=
\Bigl < &\left(\begin{array}{cccc}
\frac{1}{2}\\
0\\
0\\
0
\end{array}\right)
\left(\begin{array}{cccc}
1& 0 & 0 & 0\\
0& 1 & 0 & 0\\
0& 0 & 1 & 0\\
0& 0 & 0 & -1
\end{array}\right),
\left(\begin{array}{cccc}
0\\
\frac{1}{2}\\
0\\
0
\end{array}\right)
\left(\begin{array}{cccc}
1& 0 & 0 & 0\\
0& 1 & 0 & 0\\
0& 0 & -1 & 0\\
0& 0 & 0 & 1
\end{array}\right),
\\
&\left(\begin{array}{cccc}
0\\
0\\
\frac{1}{2}\\
0
\end{array}\right)
\left(\begin{array}{cccc}
1& 0 & 0 & 0\\
0& 1 & 0 & 0\\
0& 0 & 1 & 0\\
0& 0 & 0 & 1
\end{array}\right),
\left(\begin{array}{cccc}
0\\
0\\
0\\
\frac{1}{2}
\end{array}\right)
\left(\begin{array}{cccc}
1& 0 & 0 & 0\\
0& 1 & 0 & 0\\
0& 0 & 1 & 0\\
0& 0 & 0 & 1
\end{array}\right)
\Bigr >
\end{align*}

\begin{align*}
\Gamma_{15}= 
\Bigl < & \left(\begin{array}{cccc}
\frac{1}{2}\\
0\\
0\\
0
\end{array}\right)
\left(\begin{array}{cccc}
1& 0 & 0 & 0\\
0& 1 & 0 & 0\\
0& 0 & 1 & 0\\
0& 0 & 0 & -1
\end{array}\right),
\left(\begin{array}{cccc}
\frac{1}{2}\\
\frac{1}{2}\\
0\\
0
\end{array}\right)
\left(\begin{array}{cccc}
1& 0 & 0 & 0\\
0& 1 & 0 & 0\\
0& 0 & -1 & 0\\
0& 0 & 0 & 1
\end{array}\right),
\\
&\left(\begin{array}{cccc}
0\\
0\\
\frac{1}{2}\\
0
\end{array}\right)
\left(\begin{array}{cccc}
1& 0 & 0 & 0\\
0& 1 & 0 & 0\\
0& 0 & 1 & 0\\
0& 0 & 0 & 1
\end{array}\right),
\left(\begin{array}{cccc}
0\\
0\\
0\\
\frac{1}{2}
\end{array}\right)
\left(\begin{array}{cccc}
1& 0 & 0 & 0\\
0& 1 & 0 & 0\\
0& 0 & 1 & 0\\
0& 0 & 0 & 1
\end{array}\right)
\Bigr >.
\end{align*}

\normalsize
\noindent
Let $\varphi(z_1,z_2,z_3,z_4)=(z_1z_2,z_2,z_3,z_4)$, we get these commutative diagrams

\footnotesize
\[
\begin{CD}
(z_1,z_2,z_3,z_4) @>\varphi>>(z_1z_2,z_2,z_3,z_4)\\
@Vg_1VV @Vh_1 VV\\
(-z_1,z_2,z_3,\bar z_4) @>\varphi>> (-z_1z_2,z_2,z_3,\bar z_4)
\end{CD}
\hspace{.1cm}
\begin{CD}
(z_1,z_2,z_3,z_4) @>\varphi>>(z_1z_2,z_2,z_3,z_4)\\
@Vg_3VV @Vh_3 VV\\
(z_1,z_2,-z_3,z_4) @>\varphi>> (z_1z_2,z_2,-z_3,z_4)
\end{CD}
\]
\[ 
\begin{CD}
(z_1,z_2,z_3,z_4) @>\varphi>>(z_1z_2,z_2,z_3,z_4)\\
@Vg_2VV @Vh_2h_1 VV\\
(z_1,-z_2,\bar z_3,z_4) @>\varphi>> (-z_1z_2,-z_2,\bar z_3,z_4)
\end{CD}
\hspace{.1cm}
\begin{CD}
(z_1,z_2,z_3,z_4) @>\varphi>>(z_1z_2,z_2,z_3,z_4)\\
@Vg_4VV @Vh_4 VV\\
(z_1,z_2,z_3,-z_4) @>\varphi>> (z_1z_2,z_2,z_3,-z_4).
\end{CD}
\]

\normalsize
\noindent 
Therefore $\exists$ 
$\gamma 
\footnotesize
=
\Bigl( \left(\begin{array}{cccc}
0\\
0\\
0\\
0
\end{array}\right)
\left(\begin{array}{cccc}
1& 1 & 0 & 0\\
0& 1 & 0 & 0\\
0& 0 & 1 & 0\\
0& 0 & 0 & 1
\end{array}\right)
\Bigr)
\normalsize
\in \mathbb{A}(n)$ 
s.t. $\gamma \Gamma_{13}\gamma ^{-1}=\Gamma_{15}$.\\ 
%---------------------------------------------------------------------------

\noindent 
$\bullet$ $M(A_{13})\approx  M(A_{19})$.\\
For
\footnotesize
\begin{align*}
&A_{13}=
\left(\begin{array}{cccc}
1& 0 & 0 & 1 \\
0& 1 & 1 & 0\\
0& 0 & 1 & 0\\
0& 0 & 0 & 1
\end{array}\right)
\hspace{.1cm}
&A_{19}=
\left(\begin{array}{cccc}
1& 0 & 1 & 0 \\
0& 1 & 0 & 1\\
0& 0 & 1 & 0\\
0& 0 & 0 & 1
\end{array}\right)
\\
&g_1(z_1,z_2,z_3,z_4)=(-z_1,z_2,z_3,\bar z_4).\hspace{.1cm}   &h_1(z_1,z_2,z_3,z_4)=(-z_1,z_2,\bar z_3,z_4)\\
&g_2(z_1,z_2,z_3,z_4)=(z_1,-z_2,\bar z_3,z_4).            &h_2(z_1,z_2,z_3,z_4)=(z_1,-z_2,z_3,\bar z_4)\\
&g_3(z_1,z_2,z_3,z_4)=(z_1,z_2,-z_3,z_4).           &h_3(z_1,z_2,z_3,z_4)=(z_1,z_2,-z_3,z_4)\\
&g_4(z_1,z_2,z_3,z_4)=(z_1,z_2,z_3,-z_4).                &h_4(z_1,z_2,z_3,z_4)=(z_1,z_2,z_3,-z_4)
\end{align*}
\normalsize then \footnotesize
\begin{align*}
\Gamma_{13}=
\Bigl < &\left(\begin{array}{cccc}
\frac{1}{2}\\
0\\
0\\
0
\end{array}\right)
\left(\begin{array}{cccc}
1& 0 & 0 & 0\\
0& 1 & 0 & 0\\
0& 0 & 1 & 0\\
0& 0 & 0 & -1
\end{array}\right),
\left(\begin{array}{cccc}
0\\
\frac{1}{2}\\
0\\
0
\end{array}\right)
\left(\begin{array}{cccc}
1& 0 & 0 & 0\\
0& 1 & 0 & 0\\
0& 0 & -1 & 0\\
0& 0 & 0 & 1
\end{array}\right),
\\
&\left(\begin{array}{cccc}
0\\
0\\
\frac{1}{2}\\
0
\end{array}\right)
\left(\begin{array}{cccc}
1& 0 & 0 & 0\\
0& 1 & 0 & 0\\
0& 0 & 1 & 0\\
0& 0 & 0 & 1
\end{array}\right),
\left(\begin{array}{cccc}
0\\
0\\
0\\
\frac{1}{2}
\end{array}\right)
\left(\begin{array}{cccc}
1& 0 & 0 & 0\\
0& 1 & 0 & 0\\
0& 0 & 1 & 0\\
0& 0 & 0 & 1
\end{array}\right)
\Bigr >
\end{align*}

\begin{align*}
\Gamma_{19}= 
\Bigl < & \left(\begin{array}{cccc}
\frac{1}{2}\\
0\\
0\\
0
\end{array}\right)
\left(\begin{array}{cccc}
1& 0 & 0 & 0\\
0& 1 & 0 & 0\\
0& 0 & -1 & 0\\
0& 0 & 0 & 1
\end{array}\right),
\left(\begin{array}{cccc}
0\\
\frac{1}{2}\\
0\\
0
\end{array}\right)
\left(\begin{array}{cccc}
1& 0 & 0 & 0\\
0& 1 & 0 & 0\\
0& 0 & 1 & 0\\
0& 0 & 0 & -1
\end{array}\right),
\\
&\left(\begin{array}{cccc}
0\\
0\\
\frac{1}{2}\\
0
\end{array}\right)
\left(\begin{array}{cccc}
1& 0 & 0 & 0\\
0& 1 & 0 & 0\\
0& 0 & 1 & 0\\
0& 0 & 0 & 1
\end{array}\right),
\left(\begin{array}{cccc}
0\\
0\\
0\\
\frac{1}{2}
\end{array}\right)
\left(\begin{array}{cccc}
1& 0 & 0 & 0\\
0& 1 & 0 & 0\\
0& 0 & 1 & 0\\
0& 0 & 0 & 1
\end{array}\right)
\Bigr >.
\end{align*}

\normalsize
\noindent
Let $\varphi(z_1,z_2,z_3,z_4)=(z_1,z_2,z_4,z_3)$, we get these commutative diagrams

\footnotesize
\[
\begin{CD}
(z_1,z_2,z_3,z_4) @>\varphi>>(z_1,z_2,z_4,z_3)\\
@Vg_1VV @Vh_1 VV\\
(-z_1,z_2,z_3,\bar z_4) @>\varphi>> (-z_1,z_2,\bar z_4,z_3)
\end{CD}
\hspace{.1cm}
\begin{CD}
(z_1,z_2,z_3,z_4) @>\varphi>>(z_1,z_2,z_4,z_3)\\
@Vg_3VV @Vh_4 VV\\
(z_1,z_2,-z_3,z_4) @>\varphi>> (z_1,z_2,z_4,-z_3)
\end{CD}
\]
\[ 
\begin{CD}
(z_1,z_2,z_3,z_4) @>\varphi>>(z_1,z_2,z_4,z_3)\\
@Vg_2VV @Vh_2 VV\\
(z_1,-z_2,\bar z_3,z_4) @>\varphi>> (z_1,-z_2,z_4,\bar z_3)
\end{CD}
\hspace{.1cm}
\begin{CD}
(z_1,z_2,z_3,z_4) @>\varphi>>(z_1,z_2,z_4,z_3)\\
@Vg_4VV @Vh_3 VV\\
(z_1,z_2,z_3,-z_4) @>\varphi>> (z_1,z_2,-z_4,z_3).
\end{CD}
\]

\normalsize
\noindent 
Therefore $\exists$ 
$\gamma 
\footnotesize
=
\Bigl( \left(\begin{array}{cccc}
0\\
0\\
0\\
0
\end{array}\right)
\left(\begin{array}{cccc}
1& 0 & 0 & 0\\
0& 1 & 0 & 0\\
0& 0 & 0 & 1\\
0& 0 & 1 & 0
\end{array}\right)
\Bigr)
\normalsize
\in \mathbb{A}(n)$ 
s.t. $\gamma \Gamma_{13}\gamma ^{-1}=\Gamma_{19}$.\\ 
%---------------------------------------------------------------------------

\noindent 
$\bullet$ $M(A_{13})\approx  M(A_{23})$.\\
For
\footnotesize
\begin{align*}
&A_{13}=
\left(\begin{array}{cccc}
1& 0 & 0 & 1 \\
0& 1 & 1 & 0\\
0& 0 & 1 & 0\\
0& 0 & 0 & 1
\end{array}\right)
\hspace{.1cm}
&A_{23}=
\left(\begin{array}{cccc}
1& 0 & 1 & 1 \\
0& 1 & 0 & 1\\
0& 0 & 1 & 0\\
0& 0 & 0 & 1
\end{array}\right)
\\
&g_1(z_1,z_2,z_3,z_4)=(-z_1,z_2,z_3,\bar z_4).\hspace{.1cm}   &h_1h_2(z_1,z_2,z_3,z_4)=(-z_1,-z_2,\bar z_3,z_4)\\
&g_2(z_1,z_2,z_3,z_4)=(z_1,-z_2,\bar z_3,z_4).            &h_2(z_1,z_2,z_3,z_4)=(z_1,-z_2,z_3,\bar z_4)\\
&g_3(z_1,z_2,z_3,z_4)=(z_1,z_2,-z_3,z_4).           &h_3(z_1,z_2,z_3,z_4)=(z_1,z_2,-z_3,z_4)\\
&g_4(z_1,z_2,z_3,z_4)=(z_1,z_2,z_3,-z_4).                &h_4(z_1,z_2,z_3,z_4)=(z_1,z_2,z_3,-z_4)
\end{align*}
\normalsize then \footnotesize
\begin{align*}
\Gamma_{13}=
\Bigl < &\left(\begin{array}{cccc}
\frac{1}{2}\\
0\\
0\\
0
\end{array}\right)
\left(\begin{array}{cccc}
1& 0 & 0 & 0\\
0& 1 & 0 & 0\\
0& 0 & 1 & 0\\
0& 0 & 0 & -1
\end{array}\right),
\left(\begin{array}{cccc}
0\\
\frac{1}{2}\\
0\\
0
\end{array}\right)
\left(\begin{array}{cccc}
1& 0 & 0 & 0\\
0& 1 & 0 & 0\\
0& 0 & -1 & 0\\
0& 0 & 0 & 1
\end{array}\right),
\\
&\left(\begin{array}{cccc}
0\\
0\\
\frac{1}{2}\\
0
\end{array}\right)
\left(\begin{array}{cccc}
1& 0 & 0 & 0\\
0& 1 & 0 & 0\\
0& 0 & 1 & 0\\
0& 0 & 0 & 1
\end{array}\right),
\left(\begin{array}{cccc}
0\\
0\\
0\\
\frac{1}{2}
\end{array}\right)
\left(\begin{array}{cccc}
1& 0 & 0 & 0\\
0& 1 & 0 & 0\\
0& 0 & 1 & 0\\
0& 0 & 0 & 1
\end{array}\right)
\Bigr >
\end{align*}

\begin{align*}
\Gamma_{23}= 
\Bigl < & \left(\begin{array}{cccc}
\frac{1}{2}\\
\frac{1}{2}\\
0\\
0
\end{array}\right)
\left(\begin{array}{cccc}
1& 0 & 0 & 0\\
0& 1 & 0 & 0\\
0& 0 & -1 & 0\\
0& 0 & 0 & 1
\end{array}\right),
\left(\begin{array}{cccc}
0\\
\frac{1}{2}\\
0\\
0
\end{array}\right)
\left(\begin{array}{cccc}
1& 0 & 0 & 0\\
0& 1 & 0 & 0\\
0& 0 & 1 & 0\\
0& 0 & 0 & -1
\end{array}\right),
\\
&\left(\begin{array}{cccc}
0\\
0\\
\frac{1}{2}\\
0
\end{array}\right)
\left(\begin{array}{cccc}
1& 0 & 0 & 0\\
0& 1 & 0 & 0\\
0& 0 & 1 & 0\\
0& 0 & 0 & 1
\end{array}\right),
\left(\begin{array}{cccc}
0\\
0\\
0\\
\frac{1}{2}
\end{array}\right)
\left(\begin{array}{cccc}
1& 0 & 0 & 0\\
0& 1 & 0 & 0\\
0& 0 & 1 & 0\\
0& 0 & 0 & 1
\end{array}\right)
\Bigr >.
\end{align*}

\normalsize
\noindent
Let $\varphi(z_1,z_2,z_3,z_4)=(z_1,z_1z_2,z_4,z_3)$, we get these commutative diagrams

\footnotesize
\[
\begin{CD}
(z_1,z_2,z_3,z_4) @>\varphi>>(z_1,z_1z_2,z_4,z_3)\\
@Vg_1VV @Vh_1h_2 VV\\
(-z_1,z_2,z_3,\bar z_4) @>\varphi>> (-z_1,-z_1z_2,\bar z_4,z_3)
\end{CD}
\hspace{.1cm}
\begin{CD}
(z_1,z_2,z_3,z_4) @>\varphi>>(z_1,z_1z_2,z_4,z_3)\\
@Vg_3VV @Vh_4 VV\\
(z_1,z_2,-z_3,z_4) @>\varphi>> (z_1,z_1z_2,z_4,-z_3)
\end{CD}
\]
\[ 
\begin{CD}
(z_1,z_2,z_3,z_4) @>\varphi>>(z_1,z_1z_2,z_4,z_3)\\
@Vg_2VV @Vh_2 VV\\
(z_1,-z_2,\bar z_3,z_4) @>\varphi>> (z_1,-z_1z_2,z_4,\bar z_3)
\end{CD}
\hspace{.1cm}
\begin{CD}
(z_1,z_2,z_3,z_4) @>\varphi>>(z_1,z_1z_2,z_4,z_3)\\
@Vg_4VV @Vh_3 VV\\
(z_1,z_2,z_3,-z_4) @>\varphi>> (z_1,z_2,-z_4,z_3).
\end{CD}
\]

\normalsize
\noindent 
Therefore $\exists$ 
$\gamma 
\footnotesize
=
\Bigl( \left(\begin{array}{cccc}
0\\
0\\
0\\
0
\end{array}\right)
\left(\begin{array}{cccc}
1& 0 & 0 & 0\\
1& 1 & 0 & 0\\
0& 0 & 0 & 1\\
0& 0 & 1 & 0
\end{array}\right)
\Bigr)
\normalsize
\in \mathbb{A}(n)$ 
s.t. $\gamma \Gamma_{13}\gamma ^{-1}=\Gamma_{23}$.\\ 
%---------------------------------------------------------------------------

\noindent 
$\bullet$ $M(A_{13})\approx  M(A_{27})$.\\
For
\footnotesize
\begin{align*}
&A_{13}=
\left(\begin{array}{cccc}
1& 0 & 0 & 1 \\
0& 1 & 1 & 0\\
0& 0 & 1 & 0\\
0& 0 & 0 & 1
\end{array}\right)
\hspace{.1cm}
&A_{27}=
\left(\begin{array}{cccc}
1& 0 & 1 & 0 \\
0& 1 & 1 & 1\\
0& 0 & 1 & 0\\
0& 0 & 0 & 1
\end{array}\right)
\\
&g_1(z_1,z_2,z_3,z_4)=(-z_1,z_2,z_3,\bar z_4).\hspace{.1cm}   &h_1(z_1,z_2,z_3,z_4)=(-z_1,z_2,\bar z_3,z_4)\\
&g_2(z_1,z_2,z_3,z_4)=(z_1,-z_2,\bar z_3,z_4).            &h_2h_1(z_1,z_2,z_3,z_4)=(-z_1,-z_2,z_3,\bar z_4)\\
&g_3(z_1,z_2,z_3,z_4)=(z_1,z_2,-z_3,z_4).           &h_3(z_1,z_2,z_3,z_4)=(z_1,z_2,-z_3,z_4)\\
&g_4(z_1,z_2,z_3,z_4)=(z_1,z_2,z_3,-z_4).                &h_4(z_1,z_2,z_3,z_4)=(z_1,z_2,z_3,-z_4)
\end{align*}
\normalsize then \footnotesize
\begin{align*}
\Gamma_{13}=
\Bigl < &\left(\begin{array}{cccc}
\frac{1}{2}\\
0\\
0\\
0
\end{array}\right)
\left(\begin{array}{cccc}
1& 0 & 0 & 0\\
0& 1 & 0 & 0\\
0& 0 & 1 & 0\\
0& 0 & 0 & -1
\end{array}\right),
\left(\begin{array}{cccc}
0\\
\frac{1}{2}\\
0\\
0
\end{array}\right)
\left(\begin{array}{cccc}
1& 0 & 0 & 0\\
0& 1 & 0 & 0\\
0& 0 & -1 & 0\\
0& 0 & 0 & 1
\end{array}\right),
\\
&\left(\begin{array}{cccc}
0\\
0\\
\frac{1}{2}\\
0
\end{array}\right)
\left(\begin{array}{cccc}
1& 0 & 0 & 0\\
0& 1 & 0 & 0\\
0& 0 & 1 & 0\\
0& 0 & 0 & 1
\end{array}\right),
\left(\begin{array}{cccc}
0\\
0\\
0\\
\frac{1}{2}
\end{array}\right)
\left(\begin{array}{cccc}
1& 0 & 0 & 0\\
0& 1 & 0 & 0\\
0& 0 & 1 & 0\\
0& 0 & 0 & 1
\end{array}\right)
\Bigr >
\end{align*}

\begin{align*}
\Gamma_{27}= 
\Bigl < & \left(\begin{array}{cccc}
\frac{1}{2}\\
0\\
0\\
0
\end{array}\right)
\left(\begin{array}{cccc}
1& 0 & 0 & 0\\
0& 1 & 0 & 0\\
0& 0 & -1 & 0\\
0& 0 & 0 & 1
\end{array}\right),
\left(\begin{array}{cccc}
\frac{1}{2}\\
\frac{1}{2}\\
0\\
0
\end{array}\right)
\left(\begin{array}{cccc}
1& 0 & 0 & 0\\
0& 1 & 0 & 0\\
0& 0 & 1 & 0\\
0& 0 & 0 & -1
\end{array}\right),
\\
&\left(\begin{array}{cccc}
0\\
0\\
\frac{1}{2}\\
0
\end{array}\right)
\left(\begin{array}{cccc}
1& 0 & 0 & 0\\
0& 1 & 0 & 0\\
0& 0 & 1 & 0\\
0& 0 & 0 & 1
\end{array}\right),
\left(\begin{array}{cccc}
0\\
0\\
0\\
\frac{1}{2}
\end{array}\right)
\left(\begin{array}{cccc}
1& 0 & 0 & 0\\
0& 1 & 0 & 0\\
0& 0 & 1 & 0\\
0& 0 & 0 & 1
\end{array}\right)
\Bigr >.
\end{align*}

\normalsize
\noindent
Let $\varphi(z_1,z_2,z_3,z_4)=(z_1z_2,z_2,z_4,z_3)$, we get these commutative diagrams

\footnotesize
\[
\begin{CD}
(z_1,z_2,z_3,z_4) @>\varphi>>(z_1z_2,z_2,z_4,z_3)\\
@Vg_1VV @Vh_1 VV\\
(-z_1,z_2,z_3,\bar z_4) @>\varphi>> (-z_1z_2,z_2,\bar z_4,z_3)
\end{CD}
\hspace{.1cm}
\begin{CD}
(z_1,z_2,z_3,z_4) @>\varphi>>(z_1z_2,z_2,z_4,z_3)\\
@Vg_3VV @Vh_4 VV\\
(z_1,z_2,-z_3,z_4) @>\varphi>> (z_1z_2,z_2,z_4,-z_3)
\end{CD}
\]
\[ 
\begin{CD}
(z_1,z_2,z_3,z_4) @>\varphi>>(z_1z_2,z_2,z_4,z_3)\\
@Vg_2VV @Vh_2h_1 VV\\
(z_1,-z_2,\bar z_3,z_4) @>\varphi>> (-z_1z_2,-z_2,z_4,\bar z_3)
\end{CD}
\hspace{.1cm}
\begin{CD}
(z_1,z_2,z_3,z_4) @>\varphi>>(z_1z_2,z_2,z_4,z_3)\\
@Vg_4VV @Vh_3 VV\\
(z_1,z_2,z_3,-z_4) @>\varphi>> (z_1z_2,z_2,-z_4,z_3).
\end{CD}
\]

\normalsize
\noindent 
Therefore $\exists$ 
$\gamma 
\footnotesize
=
\Bigl( \left(\begin{array}{cccc}
0\\
0\\
0\\
0
\end{array}\right)
\left(\begin{array}{cccc}
1& 1 & 0 & 0\\
0& 1 & 0 & 0\\
0& 0 & 0 & 1\\
0& 0 & 1 & 0
\end{array}\right)
\Bigr)
\normalsize
\in \mathbb{A}(n)$ 
s.t. $\gamma \Gamma_{13}\gamma ^{-1}=\Gamma_{27}$.\\ 
%---------------------------------------------------------------------------

\noindent 
$\bullet$ $M(A_{13})\approx  M(A_{29})$.\\
For 
\footnotesize
\begin{align*}
&A_{13}=
\left(\begin{array}{cccc}
1& 0 & 0 & 1 \\
0& 1 & 1 & 0\\
0& 0 & 1 & 0\\
0& 0 & 0 & 1
\end{array}\right)
\hspace{.1cm}
&A_{29}=
\left(\begin{array}{cccc}
1& 0 & 1 & 1 \\
0& 1 & 1 & 0\\
0& 0 & 1 & 0\\
0& 0 & 0 & 1
\end{array}\right)
\\
&g_1(z_1,z_2,z_3,z_4)=(-z_1,z_2,z_3,\bar z_4).\hspace{.1cm}   &h_1h_2(z_1,z_2,z_3,z_4)=(-z_1,-z_2,z_3,\bar z_4)\\
&g_2(z_1,z_2,z_3,z_4)=(z_1,-z_2,\bar z_3,z_4).            &h_2(z_1,z_2,z_3,z_4)=(z_1,-z_2,\bar z_3,z_4)\\
&g_3(z_1,z_2,z_3,z_4)=(z_1,z_2,-z_3,z_4).           &h_3(z_1,z_2,z_3,z_4)=(z_1,z_2,-z_3,z_4)\\
&g_4(z_1,z_2,z_3,z_4)=(z_1,z_2,z_3,-z_4).                &h_4(z_1,z_2,z_3,z_4)=(z_1,z_2,z_3,-z_4)
\end{align*}
\normalsize then \footnotesize
\begin{align*}
\Gamma_{13}=
\Bigl < &\left(\begin{array}{cccc}
\frac{1}{2}\\
0\\
0\\
0
\end{array}\right)
\left(\begin{array}{cccc}
1& 0 & 0 & 0\\
0& 1 & 0 & 0\\
0& 0 & 1 & 0\\
0& 0 & 0 & -1
\end{array}\right),
\left(\begin{array}{cccc}
0\\
\frac{1}{2}\\
0\\
0
\end{array}\right)
\left(\begin{array}{cccc}
1& 0 & 0 & 0\\
0& 1 & 0 & 0\\
0& 0 & -1 & 0\\
0& 0 & 0 & 1
\end{array}\right),
\\
&\left(\begin{array}{cccc}
0\\
0\\
\frac{1}{2}\\
0
\end{array}\right)
\left(\begin{array}{cccc}
1& 0 & 0 & 0\\
0& 1 & 0 & 0\\
0& 0 & 1 & 0\\
0& 0 & 0 & 1
\end{array}\right),
\left(\begin{array}{cccc}
0\\
0\\
0\\
\frac{1}{2}
\end{array}\right)
\left(\begin{array}{cccc}
1& 0 & 0 & 0\\
0& 1 & 0 & 0\\
0& 0 & 1 & 0\\
0& 0 & 0 & 1
\end{array}\right)
\Bigr >
\end{align*}

\begin{align*}
\Gamma_{29}= 
\Bigl < & \left(\begin{array}{cccc}
\frac{1}{2}\\
\frac{1}{2}\\
0\\
0
\end{array}\right)
\left(\begin{array}{cccc}
1& 0 & 0 & 0\\
0& 1 & 0 & 0\\
0& 0 & 1 & 0\\
0& 0 & 0 & -1
\end{array}\right),
\left(\begin{array}{cccc}
0\\
\frac{1}{2}\\
0\\
0
\end{array}\right)
\left(\begin{array}{cccc}
1& 0 & 0 & 0\\
0& 1 & 0 & 0\\
0& 0 & -1 & 0\\
0& 0 & 0 & 1
\end{array}\right),
\\
&\left(\begin{array}{cccc}
0\\
0\\
\frac{1}{2}\\
0
\end{array}\right)
\left(\begin{array}{cccc}
1& 0 & 0 & 0\\
0& 1 & 0 & 0\\
0& 0 & 1 & 0\\
0& 0 & 0 & 1
\end{array}\right),
\left(\begin{array}{cccc}
0\\
0\\
0\\
\frac{1}{2}
\end{array}\right)
\left(\begin{array}{cccc}
1& 0 & 0 & 0\\
0& 1 & 0 & 0\\
0& 0 & 1 & 0\\
0& 0 & 0 & 1
\end{array}\right)
\Bigr >.
\end{align*}

\normalsize
\noindent
Let $\varphi(z_1,z_2,z_3,z_4)=(z_1,z_1z_2,z_3,z_4)$, we get these commutative diagrams

\footnotesize
\[
\begin{CD}
(z_1,z_2,z_3,z_4) @>\varphi>>(z_1,z_1z_2,z_3,z_4)\\
@Vg_1VV @Vh_1h_2 VV\\
(-z_1,z_2,z_3,\bar z_4) @>\varphi>> (-z_1,-z_1z_2,z_3,\bar z_4)
\end{CD}
\hspace{.1cm}
\begin{CD}
(z_1,z_2,z_3,z_4) @>\varphi>>(z_1,z_1z_2,z_3,z_4)\\
@Vg_3VV @Vh_3 VV\\
(z_1,z_2,-z_3,z_4) @>\varphi>> (z_1,z_1z_2,-z_3,z_4)
\end{CD}
\]
\[ 
\begin{CD}
(z_1,z_2,z_3,z_4) @>\varphi>>(z_1,z_1z_2,z_3,z_4)\\
@Vg_2VV @Vh_2 VV\\
(z_1,-z_2,\bar z_3,z_4) @>\varphi>> (z_1,-z_1z_2,\bar z_3,z_4)
\end{CD}
\hspace{.1cm}
\begin{CD}
(z_1,z_2,z_3,z_4) @>\varphi>>(z_1,z_1z_2,z_3,z_4)\\
@Vg_4VV @Vh_4 VV\\
(z_1,z_2,z_3,-z_4) @>\varphi>> (z_1,z_1z_2,z_3,-z_4).
\end{CD}
\]

\normalsize
\noindent 
Therefore $\exists$ 
$\gamma 
\footnotesize
=
\Bigl( \left(\begin{array}{cccc}
0\\
0\\
0\\
0
\end{array}\right)
\left(\begin{array}{cccc}
1& 0 & 0 & 0\\
1& 1 & 0 & 0\\
0& 0 & 1 & 0\\
0& 0 & 0 & 1
\end{array}\right)
\Bigr)
\normalsize
\in \mathbb{A}(n)$ 
s.t. $\gamma \Gamma_{13}\gamma ^{-1}=\Gamma_{29}$.\\ 
%---------------------------------------------------------------------------

\noindent 
$\bullet$ $M(A_{13})\approx  M(A_{a2})$.\\
For 
\footnotesize
\begin{align*}
&A_{13}=
\left(\begin{array}{cccc}
1& 0 & 0 & 1 \\
0& 1 & 1 & 0\\
0& 0 & 1 & 0\\
0& 0 & 0 & 1
\end{array}\right)
\hspace{.1cm}
&A_{a2}=
\left(\begin{array}{cccc}
1& 1 & 0 & 0 \\
0& 1 & 0 & 0\\
0& 0 & 1 & 1\\
0& 0 & 0 & 1
\end{array}\right)
\\
&g_1(z_1,z_2,z_3,z_4)=(-z_1,z_2,z_3,\bar z_4).\hspace{.1cm}   &h_1(z_1,z_2,z_3,z_4)=(-z_1,\bar z_2,z_3,z_4)\\
&g_2(z_1,z_2,z_3,z_4)=(z_1,-z_2,\bar z_3,z_4).            &h_2(z_1,z_2,z_3,z_4)=(z_1,-z_2,z_3,z_4)\\
&g_3(z_1,z_2,z_3,z_4)=(z_1,z_2,-z_3,z_4).           &h_3(z_1,z_2,z_3,z_4)=(z_1,z_2,-z_3,\bar z_4)\\
&g_4(z_1,z_2,z_3,z_4)=(z_1,z_2,z_3,-z_4).                &h_4(z_1,z_2,z_3,z_4)=(z_1,z_2,z_3,-z_4)
\end{align*}
\normalsize then \footnotesize
\begin{align*}
\Gamma_{13}=
\Bigl < &\left(\begin{array}{cccc}
\frac{1}{2}\\
0\\
0\\
0
\end{array}\right)
\left(\begin{array}{cccc}
1& 0 & 0 & 0\\
0& 1 & 0 & 0\\
0& 0 & 1 & 0\\
0& 0 & 0 & -1
\end{array}\right),
\left(\begin{array}{cccc}
0\\
\frac{1}{2}\\
0\\
0
\end{array}\right)
\left(\begin{array}{cccc}
1& 0 & 0 & 0\\
0& 1 & 0 & 0\\
0& 0 & -1 & 0\\
0& 0 & 0 & 1
\end{array}\right),
\\
&\left(\begin{array}{cccc}
0\\
0\\
\frac{1}{2}\\
0
\end{array}\right)
\left(\begin{array}{cccc}
1& 0 & 0 & 0\\
0& 1 & 0 & 0\\
0& 0 & 1 & 0\\
0& 0 & 0 & 1
\end{array}\right),
\left(\begin{array}{cccc}
0\\
0\\
0\\
\frac{1}{2}
\end{array}\right)
\left(\begin{array}{cccc}
1& 0 & 0 & 0\\
0& 1 & 0 & 0\\
0& 0 & 1 & 0\\
0& 0 & 0 & 1
\end{array}\right)
\Bigr >
\end{align*}

\begin{align*}
\Gamma_{a2}= 
\Bigl < & \left(\begin{array}{cccc}
\frac{1}{2}\\
0\\
0\\
0
\end{array}\right)
\left(\begin{array}{cccc}
1& 0 & 0 & 0\\
0& -1 & 0 & 0\\
0& 0 & 1 & 0\\
0& 0 & 0 & 1
\end{array}\right),
\left(\begin{array}{cccc}
0\\
\frac{1}{2}\\
0\\
0
\end{array}\right)
\left(\begin{array}{cccc}
1& 0 & 0 & 0\\
0& 1 & 0 & 0\\
0& 0 & 1 & 0\\
0& 0 & 0 & 1
\end{array}\right),
\\
&\left(\begin{array}{cccc}
0\\
0\\
\frac{1}{2}\\
0
\end{array}\right)
\left(\begin{array}{cccc}
1& 0 & 0 & 0\\
0& 1 & 0 & 0\\
0& 0 & 1 & 0\\
0& 0 & 0 & -1
\end{array}\right),
\left(\begin{array}{cccc}
0\\
0\\
0\\
\frac{1}{2}
\end{array}\right)
\left(\begin{array}{cccc}
1& 0 & 0 & 0\\
0& 1 & 0 & 0\\
0& 0 & 1 & 0\\
0& 0 & 0 & 1
\end{array}\right)
\Bigr >.
\end{align*}

\normalsize
\noindent
Let $\varphi(z_1,z_2,z_3,z_4)=(z_2,z_3,z_1,z_4)$, we get these commutative diagrams

\footnotesize
\[
\begin{CD}
(z_1,z_2,z_3,z_4) @>\varphi>>(z_2,z_3,z_1,z_4)\\
@Vg_1VV @Vh_3 VV\\
(-z_1,z_2,z_3,\bar z_4) @>\varphi>> (z_2,z_3,-z_1,\bar z_4)
\end{CD}
\hspace{.1cm}
\begin{CD}
(z_1,z_2,z_3,z_4) @>\varphi>>(z_2,z_3,z_1,z_4)\\
@Vg_3VV @Vh_2 VV\\
(z_1,z_2,-z_3,z_4) @>\varphi>> (z_2,-z_3,z_1,z_4)
\end{CD}
\]
\[ 
\begin{CD}
(z_1,z_2,z_3,z_4) @>\varphi>>(z_2,z_3,z_1,z_4)\\
@Vg_2VV @Vh_1 VV\\
(z_1,-z_2,\bar z_3,z_4) @>\varphi>> (-z_2,\bar z_3,z_1,z_4)
\end{CD}
\hspace{.1cm}
\begin{CD}
(z_1,z_2,z_3,z_4) @>\varphi>>(z_2,z_3,z_1,z_4)\\
@Vg_4VV @Vh_4 VV\\
(z_1,z_2,z_3,-z_4) @>\varphi>> (z_2,z_3,z_1,-z_4).
\end{CD}
\]

\normalsize
\noindent 
Therefore $\exists$ 
$\gamma 
\footnotesize
=
\Bigl( \left(\begin{array}{cccc}
0\\
0\\
0\\
0
\end{array}\right)
\left(\begin{array}{cccc}
0& 1 & 0 & 0\\
0& 0 & 1 & 0\\
1& 0 & 0 & 0\\
0& 0 & 0 & 1
\end{array}\right)
\Bigr)
\normalsize
\in \mathbb{A}(n)$ 
s.t. $\gamma \Gamma_{13}\gamma ^{-1}=\Gamma_{a2}$.\\ 
%---------------------------------------------------------------------------

\noindent 
$\bullet$ $M(A_{a2})\approx  M(A_{a6})$.\\
For
\footnotesize
\begin{align*}
&A_{a2}=
\left(\begin{array}{cccc}
1& 1 & 0 & 0 \\
0& 1 & 0 & 0\\
0& 0 & 1 & 1\\
0& 0 & 0 & 1
\end{array}\right)
\hspace{.1cm}
&A_{a6}=
\left(\begin{array}{cccc}
1& 1 & 0 & 1 \\
0& 1 & 0 & 0\\
0& 0 & 1 & 1\\
0& 0 & 0 & 1
\end{array}\right)
\\
&g_1(z_1,z_2,z_3,z_4)=(-z_1,\bar z_2,z_3,z_4).\hspace{.1cm}   &h_1h_3(z_1,z_2,z_3,z_4)=(-z_1,\bar z_2,-z_3,z_4)\\
&g_2(z_1,z_2,z_3,z_4)=(z_1,-z_2,z_3,z_4).            &h_2(z_1,z_2,z_3,z_4)=(z_1,-z_2,z_3,z_4)\\
&g_3(z_1,z_2,z_3,z_4)=(z_1,z_2,-z_3,\bar z_4).           &h_3(z_1,z_2,z_3,z_4)=(z_1,z_2,-z_3,\bar z_4)\\
&g_4(z_1,z_2,z_3,z_4)=(z_1,z_2,z_3,-z_4).                &h_4(z_1,z_2,z_3,z_4)=(z_1,z_2,z_3,-z_4)
\end{align*}
\normalsize then \footnotesize
\begin{align*}
\Gamma_{a2}= 
\Bigl < & \left(\begin{array}{cccc}
\frac{1}{2}\\
0\\
0\\
0
\end{array}\right)
\left(\begin{array}{cccc}
1& 0 & 0 & 0\\
0& -1 & 0 & 0\\
0& 0 & 1 & 0\\
0& 0 & 0 & 1
\end{array}\right),
\left(\begin{array}{cccc}
0\\
\frac{1}{2}\\
0\\
0
\end{array}\right)
\left(\begin{array}{cccc}
1& 0 & 0 & 0\\
0& 1 & 0 & 0\\
0& 0 & 1 & 0\\
0& 0 & 0 & 1
\end{array}\right),
\\
&\left(\begin{array}{cccc}
0\\
0\\
\frac{1}{2}\\
0
\end{array}\right)
\left(\begin{array}{cccc}
1& 0 & 0 & 0\\
0& 1 & 0 & 0\\
0& 0 & 1 & 0\\
0& 0 & 0 & -1
\end{array}\right),
\left(\begin{array}{cccc}
0\\
0\\
0\\
\frac{1}{2}
\end{array}\right)
\left(\begin{array}{cccc}
1& 0 & 0 & 0\\
0& 1 & 0 & 0\\
0& 0 & 1 & 0\\
0& 0 & 0 & 1
\end{array}\right)
\Bigr >
\end{align*}

\begin{align*}
\Gamma_{a6}=
\Bigl < &\left(\begin{array}{cccc}
\frac{1}{2}\\
0\\
\frac{1}{2}\\
0
\end{array}\right)
\left(\begin{array}{cccc}
1& 0 & 0 & 0\\
0& -1 & 0 & 0\\
0& 0 & 1 & 0\\
0& 0 & 0 & 1
\end{array}\right),
\left(\begin{array}{cccc}
0\\
\frac{1}{2}\\
0\\
0
\end{array}\right)
\left(\begin{array}{cccc}
1& 0 & 0 & 0\\
0& 1 & 0 & 0\\
0& 0 & 1 & 0\\
0& 0 & 0 & 1
\end{array}\right),
\\
&\left(\begin{array}{cccc}
0\\
0\\
\frac{1}{2}\\
0
\end{array}\right)
\left(\begin{array}{cccc}
1& 0 & 0 & 0\\
0& 1 & 0 & 0\\
0& 0 & 1 & 0\\
0& 0 & 0 & -1
\end{array}\right),
\left(\begin{array}{cccc}
0\\
0\\
0\\
\frac{1}{2}
\end{array}\right)
\left(\begin{array}{cccc}
1& 0 & 0 & 0\\
0& 1 & 0 & 0\\
0& 0 & 1 & 0\\
0& 0 & 0 & 1
\end{array}\right)
\Bigr >.
\end{align*}

\normalsize
\noindent
Let $\varphi(z_1,z_2,z_3,z_4)=(z_1,z_2,z_1z_3,z_4)$, we get these commutative diagrams

\footnotesize
\[
\begin{CD}
(z_1,z_2,z_3,z_4) @>\varphi>>(z_1,z_2,z_1z_3,z_4)\\
@Vg_1VV @Vh_1h_3 VV\\
(-z_1,\bar z_2,z_3,z_4) @>\varphi>> (-z_1,\bar z_2,-z_1z_3,z_4)
\end{CD}
\hspace{.1cm}
\begin{CD}
(z_1,z_2,z_3,z_4) @>\varphi>>(z_1,z_2,z_1z_3,z_4)\\
@Vg_3VV @Vh_3 VV\\
(z_1,z_2,-z_3,\bar z_4) @>\varphi>> (z_1,z_2,-z_1z_3,\bar z_4)
\end{CD}
\]
\[ 
\begin{CD}
(z_1,z_2,z_3,z_4) @>\varphi>>(z_1,z_2,z_1z_3,z_4)\\
@Vg_2VV @Vh_2 VV\\
(z_1,-z_2,z_3,z_4) @>\varphi>> (z_1,-z_2,z_1z_3,z_4)
\end{CD}
\hspace{.1cm}
\begin{CD}
(z_1,z_2,z_3,z_4) @>\varphi>>(z_1,z_2,z_1z_3,z_4)\\
@Vg_4VV @Vh_4 VV\\
(z_1,z_2,z_3,-z_4) @>\varphi>> (z_1,z_2,z_1z_3,-z_4).
\end{CD}
\]

\normalsize
\noindent 
Therefore $\exists$ 
$\gamma 
\footnotesize
=
\Bigl( \left(\begin{array}{cccc}
0\\
0\\
0\\
0
\end{array}\right)
\left(\begin{array}{cccc}
1& 0 & 0 & 0\\
0& 1 & 0 & 0\\
1& 0 & 1 & 0\\
0& 0 & 0 & 1
\end{array}\right)
\Bigr)
\normalsize
\in \mathbb{A}(n)$ 
s.t. $\gamma \Gamma_{a2}\gamma ^{-1}=\Gamma_{a6}$.\\ 

%============================================================================
\noindent {\bf (24)}. Similar to (4).\\

%============================================================================
\noindent {\bf (25)}. Similar to (11).\\

%============================================================================
\noindent {\bf (26)}. $M(A_{13})$ is not diffeomorphic to $M(A_{a3})$.\\

Let \footnotesize
\begin{align*}
\Gamma_{a3}= <&\tilde {g}_1,\tilde {g}_2,t_3,t_4>\\
=\Bigl < &\left(\begin{array}{cccc}
\frac{1}{2}\\
0\\
0\\
0
\end{array}\right)
\left(\begin{array}{cccc}
1& 0 & 0 & 0\\
0& -1 & 0 & 0\\
0& 0 & 1 & 0\\
0& 0 & 0 & 1
\end{array}\right),
\left(\begin{array}{cccc}
0\\
\frac{1}{2}\\
0\\
0
\end{array}\right)
\left(\begin{array}{cccc}
1& 0 & 0 & 0\\
0& 1 & 0 & 0\\
0& 0 & 1 & 0\\
0& 0 & 0 & -1
\end{array}\right),
\\
&\left(\begin{array}{cccc}
0\\
0\\
\frac{1}{2}\\
0
\end{array}\right)
\left(\begin{array}{cccc}
1& 0 & 0 & 0\\
0& 1 & 0 & 0\\
0& 0 & 1 & 0\\
0& 0 & 0 & 1
\end{array}\right),
\left(\begin{array}{cccc}
0\\
0\\
0\\
\frac{1}{2}
\end{array}\right)
\left(\begin{array}{cccc}
1& 0 & 0 & 0\\
0& 1 & 0 & 0\\
0& 0 & 1 & 0\\
0& 0 & 0 & 1
\end{array}\right)
\Bigr >
\\
\Gamma_{13}= <&\tilde {h}_1,\tilde {h}_2,t_3,t_4>\\
=\Bigl < & \left(\begin{array}{cccc}
\frac{1}{2}\\
0\\
0\\
0
\end{array}\right)
\left(\begin{array}{cccc}
1& 0 & 0 & 0\\
0& 1 & 0 & 0\\
0& 0 & 1 & 0\\
0& 0 & 0 & -1
\end{array}\right),
\left(\begin{array}{cccc}
0\\
\frac{1}{2}\\
0\\
0
\end{array}\right)
\left(\begin{array}{cccc}
1& 0 & 0 & 0\\
0& 1 & 0 & 0\\
0& 0 & -1 & 0\\
0& 0 & 0 & 1
\end{array}\right),
\\
&\left(\begin{array}{cccc}
0\\
0\\
\frac{1}{2}\\
0
\end{array}\right)
\left(\begin{array}{cccc}
1& 0 & 0 & 0\\
0& 1 & 0 & 0\\
0& 0 & 1 & 0\\
0& 0 & 0 & 1
\end{array}\right),
\left(\begin{array}{cccc}
0\\
0\\
0\\
\frac{1}{2}
\end{array}\right)
\left(\begin{array}{cccc}
1& 0 & 0 & 0\\
0& 1 & 0 & 0\\
0& 0 & 1 & 0\\
0& 0 & 0 & 1
\end{array}\right)
\Bigr >.
\end{align*}

\normalsize
We know from (15) above, that the center $\displaystyle \mathcal C(\Gamma_{a3})= \langle t_1,
t_3\rangle$ where $t_1=\tilde {g}_1^2$. There is an extension
\begin{align*}
1\rightarrow  \langle t_1,t_3\rangle \rightarrow  \Gamma_{a3}\rightarrow  \Delta_{a3}\rightarrow  1,
\end{align*}
where $\Delta_{a3}$
is isomorphic to
\begin{align*}
\Bigl <&\gamma_1=\left(\begin{array}{c}
\frac{1}{2}\\
0
\end{array}\right)
\left(\begin{array}{cc}
1& 0 \\
0& -1
\end{array}\right), \gamma_2=\left(\begin{array}{c}
0\\
\frac{1}{2}
\end{array}\right) \left(\begin{array}{cc}
1& 0 \\
0& 1
\end{array}\right)
\Bigr >\rtimes \langle \beta \rangle
\end{align*}
where
$\beta =\Bigl(\left(\begin{array}{c}
0\\
0
\end{array}\right)
\left(\begin{array}{ccc}
-1& 0 \\
0& 1
\end{array}\right)\Bigr)$.

For $\Gamma_{13}$, since 
%$\tilde {h}_1\tilde {h}_2\tilde {h}_1^{-1}=\tilde {h}_2$,
 $\tilde {h}_1 t_3\tilde {h}_1^{-1}=t_3$, 
$\tilde {h}_1 t_4\tilde {h}_1^{-1}=t_4^{-1}$, 
%$h_2\tilde {h}_1 h_2^{-1}=\tilde {h}_1$, 
 $h_2 t_3 h_2^{-1}=t_3^{-1}$ and
$h_2 t_4 h_2^{-1}=t_4$, then $\Gamma_{13}=\langle t_3,t_4 \rangle \rtimes \langle \tilde {h}_1,\tilde {h}_2 \rangle$.
Then center $\mathcal C(\Gamma_{13})= \langle t_1,t_2\rangle$.
It induces an extension
\begin{align*}
1\rightarrow \langle t_1,t_2\rangle \rightarrow  \Gamma_{13}\rightarrow  \Delta_{13}\rightarrow  1,
\end{align*}
where $\Delta_{13}$ is isomorphic to
\begin{align*}
\Bigl <& s_1=\left(\begin{array}{c}
\frac{1}{2}\\
0
\end{array}\right)
\left(\begin{array}{ccc}
1& 0 \\
0& 1
\end{array}\right), s_2=\left(\begin{array}{c}
0\\
\frac{1}{2}
\end{array}\right)
\left(\begin{array}{ccc}
1& 0\\
0& 1
\end{array}\right)
\Bigr >\rtimes \langle \alpha ,\beta \rangle
\end{align*}
with $\alpha =\Bigr(\left(\begin{array}{c}
0\\
0
\end{array}\right)
\left(\begin{array}{ccc}
1& 0 \\
0 & -1
\end{array}\right)\Bigr),\beta =\Bigl(
\left(\begin{array}{c}
0\\
0
\end{array}\right)
\left(\begin{array}{ccc}
-1& 0 \\
0 &1
\end{array}\right)\Bigr)$.
\\

Suppose that $\Gamma_{a3}$ is isomorphic to $\Gamma_{13}$ with isomorphism  $\varphi :\Gamma_{a3}\rightarrow \Gamma_{13}$
which induces an isomorphism $\hat \varphi : \Delta _{a3}\rightarrow \Delta _{13}$
\begin{align*}
\begin{CD}
1 @>>>  \langle t_1,t_3 \rangle @>>> \Gamma_{a3}@>>>  \Delta_{a3}@>>>  1\\
  @.                @.                @V \varphi VV           @V \hat \varphi VV\\
1 @>>>  \langle t_1,t_2 \rangle @>>> \Gamma_{13}@>>>  \Delta_{13}@>>>  1.
\end{CD}
\end{align*}
By the same argument as (15) above, 
consider $\hat \varphi (\gamma_1^{a_1} \gamma_2^{a_2} \beta )=\alpha $ and $\hat \varphi (\gamma_1^{b_1} \gamma_2^{b_2} \beta )=\beta  $ 
where $a_1, a_2, b_1, b_2 \in \mathbb{Z}$. 
Then 
\begin{align*}
\alpha\beta=&\hat \varphi (\gamma_1^{a_1} \gamma_2^{a_2} \beta  \gamma_1^{b_1} \gamma_2^{b_2} \beta)  \\
      =&\hat \varphi (\gamma_1^{a_1} \gamma_2^{a_2} \gamma_1^{-b_1} \beta  \gamma_2^{b_2} \beta)  \hspace{.5cm} (\beta \gamma_1 \beta^{-1}=\gamma_1^{-1})\\
      =&\hat \varphi (\gamma_1^{a_1} \gamma_2^{a_2} \gamma_1^{-b_1} \gamma_2^{b_2}) \hspace{.9cm} (\beta \gamma_2 \beta^{-1}=\gamma_2).
\end{align*}
Since $\alpha \beta $ is a torsion element and $\hat \varphi $ is an isomorphism, then 
$\gamma_1^{a_1} \gamma_2^{a_2} \gamma_1^{-b_1} \gamma_2^{b_2}$ is also a torsion element.
Let
\begin{align*}
g=&\gamma_1^{a_1} \gamma_2^{a_2} \gamma_1^{-b_1} \gamma_2^{b_2} \\
 =&
\Bigl ( \left(\begin{array}{c}
\frac{1}{2}a_1\\
(-1)^{a_1}\frac{1}{2}a_2
\end{array}\right)
\left(\begin{array}{cc}
1& 0\\
0& (-1)^{a_1}\end{array}\right)\Bigr)
\Bigl ( \left(\begin{array}{c}
-\frac{1}{2}b_1\\
(-1)^{-b_1}\frac{1}{2}b_2
\end{array}\right)
\left(\begin{array}{cc}
1& 0\\
0& (-1)^{-b_1}\end{array}\right)\Bigr)\\
=&
\Bigl ( \left(\begin{array}{c}
\frac{1}{2}(a_1 - b_1)\\
\frac{1}{2}((-1)^{a_1}a_2+(-1)^{(a_1-b_1)}b_2)
\end{array}\right)
\left(\begin{array}{cc}
1& 0\\
0& (-1)^{a_1-b_1}\end{array}\right)\Bigr).
\end{align*}
We have
\begin{align*}
g^2=&\Bigl ( \left(\begin{array}{c}
a_1 - b_1\\
(1+(-1)^{(a_1-b_1)})\frac{1}{2}((-1)^{a_1}a_2+(-1)^{(a_1-b_1)}b_2
\end{array}\right)
\left(\begin{array}{cc}
1& 0\\
0& 1 \end{array}\right)\Bigr)\\
=&
\Bigl ( \left(\begin{array}{c}
0\\
0
\end{array}\right)
\left(\begin{array}{cc}
1& 0\\
0& 1 \end{array}\right)\Bigr).
\end{align*}
Therefore we get $a_1=b_1$ and $(-1)^{a_1}a_2+b_2=0$.
Hence $\hat \varphi (g)=\hat \varphi 
\Bigl ( \left(\begin{array}{c}
0\\
0
\end{array}\right)
\left(\begin{array}{cc}
1& 0\\
0& 1 \end{array}\right)\Bigr)=
\Bigl ( \left(\begin{array}{c}
0\\
0
\end{array}\right)
\left(\begin{array}{cc}
-1& 0\\
0& -1 \end{array}\right)\Bigr)$.
This yields a contradiction.
\\

%============================================================================
\noindent {\bf (27)}. $M(A_{13})$ is not diffeomorphic to $M(A_{a4})$.\\

Let \footnotesize
\begin{align*}
\Gamma_{a4}= <&\tilde {g}_1,\tilde {g}_2,\tilde {g}_3,t_4>\\
=\Bigl < &\left(\begin{array}{cccc}
\frac{1}{2}\\
0\\
0\\
0
\end{array}\right)
\left(\begin{array}{cccc}
1& 0 & 0 & 0\\
0& -1 & 0 & 0\\
0& 0 & 1 & 0\\
0& 0 & 0 & 1
\end{array}\right),
\left(\begin{array}{cccc}
0\\
\frac{1}{2}\\
0\\
0
\end{array}\right)
\left(\begin{array}{cccc}
1& 0 & 0 & 0\\
0& 1 & 0 & 0\\
0& 0 & 1 & 0\\
0& 0 & 0 & -1
\end{array}\right),
\\
&\left(\begin{array}{cccc}
0\\
0\\
\frac{1}{2}\\
0
\end{array}\right)
\left(\begin{array}{cccc}
1& 0 & 0 & 0\\
0& 1 & 0 & 0\\
0& 0 & 1 & 0\\
0& 0 & 0 & -1
\end{array}\right),
\left(\begin{array}{cccc}
0\\
0\\
0\\
\frac{1}{2}
\end{array}\right)
\left(\begin{array}{cccc}
1& 0 & 0 & 0\\
0& 1 & 0 & 0\\
0& 0 & 1 & 0\\
0& 0 & 0 & 1
\end{array}\right)
\Bigr >\\
%=<&\tilde {g}_1,\tilde {g}_2\tilde {g}_3,\tilde {g}_3,t_4>\\
%=\Bigl < &\left(\begin{array}{cccc}
%\frac{1}{2}\\
%0\\
%0\\
%0
%\end{array}\right)
%\left(\begin{array}{cccc}
%1& 0 & 0 & 0\\
%0& -1 & 0 & 0\\
%0& 0 & 1 & 0\\
%0& 0 & 0 & 1
%\end{array}\right),
%\left(\begin{array}{cccc}
%0\\
%\frac{1}{2}\\
%\frac{1}{2}\\
%0
%\end{array}\right)
%\left(\begin{array}{cccc}
%1& 0 & 0 & 0\\
%0& 1 & 0 & 0\\
%0& 0 & 1 & 0\\
%0& 0 & 0 & 1
%\end{array}\right),
%\\
%&\left(\begin{array}{cccc}
%0\\
%0\\
%\frac{1}{2}\\
%0
%\end{array}\right)
%\left(\begin{array}{cccc}
%1& 0 & 0 & 0\\
%0& 1 & 0 & 0\\
%0& 0 & 1 & 0\\
%0& 0 & 0 & -1
%\end{array}\right),
%\left(\begin{array}{cccc}
%0\\
%0\\
%0\\
%\frac{1}{2}
%\end{array}\right)
%\left(\begin{array}{cccc}
%1& 0 & 0 & 0\\
%0& 1 & 0 & 0\\
%0& 0 & 1 & 0\\
%0& 0 & 0 & 1
%\end{array}\right)
%\Bigr >
\\
\Gamma_{13}= <&\tilde {h}_1,\tilde {h}_2,t_3,t_4>\\
=\Bigl < & \left(\begin{array}{cccc}
\frac{1}{2}\\
0\\
0\\
0
\end{array}\right)
\left(\begin{array}{cccc}
1& 0 & 0 & 0\\
0& 1 & 0 & 0\\
0& 0 & 1 & 0\\
0& 0 & 0 & -1
\end{array}\right),
\left(\begin{array}{cccc}
0\\
\frac{1}{2}\\
0\\
0
\end{array}\right)
\left(\begin{array}{cccc}
1& 0 & 0 & 0\\
0& 1 & 0 & 0\\
0& 0 & -1 & 0\\
0& 0 & 0 & 1
\end{array}\right),
\\
&\left(\begin{array}{cccc}
0\\
0\\
\frac{1}{2}\\
0
\end{array}\right)
\left(\begin{array}{cccc}
1& 0 & 0 & 0\\
0& 1 & 0 & 0\\
0& 0 & 1 & 0\\
0& 0 & 0 & 1
\end{array}\right),
\left(\begin{array}{cccc}
0\\
0\\
0\\
\frac{1}{2}
\end{array}\right)
\left(\begin{array}{cccc}
1& 0 & 0 & 0\\
0& 1 & 0 & 0\\
0& 0 & 1 & 0\\
0& 0 & 0 & 1
\end{array}\right)
\Bigr >.
\end{align*}

\normalsize
Since $\tilde {g}_1\tilde {g}_2\tilde {g}_1^{-1}=\tilde {g}_2^{-1}$, 
%$\tilde {g}_1 \tilde {g}_3\tilde {g}_1^{-1}=\tilde {g}_3$, 
$\tilde {g}_1 t_4\tilde {g}_1^{-1}=t_4$, 
%$\tilde {g}_3\tilde {g}_1\tilde {g}_3^{-1}=\tilde {g}_1$, 
$\tilde {g}_3 \tilde {g}_2\tilde {g}_3^{-1}=\tilde {g}_2$, 
$\tilde {g}_3 t_4\tilde {g}_3^{-1}=t_4^{-1}$, then $\Gamma_{a4}=\langle \tilde {g}_2,t_4 \rangle \rtimes \langle \tilde {g}_1,\tilde {g}_3 \rangle$.
Then the center $\displaystyle \mathcal C(\Gamma_{a4})= \langle t_1,
t_3\rangle$ where $t_1=\tilde {g}_1^2$ and $t_3=\tilde {g}_3^2$. There is an extension
\begin{align*}
1\rightarrow  \langle t_1,t_3\rangle \rightarrow  \Gamma_{a4}\rightarrow  \Delta_{a4}\rightarrow  1,
\end{align*}
where $\Delta_{a4}$ is isomorphic to
\begin{align*}
\Bigl <&\gamma _1=\left(\begin{array}{c}
\frac{1}{2}\\
0
\end{array}\right)
\left(\begin{array}{cc}
1& 0 \\
0& -1
\end{array}\right), \gamma _2=\left(\begin{array}{c}
0\\
\frac{1}{2}
\end{array}\right) \left(\begin{array}{cc}
1& 0 \\
0& 1
\end{array}\right)
\Bigr > \rtimes \langle \alpha_1,\alpha_2 \rangle
\end{align*}
with
$\alpha_1=\Bigl(\left(\begin{array}{c}
0\\
0
\end{array}\right)
\left(\begin{array}{ccc}
-1& 0 \\
0& 1
\end{array}\right)\Bigr)$ and
$\alpha_2=\Bigl(\left(\begin{array}{c}
0\\
0
\end{array}\right)
\left(\begin{array}{ccc}
1& 0 \\
0& -1
\end{array}\right)\Bigr)$.

For $\Gamma_{13}$, the center $\mathcal C(\Gamma_{13})= \langle t_1,t_2\rangle$.
It induces an extension
\begin{align*}
1\rightarrow \langle t_1,t_2\rangle \rightarrow  \Gamma_{13}\rightarrow  \Delta_{13}\rightarrow  1,
\end{align*}
where $\Delta_{13}$ is isomorphic to
\begin{align*}
\Bigl <& s_1=\left(\begin{array}{c}
\frac{1}{2}\\
0
\end{array}\right)
\left(\begin{array}{ccc}
1& 0 \\
0& 1
\end{array}\right), s_2=\left(\begin{array}{c}
0\\
\frac{1}{2}
\end{array}\right)
\left(\begin{array}{ccc}
1& 0\\
0& 1
\end{array}\right)
\Bigr >\rtimes \langle \alpha ,\beta \rangle
\end{align*}
with $\alpha=\alpha_1$ and $\beta=\alpha_2$.

Suppose that $\Gamma_{13}$ is isomorphic to $\Gamma_{a4}$ with isomorphism  $\varphi :\Gamma_{13}\rightarrow \Gamma_{a4}$
which induces an isomorphism $\hat \varphi : \Delta _{13}\rightarrow \Delta _{a4}$
\begin{align*}
\begin{CD}
1 @>>>  \langle t_1,t_3 \rangle @>>> \Gamma_{a4}@>>>  \Delta_{a4}@>>>  1\\
  @.                @.                @A \varphi AA           @A \hat \varphi AA\\
1 @>>>  \langle t_1,t_2 \rangle @>>> \Gamma_{13}@>>>  \Delta_{13}@>>>  1.
\end{CD}
\end{align*}

Since $\ti g_2 \ti g_3=\ti g_3 \ti g_2$ then
\begin{align*}
\ti g_2^{-1} \ti g_3 &=\ti g_3^{-1} \ti g_3 \ti g_2^{-1} \ti g_3\\
			   &=\ti g_3^{-1} \ti g_2^{-1} \ti g_3^{2}\\
			   &=(\ti g_2 \ti g_3)^{-1}\ti g_3^{2}\\
\ti g_2^{-1}\ti g_1 \ti g_3 \ti g_1^{-1} &=(\ti g_2 \ti g_3)^{-1}\ti g_3^{2}.
\end{align*}
We get 
\begin{equation}\label{Eq1}
\tilde g_1(\tilde g_2\tilde g_3)\tilde {g_1}^{-1}= (\tilde g_2\tilde g_3)^{-1} \tilde g_3^{2}.
\end{equation}

Now we want to find the elements $h,h'\in \GA_{13}$ such that $\f(h)=\ti g_1$ and $\f(h')=\tilde g_2\tilde g_3$.
According to the above diagram, let us consider the following diagram
\[
\begin{CD}
\footnotesize \biggl ( \left(\begin{array}{cccc}
0\\
\frac{1}{2}\\
\frac{1}{2}\\
0
\end{array}\right),I \biggr)= \tilde g_2\tilde g_3 @>>> 
\gamma_1\alpha_2=
\biggl ( \left(\begin{array}{cc}
\frac{1}{2}\\
0
\end{array}\right),I \biggr)\\
@A \varphi AA           @A \hat \varphi AA\\
\biggl ( \left(\begin{array}{cccc}
c\\
d\\
\frac{1}{2}a\\
\frac{1}{2}b
\end{array}\right),I \biggr)= {t_3}^a{t_4}^b{t_1}^c{t_2}^d @>>> 
{s_1}^a {s_2}^b=
\biggl ( \left(\begin{array}{cc}
\frac{1}{2}a\\
\frac{1}{2}b
\end{array}\right),I \biggr)
\end{CD}
\]
so,
\begin{equation}\label{Eq2}
\tilde g_2\tilde g_3=\varphi ({t_3}^a{t_4}^b{t_1}^c{t_2}^d ),
\end{equation}
and from the diagrams
\[
\begin{CD}
\footnotesize \biggl ( \left(\begin{array}{cccc}
\frac{1}{2}\\
0\\
0\\
0
\end{array}\right),\left(\begin{array}{cccc}
1& 0 & 0 & 0\\
0& -1 & 0 & 0\\
0& 0 & 1 & 0\\
0& 0 & 0 & 1
\end{array}\right) \biggr)= \tilde g_1 @>>> 
\alpha_1=
\biggl ( \left(\begin{array}{cc}
0\\
0
\end{array}\right),\left(\begin{array}{cc}
-1& 0 \\
0& 1  
\end{array}\right) \biggr)\\
@A \varphi AA           @A \hat \varphi AA\\
{\tilde h_1}{t_3}^{a'}{t_4}^{b'}{t_1}^{c'}{t_2}^{d'}= @>>> \beta{s_1}^{a'} {s_2}^{b'}=\\
%@. @.\\
\biggl ( \left(\begin{array}{cccc}
c'+\frac{1}{2}\\
d'\\
\frac{1}{2}{a'}\\
-\frac{1}{2}{b'}
\end{array}\right),\left(\begin{array}{cccc}
1& 0 & 0 & 0\\
0& 1 & 0 & 0\\
0& 0 & 1 & 0\\
0& 0 & 0 & -1
\end{array}\right) \biggr) @.
%\hspace {2cm}
\biggl ( \left(\begin{array}{cc}
\frac{1}{2}a'\\
-\frac{1}{2}b'
\end{array}\right),\left(\begin{array}{cc}
1& 0 \\
0& -1  
\end{array}\right) \biggr)
\end{CD}
\] 
or
\[
\begin{CD}
\footnotesize \biggl ( \left(\begin{array}{cccc}
\frac{1}{2}\\
0\\
0\\
0
\end{array}\right),\left(\begin{array}{cccc}
1& 0 & 0 & 0\\
0& -1 & 0 & 0\\
0& 0 & 1 & 0\\
0& 0 & 0 & 1
\end{array}\right) \biggr)= \tilde g_1 @>>> 
\alpha_1=
\biggl ( \left(\begin{array}{cc}
0\\
0
\end{array}\right),\left(\begin{array}{cc}
-1& 0 \\
0& 1  
\end{array}\right) \biggr)\\
@A \varphi AA           @A \hat \varphi AA\\
{\tilde h_2}{t_3}^{a'}{t_4}^{b'}{t_1}^{c'}{t_2}^{d'}= @>>> \al {s_1}^{a'} {s_2}^{b'}=\\
%@. @.\\
\biggl ( \left(\begin{array}{cccc}
c'\\
d'+\frac{1}{2}\\
-\frac{1}{2}{a'}\\
\frac{1}{2}{b'}
\end{array}\right),\left(\begin{array}{cccc}
1& 0 & 0 & 0\\
0& 1 & 0 & 0\\
0& 0 & -1 & 0\\
0& 0 & 0 & 1
\end{array}\right) \biggr) @.
%\hspace {2cm}
\biggl ( \left(\begin{array}{cc}
-\frac{1}{2}a'\\
\frac{1}{2}b'
\end{array}\right),\left(\begin{array}{cc}
-1& 0 \\
0& 1  
\end{array}\right) \biggr)
\end{CD}
\]
we get
\begin{equation}\label{Eq3}
\tilde {g_1}=\varphi (\tilde h_1{t_3}^{a'}{t_4}^{b'}{t_1}^{c'}{t_2}^{d'} )=\varphi (\tilde h_1t),
\end{equation}
or
\begin{equation}\label{Eq3b}
\tilde {g_1}=\varphi (\tilde h_2{t_3}^{a'}{t_4}^{b'}{t_1}^{c'}{t_2}^{d'} )=\varphi (\tilde h_2t),
\end{equation}
where $t={t_3}^{a'}{t_4}^{b'}{t_1}^{c'}{t_2}^{d'}$.

Next we get
\begin{align*}
\tilde g_1 (\tilde g_2\tilde g_3)\tilde {g_1}^{-1}=&\varphi (\tilde h_1 t)\varphi ({t_3}^{a}{t_4}^{b}{t_1}^{c}{t_2}^{d} )\varphi (\tilde h_1 t)^{-1}
									 \hspace{.5cm} (\text {by} \, \eqref{Eq2},\, \eqref{Eq3})\\
(\tilde g_2\tilde g_3)^{-1}t_3=&\varphi (\tilde h_1 t ({t_3}^{a}{t_4}^{b}{t_1}^{c}{t_2}^{d})t^{-1}\tilde h_1^{-1})
									 \hspace{1.1cm} (\text {by} \, \eqref{Eq1})\\
					=&\varphi (\tilde h_1 ({t_3}^{a}{t_4}^{b}{t_1}^{c}{t_2}^{d})\tilde h_1^{-1})\\
					=&\varphi (\tilde h_1 {t_3}^{a}{t_4}^{b}\tilde h_1^{-1}{t_1}^{c}{t_2}^{d})\\
					=&\varphi (({t_3}^{a}{t_4}^{b})^{-1}t_3^{2a}{t_1}^{c}{t_2}^{d}) 
									\hspace{1.5cm} (h_1 {t_3}^{a}{t_4}^{b}h_1^{-1}={t_3}^{a}{t_4}^{-b})\\
					=&\varphi (({t_3}^{a}{t_4}^{b})^{-1}({t_1}^{c}{t_2}^{d})t_3^{2a})\\
					=&\varphi (({t_3}^{a}{t_4}^{b})^{-1}({t_1}^{c}{t_2}^{d})^{-1}(({t_1}^{c}{t_2}^{d})t_3^{a})^2)\\
					=&\varphi (({t_3}^{a}{t_4}^{b}{t_1}^{c}{t_2}^{d})^{-1}({t_1}^{c}{t_2}^{d}t_3^{a})^2)\\
					=&(\tilde g_2\tilde g_3)^{-1}\varphi (({t_1}^{c}{t_2}^{d}t_3^{a})^2),
\end{align*}
or
\begin{align*}
\tilde g_1 (\tilde g_2\tilde g_3)\tilde {g_1}^{-1}=&\varphi (\tilde h_2 t)\varphi ({t_3}^{a}{t_4}^{b}{t_1}^{c}{t_2}^{d} )\varphi (\tilde h_2 t)^{-1}
									 \hspace{.5cm} (\text {by} \, \eqref{Eq2},\, \eqref{Eq3b})\\
(\tilde g_2\tilde g_3)^{-1}t_3=&\varphi (\tilde h_2 t ({t_3}^{a}{t_4}^{b}{t_1}^{c}{t_2}^{d})t^{-1}\tilde h_2^{-1})
									 \hspace{1.1cm} (\text {by} \, \eqref{Eq1})\\
					=&\varphi (\tilde h_2 ({t_3}^{a}{t_4}^{b}{t_1}^{c}{t_2}^{d})\tilde h_2^{-1})\\
					=&\varphi (\tilde h_2 {t_3}^{a}{t_4}^{b}\tilde h_1^{-1}{t_1}^{c}{t_2}^{d})\\
					=&\varphi (({t_3}^{a}{t_4}^{b})^{-1}t_4^{2b}{t_1}^{c}{t_2}^{d}) 
									\hspace{1.5cm} (h_2 {t_3}^{a}{t_4}^{b}h_2^{-1}={t_3}^{-a}{t_4}^{b})\\
					=&\varphi (({t_3}^{a}{t_4}^{b})^{-1}({t_1}^{c}{t_2}^{d})t_4^{2b})\\
					=&\varphi (({t_3}^{a}{t_4}^{b})^{-1}({t_1}^{c}{t_2}^{d})^{-1}(({t_1}^{c}{t_2}^{d})t_4^{b})^2)\\
					=&\varphi (({t_3}^{a}{t_4}^{b}{t_1}^{c}{t_2}^{d})^{-1}({t_1}^{c}{t_2}^{d}t_4^{b})^2)\\
					=&(\tilde g_2\tilde g_3)^{-1}\varphi (({t_1}^{c}{t_2}^{d}t_4^{b})^2).
\end{align*}
Hence
\begin{equation}\label{Eq4}
	t_3=\varphi (({t_1}^{c}{t_2}^{d}t_3^{a})^2),
\end{equation}
or
\begin{equation}\label{Eq4b}
	t_3=\varphi (({t_1}^{c}{t_2}^{d}t_4^{b})^2).
\end{equation}
Let us consider the diagrams below
\[
\begin{CD}
\footnotesize \biggl ( \left(\begin{array}{cccc}
p\\
\frac{1}{2}m\\
\frac{1}{2}m+l\\
\frac{1}{2}q
\end{array}\right),I \biggr)= t_3^l(\tilde g_2\tilde g_3)^{m}(t_1^pt_4^q) @>>> 
(\gamma_1\alpha_2)^m\gamma_2^q=
\biggl ( \left(\begin{array}{cc}
\frac{1}{2}m\\
\frac{1}{2}q
\end{array}\right),I \biggr)\\
@A \varphi AA           @A \hat \varphi AA\\
\biggl ( \left(\begin{array}{cccc}
c\\
d\\
\frac{1}{2}{a}\\
0
\end{array}\right),I \biggr) ={t_1}^{c}{t_2}^{d}t_3^a @>>> {s_1}^{a}=
\biggl ( \left(\begin{array}{cc}
\frac{1}{2}a\\
0
\end{array}\right),I \biggr)
\end{CD}
\]
or
\[
\begin{CD}
\footnotesize \biggl ( \left(\begin{array}{cccc}
p\\
\frac{1}{2}m\\
\frac{1}{2}m+l\\
\frac{1}{2}q
\end{array}\right),I \biggr)= t_3^l(\tilde g_2\tilde g_3)^{m}(t_1^pt_4^q) @>>> 
(\gamma_1\alpha_2)^m\gamma_2^q=
\biggl ( \left(\begin{array}{cc}
\frac{1}{2}m\\
\frac{1}{2}q
\end{array}\right),I \biggr)\\
@A \varphi AA           @A \hat \varphi AA\\
\biggl ( \left(\begin{array}{cccc}
c\\
d\\
0\\
\frac{1}{2}{b}
\end{array}\right),I \biggr) ={t_1}^{c}{t_2}^{d}t_4^b @>>> {s_2}^{b}=
\biggl ( \left(\begin{array}{cc}
0\\
\frac{1}{2}b
\end{array}\right),I \biggr),
\end{CD}
\]
 \normalsize
we have
\begin{align*}
\varphi ({t_1}^{c}{t_2}^{d}t_3^a)=&t_3^l(\tilde g_2\tilde g_3)^{m}(t_1^pt_4^q)\\
\varphi (({t_1}^{c}{t_2}^{d}t_3^a)^2)=&t_3^{2l}(\tilde g_2\tilde g_3)^{2m}(t_1^{2p}t_4^{2q})\\
						 =&t_3^{2l}(t_2t_3)^{m}(t_1^{2p}t_4^{2q}) \hspace{1cm}  ((\tilde g_2\tilde g_3)^{2}=t_2t_3)\\
						 =&t_1^{2p}t_2^m t_3^{m+2l}t_4^{2q},
\end{align*}
or
\begin{align*}
\varphi ({t_1}^{c}{t_2}^{d}t_4^b)=&t_3^l(\tilde g_2\tilde g_3)^{m}(t_1^pt_4^q)\\
\varphi (({t_1}^{c}{t_2}^{d}t_4^b)^2)=&t_3^{2l}(\tilde g_2\tilde g_3)^{2m}(t_1^{2p}t_4^{2q})\\
						 =&t_3^{2l}(t_2t_3)^{m}(t_1^{2p}t_4^{2q}) \hspace{1cm}  ((\tilde g_2\tilde g_3)^{2}=t_2t_3)\\
						 =&t_1^{2p}t_2^m t_3^{m+2l}t_4^{2q}.
\end{align*}
Therefore from \eqref{Eq4} or \eqref{Eq4b} $t_3=t_1^{2p}t_2^m t_3^{m+2l}t_4^{2q}$. 
Hence $p=m=q=0$ and $m+2l=1$ or $2l=1$. This yields a contradiction.\\

%============================================================================
\noindent {\bf (28)}. \\
\noindent 
$\bullet$ $M(A_{a10})\approx  M(A_{a12})$.\\
For 
\footnotesize
\begin{align*}
&A_{a10}=
\left(\begin{array}{cccc}
1& 1 & 0 & 0 \\
0& 1 & 1 & 0\\
0& 0 & 1 & 1\\
0& 0 & 0 & 1
\end{array}\right)
\hspace{.1cm}
&A_{a12}=
\left(\begin{array}{cccc}
1& 1 & 0 & 0 \\
0& 1 & 1 & 1\\
0& 0 & 1 & 1\\
0& 0 & 0 & 1
\end{array}\right)
\\
&g_1(z_1,z_2,z_3,z_4)=(-z_1,\bar z_2,z_3,z_4).\hspace{.1cm}   &h_1(z_1,z_2,z_3,z_4)=(-z_1,\bar z_2,z_3,z_4)\\
&g_2(z_1,z_2,z_3,z_4)=(z_1,-z_2,\bar z_3,z_4).            &h_2h_3(z_1,z_2,z_3,z_4)=(z_1,-z_2,-\bar z_3,z_4)\\
&g_3(z_1,z_2,z_3,z_4)=(z_1,z_2,-z_3,\bar z_4).           &h_3(z_1,z_2,z_3,z_4)=(z_1,z_2,-z_3,\bar z_4)\\
&g_4(z_1,z_2,z_3,z_4)=(z_1,z_2,z_3,-z_4).                &h_4(z_1,z_2,z_3,z_4)=(z_1,z_2,z_3,-z_4)
\end{align*}
\normalsize then \footnotesize
\begin{align*}
\Gamma_{a10}=
\Bigl < &\left(\begin{array}{cccc}
\frac{1}{2}\\
0\\
0\\
0
\end{array}\right)
\left(\begin{array}{cccc}
1& 0 & 0 & 0\\
0& -1 & 0 & 0\\
0& 0 & 1 & 0\\
0& 0 & 0 & 1
\end{array}\right),
\left(\begin{array}{cccc}
0\\
\frac{1}{2}\\
0\\
0
\end{array}\right)
\left(\begin{array}{cccc}
1& 0 & 0 & 0\\
0& 1 & 0 & 0\\
0& 0 & -1 & 0\\
0& 0 & 0 & 1
\end{array}\right),
\\
&\left(\begin{array}{cccc}
0\\
0\\
\frac{1}{2}\\
0
\end{array}\right)
\left(\begin{array}{cccc}
1& 0 & 0 & 0\\
0& 1 & 0 & 0\\
0& 0 & 1 & 0\\
0& 0 & 0 & -1
\end{array}\right),
\left(\begin{array}{cccc}
0\\
0\\
0\\
\frac{1}{2}
\end{array}\right)
\left(\begin{array}{cccc}
1& 0 & 0 & 0\\
0& 1 & 0 & 0\\
0& 0 & 1 & 0\\
0& 0 & 0 & 1
\end{array}\right)
\Bigr >
\end{align*}

\begin{align*}
\Gamma_{a12}= 
\Bigl < & \left(\begin{array}{cccc}
\frac{1}{2}\\
0\\
0\\
0
\end{array}\right)
\left(\begin{array}{cccc}
1& 0 & 0 & 0\\
0& -1 & 0 & 0\\
0& 0 & 1 & 0\\
0& 0 & 0 & 1
\end{array}\right),
\left(\begin{array}{cccc}
0\\
\frac{1}{2}\\
\frac{1}{2}\\
0
\end{array}\right)
\left(\begin{array}{cccc}
1& 0 & 0 & 0\\
0& 1 & 0 & 0\\
0& 0 & -1 & 0\\
0& 0 & 0 & 1
\end{array}\right),
\\
&\left(\begin{array}{cccc}
0\\
0\\
\frac{1}{2}\\
0
\end{array}\right)
\left(\begin{array}{cccc}
1& 0 & 0 & 0\\
0& 1 & 0 & 0\\
0& 0 & 1 & 0\\
0& 0 & 0 & -1
\end{array}\right),
\left(\begin{array}{cccc}
0\\
0\\
0\\
\frac{1}{2}
\end{array}\right)
\left(\begin{array}{cccc}
1& 0 & 0 & 0\\
0& 1 & 0 & 0\\
0& 0 & 1 & 0\\
0& 0 & 0 & 1
\end{array}\right)
\Bigr >.
\end{align*}

\normalsize
\noindent
Let $\varphi(z_1,z_2,z_3,z_4)=(z_1,z_2,iz_3,z_4)$, we get these commutative diagrams

\footnotesize
\[
\begin{CD}
(z_1,z_2,z_3,z_4) @>\varphi>>(z_1,z_2,iz_3,z_4)\\
@Vg_1VV @Vh_1 VV\\
(-z_1,\bar z_2,z_3,z_4) @>\varphi>> (-z_1,\bar z_2,iz_3,z_4)
\end{CD}
\hspace{.1cm}
\begin{CD}
(z_1,z_2,z_3,z_4) @>\varphi>>(z_1,z_2,iz_3,z_4)\\
@Vg_3VV @Vh_3 VV\\
(z_1,z_2,-z_3,\bar z_4) @>\varphi>> (z_1,z_2,-iz_3,\bar z_4)
\end{CD}
\]
\[ 
\begin{CD}
(z_1,z_2,z_3,z_4) @>\varphi>>(z_1,z_2,iz_3,z_4)\\
@Vg_2VV @Vh_2h_3 VV\\
(z_1,-z_2,\bar z_3,z_4) @>\varphi>> (z_1,-z_2,-\bar{iz}_3,z_4)
\end{CD}
\hspace{.1cm}
\begin{CD}
(z_1,z_2,z_3,z_4) @>\varphi>>(z_1,z_2,iz_3,z_4)\\
@Vg_4VV @Vh_4 VV\\
(z_1,z_2,z_3,-z_4) @>\varphi>> (z_1,z_2,iz_3,-z_4).
\end{CD}
\]

\normalsize
\noindent 
Therefore $\exists$ 
$\gamma 
\footnotesize
=
\Bigl( \left(\begin{array}{cccc}
0\\
0\\
\frac{1}{4}\\
0
\end{array}\right)
\left(\begin{array}{cccc}
1& 0 & 0 & 0\\
0& 1 & 0 & 0\\
0& 0 & 1 & 0\\
0& 0 & 0 & 1
\end{array}\right)
\Bigr)
\normalsize
\in \mathbb{A}(n)$ 
s.t. $\gamma \Gamma_{a10}\gamma ^{-1}=\Gamma_{a12}$.\\
%---------------------------------------------------------------------------

\noindent 
$\bullet$ $M(A_{a10})\approx  M(A_{a26})$.\\
For 
\footnotesize
\begin{align*}
&A_{a10}=
\left(\begin{array}{cccc}
1& 1 & 0 & 0 \\
0& 1 & 1 & 0\\
0& 0 & 1 & 1\\
0& 0 & 0 & 1
\end{array}\right)
\hspace{1cm}
&A_{a26}=
\left(\begin{array}{cccc}
1& 1 & 1 & 0 \\
0& 1 & 1 & 0\\
0& 0 & 1 & 1\\
0& 0 & 0 & 1
\end{array}\right)
\\
&g_1(z_1,z_2,z_3,z_4)=(-z_1,\bar z_2,z_3,z_4).\hspace{.1cm}   &h_1h_2(z_1,z_2,z_3,z_4)=(-z_1,-\bar z_2,z_3,z_4)\\
&g_2(z_1,z_2,z_3,z_4)=(z_1,-z_2,\bar z_3,z_4).            &h_2(z_1,z_2,z_3,z_4)=(z_1,-z_2,\bar z_3,z_4)\\
&g_3(z_1,z_2,z_3,z_4)=(z_1,z_2,-z_3,\bar z_4).           &h_3(z_1,z_2,z_3,z_4)=(z_1,z_2,-z_3,\bar z_4)\\
&g_4(z_1,z_2,z_3,z_4)=(z_1,z_2,z_3,-z_4).                &h_4(z_1,z_2,z_3,z_4)=(z_1,z_2,z_3,-z_4)
\end{align*}
\normalsize then \footnotesize
\begin{align*}
\Gamma_{a10}=
\Bigl < &\left(\begin{array}{cccc}
\frac{1}{2}\\
0\\
0\\
0
\end{array}\right)
\left(\begin{array}{cccc}
1& 0 & 0 & 0\\
0& -1 & 0 & 0\\
0& 0 & 1 & 0\\
0& 0 & 0 & 1
\end{array}\right),
\left(\begin{array}{cccc}
0\\
\frac{1}{2}\\
0\\
0
\end{array}\right)
\left(\begin{array}{cccc}
1& 0 & 0 & 0\\
0& 1 & 0 & 0\\
0& 0 & -1 & 0\\
0& 0 & 0 & 1
\end{array}\right),
\\
&\left(\begin{array}{cccc}
0\\
0\\
\frac{1}{2}\\
0
\end{array}\right)
\left(\begin{array}{cccc}
1& 0 & 0 & 0\\
0& 1 & 0 & 0\\
0& 0 & 1 & 0\\
0& 0 & 0 & -1
\end{array}\right),
\left(\begin{array}{cccc}
0\\
0\\
0\\
\frac{1}{2}
\end{array}\right)
\left(\begin{array}{cccc}
1& 0 & 0 & 0\\
0& 1 & 0 & 0\\
0& 0 & 1 & 0\\
0& 0 & 0 & 1
\end{array}\right)
\Bigr >
\end{align*}

\begin{align*}
\Gamma_{a26}= 
\Bigl < & \left(\begin{array}{cccc}
\frac{1}{2}\\
\frac{1}{2}\\
0\\
0
\end{array}\right)
\left(\begin{array}{cccc}
1& 0 & 0 & 0\\
0& -1 & 0 & 0\\
0& 0 & 1 & 0\\
0& 0 & 0 & 1
\end{array}\right),
\left(\begin{array}{cccc}
0\\
\frac{1}{2}\\
0\\
0
\end{array}\right)
\left(\begin{array}{cccc}
1& 0 & 0 & 0\\
0& 1 & 0 & 0\\
0& 0 & -1 & 0\\
0& 0 & 0 & 1
\end{array}\right),
\\
&\left(\begin{array}{cccc}
0\\
0\\
\frac{1}{2}\\
0
\end{array}\right)
\left(\begin{array}{cccc}
1& 0 & 0 & 0\\
0& 1 & 0 & 0\\
0& 0 & 1 & 0\\
0& 0 & 0 & -1
\end{array}\right),
\left(\begin{array}{cccc}
0\\
0\\
0\\
\frac{1}{2}
\end{array}\right)
\left(\begin{array}{cccc}
1& 0 & 0 & 0\\
0& 1 & 0 & 0\\
0& 0 & 1 & 0\\
0& 0 & 0 & 1
\end{array}\right)
\Bigr >.
\end{align*}

\normalsize
\noindent
Let $\varphi(z_1,z_2,z_3,z_4)=(z_1,iz_2,z_3,z_4)$, we get these commutative diagrams

\footnotesize
\[
\begin{CD}
(z_1,z_2,z_3,z_4) @>\varphi>>(z_1,iz_2,z_3,z_4)\\
@Vg_1VV @Vh_1h_2 VV\\
(-z_1,\bar z_2,z_3,z_4) @>\varphi>> (-z_1,-\bar {iz}_2,z_3,z_4)
\end{CD}
\hspace{.1cm}
\begin{CD}
(z_1,z_2,z_3,z_4) @>\varphi>>(z_1,iz_2,z_3,z_4)\\
@Vg_3VV @Vh_3 VV\\
(z_1,z_2,-z_3,\bar z_4) @>\varphi>> (z_1,iz_2,-z_3,\bar z_4)
\end{CD}
\]
\[ 
\begin{CD}
(z_1,z_2,z_3,z_4) @>\varphi>>(z_1,iz_2,z_3,z_4)\\
@Vg_2VV @Vh_2 VV\\
(z_1,-z_2,\bar z_3,z_4) @>\varphi>> (z_1,-iz_2,\bar z_3,z_4)
\end{CD}
\hspace{.1cm}
\begin{CD}
(z_1,z_2,z_3,z_4) @>\varphi>>(z_1,iz_2,z_3,z_4)\\
@Vg_4VV @Vh_4 VV\\
(z_1,z_2,z_3,-z_4) @>\varphi>> (z_1,iz_2,z_3,-z_4).
\end{CD}
\]

\normalsize
\noindent 
Therefore $\exists$ 
$\gamma 
\footnotesize
=
\Bigl( \left(\begin{array}{cccc}
0\\
\frac{1}{4}\\
0\\
0
\end{array}\right)
\left(\begin{array}{cccc}
1& 0 & 0 & 0\\
0& 1 & 0 & 0\\
0& 0 & 1 & 0\\
0& 0 & 0 & 1
\end{array}\right)
\Bigr)
\normalsize
\in \mathbb{A}(n)$ 
s.t. $\gamma \Gamma_{a10}\gamma ^{-1}=\Gamma_{a26}$.\\
%---------------------------------------------------------------------------

\noindent 
$\bullet$ $M(A_{a12})\approx  M(A_{a32})$.\\
For
\footnotesize
\begin{align*}
&A_{a12}=
\left(\begin{array}{cccc}
1& 1 & 0 & 0 \\
0& 1 & 1 & 1\\
0& 0 & 1 & 1\\
0& 0 & 0 & 1
\end{array}\right)
\hspace{.1cm}
&A_{a32}=
\left(\begin{array}{cccc}
1& 1 & 1 & 1 \\
0& 1 & 1 & 1\\
0& 0 & 1 & 1\\
0& 0 & 0 & 1
\end{array}\right)
\\
&g_1(z_1,z_2,z_3,z_4)=(-z_1,\bar z_2,z_3,z_4).\hspace{.1cm}   &h_1h_2(z_1,z_2,z_3,z_4)=(-z_1,-\bar z_2,z_3,z_4)\\
&g_2(z_1,z_2,z_3,z_4)=(z_1,-z_2,\bar z_3,\bar z_4).            &h_2(z_1,z_2,z_3,z_4)=(z_1,-z_2,\bar z_3,\bar z_4)\\
&g_3(z_1,z_2,z_3,z_4)=(z_1,z_2,-z_3,\bar z_4).           &h_3(z_1,z_2,z_3,z_4)=(z_1,z_2,-z_3,\bar z_4)\\
&g_4(z_1,z_2,z_3,z_4)=(z_1,z_2,z_3,-z_4).                &h_4(z_1,z_2,z_3,z_4)=(z_1,z_2,z_3,-z_4)
\end{align*}
\normalsize then \footnotesize
\begin{align*}
\Gamma_{a12}=
\Bigl < &\left(\begin{array}{cccc}
\frac{1}{2}\\
0\\
0\\
0
\end{array}\right)
\left(\begin{array}{cccc}
1& 0 & 0 & 0\\
0& -1 & 0 & 0\\
0& 0 & 1 & 0\\
0& 0 & 0 & 1
\end{array}\right),
\left(\begin{array}{cccc}
0\\
\frac{1}{2}\\
0\\
0
\end{array}\right)
\left(\begin{array}{cccc}
1& 0 & 0 & 0\\
0& 1 & 0 & 0\\
0& 0 & -1 & 0\\
0& 0 & 0 & -1
\end{array}\right),
\\
&\left(\begin{array}{cccc}
0\\
0\\
\frac{1}{2}\\
0
\end{array}\right)
\left(\begin{array}{cccc}
1& 0 & 0 & 0\\
0& 1 & 0 & 0\\
0& 0 & 1 & 0\\
0& 0 & 0 & -1
\end{array}\right),
\left(\begin{array}{cccc}
0\\
0\\
0\\
\frac{1}{2}
\end{array}\right)
\left(\begin{array}{cccc}
1& 0 & 0 & 0\\
0& 1 & 0 & 0\\
0& 0 & 1 & 0\\
0& 0 & 0 & 1
\end{array}\right)
\Bigr >
\end{align*}

\begin{align*}
\Gamma_{a32}= 
\Bigl < & \left(\begin{array}{cccc}
\frac{1}{2}\\
\frac{1}{2}\\
0\\
0
\end{array}\right)
\left(\begin{array}{cccc}
1& 0 & 0 & 0\\
0& -1 & 0 & 0\\
0& 0 & 1 & 0\\
0& 0 & 0 & 1
\end{array}\right),
\left(\begin{array}{cccc}
0\\
\frac{1}{2}\\
0\\
0
\end{array}\right)
\left(\begin{array}{cccc}
1& 0 & 0 & 0\\
0& 1 & 0 & 0\\
0& 0 & -1 & 0\\
0& 0 & 0 & -1
\end{array}\right),
\\
&\left(\begin{array}{cccc}
0\\
0\\
\frac{1}{2}\\
0
\end{array}\right)
\left(\begin{array}{cccc}
1& 0 & 0 & 0\\
0& 1 & 0 & 0\\
0& 0 & 1 & 0\\
0& 0 & 0 & -1
\end{array}\right),
\left(\begin{array}{cccc}
0\\
0\\
0\\
\frac{1}{2}
\end{array}\right)
\left(\begin{array}{cccc}
1& 0 & 0 & 0\\
0& 1 & 0 & 0\\
0& 0 & 1 & 0\\
0& 0 & 0 & 1
\end{array}\right)
\Bigr >.
\end{align*}

\normalsize
\noindent
Let $\varphi(z_1,z_2,z_3,z_4)=(z_1,iz_2,z_3,z_4)$, we get these commutative diagrams

\footnotesize
\[
\begin{CD}
(z_1,z_2,z_3,z_4) @>\varphi>>(z_1,iz_2,z_3,z_4)\\
@Vg_1VV @Vh_1h_2 VV\\
(-z_1,\bar z_2,z_3,z_4) @>\varphi>> (-z_1,-\bar {iz}_2,z_3,z_4)
\end{CD}
\hspace{.1cm}
\begin{CD}
(z_1,z_2,z_3,z_4) @>\varphi>>(z_1,iz_2,z_3,z_4)\\
@Vg_3VV @Vh_3 VV\\
(z_1,z_2,-z_3,\bar z_4) @>\varphi>> (z_1,iz_2,-z_3,\bar z_4)
\end{CD}
\]
\[ 
\begin{CD}
(z_1,z_2,z_3,z_4) @>\varphi>>(z_1,iz_2,z_3,z_4)\\
@Vg_2VV @Vh_2 VV\\
(z_1,-z_2,\bar z_3,\bar z_4) @>\varphi>> (z_1,-iz_2,\bar{z}_3,\bar z_4)
\end{CD}
\hspace{.1cm}
\begin{CD}
(z_1,z_2,z_3,z_4) @>\varphi>>(z_1,iz_2,z_3,z_4)\\
@Vg_4VV @Vh_4 VV\\
(z_1,z_2,z_3,-z_4) @>\varphi>> (z_1,iz_2,z_3,-z_4).
\end{CD}
\]

\normalsize
\noindent 
Therefore $\exists$ 
$\gamma 
\footnotesize
=
\Bigl( \left(\begin{array}{cccc}
0\\
\frac{1}{4}\\
0\\
0
\end{array}\right)
\left(\begin{array}{cccc}
1& 0 & 0 & 0\\
0& 1 & 0 & 0\\
0& 0 & 1 & 0\\
0& 0 & 0 & 1
\end{array}\right)
\Bigr)
\normalsize
\in \mathbb{A}(n)$ 
s.t. $\gamma \Gamma_{a12}\gamma ^{-1}=\Gamma_{a32}$.\\
%---------------------------------------------------------------------------

\noindent 
$\bullet$ $M(A_{a14})\approx  M(A_{a16})$.\\
For
\footnotesize
\begin{align*}
&A_{a14}=
\left(\begin{array}{cccc}
1& 1 & 0 & 1 \\
0& 1 & 1 & 0\\
0& 0 & 1 & 1\\
0& 0 & 0 & 1
\end{array}\right)
\hspace{.1cm}
&A_{a16}=
\left(\begin{array}{cccc}
1& 1 & 0 & 1 \\
0& 1 & 1 & 1\\
0& 0 & 1 & 1\\
0& 0 & 0 & 1
\end{array}\right)
\\
&g_1(z_1,z_2,z_3,z_4)=(-z_1,\bar z_2,z_3,\bar z_4).\hspace{.1cm}   &h_1(z_1,z_2,z_3,z_4)=(-z_1,\bar z_2,z_3,\bar z_4)\\
&g_2(z_1,z_2,z_3,z_4)=(z_1,-z_2,\bar z_3,z_4).            &h_2h_3(z_1,z_2,z_3,z_4)=(z_1,-z_2,-\bar z_3,z_4)\\
&g_3(z_1,z_2,z_3,z_4)=(z_1,z_2,-z_3,\bar z_4).           &h_3(z_1,z_2,z_3,z_4)=(z_1,z_2,-z_3,\bar z_4)\\
&g_4(z_1,z_2,z_3,z_4)=(z_1,z_2,z_3,-z_4).                &h_4(z_1,z_2,z_3,z_4)=(z_1,z_2,z_3,-z_4)
\end{align*}
\normalsize then \footnotesize
\begin{align*}
\Gamma_{a14}=
\Bigl < &\left(\begin{array}{cccc}
\frac{1}{2}\\
0\\
0\\
0
\end{array}\right)
\left(\begin{array}{cccc}
1& 0 & 0 & 0\\
0& -1 & 0 & 0\\
0& 0 & 1 & 0\\
0& 0 & 0 & -1
\end{array}\right),
\left(\begin{array}{cccc}
0\\
\frac{1}{2}\\
0\\
0
\end{array}\right)
\left(\begin{array}{cccc}
1& 0 & 0 & 0\\
0& 1 & 0 & 0\\
0& 0 & -1 & 0\\
0& 0 & 0 & 1
\end{array}\right),
\\
&\left(\begin{array}{cccc}
0\\
0\\
\frac{1}{2}\\
0
\end{array}\right)
\left(\begin{array}{cccc}
1& 0 & 0 & 0\\
0& 1 & 0 & 0\\
0& 0 & 1 & 0\\
0& 0 & 0 & -1
\end{array}\right),
\left(\begin{array}{cccc}
0\\
0\\
0\\
\frac{1}{2}
\end{array}\right)
\left(\begin{array}{cccc}
1& 0 & 0 & 0\\
0& 1 & 0 & 0\\
0& 0 & 1 & 0\\
0& 0 & 0 & 1
\end{array}\right)
\Bigr >
\end{align*}

\begin{align*}
\Gamma_{a16}= 
\Bigl < & \left(\begin{array}{cccc}
\frac{1}{2}\\
0\\
0\\
0
\end{array}\right)
\left(\begin{array}{cccc}
1& 0 & 0 & 0\\
0& -1 & 0 & 0\\
0& 0 & 1 & 0\\
0& 0 & 0 & -1
\end{array}\right),
\left(\begin{array}{cccc}
0\\
\frac{1}{2}\\
\frac{1}{2}\\
0
\end{array}\right)
\left(\begin{array}{cccc}
1& 0 & 0 & 0\\
0& 1 & 0 & 0\\
0& 0 & -1 & 0\\
0& 0 & 0 & 1
\end{array}\right),
\\
&\left(\begin{array}{cccc}
0\\
0\\
\frac{1}{2}\\
0
\end{array}\right)
\left(\begin{array}{cccc}
1& 0 & 0 & 0\\
0& 1 & 0 & 0\\
0& 0 & 1 & 0\\
0& 0 & 0 & -1
\end{array}\right),
\left(\begin{array}{cccc}
0\\
0\\
0\\
\frac{1}{2}
\end{array}\right)
\left(\begin{array}{cccc}
1& 0 & 0 & 0\\
0& 1 & 0 & 0\\
0& 0 & 1 & 0\\
0& 0 & 0 & 1
\end{array}\right)
\Bigr >.
\end{align*}

\normalsize
\noindent
Let $\varphi(z_1,z_2,z_3,z_4)=(z_1,z_2,iz_3,z_4)$, we get these commutative diagrams

\footnotesize
\[
\begin{CD}
(z_1,z_2,z_3,z_4) @>\varphi>>(z_1,z_2,iz_3,z_4)\\
@Vg_1VV @Vh_1 VV\\
(-z_1,\bar z_2,z_3,\bar z_4) @>\varphi>> (-z_1,\bar {z}_2,iz_3,\bar z_4)
\end{CD}
\hspace{.1cm}
\begin{CD}
(z_1,z_2,z_3,z_4) @>\varphi>>(z_1,z_2,iz_3,z_4)\\
@Vg_3VV @Vh_3 VV\\
(z_1,z_2,-z_3,\bar z_4) @>\varphi>> (z_1,z_2,-iz_3,\bar z_4)
\end{CD}
\]
\[ 
\begin{CD}
(z_1,z_2,z_3,z_4) @>\varphi>>(z_1,z_2,iz_3,z_4)\\
@Vg_2VV @Vh_2h_3 VV\\
(z_1,-z_2,\bar z_3,z_4) @>\varphi>> (z_1,-z_2,-\bar{iz}_3,z_4)
\end{CD}
\hspace{.1cm}
\begin{CD}
(z_1,z_2,z_3,z_4) @>\varphi>>(z_1,z_2,iz_3,z_4)\\
@Vg_4VV @Vh_4 VV\\
(z_1,z_2,z_3,-z_4) @>\varphi>> (z_1,z_2,iz_3,-z_4).
\end{CD}
\]

\normalsize
\noindent 
Therefore $\exists$ 
$\gamma 
\footnotesize
=
\Bigl( \left(\begin{array}{cccc}
0\\
0\\
\frac{1}{4}\\
0
\end{array}\right)
\left(\begin{array}{cccc}
1& 0 & 0 & 0\\
0& 1 & 0 & 0\\
0& 0 & 1 & 0\\
0& 0 & 0 & 1
\end{array}\right)
\Bigr)
\normalsize
\in \mathbb{A}(n)$ 
s.t. $\gamma \Gamma_{a14}\gamma ^{-1}=\Gamma_{a16}$.\\
%---------------------------------------------------------------------------

\noindent 
$\bullet$ $M(A_{a14})\approx  M(A_{a30})$.\\
For
\footnotesize
\begin{align*}
&A_{a14}=
\left(\begin{array}{cccc}
1& 1 & 0 & 1 \\
0& 1 & 1 & 0\\
0& 0 & 1 & 1\\
0& 0 & 0 & 1
\end{array}\right)
\hspace{.1cm}
&A_{a30}=
\left(\begin{array}{cccc}
1& 1 & 1 & 1 \\
0& 1 & 1 & 0\\
0& 0 & 1 & 1\\
0& 0 & 0 & 1
\end{array}\right)
\\
&g_1(z_1,z_2,z_3,z_4)=(-z_1,\bar z_2,z_3,\bar z_4).\hspace{.1cm}   &h_1h_2(z_1,z_2,z_3,z_4)=(-z_1,-\bar z_2,z_3,\bar z_4)\\
&g_2(z_1,z_2,z_3,z_4)=(z_1,-z_2,\bar z_3,z_4).            &h_2(z_1,z_2,z_3,z_4)=(z_1,-z_2,\bar z_3,z_4)\\
&g_3(z_1,z_2,z_3,z_4)=(z_1,z_2,-z_3,\bar z_4).           &h_3(z_1,z_2,z_3,z_4)=(z_1,z_2,-z_3,\bar z_4)\\
&g_4(z_1,z_2,z_3,z_4)=(z_1,z_2,z_3,-z_4).                &h_4(z_1,z_2,z_3,z_4)=(z_1,z_2,z_3,-z_4)
\end{align*}
\normalsize then \footnotesize
\begin{align*}
\Gamma_{a14}=
\Bigl < &\left(\begin{array}{cccc}
\frac{1}{2}\\
0\\
0\\
0
\end{array}\right)
\left(\begin{array}{cccc}
1& 0 & 0 & 0\\
0& -1 & 0 & 0\\
0& 0 & 1 & 0\\
0& 0 & 0 & -1
\end{array}\right),
\left(\begin{array}{cccc}
0\\
\frac{1}{2}\\
0\\
0
\end{array}\right)
\left(\begin{array}{cccc}
1& 0 & 0 & 0\\
0& 1 & 0 & 0\\
0& 0 & -1 & 0\\
0& 0 & 0 & 1
\end{array}\right),
\\
&\left(\begin{array}{cccc}
0\\
0\\
\frac{1}{2}\\
0
\end{array}\right)
\left(\begin{array}{cccc}
1& 0 & 0 & 0\\
0& 1 & 0 & 0\\
0& 0 & 1 & 0\\
0& 0 & 0 & -1
\end{array}\right),
\left(\begin{array}{cccc}
0\\
0\\
0\\
\frac{1}{2}
\end{array}\right)
\left(\begin{array}{cccc}
1& 0 & 0 & 0\\
0& 1 & 0 & 0\\
0& 0 & 1 & 0\\
0& 0 & 0 & 1
\end{array}\right)
\Bigr >
\end{align*}

\begin{align*}
\Gamma_{a30}= 
\Bigl < & \left(\begin{array}{cccc}
\frac{1}{2}\\
\frac{1}{2}\\
0\\
0
\end{array}\right)
\left(\begin{array}{cccc}
1& 0 & 0 & 0\\
0& -1 & 0 & 0\\
0& 0 & 1 & 0\\
0& 0 & 0 & -1
\end{array}\right),
\left(\begin{array}{cccc}
0\\
\frac{1}{2}\\
0\\
0
\end{array}\right)
\left(\begin{array}{cccc}
1& 0 & 0 & 0\\
0& 1 & 0 & 0\\
0& 0 & -1 & 0\\
0& 0 & 0 & 1
\end{array}\right),
\\
&\left(\begin{array}{cccc}
0\\
0\\
\frac{1}{2}\\
0
\end{array}\right)
\left(\begin{array}{cccc}
1& 0 & 0 & 0\\
0& 1 & 0 & 0\\
0& 0 & 1 & 0\\
0& 0 & 0 & -1
\end{array}\right),
\left(\begin{array}{cccc}
0\\
0\\
0\\
\frac{1}{2}
\end{array}\right)
\left(\begin{array}{cccc}
1& 0 & 0 & 0\\
0& 1 & 0 & 0\\
0& 0 & 1 & 0\\
0& 0 & 0 & 1
\end{array}\right)
\Bigr >.
\end{align*}

\normalsize
\noindent
Let $\varphi(z_1,z_2,z_3,z_4)=(z_1,iz_2,z_3,z_4)$, we get these commutative diagrams

\footnotesize
\[
\begin{CD}
(z_1,z_2,z_3,z_4) @>\varphi>>(z_1,iz_2,z_3,z_4)\\
@Vg_1VV @Vh_1h_2 VV\\
(-z_1,\bar z_2,z_3,\bar z_4) @>\varphi>> (-z_1,-\bar {iz}_2,z_3,\bar z_4)
\end{CD}
\hspace{.1cm}
\begin{CD}
(z_1,z_2,z_3,z_4) @>\varphi>>(z_1,iz_2,z_3,z_4)\\
@Vg_3VV @Vh_3 VV\\
(z_1,z_2,-z_3,\bar z_4) @>\varphi>> (z_1,iz_2,-z_3,\bar z_4)
\end{CD}
\]
\[ 
\begin{CD}
(z_1,z_2,z_3,z_4) @>\varphi>>(z_1,iz_2,z_3,z_4)\\
@Vg_2VV @Vh_2 VV\\
(z_1,-z_2,\bar z_3,z_4) @>\varphi>> (z_1,-iz_2,\bar{z}_3,z_4)
\end{CD}
\hspace{.1cm}
\begin{CD}
(z_1,z_2,z_3,z_4) @>\varphi>>(z_1,iz_2,z_3,z_4)\\
@Vg_4VV @Vh_4 VV\\
(z_1,z_2,z_3,-z_4) @>\varphi>> (z_1,iz_2,z_3,-z_4).
\end{CD}
\]

\normalsize
\noindent 
Therefore $\exists$ 
$\gamma 
\footnotesize
=
\Bigl( \left(\begin{array}{cccc}
0\\
\frac{1}{4}\\
0\\
0
\end{array}\right)
\left(\begin{array}{cccc}
1& 0 & 0 & 0\\
0& 1 & 0 & 0\\
0& 0 & 1 & 0\\
0& 0 & 0 & 1
\end{array}\right)
\Bigr)
\normalsize
\in \mathbb{A}(n)$ 
s.t. $\gamma \Gamma_{a14}\gamma ^{-1}=\Gamma_{a30}$.\\
%---------------------------------------------------------------------------

\noindent 
$\bullet$ $M(A_{a28})\approx  M(A_{a30})$.\\
For
\footnotesize
\begin{align*}
&A_{a28}=
\left(\begin{array}{cccc}
1& 1 & 1 & 0 \\
0& 1 & 1 & 1\\
0& 0 & 1 & 1\\
0& 0 & 0 & 1
\end{array}\right)
\hspace{.1cm}
&A_{a30}=
\left(\begin{array}{cccc}
1& 1 & 1 & 1 \\
0& 1 & 1 & 0\\
0& 0 & 1 & 1\\
0& 0 & 0 & 1
\end{array}\right)
\\
&g_1(z_1,z_2,z_3,z_4)=(-z_1,\bar z_2,\bar z_3,z_4).\hspace{.1cm}   &h_1h_3(z_1,z_2,z_3,z_4)=(-z_1,\bar z_2,-\bar z_3,z_4)\\
&g_2(z_1,z_2,z_3,z_4)=(z_1,-z_2,\bar z_3,\bar z_4).            &h_2h_3(z_1,z_2,z_3,z_4)=(z_1,-z_2,-\bar z_3,\bar z_4)\\
&g_3(z_1,z_2,z_3,z_4)=(z_1,z_2,-z_3,\bar z_4).           &h_3(z_1,z_2,z_3,z_4)=(z_1,z_2,-z_3,\bar z_4)\\
&g_4(z_1,z_2,z_3,z_4)=(z_1,z_2,z_3,-z_4).                &h_4(z_1,z_2,z_3,z_4)=(z_1,z_2,z_3,-z_4)
\end{align*}
\normalsize then \footnotesize
\begin{align*}
\Gamma_{a28}=
\Bigl < &\left(\begin{array}{cccc}
\frac{1}{2}\\
0\\
0\\
0
\end{array}\right)
\left(\begin{array}{cccc}
1& 0 & 0 & 0\\
0& -1 & 0 & 0\\
0& 0 & -1 & 0\\
0& 0 & 0 & 1
\end{array}\right),
\left(\begin{array}{cccc}
0\\
\frac{1}{2}\\
0\\
0
\end{array}\right)
\left(\begin{array}{cccc}
1& 0 & 0 & 0\\
0& 1 & 0 & 0\\
0& 0 & -1 & 0\\
0& 0 & 0 & -1
\end{array}\right),
\\
&\left(\begin{array}{cccc}
0\\
0\\
\frac{1}{2}\\
0
\end{array}\right)
\left(\begin{array}{cccc}
1& 0 & 0 & 0\\
0& 1 & 0 & 0\\
0& 0 & 1 & 0\\
0& 0 & 0 & -1
\end{array}\right),
\left(\begin{array}{cccc}
0\\
0\\
0\\
\frac{1}{2}
\end{array}\right)
\left(\begin{array}{cccc}
1& 0 & 0 & 0\\
0& 1 & 0 & 0\\
0& 0 & 1 & 0\\
0& 0 & 0 & 1
\end{array}\right)
\Bigr >
\end{align*}

\begin{align*}
\Gamma_{a30}= 
\Bigl < & \left(\begin{array}{cccc}
\frac{1}{2}\\
0\\
\frac{1}{2}\\
0
\end{array}\right)
\left(\begin{array}{cccc}
1& 0 & 0 & 0\\
0& -1 & 0 & 0\\
0& 0 & -1 & 0\\
0& 0 & 0 & 1
\end{array}\right),
\left(\begin{array}{cccc}
0\\
\frac{1}{2}\\
\frac{1}{2}\\
0
\end{array}\right)
\left(\begin{array}{cccc}
1& 0 & 0 & 0\\
0& 1 & 0 & 0\\
0& 0 & -1 & 0\\
0& 0 & 0 & -1
\end{array}\right),
\\
&\left(\begin{array}{cccc}
0\\
0\\
\frac{1}{2}\\
0
\end{array}\right)
\left(\begin{array}{cccc}
1& 0 & 0 & 0\\
0& 1 & 0 & 0\\
0& 0 & 1 & 0\\
0& 0 & 0 & -1
\end{array}\right),
\left(\begin{array}{cccc}
0\\
0\\
0\\
\frac{1}{2}
\end{array}\right)
\left(\begin{array}{cccc}
1& 0 & 0 & 0\\
0& 1 & 0 & 0\\
0& 0 & 1 & 0\\
0& 0 & 0 & 1
\end{array}\right)
\Bigr >.
\end{align*}

\normalsize
\noindent
Let $\varphi(z_1,z_2,z_3,z_4)=(z_1,z_2,iz_3,z_4)$, we get these commutative diagrams

\footnotesize
\[
\begin{CD}
(z_1,z_2,z_3,z_4) @>\varphi>>(z_1,z_2,iz_3,z_4)\\
@Vg_1VV @Vh_1h_3 VV\\
(-z_1,\bar z_2,\bar z_3,z_4) @>\varphi>> (-z_1,\bar {z}_2,-\bar {iz}_3,z_4)
\end{CD}
\hspace{.1cm}
\begin{CD}
(z_1,z_2,z_3,z_4) @>\varphi>>(z_1,z_2,iz_3,z_4)\\
@Vg_3VV @Vh_3 VV\\
(z_1,z_2,-z_3,\bar z_4) @>\varphi>> (z_1,z_2,-iz_3,\bar z_4)
\end{CD}
\]
\[ 
\begin{CD}
(z_1,z_2,z_3,z_4) @>\varphi>>(z_1,z_2,iz_3,z_4)\\
@Vg_2VV @Vh_2h_3 VV\\
(z_1,-z_2,\bar z_3,\bar z_4) @>\varphi>> (z_1,-z_2,-\bar{iz}_3,\bar z_4)
\end{CD}
\hspace{.1cm}
\begin{CD}
(z_1,z_2,z_3,z_4) @>\varphi>>(z_1,z_2,iz_3,z_4)\\
@Vg_4VV @Vh_4 VV\\
(z_1,z_2,z_3,-z_4) @>\varphi>> (z_1,z_2,iz_3,-z_4).
\end{CD}
\]

\normalsize
\noindent 
Therefore $\exists$ 
$\gamma 
\footnotesize
=
\Bigl( \left(\begin{array}{cccc}
0\\
0\\
\frac{1}{4}\\
0
\end{array}\right)
\left(\begin{array}{cccc}
1& 0 & 0 & 0\\
0& 1 & 0 & 0\\
0& 0 & 1 & 0\\
0& 0 & 0 & 1
\end{array}\right)
\Bigr)
\normalsize
\in \mathbb{A}(n)$ 
s.t. $\gamma \Gamma_{a28}\gamma ^{-1}=\Gamma_{a30}$.\\
%--------------------------------------------------

\noindent {\bf (29)}. Similar to (4), because $\Phi_{a10}=\Phi_{a14}=\mathbb{Z}_2 \times \mathbb{Z}_2 \times \mathbb{Z}_2 $.
\\

%=======================================================
\noindent
{\bf (30).} $M(A_{a10})$ is not diffeomorphic to $M(A_{a14})$.\\
For 
\footnotesize
\begin{align*}
&A_{a10}=
\left(\begin{array}{cccc}
1& 1 & 0 & 0 \\
0& 1 & 1 & 0\\
0& 0 & 1 & 1\\
0& 0 & 0 & 1
\end{array}\right)
\hspace{.1cm}
&A_{a14}=
\left(\begin{array}{cccc}
1& 1 & 0 & 1 \\
0& 1 & 1 & 0\\
0& 0 & 1 & 1\\
0& 0 & 0 & 1
\end{array}\right)
\\
&g_1(z_1,z_2,z_3,z_4)=(-z_1,\bar z_2,z_3,z_4).\hspace{.1cm}   &h_1(z_1,z_2,z_3,z_4)=(-z_1,\bar z_2,z_3,\bar z_4)\\
&g_2(z_1,z_2,z_3,z_4)=(z_1,-z_2,\bar z_3,z_4).            &h_2(z_1,z_2,z_3,z_4)=(z_1,-z_2,\bar z_3,z_4)\\
&g_3(z_1,z_2,z_3,z_4)=(z_1,z_2,-z_3,\bar z_4).           &h_3(z_1,z_2,z_3,z_4)=(z_1,z_2,-z_3,\bar z_4)\\
&g_4(z_1,z_2,z_3,z_4)=(z_1,z_2,z_3,-z_4).                &h_4(z_1,z_2,z_3,z_4)=(z_1,z_2,z_3,-z_4)
\end{align*}
\normalsize then \footnotesize
\begin{align*}
\Gamma_{a10}=&\langle \tilde {g}_1, \tilde {g}_2, \tilde {g}_3, t_4 \rangle\\
\Bigl < &\left(\begin{array}{cccc}
\frac{1}{2}\\
0\\
0\\
0
\end{array}\right)
\left(\begin{array}{cccc}
1& 0 & 0 & 0\\
0& -1 & 0 & 0\\
0& 0 & 1 & 0\\
0& 0 & 0 & 1
\end{array}\right),
\left(\begin{array}{cccc}
0\\
\frac{1}{2}\\
0\\
0
\end{array}\right)
\left(\begin{array}{cccc}
1& 0 & 0 & 0\\
0& 1 & 0 & 0\\
0& 0 & -1 & 0\\
0& 0 & 0 & 1
\end{array}\right),
\\
&\left(\begin{array}{cccc}
0\\
0\\
\frac{1}{2}\\
0
\end{array}\right)
\left(\begin{array}{cccc}
1& 0 & 0 & 0\\
0& 1 & 0 & 0\\
0& 0 & 1 & 0\\
0& 0 & 0 & -1
\end{array}\right),
\left(\begin{array}{cccc}
0\\
0\\
0\\
\frac{1}{2}
\end{array}\right)
\left(\begin{array}{cccc}
1& 0 & 0 & 0\\
0& 1 & 0 & 0\\
0& 0 & 1 & 0\\
0& 0 & 0 & 1
\end{array}\right)
\Bigr >
\end{align*}

\begin{align*}
\Gamma_{a14}=&\langle \tilde {h}_1, \tilde {h}_2, \tilde {h}_3, t_4 \rangle\\
\Bigl < &\left(\begin{array}{cccc}
\frac{1}{2}\\
0\\
0\\
0
\end{array}\right)
\left(\begin{array}{cccc}
1& 0 & 0 & 0\\
0& -1 & 0 & 0\\
0& 0 & 1 & 0\\
0& 0 & 0 & -1
\end{array}\right),
\left(\begin{array}{cccc}
0\\
\frac{1}{2}\\
0\\
0
\end{array}\right)
\left(\begin{array}{cccc}
1& 0 & 0 & 0\\
0& 1 & 0 & 0\\
0& 0 & -1 & 0\\
0& 0 & 0 & 1
\end{array}\right),
\\
&\left(\begin{array}{cccc}
0\\
0\\
\frac{1}{2}\\
0
\end{array}\right)
\left(\begin{array}{cccc}
1& 0 & 0 & 0\\
0& 1 & 0 & 0\\
0& 0 & 1 & 0\\
0& 0 & 0 & -1
\end{array}\right),
\left(\begin{array}{cccc}
0\\
0\\
0\\
\frac{1}{2}
\end{array}\right)
\left(\begin{array}{cccc}
1& 0 & 0 & 0\\
0& 1 & 0 & 0\\
0& 0 & 1 & 0\\
0& 0 & 0 & 1
\end{array}\right)
\Bigr >.
\end{align*}
\normalsize
Since $\tilde {g}_1\tilde {g}_2\tilde {g}_1^{-1}=\tilde {g}_2^{-1}$, $\tilde {g}_1 \tilde {g}_3\tilde {g}_1^{-1}=\tilde {g}_3$, and
$\tilde {g}_1 t_4\tilde {g}_1^{-1}=t_4$, 
then $\Gamma_{a10}=\langle \tilde {g}_2, \tilde {g}_3, t_4 \rangle \rtimes \langle \tilde {g}_1 \rangle$.
Then the center $\displaystyle \mathcal C(\Gamma_{a10})= \langle t_1 \rangle$ where $t_1=\tilde {g}_1^2$.
There is an extension
\begin{align*}
1\rightarrow  \langle t_1 \rangle \rightarrow  \Gamma_{a10}\rightarrow  \Delta_{a10}\rightarrow  1,
\end{align*}
where $\Delta_{a10}=\Delta^f_{a10}\rtimes \langle \alpha   \rangle
                   =\langle p,q,r \rangle \rtimes \langle \alpha   \rangle$ is isomorphic to \footnotesize
\begin{align*}
\Bigl <\left(\begin{array}{c}
\frac{1}{2}\\
0\\
0
\end{array}\right)
\left(\begin{array}{ccc}
1& 0 & 0\\
0& -1 & 0 \\
0& 0 & 1
\end{array}\right), \left(\begin{array}{c}
0\\
\frac{1}{2}\\
0
\end{array}\right) \left(\begin{array}{ccc}
1& 0 & 0\\
0& 1 & 0 \\
0& 0 & -1
\end{array}\right),
\left(\begin{array}{c}
0\\
0\\
\frac{1}{2}
\end{array}\right) \left(\begin{array}{ccc}
1& 0 & 0\\
0& 1 & 0 \\
0& 0 & 1
\end{array}\right)
\Bigr > \rtimes \\
\Big < \left(\begin{array}{c}
0\\
0\\
0
\end{array}\right) \left(\begin{array}{ccc}
-1& 0 & 0\\
0& 1 & 0 \\
0& 0 & 1
\end{array}\right)
\Bigr >.
\end{align*}
\normalsize
Then, since $\tilde {h}_1\tilde {h}_2\tilde {h}_1^{-1}=\tilde {h}_2^{-1}$, $\tilde {h}_1 \tilde {h}_3\tilde {h}_1^{-1}=\tilde {h}_3$, and
$\tilde {h}_1 t_4\tilde {h}_1^{-1}=t_4^{-1}$, 
then $\Gamma_{a14}=\langle \tilde {h}_2, \tilde {h}_3, t_4 \rangle \rtimes \langle \tilde {h}_1 \rangle$.
Then the center $\displaystyle \mathcal C(\Gamma_{a14})= \langle t_1 \rangle$ where $t_1=\tilde {h}_1^2$.
There is an extension
\begin{align*}
1\rightarrow  \langle t_1 \rangle \rightarrow  \Gamma_{a14}\rightarrow  \Delta_{a14}\rightarrow  1,
\end{align*}
where $\Delta_{a14}=\Delta^f_{a14}\rtimes \langle \beta   \rangle=\langle s,t,u \rangle \rtimes \langle \beta  \rangle$ is isomorphic to \footnotesize
\begin{align*}
\Bigl <\left(\begin{array}{c}
\frac{1}{2}\\
0\\
0
\end{array}\right)
\left(\begin{array}{ccc}
1& 0 & 0\\
0& -1 & 0 \\
0& 0 & 1
\end{array}\right), \left(\begin{array}{c}
0\\
\frac{1}{2}\\
0
\end{array}\right) \left(\begin{array}{ccc}
1& 0 & 0\\
0& 1 & 0 \\
0& 0 & -1
\end{array}\right),
\left(\begin{array}{c}
0\\
0\\
\frac{1}{2}
\end{array}\right) \left(\begin{array}{ccc}
1& 0 & 0\\
0& 1 & 0 \\
0& 0 & 1
\end{array}\right)
\Bigr > \rtimes \\
\Big < \left(\begin{array}{c}
0\\
0\\
0
\end{array}\right) \left(\begin{array}{ccc}
-1& 0 & 0\\
0& 1 & 0 \\
0& 0 & -1
\end{array}\right)
\Bigr > .
\end{align*}
\normalsize
Suppose that $\Gamma_{a10}$ is isomorphic to $\Gamma_{a14}$
 with isomorphism $\varphi:\Gamma_{a10}\rightarrow  \Gamma_{a14}$.
Then it induces an isomorphism $\hat\varphi:\Delta_{a10}\rightarrow 
\Delta_{a14}$ 
\begin{align*}
\begin{CD}
1 @>>>  \langle t_1 \rangle @>>> \Gamma_{a10}@>>>  \Delta_{a10}@>>>  1\\
  @.                @.                @V \varphi VV           @V \hat \varphi VV\\
1 @>>>  \langle t_1 \rangle @>>> \Gamma_{a14}@>>>  \Delta_{a14}@>>>  1.
\end{CD}
\end{align*}
{\bf Claim 1:}\\
Given an isomorphism $\hat \varphi:\Delta_{a10}\rightarrow \Delta_{a14}$, then  $\hat \varphi (\Delta_{a10}^f)=\Delta_{a14}^f$, 
that is, $\hat \varphi $ maps isomorphically.
\\

Since $\hat \varphi (\alpha )$ is a torsion element, we may write $\hat \varphi (\alpha )=h\beta $ 
for some $h\in \Delta_{a14}^f$. Since $\hat \varphi $ is an isomorphism, there exist $g\in \Delta_{a10}^f$
such that $\hat \varphi (g\alpha )=\beta $. Hence $\beta =\hat \varphi (g)\hat \varphi (\alpha )=\hat \varphi (g)h\beta $.
So $\hat \varphi (g)h=1$ or $h =\hat \varphi (g)^{-1}=\hat \varphi (g^{-1})$. Hence 
\begin{equation}\label{E1}
\hat \varphi (\alpha )=\hat \varphi (g^{-1})\beta.
\end{equation}
By using (\ref{E1}), 
\begin{align*}
\hat \varphi (\Delta_{a10})=&\hat \varphi (\Delta_{a10}^f) \rtimes \hat \varphi (\alpha )\\
                           =&\hat \varphi (\Delta_{a10}^f) \rtimes \hat \varphi (g^{-1})\beta \\
				   =&\hat \varphi (\Delta_{a10}^f) \rtimes \beta  \hspace{1cm} (\hat \varphi (g^{-1})\in \hat \varphi (\Delta_{a10}^f))
\end{align*}
while $\hat \varphi (\Delta_{a10})=\Delta_{a14}$. Therefore
\begin{equation}\label{E2}
\hat \varphi (\Delta_{a10}^f) \rtimes \beta =\Delta_{a14}^f \rtimes \beta.
\end{equation}

Let $s\in \Delta_{a14}^f$, then $s\beta =
\footnotesize
\Big ( \left(\begin{array}{c}
\frac{1}{2}\\
0\\
0
\end{array}\right) \left(\begin{array}{ccc}
-1& 0 & 0\\
0& -1 & 0 \\
0& 0 & -1
\end{array}\right)
\Bigr) $ \normalsize
which is a torsion element. By using (\ref{E2}) $\hat \varphi (g'_2)\beta =s\beta $
for some $g'_2\in \Delta_{a10}^f$. Hence 
\begin{equation}\label{E3}
\hat \varphi (g'_2)=s.
\end{equation}
Let $t\in \Delta_{a14}^f$, then $st\beta =
\footnotesize
\Big ( \left(\begin{array}{c}
\frac{1}{2}\\
-\frac{1}{2}\\
0
\end{array}\right) \left(\begin{array}{ccc}
-1& 0 & 0\\
0& -1 & 0 \\
0& 0 & 1
\end{array}\right)
\Bigr) $ \normalsize
which is a torsion element. By using (\ref{E2}) there exist $g'_3\in \Delta_{a10}^f$ such that
$\hat \varphi (g'_3)\beta =st\beta $. So, $\hat \varphi (g'_3)=st$ or $s^{-1}\hat \varphi (g'_3) =t$, and by (\ref{E3})
\begin{equation}\label{E4}
\hat \varphi ((g'_2)^{-1} g'_3)=t
\end{equation}
where $(g'_2)^{-1}g'_3\in \Delta_{a10}^f$.

As $u\beta =
\footnotesize
\Big ( \left(\begin{array}{c}
0\\
0\\
\frac{1}{2}
\end{array}\right) \left(\begin{array}{ccc}
-1& 0 & 0\\
0& 1 & 0 \\
0& 0 & -1
\end{array}\right)
\Bigr) $ \normalsize
is a torsion element. By using (\ref{E2}) there exist $g''\in \Delta_{a10}^f$ such that
$\hat \varphi (g'')\beta =u\beta $. Hence
\begin{equation}\label{E5}
\hat \varphi (g'')=u.
\end{equation}

If $g\in \Delta_{a10}^f$, we have $\hat \varphi (g)=h\in \Delta_{a14}^f$. If this is not true, then $\hat \varphi (g)=h\beta $.
But $h=\hat \varphi (g')$ for some $g'\in \Delta_{a10}^f$. So, the above implies $h^{-1}\hat \varphi (g)=\beta $ or 
$\hat \varphi (g'^{-1}g)=\beta $ where $g'^{-1}g\in \Delta_{a10}^f$.
Since $\Delta_{a10}^f$ is torsion free, it is impossible. Therefore,
since for any $g\in \Delta_{a10}^f$ $\hat \varphi (g)=h\in\Delta_{a14}^f $ then $\hat \varphi (\Delta_{a10}^f)\subset \Delta_{a14}^f$,
and as (\ref{E3}) (\ref{E4}) and (\ref{E5}) imply surjectivity, then $\hat \varphi (\Delta_{a10}^f)= \Delta_{a14}^f$.
\\
\\
{\bf Claim 2:}\\
Suppose $\Delta_{a10}$ and $\Delta_{a14}$ are isomorphic by $\hat \varphi$. Since  
$\Delta_{a10}$ and $\Delta_{a14}$ are crystallographics groups (rigid motion of $\mathbb{R}^{3}$),
there are affinely conjugate, i.e. there exist an equivariant diffeomorphism (=affine transformation)
\begin{equation}\label{E6}
(\hat \varphi , S):(\Delta_{a10}, \mathbb{R}^{3})\longrightarrow (\Delta_{a14}, \mathbb{R}^{3})
\end{equation}
such that $S(\gamma v)=\hat \varphi (\gamma )S(v)$ or equivalently 
$S\gamma =\hat \varphi (\gamma )S$, $S\gamma S^{-1}=\hat \varphi (\gamma )$.
\\

First we note that $\hat \varphi $ induces an isomorphism
\begin{align*}
\bar \varphi : \Delta_{a10}/{\Delta_{a10}^f}=\langle \bar \alpha  \rangle \longrightarrow 
\Delta_{a14}/{\Delta_{a14}^f}=\langle \bar \beta  \rangle
\end{align*}
by $\bar \varphi (\bar \gamma )=\overline {\hat {\varphi} (\gamma )}$. 
For $\bar \gamma =\bar \gamma'$ $(\gamma ,\gamma'\in \Delta_{a10})$, then $\gamma =r\gamma'$ $(r\in \Delta_{a10}^f)$.
Then $\hat \varphi (\gamma) =\hat \varphi (r) \hat \varphi (\gamma')=r'\hat \varphi (\gamma')$ for
some $r'\in \Delta_{a14}^f(\because \hat \varphi (\Delta_{a10}^f)= \Delta_{a14}^f)$. Hence 
$\overline {\hat {\varphi} (\gamma )}=\overline {\hat {\varphi} (\gamma' )}$ which show that it is well defined.

Note that, if $\gamma=\tilde g \alpha$, $\tilde g\in \Delta _{a10}^f$, then $\bar \gamma =\overline{\tilde g \alpha }=\bar \alpha $
and 
\begin{align*}
\hat \varphi (\gamma)=&\hat \varphi (\tilde g) \hat \varphi  (\alpha)\\
                   =&\hat \varphi (\tilde g) \hat \varphi  (g^{-1}) \beta    \hspace{1cm} \text {by} \hspace{0.2cm}\eqref{E1}\\
                   =&\hat \varphi (\tilde g g^{-1})\beta.
\end{align*}
Since $g,\tilde g\in \Delta _{a10}^f$ then $\hat \varphi (\tilde g g^{-1})\in \Delta _{a14}^f$ 
and $\hat \varphi (\gamma) \in \Delta _{a14}$ by assumption. Hence $\overline {\hat \varphi (\gamma)  }=\bar \beta $, 
and by definition
\begin{equation}\label{E7}
\bar  \varphi {(\bar \alpha) }=\bar \beta.
\end{equation} 

Note also, $S$ induces a diffeomorphism
\begin{align*}
\bar S: M(\Delta _{a10}^f)=\mathbb{R}^{3}/\Delta _{a10}^f \longrightarrow M(\Delta _{a14}^f)=\mathbb{R}^{3}/\Delta _{a14}^f
\end{align*}
by $\bar S(\bar v)=\overline {S(v)}$.
For this, if $\bar v=\bar v'$ then $v'=rv$ $(r\in\Delta _{a10}^f )$. $S(v')=S(rv)=\hat \varphi (r)S(v)$ by \eqref{E6}
where $\hat \varphi (r)\in \Delta _{a14}^f$ (because of Claim 1). Thus $\overline {S(v')}=\overline {S(v)}$.

Moreover $\bar S$ is an equivariant diffeomorphism
\begin{align*}
(\bar \varphi ,\bar S): (\langle \bar \alpha \rangle, M(\Delta _{a10}^f))\longrightarrow (\langle \bar \beta  \rangle,M(\Delta _{a14}^f)).
\end{align*}
For this, 
\begin{align*}
\bar S (\bar \alpha \bar v)=&\bar S (\overline{\alpha v})=\overline {S(\alpha v)}\\
                           =&\overline {\hat \varphi (\alpha ) S(v)}  \hspace {0.9cm} (\text {by} \quad \eqref{E6})\\
				   =&\overline {\hat {\varphi} (\alpha )} \quad \overline {S(v)}  \hspace {0.5cm} (\text {by \,\,($\ast$)\,\, below})\\
				   =&\bar \beta \bar S(\bar v)
\end{align*}
$(\ast )$\\
The action of $\Delta _{a10}/\Delta _{a10}^f=\langle \bar \alpha \rangle$ on $M(\Delta _{a10}^f)$ is obtained: 
$ \bar \gamma \bar v =\overline {\gamma v}$  where $\bar \gamma \in \Delta _{a10}/\Delta _{a10}^f$ and
$\bar v \in \mathbb{R}^3/\Delta _{a10}^f$. If $\bar\gamma=\bar {\gamma'}$ then
$\gamma'=g \gamma$, where $\gamma', \gamma \in \Delta _{a10}$ and $g\in \Delta _{a10}^f$. Therefore
$\overline{\gamma' v}=\overline {g\gamma v}=\overline{g(\gamma v)}=\overline{\gamma v}$.
If $\bar v=\bar {v'}$ then $v'=g'v$ where $g'\in \Delta _{a10}^f$. Therefore $\overline {\gamma' v'}=
\overline {\gamma' g' v}=\overline {(\gamma'g' \gamma'^{-1}) \gamma'v}=\overline {g'' \gamma' v}=\overline {\gamma' v}$,
where $g''\in \Delta _{a10}^f$ normal subgroup.
So, when $\bar\gamma=\bar {\gamma'}$, $\bar v=\bar {v'}$ then $\bar \gamma' \bar {v'}=
\overline {\gamma' v'}=\overline {\gamma' v}=\overline {\gamma v}=\bar \gamma \bar v$.

We have this diagram:\\
\[
\begin{CD}
\Delta _{a10}^f @>\varphi>>\Delta _{a14}^f\\
@VVV @VVV\\
(\Delta _{a10},\mathbb{R}^{3}) @>(\varphi,S)>> (\Delta _{a14},\mathbb{R}^{3})\\
@VVV @VVV\\
(\Delta _{a10}/\Delta _{a10}^f,\mathbb{R}^{3}/\Delta _{a10}^f) @>(\bar \varphi,\bar S)>> (\Delta _{a14}/\Delta _{a14}^f,\mathbb{R}^{3}/\Delta _{a14}^f)\\
\end{CD}
\]
where $\bar S (\bar \alpha \bar v)=\bar \beta \bar S (\bar v)$ as $\bar S$ is an equivariant diffeomorphism.
\\

We note that, the fixed point set of $\bar \alpha$, $Fix(\bar \alpha)=\{ \bar v\in M(\Delta _{a10}^f)|\bar \alpha \bar v=\bar v\}$
mapped diffeomorphically onto $Fix(\bar \beta)$.
\\

Now, to get $Fix(\bar \al)$ and $Fix(\bar \be)$, the argument is as follows;
By considering the exact sequence 
\[
\ZZ^3\lra \Delta_{a10}^f \lra \ZZ_2 \times \ZZ_2,
\]
the action of $\ZZ_2 \times \ZZ_2=<\tau , \mu >$ on $T^3$ by
$\tau \footnotesize \overline {\left(\begin{array}{ccc}
x_1\\
x_2\\
x_3
\end{array}\right)}=\overline {\left(\begin{array}{ccc}
\frac{1}{2}+x_1\\
-x_2\\
x_3
\end{array}\right)}={\left(\begin{array}{ccc}
-z_1\\
\overline {z_2}\\
z_3
\end{array}\right)},$
\normalsize
$\mu \footnotesize \overline {\left(\begin{array}{ccc}
x_1\\
x_2\\
x_3
\end{array}\right)}=\overline {\left(\begin{array}{ccc}
x_1\\
\frac{1}{2}+x_2\\
-x_3
\end{array}\right)}={\left(\begin{array}{ccc}
z_1\\
-z_2\\
\overline {z_3}
\end{array}\right)}$,
\normalsize 
and the action of $< \bar\al>$ on $\RR^3/\Delta_{a10}^f$ by $\bar \al 
\footnotesize \overline {\left(\begin{array}{ccc}
x_1\\
x_2\\
x_3
\end{array}\right)}=\footnotesize \overline {\left(\begin{array}{ccc}
-x_1\\
x_2\\
x_3
\end{array}\right)}$, \normalsize we have the commutative diagram
\[
\begin{CD}
\ZZ_2 \times \ZZ_2 @. \ZZ_2 \times \ZZ_2\\
@VVV @VVV\\
T^3=\RR^3/\ZZ^3 @>\tilde \al>> \RR^3/\ZZ^3 \\
@VpVV @VpVV\\
\RR^3/\Delta_{a10}^f @>\bar \al>> \RR^3/\Delta_{a10}^f\\
\end{CD}
\]
where $\tilde \al$ is a lift of $\bar \al$ and 
$p$ is defined by $p(\footnotesize \left(\begin{array}{ccc}
{z_1}\\
z_2\\
z_3
\end{array}\right))=\overline {\left(\begin{array}{ccc}
x_1\\
x_2\\
x_3
\end{array}\right)}$ \nz 
such that $p\circ \tilde \al=\bar \al \circ p$.

Note that $\tilde \al \footnotesize \left(\begin{array}{ccc}
{z_1}\\
z_2\\
z_3
\end{array}\right)=\footnotesize \left(\begin{array}{ccc}
{\bar z_1}\\
z_2\\
z_3
\end{array}\right)\in T^3$, \nz since $\tilde \al \footnotesize \overline {\left(\begin{array}{ccc}
{x_1}\\
x_2\\
x_3
\end{array}\right)}=\footnotesize \overline {\left(\begin{array}{ccc}
{-x_1}\\
x_2\\
x_3
\end{array}\right)}= \left(\begin{array}{ccc}
{\bar z_1}\\
z_2\\
z_3
\end{array}\right)\in T^3$. \nz

Let $x=\fz \left(\begin{array}{ccc}
{x_1}\\
x_2\\
x_3
\end{array}\right)$ \nz and  $z=\fz \left(\begin{array}{ccc}
{z_1}\\
z_2\\
z_3
\end{array}\right)$ \nz. 
If $z\in Fix(\tilde \al)$, $\ti \al z=z$, then $p(\ti \al z)=p(z)$. This  implies \fz
\[
\left(\begin{array}{ccc}
{\bar z_1}\\
z_2\\
z_3
\end{array}\right)=\ti \al \left(\begin{array}{ccc}
{z_1}\\
z_2\\
z_3
\end{array}\right)=\tau^m \mu^n
\left(\begin{array}{ccc}
{z_1}\\
z_2\\
z_3
\end{array}\right) \,\,\,\, m,n \in \{0,1\}.
\]
\\ \nz
So, the possibilities of $Fix(\ti \al)$ is obtained from
\begin{equation}\label{FixAl}
\fz \left(\begin{array}{ccc}
{\bar z_1}\\
z_2\\
z_3
\end{array}\right)= \nz
\begin{cases}
z\\
\tau z\\
\mu z\\
\tau \mu z.
\end{cases}
\end{equation} 

Before determining $Fix(\ti \al)$, note that $\ti \al(\tau z)=\tau (\ti \al z)$ and 
$\ti \al(\mu z)=\mu (\ti \al z)$, since $-\bar z_i=\overline {-z_i}$. Hence if $z\in Fix(\ti \al)$, $\tau z\in Fix(\ti \al)$ and
$\mu z\in Fix(\ti \al)$.

Now we calculate $Fix(\ti \al)$ for any case in \eqref{FixAl}.\\
If \fz $\left(\begin{array}{ccc}
{\bar z_1}\\
z_2\\
z_3
\end{array}\right)=\left(\begin{array}{ccc}
{z_1}\\
z_2\\
z_3
\end{array}\right)$,
$\left(\begin{array}{ccc}
{z_1}\\
z_2\\
z_3
\end{array}\right)=\left(\begin{array}{ccc}
{\pm 1}\\
z_2\\
z_3
\end{array}\right)$. \nz Since $\tau \fz \left(\begin{array}{ccc}
{1}\\
z_2\\
z_3
\end{array}\right)=\left(\begin{array}{ccc}
{-1}\\
\bar z_2\\
z_3
\end{array}\right)$, \nz so $Fix(\ti \al)=\fz \Bigl \{ \left(\begin{array}{ccc}
{1}\\
z_2\\
z_3
\end{array}\right)\Bigr \}$. \\ \nz
If
\fz $\left(\begin{array}{ccc}
{\bar z_1}\\
z_2\\
z_3
\end{array}\right)=\nz \tau \fz \left(\begin{array}{ccc}
{z_1}\\
z_2\\
z_3
\end{array}\right)= \left(\begin{array}{ccc}
{-z_1}\\
\bar z_2\\
z_3
\end{array}\right)$,
$\left(\begin{array}{ccc}
{z_1}\\
z_2\\
z_3
\end{array}\right)=\left(\begin{array}{ccc}
{\pm i}\\
\pm 1\\
z_3
\end{array}\right)$. \nz 
Since $\tau \fz \left(\begin{array}{ccc}
{i}\\
\pm 1\\
z_3
\end{array}\right)=\left(\begin{array}{ccc}
{-i}\\
\pm 1\\
z_3
\end{array}\right)$, \nz 
and 
$\mu \fz \left(\begin{array}{ccc}
{i}\\
1\\
z_3
\end{array}\right)=\left(\begin{array}{ccc}
{i}\\
-1\\
\bar z_3
\end{array}\right)$
\nz so $Fix(\ti \al)=\fz \Bigl \{ \left(\begin{array}{ccc}
{i}\\
1\\
z_3
\end{array}\right)\Bigr \}$. \\ \nz
If
\fz $\left(\begin{array}{ccc}
{\bar z_1}\\
z_2\\
z_3
\end{array}\right)=\nz \mu \fz \left(\begin{array}{ccc}
{z_1}\\
z_2\\
z_3
\end{array}\right)= \left(\begin{array}{ccc}
{z_1}\\
-z_2\\
\bar z_3
\end{array}\right)$, there is no $z_2=-z_2$. \nz So $Fix(\ti \al)=\emptyset $.\\
If
\fz $\left(\begin{array}{ccc}
{\bar z_1}\\
z_2\\
z_3
\end{array}\right)=\nz \tau \mu \fz \left(\begin{array}{ccc}
{z_1}\\
z_2\\
z_3
\end{array}\right)= \left(\begin{array}{ccc}
{-z_1}\\
-\bar z_2\\
\bar z_3
\end{array}\right)$,
$\left(\begin{array}{ccc}
{z_1}\\
z_2\\
z_3
\end{array}\right)=\left(\begin{array}{ccc}
{\pm i}\\
\pm i\\
\pm 1
\end{array}\right)$. \nz 
Since 
\begin{align*}
\mu \fz \left(\begin{array}{ccc}
{i}\\
i\\
1
\end{array}\right)&=\left(\begin{array}{ccc}
{i}\\
-i\\
1
\end{array}\right),  
&\mu \fz \left(\begin{array}{ccc}
{i}\\
i\\
-1
\end{array}\right)=\left(\begin{array}{ccc}
{i}\\
-i\\
-1
\end{array}\right)\\
\mu \fz \left(\begin{array}{ccc}
{-i}\\
i\\
1
\end{array}\right)&=\left(\begin{array}{ccc}
{-i}\\
-i\\
1
\end{array}\right),  
&\mu \fz \left(\begin{array}{ccc}
{-i}\\
i\\
-1
\end{array}\right)=\left(\begin{array}{ccc}
{-i}\\
-i\\
-1
\end{array}\right)\\
\tau \mu \fz \left(\begin{array}{ccc}
{i}\\
-i\\
1
\end{array}\right)&=\left(\begin{array}{ccc}
{-i}\\
-i\\
1
\end{array}\right),  
&\tau \mu \fz \left(\begin{array}{ccc}
{i}\\
-i\\
-1
\end{array}\right)=\left(\begin{array}{ccc}
{-i}\\
-i\\
-1
\end{array}\right)
\end{align*}
\nz so $Fix(\ti \al)=\fz \Bigl \{ \left(\begin{array}{ccc}
{i}\\
-i\\
1
\end{array}\right), \left(\begin{array}{ccc}
{i}\\
-i\\
-1
\end{array}\right)\Bigr \}$. 

To determine $Fix(\bar \al)$, the argument is as follows.
If $z\in Fix(\ti \al)$, 
\begin{align*}
\bar \al \bar x &=\bar \al p(z)
                =p(\ti \al z) 
                =p(z)\\
	          &=\bar x,
\end{align*}
hence $p(Fix(\ti \al))=Fix(\bar \al)$.
Now we get 
\begin{align*}
Fix(\bar \al)=
\begin{cases} 
\Bigl \{ p \fz \left(\begin{array}{ccc}
{1}\\
z_2\\
z_3
\end{array}\right)
  \Bigr \}=
\Bigl \{ \fz \overline {\left(\begin{array}{ccc}
{1}\\
x_2\\
x_3
\end{array}\right)}
  \Bigr \}=T^2\\
\Bigl \{ p \fz \left(\begin{array}{ccc}
{i}\\
1\\
z_3
\end{array}\right)
  \Bigr \}=
\Bigl \{ \fz \overline {\left(\begin{array}{ccc}
\frac{1}{4}\\
1\\
x_3
\end{array}\right)}
  \Bigr \}=S^1\\
\Bigl \{ p \fz \left(\begin{array}{ccc}
{i}\\
i\\
1
\end{array}\right),
p \fz \left(\begin{array}{ccc}
{i}\\
i\\
-1
\end{array}\right)
  \Bigr \}=
\Bigl \{ \fz \overline {\left(\begin{array}{ccc}
\frac{1}{4}\\
\frac{1}{4}\\
1
\end{array}\right)},
\overline {\left(\begin{array}{ccc}
\frac{1}{4}\\
\frac{1}{4}\\
-1
\end{array}\right)}
  \Bigr \}.
\end{cases}
\end{align*}
\\
%===========================================================================================

Similarly, we can find $Fix(\bar \be)$ as follows.
If $z\in Fix(\tilde \be)$, $\ti \be z=z$, then $p(\ti \be z)=p(z)$. This  implies \fz
\[
\left(\begin{array}{ccc}
{\bar z_1}\\
z_2\\
\bar z_3
\end{array}\right)=\ti \be \left(\begin{array}{ccc}
{z_1}\\
z_2\\
z_3
\end{array}\right)=\tau^m \mu^n
\left(\begin{array}{ccc}
{z_1}\\
z_2\\
z_3
\end{array}\right) \,\,\,\, m,n \in \{0,1\}.
\]
\\ \nz
So, the possibilities of $Fix(\ti \be)$ is obtained from
\begin{equation}\label{FixBe}
\fz \left(\begin{array}{ccc}
{\bar z_1}\\
z_2\\
\bar z_3
\end{array}\right)= \nz
\begin{cases}
z\\
\tau z\\
\mu z\\
\tau \mu z.
\end{cases}
\end{equation} 

Note that $\ti \be(\tau z)=\tau (\ti \be z)$ and 
$\ti \be(\mu z)=\mu (\ti \be z)$. Hence if $z\in Fix(\ti \be)$, $\tau z\in Fix(\ti \be)$ and
$\mu z\in Fix(\ti \be)$.

Now we calculate $Fix(\ti \be)$ for any case in \eqref{FixBe}.\\
If \fz $\left(\begin{array}{ccc}
{\bar z_1}\\
z_2\\
\bar z_3
\end{array}\right)=\left(\begin{array}{ccc}
{z_1}\\
z_2\\
z_3
\end{array}\right)$,
$\left(\begin{array}{ccc}
{z_1}\\
z_2\\
z_3
\end{array}\right)=\left(\begin{array}{ccc}
{\pm 1}\\
z_2\\
\pm 1
\end{array}\right)$. \nz Since $\tau \fz \left(\begin{array}{ccc}
{1}\\
z_2\\
1
\end{array}\right)=\left(\begin{array}{ccc}
{-1}\\
\bar z_2\\
1
\end{array}\right)$, 
$\tau \fz \left(\begin{array}{ccc}
{1}\\
z_2\\
-1
\end{array}\right)=\left(\begin{array}{ccc}
{-1}\\
\bar z_2\\
-1
\end{array}\right)$ \nz so $Fix(\ti \be)=\fz \Bigl \{ \left(\begin{array}{ccc}
{1}\\
z_2\\
1
\end{array}\right),
\left(\begin{array}{ccc}
{1}\\
z_2\\
-1
\end{array}\right)\Bigr \}$. \\ \nz
If
\fz $\left(\begin{array}{ccc}
{\bar z_1}\\
z_2\\
\bar z_3
\end{array}\right)=\nz \tau \fz \left(\begin{array}{ccc}
{z_1}\\
z_2\\
z_3
\end{array}\right)= \left(\begin{array}{ccc}
{-z_1}\\
\bar z_2\\
z_3
\end{array}\right)$,
$\left(\begin{array}{ccc}
{z_1}\\
z_2\\
z_3
\end{array}\right)=\left(\begin{array}{ccc}
{\pm i}\\
\pm 1\\
\pm 1
\end{array}\right)$. \nz 
Since 
\begin{align*}
\mu \fz \left(\begin{array}{ccc}
{i}\\
1\\
1
\end{array}\right)&=\left(\begin{array}{ccc}
{i}\\
-1\\
1
\end{array}\right),  
&\mu \fz \left(\begin{array}{ccc}
{i}\\
1\\
-1
\end{array}\right)=\left(\begin{array}{ccc}
{i}\\
-1\\
-1
\end{array}\right)\\
\mu \fz \left(\begin{array}{ccc}
{-i}\\
1\\
1
\end{array}\right)&=\left(\begin{array}{ccc}
{-i}\\
-1\\
1
\end{array}\right),  
&\mu \fz \left(\begin{array}{ccc}
{-i}\\
1\\
-1
\end{array}\right)=\left(\begin{array}{ccc}
{-i}\\
-1\\
-1
\end{array}\right)\\
\tau \mu \fz \left(\begin{array}{ccc}
{i}\\
-1\\
1
\end{array}\right)&=\left(\begin{array}{ccc}
{-i}\\
1\\
1
\end{array}\right),  
&\tau \mu \fz \left(\begin{array}{ccc}
{i}\\
-1\\
-1
\end{array}\right)=\left(\begin{array}{ccc}
{-i}\\
1\\
-1
\end{array}\right)
\end{align*}
\nz so $Fix(\ti \be)=\fz \Bigl \{ \left(\begin{array}{ccc}
{i}\\
1\\
1
\end{array}\right), \left(\begin{array}{ccc}
{i}\\
1\\
-1
\end{array}\right)\Bigr \}$.\\
If
\fz $\left(\begin{array}{ccc}
{\bar z_1}\\
z_2\\
\bar z_3
\end{array}\right)=\nz \mu \fz \left(\begin{array}{ccc}
{z_1}\\
z_2\\
z_3
\end{array}\right)= \left(\begin{array}{ccc}
{z_1}\\
-z_2\\
\bar z_3
\end{array}\right)$, there is no $z_2=-z_2$. \nz So $Fix(\ti \be)=\emptyset $.\\
If
\fz $\left(\begin{array}{ccc}
{\bar z_1}\\
z_2\\
\bar z_3
\end{array}\right)=\nz \tau \mu \fz \left(\begin{array}{ccc}
{z_1}\\
z_2\\
z_3
\end{array}\right)= \left(\begin{array}{ccc}
{-z_1}\\
-\bar z_2\\
\bar z_3
\end{array}\right)$,
$\left(\begin{array}{ccc}
{z_1}\\
z_2\\
z_3
\end{array}\right)=\left(\begin{array}{ccc}
{\pm i}\\
\pm i\\
z_3
\end{array}\right)$. \nz 
Since 
\begin{align*}
\mu \fz \left(\begin{array}{ccc}
{i}\\
i\\
z_3
\end{array}\right)&=\left(\begin{array}{ccc}
{i}\\
-i\\
z_3
\end{array}\right),  
&\mu \fz \left(\begin{array}{ccc}
{-i}\\
i\\
z_3
\end{array}\right)=\left(\begin{array}{ccc}
{-i}\\
-i\\
z_3
\end{array}\right)\\
\tau \fz \left(\begin{array}{ccc}
{i}\\
i\\
z_3
\end{array}\right)&=\left(\begin{array}{ccc}
{-i}\\
-i\\
z_3
\end{array}\right)
\end{align*}
\nz so $Fix(\ti \be)=\fz \Bigl \{ \left(\begin{array}{ccc}
{i}\\
i\\
z_3
\end{array}\right)\Bigr \}$. 

To determine $Fix(\bar \be)$, the argument is as follows.
If $z\in Fix(\ti \be)$, 
\begin{align*}
\bar \be \bar x &=\bar \be p(z)
                =p(\ti \be z) 
                =p(z)\\
	          &=\bar x,
\end{align*}
hence $p(Fix(\ti \be))=Fix(\bar \be)$.
Now we get 
\begin{align*}
Fix(\bar \be)=
\begin{cases} 
\Bigl \{ p \fz \left(\begin{array}{ccc}
{1}\\
z_2\\
-1
\end{array}\right),
p \fz \left(\begin{array}{ccc}
{1}\\
z_2\\
1
\end{array}\right)
  \Bigr \}=
\Bigl \{ \fz \overline {\left(\begin{array}{ccc}
{1}\\
x_2\\
-1
\end{array}\right)},
\overline {\left(\begin{array}{ccc}
{1}\\
x_2\\
1
\end{array}\right)}
  \Bigr \}=2S^1\\
\Bigl \{ p \fz \left(\begin{array}{ccc}
{i}\\
1\\
1
\end{array}\right),
p \fz \left(\begin{array}{ccc}
i\\
1\\
-1
\end{array}\right)
  \Bigr \}=
\Bigl \{ \fz \overline {\left(\begin{array}{ccc}
\frac{1}{4}\\
1\\
1
\end{array}\right)},
\overline {\left(\begin{array}{ccc}
\frac{1}{4}\\
1\\
-1
\end{array}\right)}
  \Bigr \}\\
\Bigl \{ p \fz \left(\begin{array}{ccc}
i\\
i\\
z_3
\end{array}\right)  \Bigr \}=
\Bigl \{ \fz \overline {\left(\begin{array}{ccc}
\frac{1}{4}\\
\frac{1}{4}\\
x_3
\end{array}\right)}\Bigr \}=S^1.
\end{cases}
\end{align*}
This yields a contradiction.\\

We observed that the sufficient condition of conjecture is true for $n=4$.

\end{proof}

%===========================================

\end{document}